\documentclass{article}

\usepackage{bbm}

\usepackage[all]{xy}\usepackage[latin1]{inputenc}  %accents
\usepackage[dvips]{graphics,graphicx}
\usepackage{amsfonts,amssymb,amsmath,color,mathrsfs, amstext,bbm}
\usepackage{amsbsy, amsopn, amscd, amsxtra, amsthm,authblk, enumerate,mathbbol}
\usepackage{upref}
\usepackage{geometry}
\usepackage{bm}
\usepackage{dsfont}
\usepackage{todonotes}
\usepackage{mathtools}
\usepackage{appendix}
\usepackage{float}

\usepackage[all]{xy}\usepackage[latin1]{inputenc}        %accents
\usepackage[dvips]{graphics,graphicx}
\usepackage{amsfonts,amssymb,amsmath,xcolor,mathrsfs, amstext}
\usepackage{amsbsy, amsopn, amscd, amsxtra, amsthm,authblk, enumerate}
\usepackage{upref}
\usepackage{geometry}
\usepackage{float}

\usepackage[colorlinks,
            linkcolor=blue,
            anchorcolor=green,
            citecolor=blue
            ]{hyperref}

\numberwithin{equation}{section}

\def\e{\varepsilon}
\def\epsilon{\varepsilon}

\newcommand{\ol}{\overline}
\newcommand{\wt}{\widetilde}
\newcommand{\wh}{\widehat}
\newcommand{\mr}{\mathring}
\def\alb#1\ale{\begin{align*}#1\end{align*}}
\newcommand{\eqb}{\begin{equation}}
\newcommand{\eqe}{\end{equation}}

\DeclareMathOperator{\supp}{supp}

\newcommand{\tcr}[1]{\textcolor{red}{#1}}

\newcommand{\bbC}{\mathbb{C}}
\newcommand{\bbD}{\mathbb{D}}
\newcommand{\bbE}{\mathbb{E}}
\newcommand{\bbH}{\mathbb{H}}

\newcommand{\bbN}{\mathbb{N}}

\newcommand{\bbR}{\mathbb{R}}
\newcommand{\bbP}{\mathbb{P}}

\newcommand{\bbZ}{\mathbb{Z}}

\newcommand{\cA}{\mathcal{A}}

\newcommand{\cE}{\mathcal{E}}
\newcommand{\cF}{\mathcal{F}}
\newcommand{\cP}{\mathcal{P}}

\newcommand{\cL}{\mathcal{L}}
\newcommand{\cZ}{\mathcal{Z}}
\newcommand{\cT}{\mathcal{T}}
\newcommand{\cJ}{\mathcal{J}}
\newcommand{\cN}{\mathcal{N}}

\newcommand{\cH}{\mathcal{H}}

\newcommand{\cV}{\mathcal{V}}

\newcommand{\cR}{\mathcal{R}}
\newcommand{\cX}{\mathcal{X}}
\newcommand{\cO}{\mathcal{O}}
\newcommand{\cY}{\mathcal{Y}}
\newcommand{\fc}{\mathfrak{c}}
\newcommand{\fR}{\mathfrak{R}}
\newcommand{\fD}{\mathfrak{D}}
\newcommand{\fU}{\mathfrak{U}}

\newcommand{\rmC}{\mathrm{C}}
\newcommand{\rmD}{\mathrm{D}}
\newcommand{\rmL}{\mathrm{L}}
\newcommand{\rmR}{\mathrm{R}}
\newcommand{\rmG}{\mathrm{G}}
\newcommand{\rmT}{\mathrm{T}}
\newcommand{\rmE}{\mathrm{E}}
\newcommand{\rmM}{\mathrm{M}}
\newcommand{\rmQ}{\mathrm{Q}}
\newcommand{\frT}{\mathfrak{T}}
\newcommand{\frt}{\mathfrak{t}}
\newcommand{\fre}{\mathfrak{e}}
\newcommand{\frv}{\mathfrak{v}}
\newcommand{\frM}{\mathfrak{M}}

\newcommand{\ceil}[1]{\left\lceil #1 \right\rceil}
\newcommand{\diam}{\mathrm{diam}}
\newcommand{\area}{\mathrm{area}}
\newcommand{\disk}{\mathrm{disk}}
\newcommand{\sph}{\mathrm{sp}}
\newcommand{\simple}{\mathrm{simple}}
\newcommand{\cell}{\mathrm{cell}}
\newcommand{\lge}{\mathrm{large}}
\newcommand{\inv}{\mathrm{inv}}
\newcommand{\frtr}{\mathfrak{tr}}

\newcommand{\MOT}{\mathrm{MOT}}
\newcommand{\SLE}{\mathrm{SLE}}

\newcommand{\IG}{\mathrm{IG}}

\newtheorem{theorem}{Theorem}[section]
\newtheorem{lemma}[theorem]{Lemma}
\newtheorem{proposition}[theorem]{Proposition}
\newtheorem*{proposition*}{Proposition}
\newtheorem{corollary}[theorem]{Corollary}
\newtheorem*{corollary*}{Corollary}

\theoremstyle{definition}
\newtheorem{definition}[theorem]{Definition}
\newtheorem*{definitions*}{Definitions}

\newtheorem*{example*}{\bf Example}

\numberwithin{equation}{section}
\title{The circle packing and Riemann uniformization embedding of tree-weighted planar maps converge to Liouville quantum gravity}
\author{Nina Holden\thanks{Courant Institute, New York University, nina.holden@nyu.edu} \qquad   Pu Yu\thanks{Courant Institute, New York University, py628@nyu.edu}}
\date{\today}

\begin{document}

\maketitle

\begin{abstract}
    We prove that in the disk, sphere and whole-plane topologies, spanning tree weighted planar maps converge to a $\sqrt{2}$-Liouville quantum gravity disk, sphere and cone, respectively, under the circle packing and Riemann uniformization embeddings. As a byproduct, we also prove that the natural path on faces of the embedded tree-weighted planar maps converge to $\SLE_8$. The proof is based on our earlier work~\cite{HY26} on circle packing and Riemann uniformization embedding for random planar maps in ergodic scale-free environments, comparisons of circle packings in different domains in a companion paper~\cite{MatedCRT}, and the convergence of tree-weighted planar maps  to $\sqrt{2}$-LQG~\cite{Tutte} by the first author.
\end{abstract}

\section{Introduction}

Planar maps are, roughly speaking, planar graphs embedded in the sphere, defined up to continuous deformations of the sphere. Planar maps are studied in many different branches of both mathematics and physics, including combinatorics, probability theory, geometry, string theory, and conformal field theory. In particular, in probability theory and theoretical physics, random planar maps serve as models for random surfaces.

A longstanding belief is that a number of random planar maps converge in the scaling limit to continuum random surfaces called Liouville quantum gravity (LQG) surfaces~\cite{david1988conformal,distler1989conformal,knizhnik1996fractal}.
There are several natural ways to phrase the convergence of (decorated) random planar maps to (SLE decorated) LQG surfaces. For  each of those, there are certain random planar maps that have been proved to converge to LQG surfaces:
\begin{enumerate}
 \item Convergence as a mated pair of trees. As constructed in~\cite{DMS14}, given a pair of Brownian motions, one can construct a pair of continuum random trees and further a space-filling $\SLE$-decorated LQG surface. On the other hand, many planar maps are encoded by a pair of random walks, and this convergence of random walk to Brownian motion is also referred to as  \emph{peanosphere convergence}. Examples include~\cite{She16,LSW17,BHS18,KMSW19}.
   \item Convergence in the sense of metric space (i.e., in Gromov-Hausdorff topology), where we view random planar maps with metric spaces endowed with graph distance. In~\cite{LeGall13,Mie13}, it is shown that uniform  random quadrangulations converge to the Brownian map, which is identified with $\sqrt{8/3}$-LQG by~\cite{MSBrownian1,MSBrownian2,MSBrownian3}. It is further shown in~\cite{GM16avoid} the self-avoiding walk on random quadrangulations converges to $\SLE_{8/3}$.
    \item Convergence under \emph{discrete conformal embeddings}. That is, one draws the planar map on the plane using conformal embeddings (e.g., circle packing, Riemann uniformization, etc.), and counts the number of vertices in each domain $D\subseteq\bbC$. For the mated-CRT map, this is done in~\cite{GMS21tutte} for the Tutte embedding  and~\cite{bertacco2025scaling} for the Smith embedding.  In~\cite{HS19}, it is shown that under the Cardy embedding, the uniform triangulations converges to $\sqrt{8/3}$-LQG.    
\end{enumerate}
 
There are a number of natural probability measures one can put on the set of planar maps. One possibility is to consider the uniform measure, where all planar maps of a given size have the same probability. Another possibility is to consider a planar map with a statistical physics model such that the law of the planar map is reweighted by the partition function of the statistical physics model. In this paper, we will consider the latter setting where the statistical physics model is the spanning tree. We prove  that under circle packing and Riemann uniformization embedding, tree-weighted planar maps and the natural path on its faces converges to $\sqrt2$-LQG surfaces decorated by   $\SLE_8$ as constructed in the mating-of-trees theory by Duplantier, Miller and Sheffield. The proof is based on a concurrent work by the first author~\cite{Tutte} where the tree-weighted planar maps converge to $\sqrt2$-LQG under the Tutte embedding, our recent work~\cite{HY26} on convergence of   circle packing and Riemann uniformization embedding for triangulations from ergodic scale-free environments defined in~\cite{GMSinvariance}, and comparisons over different circle packings in our companion paper~\cite{MatedCRT}.

\subsection{Background and context}

\subsubsection{Planar maps and discrete conformal embeddings}

 A planar map $M$ is a proper (without edge-crossing) embedding of a connected and locally finite (multi-)graph in the two-dimensional sphere $\wh\bbC:=\bbC\cup\{\infty\}$, where we identify two planar maps if one can be obtained from the other via an orientation-preserving homeomorphism of $\wh{\bbC}$. Let $\cV M$, $\cE M$ and $\cF M$ denote the set of vertices, edges and faces, respectively, of $M$. All the planar maps we consider will be rooted, meaning that there is a distinguished oriented edge $\fre\in\cE M$. The tail vertex of $\fre$ is called the root vertex and the face to the right of $\fre$ is called the root face. We also call a finite planar map a \emph{planar map with sphere topology}. A \emph{planar map with disk topology with boundary length $\ell\in\bbN$} is a finite planar map where the root face has simple boundary of length $\ell$. The vertices on the boundary of the root face are called \emph{boundary vertices} and are denoted by $\partial\cV M$ while the other vertices are called \emph{interior vertices}. A face of $M$ is called a triangular face if it has three edges on the boundary. A \emph{planar map with whole-plane topology} is an infinite planar map where all the faces are bounded in $\bbC$. We call a planar map $M$ with whole-plane (resp.\ sphere) topology  an \emph{infinite plane triangulation} (resp.\ a \emph{sphere triangulation}), if all the faces of $M$ are triangular faces. We call a planar map $M$ with disk topology a \emph{disk triangulation}, if it has at least three boundary vertices, and all the faces except for the root face are triangular faces.

A planar map $M$ can be embedded in $\wh{\bbC}$ (as described in the previous paragraph) in a number of different ways. In this paper, we focus on two natural  embeddings: the \emph{circle packing} and the \emph{Riemann uniformization embedding}. These two embeddings are canonical examples of \emph{discrete conformal embeddings}.  {This means in particular that} embeddings of the same planar map in different domains induce (or is believed to induce) a map that approximates the conformal map between the two domains. For circle packing, this was conjectured by Thurston and solved in~\cite{RodinSullivan87,kS90,kS96,HeSchramm}. For Riemann uniformization embedding, this follows from the definition.

Let $M$ be a simple planar map with vertex set $\cV M$. We say that a collection $\cP = (D_v:v\in\cV M)$ of closed disks with disjoint interiors is a \textit{circle packing} for $M$, if two disks $D_v$ and $D_{v'}$ are tangent if and only if the vertices $v$ and $v'$ are connected by an edge. This gives an embedding of the planar map, with the vertices embedded at the centers of the disks and the edges are straight line segments connecting the centers of tangent disks. By the celebrated Koebe-Andreev-Thurston circle packing theorem, if $M$ is a disk (resp.\ sphere) triangulation and is simple, then up to M\"{o}bius transforms, there exists a unique circle packing of $M$ in the unit disk $\bbD$ (resp.\ Riemann sphere $\wh\bbC$). In the case where $M$ is a disk triangulation, we say $\cP$ is a \emph{circle packing of $M$ in $\bbD$}, if all the disks in $\cP$ are in $\ol\bbD$, and the disks for boundary vertices are internally tangent to $\partial\bbD$. We further say $\cP$ is \textbf{centered at a vertex} $v$, if the disk $D_v$ for the vertex $v$ is centered at origin.  For infinite plane triangulations, the existence of a circle packing is a straightforward consequence of the tightness from the Ring Lemma~\cite{RodinSullivan87}, and  the uniqueness up to M\"{o}bius transforms was proved by Schramm~\cite{Schramm91}. In this case, there is a dichotomy: an infinite plane triangulation $T$ is \emph{circle packing parabolic} (resp.\ \emph{circle packing hyperbolic}), if there exists a circle packing $\cP$ for $T$, such that the union of all the faces and edges (a.k.a.\ the \emph{carrier} of $\cP$) is equal to the whole plane (resp.\ unit disk). In~\cite{HeSchramm-hy-pa}, it was further shown that, if the simple random walk on $T$ is recurrent, then $T$ is circle packing parabolic, while if $T$ has uniformly bounded degree and the simple random walk on $T$ is transient, then $T$ is circle packing hyperbolic. See also~\cite{circlepackingbook,Nachmiaslec} for a comprehensive introduction on circle packing and related topics.

A second natural way to embed random planar maps is through the Riemann uniformization theorem. Roughly speaking, for a planar map $M$, one associate each face of degree $p$ with a regular $p$-gon of unit side length and glue them isometrically along their edges to get a Riemann surface $\frM(M)$. To be more precise,
let $M$ be a planar map  with disk, {whole-}plane or sphere topology and has at least 3 vertices. For  each face $f$ of $M$ of degree 1, i.e., $f$ is enclosed by a single edge of $M$, we remove this edge and the face from $M$. For each face $f$ of $M$ of degree 2, i.e., $f$ is formed by a double edge of $M$, we collapse this double edge into a single edge. Now we assume that each face of $M$ has degree at least 3.  We associate to $M$ a metric space $(\frM(M), d_M)$, by regarding each face of $M$ of degree $p$ as regular $p$-gon of side length 1 and defining the distance between two points as the length of the shortest path joining them. Equivalently, $\frM(M)$ is obtained by gluing regular $p$-gons of unit side length according to their adjacency pattern in $M$. The metric space $\frM(M)$ can be equipped with a compatible Riemann surface structure by defining coordinate charts as follows.    For a face $f$, we use the regular polygon associated with $f$, viewed as a subset of the plane,  as the chart, and use the identity map as the projection.  For each edge $e$ of $M$, we define an open neighborhood of the interior of $e$ in $\frM(M)$ by adding to $e$ the two triangles formed by the endpoints of $e$ and the centers of the two faces adjacent to $e$. To define a coordinate chart on this neighborhood, we simply place the two triangles next to each other in the plane. The reason to take only a {subset of the faces} %triangle 
 and not the entire face{s} is that this chart is well defined even if both sides of $e$ are incident to the same face. For each vertex $v$ of $M$, we define an open neighborhood of $v$ in $\frM(M)$ similarly by
intersecting the corners of the faces adjacent to $v$ with open discs of radius 1/2 centered at $v$. We define a chart on this neighborhood by first laying the corners
out sequentially around the origin (with possible overlapping), and then applying the function $z\mapsto z^{2\pi/\theta(v)}$,  suitably interpreted to get an injective map into the plane. The radius is chosen so that this definition remains valid when multiple corners of the same face are located at the same vertex. The case where $M$ is a planar map with disk topology can be treated similarly, where the boundary vertices can be treated similarly using the function $z\mapsto z^{\pi/\theta(v)}$ and $\frM(M)$ is homeomorphic to the closed unit disk.

\subsubsection{Tree-weighted planar maps and the Mullin bijection}

We briefly introduce the spanning tree weighted planar maps; see Section~\ref{subsec:Mullin} for more details. Let $n,\ell\in\bbN$. Let $(M^n,\fre^n)$ be a  planar map with sphere topology.  A collection of edges $\frT^n\subseteq \cE M^n$ is called a \emph{spanning tree} if   
\begin{enumerate}
    \item for all vertices $v\in\cV M^n$ there is an edge in $\frT^n$ with $v$ as end-point;\label{cond:def-spanning-tree-1}
    \item the edges of $\frT^n$ do not form any cycle, and\label{cond:def-spanning-tree-2}
    \item $\frT^n$ is connected.\label{cond:def-spanning-tree-3}
\end{enumerate}
Likewise, for a rooted  planar map $(M^n,\fre^n)$ of disk topology with boundary length $\ell$, a collection of edges $\frT^n\subseteq \cE M^n$ is called a \emph{spanning tree with wired boundary conditions} if it satisfies Conditions~\ref{cond:def-spanning-tree-1} and~\ref{cond:def-spanning-tree-3} above, and the edges of $\frT^n$ form exactly one cycle, which is equal to the boundary of the root face. Given a tuple $(M^n,\frT^n,\fre^n)$ of disk topology and $\frT^n$ is a spanning tree with wired boundary conditions, one can also generate a triangulation $T^n$ as follows. First add vertices of the dual map of $M^n$ (except for the vertex for the root face) to $M^n$ and connect these vertices for faces of $M^n$ to the corresponding vertices of $M^n$ incident to these faces; then add the edges of the dual map of $M_n$ which are dual to edges in $\cE M^n\backslash \frT^n$ and finally erase those edges in $\cE M^n\backslash \frT^n$. This also generates a path $\frtr^n$ on  $\cF T^n$, which starts and ends at the face containing $\fre^n$, and traverses all the faces in $\cF T^n$ and does not cross any edge of the spanning tree and the dual spanning tree. The case where $(M^n,\fre^n)$ has sphere topology can be treated similarly, where we do not exclude the vertex for the root face of $M^n$ in the dual map of $M^n$.

The Mullin bijection~\cite{Mullin} is a bijection between a spanning-tree decorated map and a nearest-neighbor walk in $\bbZ^2$. Consider the set  $\cA^\disk_{\ell}$ of  nearest-neighbor walk   $(\cZ_t^{\disk,n})_{t=0,...,m} = (\cL_t^{\disk,n},\cR_t^{\disk,n})_{t=0,...,m}$  on $\bbZ^2$ such that  
  \eqb\label{eq:Mullin-disk-walk}  \cZ_0^{\disk,n} = (\ell-1,0), \ \ \cZ_m^{\disk,n} = (0,0), \ \ \cL_t^{\disk,n}\geq0,  \ \ \cR_t^{\disk,n}\geq 0\  \text{for } t=0,...,m.\eqe
Given a walk in $\cA_\ell^\disk$, one can uniquely generate a tuple $(M^n,\frT^n,\fre^n)$ where $(M^n,\fre^n)$ is a rooted planar map of disk topology with boundary length $\ell$ and $\frT^n$ is a spanning tree with wired boundary conditions. We denote this map sending walks $\cZ^{\disk,n}$ to $(M^n,\frT^n,\fre^n)$ by $\Phi_\rmM^\disk$, and let $\Phi_\rmT^\disk(\cZ^{\disk,n}) = (T^n,\frtr^n)$ be the associated triangulation and path. Similarly, consider the set $\cA^\sph_n$ of nearest-neighbor walk   $(\cZ_t^{\sph,n})_{t=0,...,m} = (\cL_t^{\sph,n},\cR_t^{\sph,n})_{t=0,...,2n}$  on $\bbZ^2$ such that  
  \eqb\label{eq:Mullin-sph-walk}  \cZ_0^{\sph,n} =   \cZ_{2n}^{\sph,n} = (0,0), \ \ \cL_t^{\sph,n}\geq0,  \ \ \cR_t^{\sph,n}\geq 0\  \text{for } t=0,...,2n.\eqe
One can uniquely generate a tuple $\Phi_\rmM^\sph( \cZ^{\sph,n}) := (M^n,\frT^n,\fre^n)$ where $(M^n,\fre^n)$ is a rooted planar map with sphere topology and $n$ edges and $\frT^n$ is a spanning tree on $M^n$, and we again set $\Phi_\rmT^\sph(\cZ^{\sph,n}) = (T^n,\frtr^n)$ be the associated triangulation and path. See Proposition~\ref{prop:Mullin-sph} and Proposition~\ref{prop:Mullin-bij-disk} for more details.

Finally, the above bijections $\Phi_\rmM$ and $\Phi_\rmT$ can be extended to the setting where $\cZ = (\cL,\cR)$ is a bi-infinite walk in $\bbZ^2$ with nearest neighbor steps, and the corresponding $(M,\fre_0,\frT)$, the \emph{uniform infinite spanning-tree decorated planar map}, is the Benjamini-Schramm~\cite{BenjaminiSchramm01} limit of uniform tuples consisting of a planar map with an oriented root edge and a distinguished spanning tree~\cite{She16,Che17}. See Section~\ref{subsubsec:Mullin-infinite} for further details.

\subsubsection{Liouville quantum gravity, Schramm Loewner evolution and mating of trees}

Liouville quantum gravity is a theory of random surfaces. A $\gamma$-LQG surface for $\gamma\in (0,2)$ is, heuristically speaking, a random surface where the metric can be written on the form $e^{\gamma h}(dx^2+dy^2)$, where $h$ is a variant of the Gaussian free field (GFF) and $dx^2+dy^2$ is the standard Euclidean metric. Following~\cite{DS11}, one can define rigorously the area measure $\mu$ as well as the boundary measure $\nu$, and also the distance function (metric) of this surface thanks to~\cite{DDDf,GMMetric}. One can represent a $\gamma$-LQG surface by a pair $(D,h)$ called an \emph{embedding}, where $D\subseteq \bbC$ and $h$ is a variant of the GFF. 
The \emph{Schramm-Loewner evolution} (SLE)~\cite{Sch00} is a family of random fractal curves $\kappa>0$, and are the scaling limit of interfaces from many  natural  statistical physics models. %The curve is a.s.\ simple for $\kappa\in(0,4]$ and space-filling for $\kappa\geq 8$.
 
One of the central results in the theory of LQG is the mating of trees~\cite{DMS14}, where given a pair of correlated Brownian motion $(L,R)$, one can generate a $\gamma$-LQG surface decorated by a  space-filling $\SLE_{16/\gamma^2}$ curve. We briefly state the result for the $\gamma=\sqrt2$ case. Consider a two dimensional Brownian motion $Z=(L,R)$ with covariance matrix chosen so that %Fix $\gamma\in(0,2)$ and  
\begin{equation}\label{eq:variance}
    \mathrm{Var}(L_t)=\mathrm{Var}(R_t) = 2|t|,  \ \ \ \mathrm{Cov}(L_t,R_t) = 0.
\end{equation}
Following~\cite[Theorem 1.9 and Theorem 1.11]{DMS14},   the Brownian motion $(L,R)$ in~\eqref{eq:variance}  (where the variance was computed in~\cite{ARS21}) a.s.\ determines a pair $(\eta,h)$, where $(\bbC,h,0,\infty)$ is an embedding of a $\gamma$-LQG surface called $\gamma$-quantum cone, and $\eta$ is a space-filling $\SLE_{16/\gamma^2}$ loop in $\wh\bbC$ from $\infty$ to $\infty$ independent from $h$ and parameterized by the $\gamma$-LQG mass induced by $h$, such that $\eta(0)=0$. Similarly, following~\cite{AG19}, one can also define a variant $Z^\disk$ of the Brownian motion as in~\eqref{eq:variance} which starts from $(1,0)$ and ends at $(0,0)$ such that it is conditioned to stay in the first quadrant, and use it to generate a pair $(h^\disk,\eta^\disk)$, where $(\bbD,h^\disk)$ is an embedding of a $\gamma$-LQG disk and $\eta^\disk$ is a space-filling $\SLE_{16/\gamma^2}$ loop independent from $h^\disk$. By~\cite{MS-mot-finite}, the same also holds where the Brownian motion is conditioned to start and end at $(0,0)$ and stay in the first quadrant, and this Brownian excursion $Z^\sph$ encodes a pair $(h^\sph,\eta^\sph)$ where $(\bbC,h^\sph)$ is an embedding of  a quantum sphere and $\eta^\sph$ is a space-filling  $\SLE_{16/\gamma^2}$ loop from $\infty$ to $\infty$.

\subsection{Main results}
The main result of this paper is, roughly speaking, that in various settings, we can couple a nearest-neighbor walk  $\cZ$ on $\bbZ^2$ with (some variant of) a Brownian motion, such that the circle packing or Riemann uniformization embedding of the planar map encoded by $\cZ$ via Mullin bijection is close to the SLE$_8$ decorated $\sqrt{2}$-LQG surface generated by $Z$ via mating of trees. Note that circle packing is defined for simple planar maps whereas the triangulations $\Phi_\rmT(\cZ)$ may have multiple edges. To resolve this, one has to turn   a triangulation $T$ with multiple edges into a triangulation that is simple. There are multiple  ways to do so, and   we are going to use the following natural operator $\Phi_{\rm simple}$ that collapses multiple edges of $T$ as below.   Let $T$ be a disk triangulation or an infinite plane triangulation. For each pair of vertices $(v_1,v_2)$ of $T$ such that $v_1,v_2$ are connected by edges $e_1,...,e_k$, we remove the vertices and edges enclosed by $e_1,...,e_k$, and further remove the edges $e_2,...,e_k$ from $T$. We let $\Phi_{\rm simple}(T)$ be the remaining map. Note that if $T$ is a disk triangulation with at least 3 edges on the boundary, then $\Phi_{\rm simple}(T)$  remains to be a disk triangulation, while if $T$ is an infinite plane triangulation, then $\Phi_{\rm simple}(T)$ is either an infinite plane triangulation or empty. If $T$ is a sphere triangulation with $n$ faces, suppose for each pair of vertices $(v_1,v_2)$ of $T$ such that $v_1,v_2$ are connected by edges $e_1,...,e_k$, one of the connected components of $T\backslash\{e_1,...,e_k\}$ has at least $3n/4$ faces; in this case, we further remove all the vertices and edges in the other connected components as well as $e_2,...,e_k$ and let $\Phi_{\rm simple}$ be the remaining map. If this is not the case, then we set $\Phi_{\rm simple}(T)=\emptyset$. If $T$ is a sphere triangulation, then  $\Phi_{\rm simple}(T)$ is either a sphere triangulation or empty.  Observe that for a disk, sphere or infinite plane triangulation $T$ such that $\Phi_{\rm simple}(T)\neq\emptyset$, if  $v\notin \cV\Phi_{\rm simple}(T)$, then there exists a unique pair of vertices $(v_1,v_2)\in \cV\Phi_{\rm simple}(T) $ such that the vertex $v$ is removed when we collapse multiple edges between $v_1$ and $v_2$ as above. See Figure~\ref{fig:PHIsimple} for an illustration.

\begin{figure}
    \centering
    \includegraphics[scale=0.8]{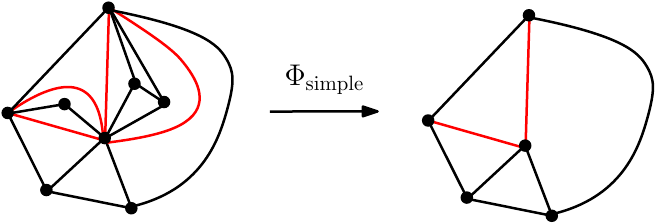}
    \caption{An illustration of the operator $\Phi_{\rm simple}$, where we collapse all the  multi-edges into single edges and remove all the vertices and edges enclosed by the multi-edge.}
    \label{fig:PHIsimple}
\end{figure}

\begin{definition}\label{def:pack-non-simple}
   Let $T$ is a disk, sphere or infinite plane triangulation with  multiple edges. If $\Phi_{\rm simple}(T)$ is empty, then we set the circle packing for $T$ to be empty. If $\Phi_{\rm simple}(T)$ is not empty, we define circle packings for $T$ as below. Let $\wt\cP = \{\wt D_v:v\in \cV\Phi_{\rm simple}(T) \}$ be a circle packing for $\Phi_{\rm simple}(T)$. For a vertex $v\in\cV T$, if $v\in \cV\Phi_{\rm simple}(T) $, then we let $D_v = \wt D_v$. If $v\notin \cV\Phi_{\rm simple}(T)$, let $(v_1,v_2)\in \cV\Phi_{\rm simple}(T) $  be the unique pair of vertices such that the vertex $v$ is removed when we collapse multiple edges between $v_1$ and $v_2$. Then we let $D_v$ be the tangent point between the disks $D_{v_1}$ and $D_{v_2}$, and we view $D_v$ as a degenerate disk. We call $\cP:=\{D_v:v\in\cV T\}$ a circle packing for $T$. If $T$ is an infinite triangulation, we say $T$ is circle packing parabolic (resp.\ hyperbolic) if $\Phi_{\rm simple}(T)$ is circle packing parabolic (resp.\ hyperbolic). 
\end{definition}

We are ready to state our main results. In the rest of the paper, we will always assume $\gamma=\sqrt{2}$ and $\kappa=16/\gamma^2=8$.  For a variant of GFF $h$ defined in $D\subseteq\bbC$, we write $\mu_h$ for the $\gamma$-LQG area measure in $D$ with respect to $h$, and if $D=\bbD$, we write $\nu_h$ for  the $\gamma$-LQG length measure on $\partial \bbD$ with respect to $h$.

\begin{theorem}\label{thm:pack-disk}
  Let $n\in\bbN$ and $\ell = \ceil{\sqrt n}$. Let $(\cZ_t^{\disk,n})_{t=0,...,m}$ be a uniform  discrete-time nearest-neighbor walk on $\bbZ^2$ sampled from $\cA_\ell^\disk$ as in~\eqref{eq:Mullin-disk-walk}
  whose law is weighted by $4^{-m}$.   Consider the disk triangulation $T^{\disk,n}$ and the   path $\mathfrak{tr}^{\disk,n}$ on $\cF T^{\disk,n}$ given by $\Phi_\rmT^{\disk}(\cZ^{\disk,n}) = (T^{\disk,n},\mathfrak{tr}^{\disk,n}).$ Let $Z^\disk=(L_t^\disk,R_t^\disk)_{t\in[0,A]}$ be a Brownian motion from~\eqref{eq:variance} conditioned on  $Z_0^\disk=(2,0)$, $Z_A^\disk=(0,0)$ and $Z_t^\disk\in\bbR_+^2$ with random duration $A$, and let $(\bbD,h^\disk,\eta^\disk)$ be the clockwise space-filling $\SLE_8$ loop decorated quantum disk generated by $Z^\disk$, embedded such that $\eta^\disk(A/2)=0$.  Then there is a coupling $((\cZ^{\disk,n})_{n\in\bbN},Z^\disk)$, such that there exists a     circle packing $\cP^{\disk,n}$ for $T^{\disk,n}$ in $\bbD$ such that as $n\to\infty$, the following holds:
  \begin{enumerate}
      \item the maximal diameter of disks in $\cP^{\disk,n}$ converges to 0 in probability;
      \item $\sup_{t\in[0,A]}|{n}^{-1/2}\cZ_{\lfloor mt/A\rfloor}^{\disk,n}-\frac{1}{2}Z_t^\disk|$ converges to 0 in probability; 
      \item  we extend $\mathfrak{tr}^{\disk,n}$ as a path    $:[0,m]\to \ol\bbD$  on embedded faces via linear interpolation, then    $$\sup_{t\in [0,A]}|\eta^\disk(t)-{\mathfrak{tr}}^{\disk,n}(mt/A)| \to 0 \text{ in probability as } n\to\infty;$$
      \item let $\mu^{\disk,n}$ be the measure on $\bbD$ that assigns mass $\frac1n$ to the center of  each disk, where  duplication is counted such that $k$ duplicated points have mass $\frac{k}{n}$ in total, and $\nu^{\disk,n}$ be the measure on $\partial\bbD$ that assigns mass $n^{-1/2}$ to each tangent point between a disk in $\cP^{\disk,n}$ and $\partial\bbD$. Then the Prokhorov distance between  {$(2\mu^{\disk,n},2\nu^{\disk,n})$} and $(\mu_{h^\disk},\nu_{h^\disk})$ tends to 0 in probability as $n\to\infty$.
  \end{enumerate}
\end{theorem}

Note that the tuple $(M^n,\fre^n,\frT^n)$ encoded by the walk $\cZ^{\disk,n}$ as above under the Mullin bijection has weight $c(M^n,\fre^n,\frT^n)=16^{-\#\cV M^n-\#\cF M^n}$. The same also holds if we sample a uniform element from $\mathcal{A}_n^\sph$ and consider the Mullin bijection in the sphere topology setting.

\begin{theorem}\label{thm:pack-sph}
    Let $n\in\bbN$, and $(\cZ^{\sph,n}_t)_{t=0,...,2n}$ be an element uniformly chosen from $\mathcal{A}_n^\sph$ as in~\eqref{eq:Mullin-sph-walk}. Let $(T^{\sph,n},\frtr^{\sph,n}) = \Phi_{\rmT}^\sph(\cZ^{\sph,n})$ be the sphere triangulation and the path on faces of $T^{\sph,n}$. Let $Z^\sph=(L_t^\sph,R_t^\sph)_{t\in[0,1]}$ be a Brownian motion from~\eqref{eq:variance} conditioned on $Z_0^\sph=Z_1^\sph=(0,0)$ and $Z_t^\sph\in\bbR_+^2$ for $t\in[0,1]$, and let $(\wh\bbC,h^\sph,\eta^\sph)$ be the space-filling $\SLE_8$ curve decorated unit area quantum sphere generated by $Z$, embedded such that $\eta^\sph(0)=\infty,\eta(1/4)=1,\eta(1/2)=0$. Then there is a coupling $((\cZ^{\sph,n})_{n\in\bbN},Z^\sph)$ and a  circle packing  $\cP^{\sph,n}$ for $T^{\sph,n}$ in $\wh\bbC$ such that with probability tending to 1 as $n\to\infty$, $\Phi_{\rm simple}(T^{\sph,n})$ is not empty, and under this event, the following holds as $n\to\infty$:
    \begin{enumerate}
        \item  the maximal spherical diameter of disks in $\cP^{\sph,n}$ tends to 0 in probability;
        \item $\sup_{t\in[0,1]}|(2n)^{-1/2}\cZ_{\lfloor 2nt\rfloor}^{\sph,n}-\frac{1}{2}Z_t^\sph|$ converges to 0 in probability;
        \item we extend $\frtr^{\sph,n}$ as a path $:[0,2n]\to\wh\bbC$ on embedded faces via linear interpolation, then 
        $$\sup_{t\in[0,1]}d_{\sph}(\eta^\sph(t),\frtr^{\sph,n}(2nt))\to0 \text{ in probability as } n\to\infty;$$
        \item let $\mu^{\sph,n}$ be the measure that assigns mass $\frac{1}{n}$ to the center of each disk, with duplication counted. Then the Prokhorov distance between   {$\mu^{\sph,n}$} and $\mu_{h^\sph}$ tends to 0 in probability as $n\to\infty$.
    \end{enumerate}
\end{theorem}

\begin{theorem}\label{thm:pack-cone}
    Let $(\check{\cZ_k^n})_{k\in\bbZ}$ be the standard nearest-neighbor walks on $\bbZ^2$, and let $(\cZ^n_t)_{t\in\bbR} = (\check \cZ_{\lfloor nt \rfloor}^n)_{t\in\bbR}$. Let $\Phi_\rmT(\cZ^n) = (T^n,\frtr^n)$ be the triangulation and path on faces of $T^n$ generated through the Mullin bijection. Let $Z=(L_t,R_t)_{t\in\bbR}$ be the Brownian motion from~\eqref{eq:variance}, and $(\bbC,h,\eta)$ be the space-filling $\SLE_8$ decorated  $\gamma$-quantum cone generated by $Z$ such that $\eta(0)=0$. Then $T^n$ is a.s.\ circle packing parabolic and $\Phi_{\rm simple}(T^n)$ is a.s.\ not empty. Furthermore,  there is a coupling between $(Z,\cZ^n)_{n\in\bbN}$ and circle packings $\cP^n$ for $T^n$ in $\bbC$, such that as $n\to\infty$, the following holds:
    \begin{enumerate}
        \item For every $R>0$, the maximal diameter of disks in $\cP^n$ intersecting $B(0;R)$ tends to 0 in probability; \label{result:thm:pack-cone-1}
        \item For every $a>0$, {$\sup_{t\in [-a,a]}|n^{-1/2}\cZ^n_t-\frac{1}{2}Z_t|$} converges to 0 in probability; \label{result:thm:pack-cone-2}
        \item If we view $\frtr^n:\bbR\to\bbC$ as a path on embedded faces, then for every $a>0$,  {$$\sup_{t\in [-a,a]}|\eta(t)-\frtr^n(t)| \to 0 \text{ in probability;}$$} \label{result:thm:pack-cone-3}
        \item Let $\mu^n$ be the measure that assigns mass $\frac{1}{n}$ to the center of each disk, with duplication counted. Then for each $R>0$, the Prokhorov distance between $2\mu^n|_{B(0;R)}$ and $\mu_h|_{B(0;R)}$ tends to 0 in probability.  \label{result:thm:pack-cone-4}
    \end{enumerate}
\end{theorem}

In the setting of the Riemann uniformization, the triangulation $T$ encoded by the walk $\cZ$ can naturally embedded on the Riemann surface $\frM(M)$ associated with the planar map $M$ encoded by $\cZ$, and we will continue to view  $\frtr$ also as a path on $\bbC$ by taking a conformal map $\varphi:\frM(M)\to\wh\bbC$.

\begin{theorem}\label{thm:unif-disk}
  Let $n\in\bbN$ and $\ell = \ceil{\sqrt n}$. Let $(\cZ_t^{\disk,n})_{t=0,...,m}$ and $Z^\disk$ be as in  Theorem~\ref{thm:pack-disk}.   Consider the planar map of disk topology $M^{\disk,n}$ and the space-filling path $\mathfrak{tr}^{\disk,n}$   encoded by the walk $\cZ^{\disk,n}$ via the map $\Phi_\rmM^\disk$ and $\Phi_\rmT^\disk$, respectively. Let $\frM(M^{\disk,n})$ be the Riemann surface associated with the map $M^{\disk,n}$. Let $(\bbD,h^\disk,\eta^\disk)$ be the clockwise space-filling $\SLE_8$ loop decorated quantum disk generated by $Z$, embedded such that $\eta^\disk(A/2)=0$.  Then there is a coupling $((\cZ^{\disk,n})_{n\in\bbN},Z)$, such that there exists a  conformal map $\varphi^{\disk,n}: \frM(M^{\disk,n})\to \bbD$ such that as $n\to\infty$, the following holds:
  \begin{enumerate}
      \item the maximal diameter of faces under the map $\varphi^{\disk,n}$ converges to 0 in probability;
      \item $\sup_{t\in[0,A]}|{n}^{-1/2}\cZ_{\lfloor mt/A\rfloor}^{\disk,n}-\frac{1}{2}Z_t^\disk|$ converges to 0 in probability; 
      \item if we view $\mathfrak{tr}^{\disk,n}:[0,m]\to \ol\bbD$ as a path on embedded faces, then    $$\sup_{t\in [0,A]}|\eta^\disk(t)-{\mathfrak{tr}}^{\disk,n}(mt/A)| \to 0 \text{ in probability as } n\to\infty;$$
      \item let $\mu^{\disk,n}$ be the measure on $\bbD$ that assigns mass $\frac1n$ to the $\varphi^{\disk,n}$-image of each vertex,   and $\nu^{\disk,n}$ be any measure on $\partial\bbD$ that assigns mass $n^{-1/2}$ to the $\varphi^{\disk,n}$-image of each edge on $\partial (\frM(M^{\disk,n}))$. Then the Prokhorov distance between  {$(4\mu^{\disk,n},2\nu^{\disk,n})$} and $(\mu_{h^\disk},\nu_{h^\disk})$ tends to 0 in probability as $n\to\infty$. The same also holds if   $\mu^{\disk,n}$ is any measure on $\bbD$ that assigns mass $\frac1n$ to the $\varphi^{\disk,n}$-image of each face.
  \end{enumerate}
\end{theorem}

\begin{theorem}\label{thm:unif-sph}
    Let $n\in\bbN$, and $(\cZ^{\sph,n}_t)_{t=0,...,2n}$ and $Z^\sph$ be as in Theorem~\ref{thm:pack-sph}. Let $M^{\sph,n}$ and $\frtr^{\sph,n}$ be the planar map and space-filling path encoded by the walk $\cZ^{\sph,n}$ via the map $\Phi_\rmM^\sph$ and $\Phi_\rmT^\sph$, respectively. Then with probability tending to 1 as $n\to\infty$, $M^{\sph,n}$ has at least 3 vertices. Let $(\wh\bbC,h^\sph,\eta^\sph)$ be the space-filling $\SLE_8$ curve decorated unit area quantum sphere generated by $Z^\sph$, embedded such that $\eta^\sph(0)=\infty,\eta(1/4)=1,\eta(1/2)=0$. Then there is a coupling $((\cZ^{\sph,n})_{n\in\bbN},Z^\sph)$ and a conformal map $\varphi^{\sph,n}:\frM(M^{\sph,n})\to\wh\bbC$ such that the following holds as $n\to\infty$:
    \begin{enumerate}
        \item The maximal spherical diameter of faces under the map $\varphi^{\sph,n}$ tends to 0 in probability;
        \item $\sup_{t\in[0,1]}|(2n)^{-1/2}\cZ_{\lfloor 2nt\rfloor}^{\sph,n}-\frac{1}{2}Z_t^\sph|$ converges to 0 in probability;
        \item If we view $\frtr^{\sph,n}:[0,2n]\to\wh\bbC$ as a path on embedded faces, then 
        $$\sup_{t\in[0,1]}d_{\sph}(\eta^\sph(t),\frtr^{\sph,n}(2nt))\to0 \text{ in probability as } n\to\infty;$$
        \item Let $\mu^{\sph,n}$ be the measure that assigns mass $\frac{1}{n}$ to the $\varphi^{\sph,n}$ image of each vertex. Then the Prokhorov distance (when viewed under the spherical metric) between  $2\mu^{\sph,n}$ and $\mu_{h^\sph}$ tends to 0 in probability as $n\to\infty$. The same also holds if $\mu^{\sph,n}$ is any measure on $\wh\bbC$ that assigns measure $\frac{1}{n}$ to the $\varphi^{\sph,n}$ image of each face.
    \end{enumerate}
\end{theorem}

\begin{theorem}\label{thm:unif-cone}
 
    Let $(\check{\cZ_k^n})_{k\in\bbZ}$ be the standard nearest-neighbor walks on $\bbZ^2$, and let $(\cZ^n_t)_{t\in\bbR} = (\check \cZ_{\lfloor nt \rfloor}^n)_{t\in\bbR}$.  Let $M^n$ and $\frtr^n$ be the planar map and space-filling path encoded by the walk $\cZ^{n}$ via the map $\Phi_\rmM$ and $\Phi_\rmT$, respectively. Let $Z=(L_t,R_t)_{t\in\bbR}$ be a 2D   Brownian motion as in~\eqref{eq:variance}, and $(\bbC,h,\eta)$ be the space-filling $\SLE_8$ decorated  $\gamma$-quantum cone generated by $Z$ such that $\eta(0)=0$. Then  the Riemann surface $\frM(M^n)$ associated with $M^n$ is a.s.\ parabolic. Furthermore,  there is a coupling between $(Z,\cZ)$ and a conformal map $\varphi^n:M(M^n)\to\bbC$, such that as $n\to\infty$, the following holds:
    \begin{enumerate}
        \item For every $R>0$, the maximal diameter of faces of $\frM(M^n)$ under the image $\varphi^n$ intersecting $B(0;R)$ tends to 0 in probability; \label{result:thm:unif-cone-1}
        \item For every $a>0$,  {$\sup_{t\in [-a,a]}|n^{-1/2}\cZ^n_t-\frac{1}{2}Z_t|$} converges to 0 in probability;\label{result:thm:unif-cone-2}
        \item If we view $\frtr^n:\bbR\to\bbC$ as a path on embedded faces, then for every $a>0$,  {$$\sup_{t\in [-a,a]}|\eta(t)-\frtr^n(t)| \to 0 \text{ in probability;}$$} \label{result:thm:unif-cone-3}
        \item Let $\mu^n$ be the measure that assigns mass $\frac{1}{n}$ to the image of each vertex under $\varphi^n$. Then for each $R>0$, the Prokhorov distance between $4\mu^n|_{B(0;R)}$ and $\mu_h|_{B(0;R)}$ tends to 0 in probability. The same applies if $\mu^n$ is any measure that assigns mass $\frac{1}{n}$ to the image of each face under $\varphi^n$. \label{result:thm:unif-cone-4}
    \end{enumerate}
 
\end{theorem}

\subsection{Proof outline}

Our proof of Theorems~\ref{thm:pack-disk}-\ref{thm:unif-cone} is based on three inputs:     a novel coupling between a random walk and Brownian motion in the concurrent work~\cite{Tutte}, and the construction of an ergodic scale-free configuration (see Section~\ref{subsec:cell-config}) whose associated map is encoded by a walk $\cZ$ in the setting of a SLE$_8$ decorated 0-quantum cone; the closeness between the  ergodic scale-free configurations and the circle packing/Riemann uniformization embedding as shown in our earlier work~\cite{HY26}; and the transfer between circle packings and Riemann uniformization embeddings in different domains established in another companion paper~\cite[Section 1.3.2]{MatedCRT}.

In Section~\ref{sec:3}, we start with the setting of a 0-quantum cone and the coupling between its boundary length process $Z$ and  a continuous-time nearest neighbor walk $\cZ$ on $\bbZ^2$ from~\cite[Theorem 4.2]{Tutte}; see Theorem~\ref{thm:KMT-0-cone}. Then in Section~\ref{subsec:time-sets}, we construct the time sets  using the $(Z,\cZ)$ in a similar way as in~\cite[Section 5]{Tutte}, and in Section~\ref{subsec:cellsystem}, we further construct cell configurations $\cH^\cell$ and $\mr\cH^\cell$ whose associated maps are encoded by $\Phi_\rmM(\cZ)$ and $\Phi_{\rm simple}\circ\Phi_\rmT(\cZ)$, respectively. In the setting of Riemann uniformization, the time sets $\cT_t^{\rmC,d}$ and the  cell configuration  $\cH^\cell$ is the same one as in~\cite{Tutte}, whereas for the circle packing, due to the collapse of the multiple-edges in the operation $\Phi_{\rm simple}$, we have to erase the $\frac{\pi}{2}$-cone excursions from the walk $\cZ$ (see the text around~\eqref{eq:double-edge-Mullin}) and modify the time sets $\mr\cT_t^{\rmC,d}$ accordingly as in Section~\ref{subsubsec:3.1.1}. This erasure of double edges will require several new estimates of the walk $\cZ$, where one need to show that not too many steps are erased, and the time sets still satisfy nice moment bounds after the erasure. The estimates for the former is in Section~\ref{subsubsec:3.1.2}, and the moment bounds are shown in Section~\ref{subsubsec:3.1.3}. In Section~\ref{subsec:cellsystem}, we define the cell configurations $\cH^\cell$ and $\mr\cH^\cell$ based on the time sets $\cT_t^{\rmC,d}$ and $\mr\cT_t^{\rmC,d}$ in a similar way as in~\cite{Tutte}, and prove that they fit into Theorem~\ref{thm:CellSystemPack} and Theorem~\ref{thm:invariance-uniformization} from our earlier work~\cite{HY26}. This gives the closeness between the cell configurations  $\mr\cH^\cell$ or  $\cH^\cell$, and  the circle packing of  $\Phi_{\rm simple}\circ\Phi_\rmT(\cZ)$ or the Riemann uniformization embedding of $\Phi_\rmM(\cZ)$, respectively, as well as  the convergence of the embedded trace $\frtr$ to SLE, as in Theorem~\ref{thm:conv-0-cone-pack-Hcell} and Theorem~\ref{thm:conv-0-cone-unif-Hcell}. In Section~\ref{subsec:3.3}, we will further prove that if we restrict to a time slice $[s,t]$, the boundary of the circle packing or Riemann uniformization embedding of the map encoded by the walk $\cZ|_{[s,t]}$ is close to $\partial\eta([s,t])$ when viewed as continuous curves. This relies on a new property of the cell configurations $\cH^\cell$ (Lemma~\ref{lem:construction-embed-cell-0}), which has not been previously shown in~\cite{Tutte}. This is the important condition to allow us to transfer to the finite volume setting in Section~\ref{subsec:pack-pf} and Section~\ref{sec:unif-pf}.

In Section~\ref{subsec:pack-pf}, we prove Theorems~\ref{thm:pack-disk}-\ref{thm:pack-cone}. In Section~\ref{subsec:He-Schramm}, we gather several results (Theorems~\ref{thm:He-Schramm-extension-1}-\ref{thm:He-Schramm-disk}) from our companion paper~\cite{MatedCRT} on the convergence of circle packing maps to conformal maps. The convergence of circle packing maps to conformal maps was originally shown in~\cite{HeSchramm}, where if one has  circle packings $\cP^n$ and $\wt\cP^n$ for triangulations  $T^n$ in  different domains $D$ and $\wt D$ and the maximal diameter of disks converges to 0 in the two packings, then the map between  $\cP^n$ and $\wt\cP^n$ that corresponds disks is close to a conformal map $f:D\to\wt D$. The original result in~\cite{HeSchramm} was stated in the form of uniform convergence on compacts, and in~\cite{MatedCRT}, we further proved that if $D,\wt D$ are Jordan, $\partial \cP^n$ is close to $\partial D$ and $\wt\cP^n$ is close to $\partial\wt D$ when viewed as curves, then the above convergence of circle packing maps can be extended to the global uniform topology. Combined with the aforementioned convergence of the boundary of the circle packing or Riemann uniformization embedding of the map encoded by the walk $\cZ|_{[s,t]}$ to $\partial\eta([s,t])$, we can now apply Theorem~\ref{thm:He-Schramm-extension-1} to prove the convergence of the circle packing or Riemann uniformization embedding of the map encoded by the walk $\cZ|_{[s,t]}$ in 0-quantum cone when embedded in the unit disk, as in Proposition~\ref{prop:cZ^n-pack-E}. To address the non-locality of the operator $\Phi_{\rm simple}$, we further trim the interval $[s,t]$ and restrict to a smaller interval $[s',t']\subseteq[s,t]$ as in Definition~\ref{def:CPtoMOT}, and the localized version of the triangulation encoded by $\cZ|_{[s,t]}$  agree with the triangulation encoded by $\cZ$ when restricted to time $[s',t']$ by our estimates from Section~\ref{subsubsec:3.1.2}. With this localization, using the local transfer from the 0-quantum cone to $\gamma$-quantum cone as in ~\cite[Section 6.1]{Tutte}, we   prove the convergence of the circle packing of the triangulation and trace $\frtr$  encoded by $\cZ^n|_{[s,t]}$ to the SLE-decorated LQG surface encoded by $Z$ in the setting where $\cZ^n$ has i.i.d.\ jumps of Poissonian rate $n$ and $Z$ is a Brownian motion, as in Lemma~\ref{lem:cZ^n-pack-Q-n}. These are done in Section~\ref{subsec:pack-iid}.  Then in Section~\ref{subsec:pack-disk-pf}, in the disk setting, for the Brownian excursion $Z:[0,A]\to\bbR^2$, we  weight the law of $(\cZ^n,Z)$ appropriately where the above convergence continues to hold for $(\cZ^n,Z)|_{[A\e,A(1-\e)]}$. As we further send $\e\to0$ sufficiently slow, by the continuity of the space-filling SLE curve, in the circle packing of the triangulation $T^{\e,n}$ encoded by $\cZ^n|_{[A\e,A(1-\e)]}$, the union of the disks for the boundary vertices which form a crosscut separating $T^{\e,n}$  from  $T^{\disk,n}$ as in Figure~\ref{fig:pack-disk-1} has small Euclidean diameter. This eventually allows us to prove Theorem~\ref{thm:pack-disk} using Theorem~\ref{thm:He-Schramm-disk}. The sphere case (Theorem~\ref{thm:pack-sph}) is treated similarly in Section~\ref{subsec:pack-sph-pf}, where we apply the comparison for circle packing  in the sphere setting as in Theorem~\ref{thm:He-Schramm-sphere}. In Section~\ref{subsec:pack-cone-pf}, we   start with the setting of intervals in Lemma~\ref{lem:cZ^n-pack-Q-n} and   zoom in  near 0, and prove Theorem~\ref{thm:pack-cone} via the comparison in Theorem~\ref{thm:He-Schramm-sphere} and scale invariance of the Brownian motion.

In Section~\ref{sec:unif-pf}, we prove Theorems~\ref{thm:unif-disk}-\ref{thm:unif-cone}. Most of the idea is the same as  Section~\ref{subsec:pack-pf}, where we start with the setting of Theorem~\ref{thm:conv-0-cone-unif-Hcell} and restrict to a time slice $[s,t]$, and further use the boundary perturbation of conformal maps from Proposition~\ref{prop:boundary-perturb} and the local transfer from ~\cite[Section 6.1]{Tutte} to prove a version of Theorems~\ref{thm:unif-disk}-\ref{thm:unif-cone} for walks with i.i.d.\ steps in a compact interval. Theorems~\ref{thm:unif-disk}-\ref{thm:unif-cone} are then proven in a similar way as Theorems~\ref{thm:pack-disk}-\ref{thm:pack-cone} where we use the comparison of conformal maps as in Propositions~\ref{prop:boundary-perturb}-\ref{prop:unif-sph}.

\medskip

\noindent \textbf{Notation.} We write $\bbN_0=\{0,1,2,3,....\}$, and write $\bbR_+ = [0,\infty)$.   For nonnegative functions $f,g:\bbN_0\to \bbR_+$, we write $f(n)\lesssim g(n)$, if there are some constants $M,N>0$ such that $f(n)\leq Mg(n)$ for every $n>N$.  {For $(a,b)\in\bbR^2$, we write $|(a,b)| = \sqrt{a^2+b^2}$ and $\|(a,b)\| = \max\{|a|,|b|\}$.}

\medskip

\noindent \textbf{Statement:} We have worked on this paper for more than one year and are currently in the process of revising it. We were prompted to post the paper today since several closely related results were released by OpenAI yesterday (October 6). To our knowledge, none of the OpenAI results imply the results of our paper, or vice versa.

\medskip

\noindent \textbf{Statement on AI usage:} We used ChatGPT Astra for proofreading and revising the first draft of the paper. AI tools were not used to develop mathematical ideas or in the writing of the manuscript. All suggested changes were reviewed by the authors, who take full responsibility for the content of the paper.

\medskip

\noindent\textbf{Acknowledgements.}   
N.H.\ was supported by grant DMS-2246820 and DMS-2553117 of the National Science Foundation, along with the Simons Collaboration Grant \emph{Probabilistic Paths to Quantum Field Theory}.

\section{Preliminaries}

\subsection{Cell configurations}\label{subsec:cell-config}

One of the main inputs to prove our main results is~\cite[Theorem 5.5]{HY26}, which states that for a certain type of ergodic scale-free cell configurations whose adjacency graph is an infinite plane triangulation, the a priori embedding is close to the circle packing and Riemann uniformization embedding of the adjacency graph. This subsection is devoted to introducing these results.

\begin{definition}\label{def:cell-config}
    A \textit{cell configuration} on $\bbC$ consists of the following objects.
    \begin{enumerate}
        \item A locally finite collection $\cH$ of  compact connected subsets of $\bbC$ (``cells") with non-empty interiors whose union is all of $\bbC$ and such that the intersection of any two elements of $\cH$ has zero Lebesgue measure.
        \item A planar map $M$, which we call the \emph{associated map} of $\cH$, with vertices $\{v_H:H\in\cH\}$. We require that if $v_H$ and $v_{H'}$ are connected by an edge of $M$ then $H\cap H'\neq\emptyset$. 
    \item A function $\fc = \fc_\cH$ (``conductance") on the edges  $M$ with values in $(0,\infty)$. 
    \end{enumerate}
\end{definition}
% \begin{definition}\label{def:cell-config}
%     A \textit{cell configuration} on $\bbC$ consists of the following objects.\footnote{The original definition of cell configurations in~\cite{GMSinvariance} also concerns conductance on edges; in this paper we will always assume that the conductance is equal to 1 on all edges and omit the conductance.}
%     \begin{enumerate}
%         \item A locally finite collection $\cH$ of  compact connected subsets of $\bbC$ (``cells") with non-empty interiors whose union is all of $\bbC$ and such that the intersection of any two elements of $\cH$ has zero Lebesgue measure.
%         \item A symmetric relation $\sim$ on $\cH\times\cH$ (``adjacency") such that if $H\sim H'$, then $H\cap H'\neq\emptyset$ and $H\neq H'$.
%    % \item A function $\fc = \fc_\cH$ (``conductance") from the set of pairs $(H,H')$ with $H\sim H'$ to $(0,\infty)$ such that $\fc(H,H')=\fc(H',H)$. 
%     \end{enumerate}
% \end{definition}
We say $H\sim H'$ if $H$ and $H'$ are connected by  {at least one} edge in $M$. Note that we use a different definition of a cell configuration than in~\cite[Definition 1.1]{GMSinvariance}, where a cell configuration is defined to be a tuple $(\cH,\sim,\fc)$. We need to also include $M$ in the definition since our goal is to prove closeness between $\cH$ and conformal embeddings of the map $M$ and, furthermore, will work with settings where $M$ has multiple edges.  
We will typically slightly abuse notation by making  the associated map $M$ and the function $\fc$ implicit, so we write $\cH$ instead of $(\cH,M,\fc )$.

For $A\subseteq\bbC$, let $\cH(A) = \{H\in\cH:H\cap A\neq\emptyset\}$, and we view $\cH(A)$ as an edge-weighted planar map with edge set $\cE\cH(A)$ consisting of all of the edges in $\cE\cH$ joining elements in $\cH(A)$ and conductances given by the restriction of $\fc$. We also define a metric on the space of cell configurations by
 \begin{equation}\label{eq:metric-cc}
     \begin{split}
         \mathbb{d}^{\rm CC}(\cH,\cH') := \int_0^\infty e^{-r}\wedge \inf\Big\{\sup_{z\in\bbC} |z-f_r(z)|+\sup_{e\in \cE\cH(B(0;r))}\big|\fc(e)-\fc'(f_r(e)) \big|   \Big\}\,dr,
     \end{split}
 \end{equation}
 where the infimum is over all homeomorphisms $f_r:\bbC\to\bbC$ such that $f_r$ takes each cell in $\cH(B(0;r))$ to $\cH'(B(0;r))$   and  preserves the associated map when restricted to cells intersecting $B(0;r)$, and $f_r^{-1}$ does the same with $\cH$ and $\cH'$ reversed. Here $f_r(\fre)$ is understood as the image of the edge $e$ under map isomorphisms. We will typically be working with a random cell configuration, i.e., a random
variable taking values in the space of cell configurations equipped with the Borel $\sigma$-algebra generated
by the above metric. 
 We write $H_0$ for the (a.s.\ unique) cell in $\cH$ containing 0. For $H\in\cH$, we write $c(H) = \frac{1}{\area(H)}\int_H z\,dz$ for the Euclidean center of $H$.

\begin{definition}\label{def:CellConfigConstraint}
    We define  the following conditions  on the cell configuration $\cH$.
    \begin{enumerate}
        \item \textbf{Translation invariance modulo scaling.} There is a (possibly random and $\cH$-dependent) increasing sequence of open sets $U_j\subseteq\bbC$,  each of which is either a square or a disk, whose union is all of $\bbC$, such that the following is true. Conditional on $\cH$ and $U_j$, let $z_j$ for $j\in\bbN$ be sampled uniformly from Lebesgue measure on $U_j$. Then there exist random numbers $C_j>0$ (possibly depending on $\cH$ and $z_j$) such that $C_j(\cH-z_j)\to\cH$ in law with respect to the metric~\eqref{eq:metric-cc}. \label{cond:def:CellConfigConstraint-1}
        \item  \textbf{Ergodicity modulo scaling.} For every real-valued measurable functions $F=F(\cH)$ which is invariant under translation and scaling, i.e., $F(C(\cH-z)) = F(\cH)$ for each $z\in\bbC$ and $C>0$, is a.s.\ equal to a deterministic constant.  \label{cond:def:CellConfigConstraint-2}
      %  \item  \textbf{Connectedness along  {macroscopic} lines.} Almost surely,  {for each given $\e>0$, there exists $R>0$ such that for $r>R$ and each horizontal or vertical line segment $\ell\subseteq B(0;r)$ with length at least $\e r$}, the subgraph of $\cH$ induced by the set of cells which intersect $\ell$ is connected. % {[*This seems to be the condition needed*]}
        \item \textbf{Rotation invariance in law.} For some $\theta_0\in(0,2\pi)\backslash\{\pi\}$, $e^{i\theta}\cH$ has the same law as $\cH$.  \label{cond:def:CellConfigConstraint-3}

       % \item \textbf{Almost planarity.} There exists a collection of points $\{z_H:H\in\cH\}$ and simple curves $\{e_{H,H'}:H\sim H'\}$ in $\bbC$ forming a planar embedding of the adjacency graph of $\cH$, and almost surely as $r\to\infty$, $$%\frac{1}{r}\max_{H\in\cH(B(0;r))}\dist(z_H,H)\to0, \ \ 
       % \frac{1}{r}\max_{\{H,H'\}\in\cE\cH(B(0;r))}\dist(e_{H,H'},H)\to0, $$
      %  where the  distance is the Hausdorff distance.
    \end{enumerate}
\end{definition}

\begin{definition}\label{def:CellConfigConstraint-large}
    Let $\cH$ and $\cH^\lge$ be cell configurations. We define  the following conditions  on the cell configurations $(\cH,\cH^\lge)$.
     \begin{enumerate}[(i)]
        \item The cell configuration  $\cH^{\rm large}$ is isomorphic to $\cH$, in the sense that $\cH$ and $\cH^{\rm large}$ have isomorphic associated maps; below we let $\phi:\cH\to\cH^{\rm large}$ denote  {the} isomorphism.   \label{cond:def:CellConfigConstraint-large-(i)}
        \item The convex hull of $H$ is contained in the convex hull of $\phi(H)$  {for any $H\in\cH$}. \label{cond:def:CellConfigConstraint-large-(ii)}
        \item Almost surely,  {for each given $\e>0$, there exists $R>0$ such that for $r>R$ and each horizontal or vertical line segment $\ell\subseteq B(0;r)$ with length at least $\e r$}, the following holds. For any two different cells $H,H'\in\cH$ intersecting $\ell$, $\phi(H),\phi(H')$ also intersect $\ell$, and we can find a path $H\sim H_1\sim\cdots\sim H_k\sim H'$, such that $\phi(H_j)$ intersects $\ell$ for each $j$. \label{cond:def:CellConfigConstraint-large-(iii)}
    \end{enumerate} 
\end{definition}

The following results are from~\cite[Theorem 5.5]{HY26}.

\begin{theorem}\label{thm:CellSystemPack}
    Let $\cH$ be a random cell configuration such that its associated map is a simple infinite plane triangulation. Suppose $\cH$   satisfies all the conditions in Definition~\ref{def:CellConfigConstraint} and the moment bound
    \begin{equation}\label{eq:thm-inv-circle}
        \bbE\Big[\frac{\diam(H_0)^2}{\area(H_0)}\deg(H_0)^4\Big]<\infty.
    \end{equation}
    Further assume that there exists another cell configuration $\cH^\lge$ satisfying the conditions in Definition~\ref{def:CellConfigConstraint} and~\eqref{eq:thm-inv-circle}, and assume that $(\cH,\cH^\lge)$ satisfies the conditions in Definition~\ref{def:CellConfigConstraint-large}. Then the associated map is a.s.\ circle packing parabolic, and there is a circle packing $\cP$ for $\cH$ such that  {up to a reflection,} the following is true.
    Almost surely, for any $\e>0$, there exists $R_0>0$ such that for every $r>R_0$, the maximal diameter of disks in $\cP$ intersecting $B(0;r)$ is less than $\e r$. Furthermore, for $H\in\cH$,  letting $o(H)$ be the center of the disk for $H$ in $\cP$, almost surely, 
    \begin{equation}\label{eq:thm:CellsystemPack}
        \lim_{r\to\infty}\frac{1}{r}\max_{H\in \cH(B(0;r))}|c(H)-o(H)| = 0.
    \end{equation}
    %Furthermore, if $\cH$ satisfies the rotation invariance in Definition~\ref{def:rotation-inv}, then $\mathbf{A}$ can be taken to be the identity matrix.
\end{theorem}

 {For a cell configuration $\cH$ such that its associated map $M$ has {whole-plane topology},  we write $\mathfrak{M}(\cH)$ for the Riemann surface formed by $M$. For $H\in\cH$, we write $v_H$ for the vertex for $H$ on $\mathfrak{M}(\cH)$. We write  $\cH^{\mathrm{face}}_0$ for the set of cells $H\in\cH$ such that $v_H$ and $v_{H_0}$ are incident to the same face of $M$. We write $d_0^f$ for the maximum degree of faces incident to $v_{H_0}$.}
\begin{theorem}\label{thm:invariance-uniformization}
    Let $\cH$ be a random cell configuration such that its associated map has  {whole-plane topology}.  Suppose $\cH$ satisfies all the conditions in Definition~\ref{def:CellConfigConstraint} and the moment bound
    \begin{equation}\label{eq:thm-inv-circle-1}
        { \bbE\Big[ \max_{H\in \cH^{\mathrm{face}}_0}\Big(\frac{\diam(H)^2}{\area(H)}\big(d_0^f+\deg(H)\big)^6\Big)\Big]<\infty. }
    \end{equation}
      Further assume that there exists another cell configuration $\cH^\lge$ satisfying the conditions in Definition~\ref{def:CellConfigConstraint} and~\eqref{eq:thm-inv-circle-1}, and assume that $(\cH,\cH^\lge)$ satisfies the conditions in Definition~\ref{def:CellConfigConstraint-large}. 
    Then the Riemann surface $\frM(\cH)$ formed by the associated map of $\cH$ is a.s.\   parabolic, and there exists a  conformal map $\varphi_0:\frM(\cH)\to\bbC$ for $\cH$ such that {up to a reflection}, the following is true.  Then {almost surely}, 
    \begin{equation}\label{eq:thm:invariance-uniform}
        \lim_{r\to\infty}\frac{1}{r}\max_{H\in \cH(B(0;r))}|c(H)-\varphi_0(v_H)| = 0.
    \end{equation}
    Furthermore,  for each $f\in \cF\cH$, we write $\mathfrak{f}_{f}$ for the face for $f$ on  $\frM(\cH)$,  
     \begin{equation}\label{eq:thm:invariance-uniform-1}
        \lim_{r\to\infty}\frac{1}{r}\max_{f\in \cF\cH(B(0;r))}\diam(\varphi_0(\mathfrak{f}_{f})) = 0,
    \end{equation}
    where $\cF\cH(B(0;r))$ is the set of faces of the associated map for $\cH$ which are incident to cells in $\cH(B(0;r))$.  {Finally, for each $\e>0$, almost surely for sufficiently large $r>0$, if the embedded face $\varphi_0(\mathfrak{f}_{f})$ intersects $B(0;r)$, then $\diam(\varphi_0(\mathfrak{f}_{f}))<\e r$.}
   % Finally, if $\cH$ further satisfies the rotation invariance in Definition~\ref{def:rotation-inv}, then $\mathbf{A}$ can be taken to be the identity matrix.
\end{theorem}

\subsection{Liouville quantum gravity and mating of trees}

\subsubsection{The Gaussian free field and Liouville quantum gravity surfaces}

%The whole plane Gaussian free field (with zero average on unit circle) is defined as follows. 
Let $X \in\{\bbC,\bbH\}$. Let $m_\bbC$ (resp.\ $m_\bbH$) be the uniform measure on the unit circle {$\partial\bbD$} (resp.\ unit semicircle $\{|z|=1,\mathrm{Im}\, z>0\}$). Define the Dirichlet inner product $\langle f,g\rangle_\nabla = (2\pi)^{-1}\int_X \nabla f\cdot\nabla g $ on the space $\{f\in C^\infty(X):\int_X|\nabla f|^2<\infty; \  \int f(z)m_X(dz)=0\},$ and let $H(X)$ be the closure of this space {with respect to} the inner product $\langle f,g\rangle_\nabla$. Let $(f_n)_{n\ge1}$ be an orthonormal basis of $H(X)$, and $(\alpha_n)_{n\ge1}$ be a collection of independent standard Gaussian variables.  Then the summation
$$h_X = \sum_{n=1}^\infty \alpha_nf_n$$
a.s.\ converges in the space of distributions on $\bbC$, and $h_X$ is the \emph{Gaussian Free Field} on $X$ normalized such that $\int h(z)m_X(dz) = 0$. Let $P_X$ be the law of $h_X$. See~\cite[Section 4.1.4]{DMS14} for more details.  As argued in \cite[Section 4.1]{DMS14},  we have the decomposition $H({X}) = H_1({X})\oplus H_2({X})$, where $H_1(X)$ is the subspace of radially symmetric functions, and $H_2(X)$ is the subspace of functions having mean 0 about all  circles $\{|z|=r\}$ if $X=\bbC$ or all  semicircles $\{|z|=r, \ \mathrm{Im}\, z>0\}$ if $X=\bbH$. As a consequence, for the GFF $h_X$ sampled from $P_X$, we can decompose \eqb\label{eq:GFF-decompose}h_X=h_{X,1}+h_{X,2}\eqe where $h_{X,1}$ and $h_{X,2}$ are independent distributions given by the projection of $h$ onto $H_1(X)$ and $ H_2(X)$, respectively. 

Now let $\gamma\in(0,2)$ and $Q=\frac{2}{\gamma}+\frac{\gamma}{2}$. Consider the space of pairs $(D,h)$, where $D\subseteq \mathbb C$ is a planar domain and $h$ is a distribution on $D$ (often some variant of the GFF). For a conformal map $g:D\to\tilde D$ and a generalized function $h$ on $D$, define the generalized function $g\bullet_\gamma h$ on $\tilde{D}$ by setting 
\begin{equation}\label{eq:lqg-changecoord}
    g\bullet_\gamma h:=h\circ g^{-1}+Q\log|(g^{-1})'|.
\end{equation}
 Define the equivalence relation $\sim_\gamma$ {as follows. We say that} $(D, h)\sim_\gamma(\wt{D}, \wt{h})$ if there is a {conformal map $g:D\to\wt{D}$ such that $\tilde h = g\bullet_\gamma h$}.
A \textit{quantum surface} $S$ is an equivalence class of pairs $(D,h)$ under the {equivalence} relation $\sim_\gamma$, and we say {that} $(D,h)$ is an \emph{embedding} of $S$ if $S = (D,h)/\mathord\sim_\gamma$.   {Abusing notation, we will sometimes call $(D,h)$ as a quantum surface, and we are then referring to the equivalence class $(D,h)/\sim_\gamma$ that it defines.} Likewise, a \emph{quantum surface with} $k$ \emph{marked points} is an equivalence class of {tuples} %elements 
of the form
	$(D, h, x_1,\dots,x_k)$, where $(D,h)$ is a quantum surface, the points  $x_i\in {\overline{D}}$, and with the further
	requirement that marked points (and their ordering) are preserved by the conformal map $\varphi$ in \eqref{eq:lqg-changecoord}, and a \emph{curve decorated quantum surface} can be defined analogously. 

     For a $\gamma$-quantum surface $(D, h, z_1, ..., z_m)$, its \textit{quantum area measure} $\mu_h$ is defined by taking the weak limit $\epsilon\to 0$ of $\mu_{h_\epsilon}:=\epsilon^{\frac{\gamma^2}{2}}e^{\gamma h_\epsilon(z)}d^2z$, where $d^2z$ is the Lebesgue area {measure} and $h_\epsilon(z)$ is the circle average of $h$ over $\partial B(z, \epsilon)$. When $D=\mathbb{H}$, we can also define the  \textit{quantum boundary length measure} $\nu_h:=\lim_{\epsilon\to 0}\epsilon^{\frac{\gamma^2}{4}}e^{\frac{\gamma}{2} h_\epsilon(x)}dx$ where $h_\epsilon (x)$ is the average of $h$ over the semicircle $\{x+\epsilon e^{i\theta}:\theta\in(0,\pi)\}$. It has been shown in \cite{DS11, SW16} that all these weak limits are well-defined  for the GFF and its variants we are considering in this paper, {and that} $\mu_{h}$ and $\nu_{h}$ {can} be conformally extended to other domains using the relation $\bullet_\gamma$.

Now we introduce the quantum surfaces we will  frequently consider in this paper. 
\begin{definition}[$\alpha$-quantum cone]\label{def:alpha-cone}
  Fix $\gamma\in(0,2)$, $\alpha<Q$ and let $(B_s)_{s\ge0}$ and $(\wt{B}_s)_{s\ge0}$ be independent standard one-dimensional Brownian motions. 
  Let 	
			\begin{equation*}
				Y_t=\left\{ \begin{array}{rcl} 
					B_{t}+\alpha t  & \mbox{for} & t\ge 0,\\
					\wt{B}_{-t} +\alpha t & \mbox{for} & t<0,
				\end{array} 
				\right.
			\end{equation*}
			conditioned on   $\wt{B}_{t} + (Q-\alpha)t>0$ for all $t>0$.  Recall the decomposition in~\eqref{eq:GFF-decompose}. The $\alpha$-quantum cone is the doubly marked quantum surface $(\bbC,h,0,\infty)/\sim_\gamma$, where $h=\psi_1+\psi_2$, and $(\psi_1,\psi_2)$ are independent random generalized functions in $H_1(\bbC)\times H_2(\bbC)$ whose law can be described as follows.  The common value of $\psi_1$ on $\{|z|=e^{-t}\}$ is equal to $Y_t$ for all $t\in\bbR$, and $\psi_2$ has the same law as $h_{\bbC,2}$ as in~\eqref{eq:GFF-decompose}.
\end{definition}
The embedding $(\bbC,h,0,\infty)$ above is also called the circle average embedding of the $\alpha$-quantum cone.

A \emph{quantum disk} is a  finite-mass quantum surface $(\bbD,h^{\mathrm{disk}})/\sim_\gamma$ defined in~\cite[Definition 4.21]{DMS14} (the precise definition is not needed here).  One can consider quantum disks with specified boundary length $\nu_{h^{\mathrm{disk}}}(\partial\bbD)$ and area $\mu_{h^{\mathrm{disk}}}(\bbD)$, as well as  quantum disks with any number of marked boundary points and interior points sampled uniformly from $\nu_h$ and $\mu_h$, respectively. We will typically work with quantum disks with unit area and boundary length and a single marked boundary point. A \emph{quantum sphere} is a  finite-mass quantum surface $(\wh\bbC,h_{\mathrm{sp}})/\sim_\gamma$ defined in~\cite[Definition 4.21]{DMS14}, where $\wh\bbC = \bbC\cup\{\infty\}$, and one can consider quantum spheres with specified area $\mu_{h_{\mathrm{sp}}}(\wh\bbC)$ as well. %One important property which we shall use frequently throughout this section is, 

\subsubsection{Space-filling SLE and mating of trees}\label{subsubsec:MatedCRT}
The Schramm-Loewner evolution (SLE) is a family of fractal curves with a single parameter $\kappa>0$~\cite{Sch00}. $\SLE_\kappa$ is a simple curve for $\kappa\in(0,4)$, self-touching but not space-filling or self-crossing for $\kappa\in(4,8)$, and space-filling for $\kappa\geq8$. As studied in the theory of Imaginary Geometry~\cite{ig1,ig4}, $\SLE_\kappa$ curves can be viewed as flow lines of GFF. For $\kappa>4$, space-filling $\SLE_\kappa$ is a variant of $\SLE_\kappa$ introduced in~\cite{ig4}, which is also generated via GFF flow lines and agrees with ordinary $\SLE_\kappa$ when $\kappa\geq8$.  The space-filling $\SLE_\kappa$ loop from $\infty$ to $\infty$ can be  generated by flow lines of  $h^{\rm IG}$, where $h^{\rm IG}$ is a whole plane GFF modulo a global additive multiple of $2\pi\chi$ with $\chi = \sqrt\kappa/2-2/\sqrt\kappa$.  See~\cite{BN23SLEnotes} for a detailed introduction to SLE, and~\cite[Section 4.3]{ig4} for the precise definition of space-filling SLE.

Now we briefly introduce the theory of mating of trees in~\cite{DMS14}. Let $\gamma\in(0,2)$ and $\kappa=16/\gamma^2$. In the rest of the paper, the letter $\rmQ$ on superscripts stands for ``quantum", and the letter $\rmE$ on superscripts stands for ``Euclidean". Consider a $\gamma$-quantum cone embedded as $(\bbC,h^\rmQ,0,\infty)$ and a space-filling $\SLE_\kappa$ loop $\eta^\rmQ:\bbR\to\bbC$ from $\infty$ to $\infty$ independent of $h^\rmQ$. Assume $\eta^\rmQ(0)=0$ and $\eta^\rmQ$ is parameterized according to LQG area measure induced by $h^\rmQ$, i.e., for each $t>s$, $\mu_{h^\rmQ}(\eta^\rmQ([s,t]))=t-s$. For $t\in\bbR$, one can define the change $L_t$ (resp.\ $R_t$) in the length of the left (resp.\ right) boundary of $\eta^\rmQ((-\infty,t])$ relative to time 0.  By ~\cite[Theorem 1.9]{DMS14} and~\cite{GHMS15}, $(L_t,R_t)_{t\in\bbR}$ evolves as 2D Brownian motion with correlation $-\cos(\pi\gamma^2/4)$, where the variance $2/\sin(\pi\gamma^2/4)$ computed in~\cite{ARS21}.  {To simplify the notations, we will multiply the quantum length measure $\nu_h$ by $\frac{1}{2}$ such that $\mathrm{Var}(L_t) = \mathrm{Var}(R_t)=\frac{1}{2}|t|$ and shift back to the original quantum length measure in the proofs of Theorem~\ref{thm:pack-disk}-\ref{thm:unif-cone}.} Moreover, by~\cite[Theorem 1.11]{DMS14}, the Brownian motion $(L_t,R_t)_{t\in\bbR}$ a.s.\ determines  $(h^\rmQ, \eta^\rmQ)$ up to a rotation. Note that for $\gamma\in(0,\sqrt{2}]$ (resp.\ $\gamma\in(\sqrt{2},2)$), for each $a<b$, $\eta^\rmQ([a,b])$ and $\eta^\rmE([a,b])$ are a.s.\ simply connected (resp.\ not simply connected). Following~\cite[Lemma 2.17]{AG19}, for every $a<b$, the curve-decorated quantum surface $(\eta^\rmQ([a,b]),\eta^\rmQ|_{[a,b]})/\sim_\gamma$ is measurable with respect to $(L_t-L_a,R_t-R_a)_{t\in[a,b]}$.  We will further embed $(\eta^\rmQ([a,b]),\eta^\rmQ|_{[a,b]})/\sim_\gamma$ onto $\bbD$ such that $\eta^\rmQ(\frac{a+b}{2})=0$ and $\eta^\rmQ(a)=1$. For a random process  $(\wt L_t,\wt R_t)_{t\in[a,b]}$ whose laws are absolutely continuous with respect to $(L_t,R_t)_{t\in[a,b]}$, one can also generated a curve-decorated quantum surface in the same way.

One can also study mating of trees   in other topologies. Consider   quantum disk with unit area and boundary length and a marked point embedded as  $(\bbD,h^{\mathrm{disk}},1)$, and an independent clockwise space-filling $\SLE_\kappa$ loop  $\eta^{\rm disk}:[0,1]\to\bbD$ from 1 to 1 parameterized by LQG area induced by $h^{\mathrm{disk}}$. By~\cite{AG19}, one can define a boundary length process $(L_t,R_t)_{t\in[0,1]}$ in terms of $(h^{\rm disk},\eta^{\rm disk})$ which evolves as a Brownian excursion in $\bbR_+^2$ from {$(1,0)$ to $(0,0)$} with duration 1, while $(L_t,R_t)_{t\in[0,1]}$ a.s.\ determines $(h^{\rm disk},\eta^{\rm disk})$.   Similarly, for a unit area quantum sphere with two marked points $(\bbC,h^{\mathrm{sp}},0,\infty)$ and an independent space-filling $\SLE_\kappa$ loop $\eta^{\rm sp}$ from $\infty$ to $\infty$ parameterized by quantum area such that $\eta^{\rm sp}(0)=\eta^{\rm sp}(1)=\infty$, one can also define a boundary length process $(L_t,R_t)_{t\in[0,1]}$ in terms of $(h^{\rm sp},\eta^{\rm sp})$ which evolves as a Brownian excursion in $\bbR_+^2$ from $(0,0)$ to $(0,0)$ with duration 1~\cite{MS-mot-finite}.

For a 0-quantum cone $(\bbC,h^\rmE,0,\infty)$ and a space-filling $\SLE_\kappa$ loop $\eta^\rmE$ from $\infty$ to $\infty$ parameterized by $\mu_{h^{\rmE}}$, as considered in~\cite{GMS21tutte}, one can also define a  process $(Z_t)_{t\in\bbR}=(L_t,R_t)_{t\in\bbR}$ reflecting the change of boundary length in the same way. 

In the rest of the paper we will constantly use the following definitions. For an interval $[a,b)$, we call the intervals $[a,\frac{a+b}{2})$ and $[\frac{a+b}{2},b)$ the dyadic children of $[a,b)$, and we say $[a,b)$ is an dyadic parent of $[a,\frac{a+b}{2})$ or $[\frac{a+b}{2},b)$. We say $I\subseteq I'$ is a dyadic descendant of $I'$, if there is $I= I_1\subseteq I_2\subseteq\cdots\subseteq I_n = I'$ such that for each $j=1,...,n-1$, $I_j$ is a dyadic children of $I_{j+1}$.
\begin{definition}\label{def:1ddyadic}
    Let $\theta$ be uniformly randomly sampled in $[0,1)$. Let $I(0) = [\theta-1,\theta)$. For $k\in\bbN$, inductively let $I(k)$ be sampled uniformly from the two dyadic parents of $I(k-1)$.  The \emph{1D uniform dyadic system} $\fD$ is defined by the set of all dyadic descendants of $I(k)$  for each $k\in\bbN$. For $k\in\bbZ$, we also write  $I(k)$ for the interval in $\fD$  containing 0 of length $2^{k}$.
    We write  $\Omega_\fD$ for the space of all possible realizations of $\fD$. For $s\in\bbR$, we write $\fD-s$ for the set $\{[a-s,b-s):[a,b)\in \fD\}$.
\end{definition}

\begin{definition}\label{def:P_EP_Q}
  Let $\fD$ be an 1D uniform dyadic system as in Definition~\ref{def:1ddyadic}, and $\fU = (U_I)_{I\in\fD}$ be a collection of i.i.d.\ uniform random variables on $[0,1]$ indexed by intervals in $\fD$.  Let $\bbP_\rmE$ be the law of $(h,h^{\rm IG},\eta,\fU)$ where $(\bbC,h,0,\infty)$ is {the circle average} embedding of the 0-quantum cone, $\eta$ is a whole-plane space-filling $\SLE_\kappa$ generated by the whole plane GFF $h^{\rm IG}$, $\fU$ is generated from a uniform dyadic system $\fD$ as above,  and the three random variables  $h,h^{\rm IG},\fD$ are all independent. We parametrize $\eta$ by the $\gamma$-LQG area induced by $h$. Let $\bbP_\rmQ$ be defined in the exact same way, except that $(\bbC,h,0,\infty)$ is   {the circle average} embedding of the $\gamma$-quantum cone instead of the 0-quantum cone. Under either of these probability measures we define the boundary length process $Z$ of  $(h,\eta)$ as in Section~\ref{subsubsec:MatedCRT}.
\end{definition}

%Let $T^{\rmE,n}_{\rm MOT}, \wh T^{\rmE,n}_{\rm MOT}, \cH^{\rmE,n}_{\rm MOT},\wh\cH^{\rmE,n}_{\rm MOT}$ be defined in the identical way using $(h^\rmE,\eta^\rmE)$ as in the $\gamma$-quantum cone case. %In Section~\ref{subsec:matedCRT-0-cone} and Section~\ref{subsec:matedCRT-gamma-cone}, we will first use Theorem~\ref{thm:CellSystemPack} to prove the convergence of circle packing of $\wh T^{\rmE,n}$ to the a priori embedding $\wh\cH^{\rmE,n}$ as $n\to\infty$, then use absolute continuity and Theorem~\ref{thm:He-Schramm-extension-1} to transfer to the mated CRT map with various topologies.

\subsubsection{Quantitative bounds on SLE and LQG}\label{subsubsec:SLE-LQG-bound}

In this section we gather estimates on SLE and LQG that will  be used in the remainder of the paper. The first property is the scale invariance property for quantum cones, which is from~\cite[Lemma 2.1]{GMS21tutte}.
\begin{lemma}\label{lem:Rb}
    Let $\gamma\in(0,2)$, $\alpha<Q$ and $(\bbC,h,0,\infty)$ be the embedding of an $\alpha$-quantum cone as in Definition~\ref{def:alpha-cone}. Let $\eta$ be an independent space-filling SLE loop from $\infty$ to $\infty$ parameterized by $\mu_h$.  For $b>0$, let 
    \begin{align}
        &\fR_b = \sup\{r>0:h_r(0)+Q\log r = \frac{1}{\gamma}\log b\}; \label{eq:Rb}
        \\  &h^b(\cdot)=h(\fR_b\cdot)+Q\log \fR_b-\frac{1}{\gamma}\log b, \qquad \eta^b(\cdot) = \fR_b^{-1}\eta(b\cdot). \label{eq:hb-etab}
    \end{align}
   (Note that by definition $\fR_1=1$.)  Then $(h^b,\eta^b)$ has the same law as $(h,\eta)$. Moreover, there is a constant $a = a(\alpha,\gamma)>0$ such that for any $b_2>b_1>0$ and $C>1$, 
     \begin{equation}\label{eq:Rb-bound}
      \bbP\Big[C^{-1}(\frac{b_2}{b_1})^{\frac{1}{\gamma(Q-\alpha)}}\leq \frac{\fR_{b_2}}{\fR_{b_1}}\leq C(\frac{b_2}{b_1})^{\frac{1}{\gamma(Q-\alpha)}}\Big]\geq 1-3\exp\big(-\frac{a (\log C)^2}{\log(b_2/b_1)+\log C}\big).
  \end{equation}
\end{lemma}
Observe that in the setting above, if $(Z_t)_{t\in\bbR}$ is the Brownian motion encoding $(h,\eta)$, then $(h^b,\eta^b)$ is encoded by $(b^{-1/2}Z_{bt})_{t\in\bbR}$.  {The following estimates are from ~\cite[Lemma A.5, Lemma A.6]{Tutte}.}
%To simplify the notations, we are going to use the following shorthand. Let $(\bbC,h,0,\infty)$  be either a $\gamma$-quantum cone or 0-quantum cone, and $\eta$ be an independent  space-filling SLE loop from $\infty$ to $\infty$ with $\eta(0)=0$ parameterized by $\mu_h$. Let $(Z_t)_{t\in\bbR} = (L_t,R_t)_{t\in\bbR}$ be the associated boundary length process determined by $(h,\eta)$. We write $\bbP_\rmE$ (resp.\ $\bbP_\rmQ$) for the law of $(h,\eta)$ if $h$ is a 0-quantum cone (resp.\ $\gamma$-quantum cone) and $\eta$ is a space-filling $\SLE_\kappa$ loop independent from $h$ and parameterized by the LQG measure $\mu_h$, and drop the superscript from $h$ or $\eta$.

% The second result is the following quantitative bound on cells in mated CRT maps, which is from~\cite[Lemma 2.7]{GMS19har} and  {Lemma 4.9 in the Tutte draft}.
% \begin{lemma}\label{lem:diam-cell-max}
%  %Let $(\bbC,h,0,\infty)$ be either a $\gamma$-quantum cone or 0-quantum cone, and $\eta$ be an independent  space-filling SLE loop from $\infty$ to $\infty$ parameterized by LQG area induced by $h$.  
%  For each $q\in(0,2(2+\gamma)^{-2})$ and $\e\in(0,1)$, 
%     $$\bbP_\rmE\Big[\diam(\eta(J))\leq \e^q \ \text{for any interval } J \text{ with } |J|\leq\e/2 \ \text{and }\eta(J)\cap B(0;3/4)\neq\emptyset\Big]\geq 1-\e^{\alpha(q,\gamma)+o_\e(1)},$$
%     where the rate of the $o_\e(1)$ depends only on $q$ and $\gamma$, and 
%     $$\alpha(q,\gamma):=\frac{q}{2\gamma^2}(\frac{1}{q}-2-\frac{\gamma^2}{2})^2-2q>0.$$
%     The same holds for $\bbP_\rmQ$ instead of $\bbP_\rmE$.
% \end{lemma}

\begin{lemma}\label{lem:SLE-contain}
    Let $\delta>0$. There exists $c>0$ depending only on $\delta$ and $\gamma$ such that  
        \begin{align}
            & \bbP_\rmE [B(0;{2^{-\delta k}}) \nsubseteq \eta([-1,1])] \lesssim 2^{-ck}; \label{eq:SLE-contain-ball}  \\
            & \bbP_\rmE [\eta([-2^{-\delta k},2^{-\delta k}])\nsubseteq B(0;1/2) ]\lesssim 2^{-ck}. \label{eq:SLE-contained-in-ball}
        \end{align}
\end{lemma}

% Recall $\fR_b$ from~\eqref{eq:Rb}, and $h^b,\eta^b$ from ~\eqref{eq:hb-etab}. By~\eqref{eq:Rb-bound}, for any $\delta>0$, for some constant $c$ depending only on $\delta$, we have
% \begin{equation}\label{eq:Rb-gammaQ}
%   1-  \bbP_\rmE[2^{(\frac{1}{\gamma Q}-\delta)k}\leq \fR_{2^k}\leq  2^{(\frac{1}{\gamma Q}+\delta)k}]\lesssim 2^{-ck}. 
% \end{equation}
% %\nina{*do you mean the complement of the event on the left side?*}
% Since $(h^b,\eta^b)$ has the same distribution as $(h,\eta)$, combined with Lemma~\ref{lem:SLE-contain} we have

\begin{lemma}\label{lem:SLE-contain-1}
    Let $\delta>0$. There exists $c>0$ depending only on $\delta$ and $\gamma$ such that  
        \begin{align}
            & \bbP_\rmE [B(0;{2^{(\frac{1}{\gamma Q}-\delta) k}}) \nsubseteq \eta([-2^k,2^k])] \lesssim 2^{-ck}; \label{eq:SLE-contain-ball-1}  \\
            & \bbP_\rmE [\eta([-2^{ k},2^{k}])\nsubseteq B(0;{2^{(\frac{1}{\gamma Q}+\delta)k}}) ]\lesssim 2^{-ck}; \label{eq:SLE-contained-in-ball-1} \\
            &  \bbP_\rmE [B(0;{2^{-(\frac{1}{\gamma Q}+\delta) k}}) \nsubseteq \eta([-2^{-k},2^{-k}])] \lesssim 2^{-ck}. \label{eq:SLE-contained-in-ball-2} 
        \end{align}
\end{lemma}

\begin{proof}
    The claims follows from Lemma~\ref{lem:SLE-contain} applied for $\delta/3$  along with the scaling as described in Lemma~\ref{lem:Rb} for $(b_1,b_2) = (1,2^{\pm k})$ and $C = 2^{\delta k/3}$.
\end{proof}

\begin{lemma}\label{lem:SLE-contain-2}
  Let $\gamma=\sqrt 2$.  There exists $c,c_1>0$ such that
    \begin{equation}
        \bbP_\rmQ [\bbD\nsubseteq \eta([-2^{c_1 k},2^{c_1k}]) ]\lesssim 2^{-ck}.
    \end{equation}
\end{lemma}

\begin{proof}
   {Following ~\cite[Lemma A.7]{Tutte} and the definition of $2^r$-embedding there, for $\fR_b$ as defined in Lemma~\ref{lem:Rb} and $\wt\eta(\cdot) = 2^r(\fR_{2^{\gamma Qr}})^{-1}\eta(\cdot)$, we have $\bbP_\rmQ[B(0;2^{-r/2})\nsubseteq \wt\eta([-2^{c_1r},2^{c_1r}])]\lesssim 2^{-cr}$. By Lemma~\ref{lem:Rb}, $\bbP_\rmQ[\fR_{2^{\gamma Qr}}\cdot 2^{-3r/2}<1]\lesssim 2^{-cr}$. This finishes the proof.}
\end{proof}

% Combining Lemma~\ref{lem:diam-cell-max} with~\eqref{eq:SLE-contained-in-ball} and~\eqref{eq:Rb-gammaQ} (for ${(1+\delta)k}$ instead of $k$), as we scale by $\fR_{2^{(1+\delta)k}}^{-1}$, we have
% \begin{lemma}\label{lem:SLE-diam-scale}
%     Let $\delta,\beta\in(0,1)$. There exists $c,c'>0$ depending only on $\delta$ and $\gamma$ such that  
%     \begin{equation}\label{eq:lem-SLE-diam-scale-0}
%         \bbP_\rmE \big[\diam(\eta(J))\leq 2^{\frac{(1+2\delta)k}{\gamma Q}-q(1-\beta+\delta)k} \ \text{for any interval } J\subseteq[-2^k,2^k] \text{ with } |J|\leq 2^{\beta k}\big]\geq 1-c'2^{-ck}.
%     \end{equation}
%     In particular,
%     \begin{equation}\label{eq:lem-SLE-diam-scale}
%         \bbP_\rmE \big[\diam(\eta(J))\leq 2^{\frac{(1+2\delta)k}{\gamma Q}-q k} \ \text{for any interval } J\subseteq[-2^k,2^k] \text{ with } |J|\leq 2^{\delta k}\big]\geq 1-c'2^{-ck}.
%     \end{equation}
% \end{lemma}

\subsection{The Mullin bijection}\label{subsec:Mullin}

The Mullin bijection~\cite{Mullin} is a bijection between a spanning-tree decorated map and a nearest- neighbor walk in $\bbZ^2$. In this subsection, we introduce the triangulations that naturally arise from the Mullin bijection.  We start with the infinite-volume version of the bijection and then we give the sphere and disk version of the bijection. 

\subsubsection{Mullin bijection for bi-infinite walks}\label{subsubsec:Mullin-infinite} 

Let $(M,\fre_0,\mathfrak{T})$ be the uniform infinite spanning-tree decorated planar map, which is the Benjamini-Schramm~\cite{BenjaminiSchramm01} limit of uniform triples consisting of a planar map with an oriented root edge and a distinguished spanning tree. This infinite-volume limit is proven to exist in~\cite{She16, Che17}. Let $M_*$ be the dual map of $M$ and $\mathfrak{T}_*$ be the dual spanning tree of $\mathfrak{T}$, i.e., the set of edges of $M_*$ which do not cross edges of $\mathfrak{T}$. Let $T$ be the triangulation with vertex set $\cV M\cup \cV M_*$, and edge set given by the union of  $\mathfrak{T}\cup\frT_*$ and edges of the form $\{v,v_*\}$, where $v\in\cV M$ is incident to the face that is dual to $v_*\in\cV M_*$. We declare that the root edge of $T$ is the edge $\{v,v_*\}$ oriented from $v$ to $v_*$ whose primal endpoint $v$ coincides with the initial end-point of $\fre_0$ and which is the first edge in clockwise order after $\fre_0$ with this property. There is a unique path  $\mathfrak{tr}:\bbZ\to\cF T$ in the dual of $T$ which hits each face of $T$ exactly once, does not cross $\frT$ or $\frT_*$ in the sense that $\mathfrak{tr}(t)$ and $\mathfrak{tr}(i-1)$ share an edge not belonging to $\frT\cup\frT_*$ for each $t\in\bbZ$, and it holds that $\mathfrak{tr}(0)$ (resp.\ $\mathfrak{tr}(1)$) is the face to the right (resp.\ left) of $\fre_0$.  We refer to e.g.\ ~\cite[Figure 6]{GHS-mot-distance} for an illustration. 
 
Consider a pair of nearest neighbor walk $\cZ = (\cL,\cR):\bbZ\to\bbZ^2$ such that
\eqb\label{eq:walk-liminf}\liminf_{n\to-\infty}\cL_n=\liminf_{n\to-\infty}\cR_n = \liminf_{n\to\infty}\cL_n=\liminf_{n\to\infty}\cR_n=-\infty.\eqe
As explained in~\cite[Section 4.1]{Che17} and~\cite[Section 3.1.2]{GHS-mot-distance}, one can recover  an instance of  the pair $(M,\fre_0,\mathfrak{T})$ using the \emph{sewing procedure}, which is an inductive relationship between certain submaps of the triangulation $T$. For $t\in\bbZ$, let $T _{t}$ be the submap of $T$ consisting of the vertices and edges on the boundaries of the triangles in $\mathfrak{tr}((-\infty,t])|_\bbZ$, and $\fre_t$ be a marked edge on the boundary of $T _t$. We can recover $(T _{t+1},\fre_{t+1})$ from $(T _{t},\fre_{t})$ by attaching a certain triangle (corresponding to $\mathfrak{tr}(t+1)$) with one edge equal to $\fre_t$ and one edge equal to $\fre_{t+1}$. If $\cZ_{t+1}-\cZ_t = (1,0)$ (resp.\ $(0,1)$), this triangle has two edges in the external face of $T _t$ and $\fre_{t+1}$ is the right (resp.\ left) one of these two edges if we face outward toward the external face. If $\cZ_{t+1}-\cZ_t = (-1,0)$ (resp.\ $(0, -1)$), then one of the edges of the new triangle coincides with the edge of $T_t$ immediately to the left (resp.\ right) of $\fre_t$, and the remaining edge lies in the external face of $T _t$ and equals $\fre_{t+1}$. The edges of the tree $\frT$ (resp.\ $\frT_*$) are precisely the edges of $T$ which are not crossed by $\mathfrak{tr}$ and which lie to the left (resp.\ right) of the path $\mathfrak{tr}$.  This generates a  locally finite (type II) infinite plane triangulation $T$ with a path $\mathfrak{tr}:\bbZ\to\cF T$ on its faces, {where we set the root edge of $T$ to be the edge $\fre_0$, oriented from the endpoint on $\frT$ to the endpoint on $\frT_*$.}  {See Figure~\ref{fig:sew0} for an illustration.}  {When $T$ is embedded on $\bbC$, we will also view $\frtr$ as a curve on $\bbC$ passing through the embedded faces.} We define the function $\Phi_\rmT$ on the set of nearest neighborhood walks on $\bbZ^2$ satisfying~\eqref{eq:walk-liminf} with $\Phi_\rmT(\cZ) = (T,\mathfrak{tr})$. The edges of the tree $\frT$ (resp.\ $\frT_*$)  are precisely the edges of $\frT$ which are not crossed by $\mathfrak{tr}$ and which lie to the left (resp.\ right) of the path $\mathfrak{tr}$. For each edge $e_*\in\frT_*$ its dual is the edge $e = \{v_1,v_2\}\in \cE M$ such that $v_1,v_2\in\cV M$ lie on the boundary of the two faces $f_1,f_2\in\cF T$ that are incident to $e_*$. Using these observations we can recover $(M,e_0,\frT)$ from T and $\mathfrak{tr}$ by letting $\cV M$ be the set of vertices to the left of $\mathfrak{tr}$ and by letting $\cE M$ be the union of $\frT$ and the dual edges to $\frT_*$. We let $\Phi_{\mathrm{M}}$ be the map sending $\cZ$ to $(M,e_0,\frT)$.  {For each edge $\fre_t$, we call its endpoint on $\frT$ (resp.\ $\frT_*$) the left (resp.\ right) endpoint of $\fre_t$.} 

%The edges of the tree $\frT$ (resp.\ $\frT_*$)  are precisely the edges of $\frT$ which are not crossed by $\mathfrak{tr}$ and which lie to the left (resp.\ right) of the path $\mathfrak{tr}$. For each edge $e_*\in\frT_*$ its dual is the edge $e = \{v_1,v_2\}\in \cE M$ such that $v_1,v_2\in\cV M$ lie on the boundary of the two faces $f_1,f_2\in\cF T$ that are incident to $e_*$.

\begin{figure}
    \centering
    \includegraphics[scale=0.8]{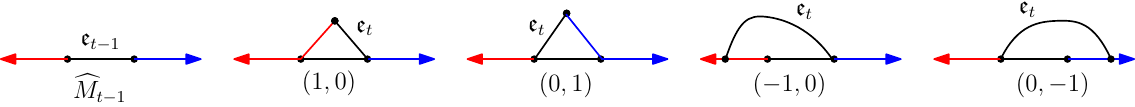}
    \caption{An illustration of the sewing procedure, where we generate the map $T _t$ and the marked edge $\fre_t$ from $T _{t-1}$ and $\fre_{t-1}$ according to the value of $\cZ_{t}-\cZ_{t-1}$. The red and blue edges are on the trees $\frT$ and $\frT_*$, respectively. The face added in the four rightmost figures is $\mathfrak{tr}(t)$.}
    \label{fig:sew0}
\end{figure}

For $T=\Phi_\rmT(\cZ)$, there is a natural correspondence between the vertex and edge sets $\cV T, \cE T$ and the steps of the walk $\cZ = (\cL,\cR)$ which we now describe. Let $\cJ_\rmL^+\subseteq\bbZ$ (resp.\ $\cJ_\rmL^-\subseteq\bbZ$) be the set of times at which $\cL$ goes up (resp.\ down), i.e.,
\eqb\label{eq:J-L-+}\cJ_\rmL^+ = \{t\in\bbZ:\cL_t = \cL_{t-1}+1\}, \ \ \ \cJ_\rmL^- = \{i\in\bbZ:\cL_t = \cL_{t-1}-1\}.\eqe
  Define $\cJ^+_\rmR$ and $\cJ^-_\rmR$ in the similar manner, and set $\cJ^- = \cJ^-_\rmL\cup\cJ^-_\rmR$, $\cJ^+ = \cJ^+_\rmL\cup\cJ^+_\rmR$.  {For $t\in\cJ:=\bbZ$, let $\frv_t$ be the vertex on the triangle $\mathfrak{tr}(t)$ that is not on the edge $\fre_t$.} Note that  {for $t\in\cJ^-$}, $t$ is the first time when $\frv_t$ is a non-boundary vertex of the map $T _t$, and $t\mapsto\frv_t$ gives  a bijection between $\cJ^-$ and $\cV T$.

 To work with the circle packing for $\Phi_\rmT(\cZ)$, we make the following observation on $\Phi_{\rm simple}\circ \Phi_\rmT(\cZ)$. Consider a pair of vertices $v_1,v_2$ of $T$ connected by edges $\fre_{t_1},...,\fre_{t_n}$ of $T$, and suppose $t_1<...<t_n$. Then from the sewing procedure, $\cL_t \geq \cL_{t_1}$, $\cR_t \geq \cR_{t_1}$ for $t_1\leq t\leq t_n$, and $(\cL_{t_n},\cR_{t_n}) = (\cL_{t_1},\cR_{t_1})$. In other words, $(\cZ_t-\cZ_{t_1})_{t_1\leq t\leq t_n}$ is an excursion on $\bbZ_{\geq0}^2$ starting and ending at $(0,0)$. In the definition of $\mr T=\Phi_{\rm simple}(T)$, the edges $(\fre_{t})_{t_1< t\leq t_n}$ and their endpoints other than $v_1$ and $v_2$ are removed from $T$, and the remaining triangulation is the same as the one generated from $\Phi_\rmT$ using the walk $(\cZ_t)_{\bbZ\backslash [t_1,t_n)}$.   From this observation, with slight abuse of notation, we define the operator $\Phi_{\rm simple}$ on  walks on $\bbZ^2$ as follows. Suppose $(\cZ_t)_{t\in \cJ}$ is  a nearest neighbor walk on $\bbZ^2$ where $\cJ=I\cap \bbZ$ and $I\subseteq\bbR$ is some interval. 
 We define $\mr\cJ$ be $\cJ$ minus all possible intervals  {$(t^-,t^+]$} where \eqb\label{eq:double-edge-Mullin}\cL_{t^-} = \cL_{t^+} = \inf_{s\in\cJ\cap [t^-,t^+]}\cL_s, \ \ \ \ \cR_{t^-} = \cR_{t^+} = \inf_{s\in\cJ\cap [t^-,t^+]}\cR_s.   \eqe
Under the event where $\mr\cJ\neq\emptyset$, let $\Phi_{\rm simple}(\cZ)$ be the  walk $\mr\cZ$ on $\bbZ^2$ such that $\mr\cZ_t = \cZ_{\mr t}$ where  {$\mr t = \sup\{\mr s\leq t: \mr s\in\mr\cJ\}$}, otherwise let $\mr\cZ$ be empty. In other words, $\Phi_{\rm simple}(\cZ)$ is the walk that erases all the $\bbZ_+^2$ excursions from $\cZ$.  We   extend the definition of $\Phi_{\rmT}$ to the setting where $\cZ_t-\cZ_{t-1}$ is allowed to be $(0,0)$, where in the sewing procedure we let the map $T _t = T _{t-1}$ and $\fre_t=\fre_{t-1}$. One can also define the space-filling path $\mr{\mathfrak{tr}}$ on $\Phi_{\rm simple}(T)$ by setting $\mr{\mathfrak{tr}}(t) = \mathfrak{tr}(\mr t)$ since the faces in $\cF\mr T$ can be identified with faces in $\cF T$, and we extend the definition of $\Phi_{\rm simple}$ to the tuple $(T,\mathfrak{tr})$.   Then from the sewing procedure and the definition we have the following.
\begin{lemma}\label{lem:basic-hat}
Let $(\cZ_t)_{t\in \cJ}$ be  a nearest neighbor walk on $\bbZ^2$ as above. 
    \begin{enumerate}
        \item Suppose $\cJ=\bbZ$ and $\cZ$ satisfies~\eqref{eq:walk-liminf}. Then   $\Phi_{\rm simple}\circ \Phi_\rmT(\cZ)=\Phi_\rmT\circ\Phi_{\rm simple} (\cZ)$. Here we  set $\Phi_\rmT(\mr \cZ)=\emptyset$ if $\mr\cZ$ is empty.
        \item  For any $t\in\mr\cJ$ and $a>0$, if the times $\inf\{s>t:\cL_s-\cL_t \leq -a\}$ and $\inf\{s>t:\cR_s-\cR_t \leq -a\}$ are finite, then they are in $\mr\cJ$.
        \item Suppose $t_1<t_2$ are in $\mr\cJ$. Then the restriction of the walk $\Phi_{\rm simple}(\cZ)$ to the interval $\mr\cJ\cap[t_1,t_2]$ is equal to $\Phi_{\rm simple}(\cZ|_{[t_1,t_2]})$.
    \end{enumerate}
\end{lemma}

\subsubsection{Mullin bijection in the finite volume setting}\label{subsubsec:Mullin-finite}

The previous section describes  functions $\Phi_\rmT$ and $\Phi_{\mathrm{M}}$ from bi-infinite nearest-neighbor walks on $\bbZ^2$ satisfying~\eqref{eq:walk-liminf} to triangulations and planar maps with a spanning tree. In this section we will extend the maps so that the input walk $\cZ = (\cZ_t)_{t=0,...,n}$ may be an arbitrary nearest-neighbor walk on $\bbZ^2$ of finite length. %Special cases of $\Phi_\mathrm{T}$ give bijections between certain classes of finite-length walks and natural families of tree-decorated planar maps. 

Consider a nearest-neighbor walk $\cZ = (\cZ_t)_{t=0,...,n}$ on $\bbZ^2$.  Let $\wt\cZ = (\wt\cZ_t)_{t\in\bbZ}$ be a nearest-neighbor walk on $\bbZ^2$ satisfying~\eqref{eq:walk-liminf} such that $\wt\cZ_t = \cZ_t$ for $t=0,...,n$. Suppose $\Phi_\rmT(\wt\cZ) = (\wt T,\wt{\mathfrak{tr}})$,  $(\wt\frT,\wt\frT_*)$ are the pair of trees generated from $\wt\cZ$, and $(\wt\fre_t)_{t\in\bbZ}$ are the labeled edges of $\wt T$.  Let $\Phi_\rmM(\wt\cZ) = (\wt M, \wt\frT,\wt\fre_0)$. We define $\Phi_\rmT(\cZ)$ to be the restriction of $(\wt T,\wt{\mathfrak{tr}})$ to time $1,...,n$, that is, we let $T$ be the submap of $\wt T$ consisted of faces from $\wt{\mathfrak{tr}}(1),...,\wt{\mathfrak{tr}}(n)$,  set $\mathfrak{tr}(t)= \wt{\mathfrak{tr}}(t)$ for $t=1,...,n$ and $\fre_t=\wt\fre_t$ for $t=0,...,n$.   We let $(\frT,\frT_*)$ be the restriction of $(\wt\frT,\wt\frT_*)$ to the edges of $T$. We let the root edge of $T$  be the first edge, oriented from the endpoint on $\frT$ to the endpoint on $\frT_*$.  {We let $M$ be the submap of $\wt M$ where the vertices are $\cV T\cap \cV\wt M$ and the edges are the edges   in $\frT$ and  edges of $\wt M$ that has both endpoints on $\cV M$ and is dual to an edge of  $\frT_*$  in the interior of $T$.} We set $\Phi_\rmM(\cZ) = (M,\frT)$, where  {the root edge of $M$ is chosen in an arbitrary way. Our definition of $\Phi_\rmM(\cZ)$ is slightly different from the version in~\cite{Tutte}, where we do not further draw an edge connecting the endpoints of the first edge and the last edge as in the definition of $\Phi_\rmM^\disk$ below; in this case the map $M$ does not necessarily have disk topology; see Figure~\ref{fig:PHIM} for an illustration. The reason here is, as in Theorem~\ref{thm:conv-0-cone-unif-Hcell}, we will first prove a version of Theorem~\ref{thm:unif-cone} in the setting of a planar map with whole-plane topology encoded by a walk $(\cZ_t)_{t\in\bbR}$ and later restrict to $(\cZ_t)_{t\in I}$ where $I$ is an interval, so it is important that the map encoded by $\Phi_\rmM((\cZ_t)_{t\in I})$ is a submap of the map encoded by  $\Phi_\rmM(\cZ)$.}    Following the sewing procedure description as in Figure~\ref{fig:sew0}, $\Phi_\rmT(\cZ)$, $\Phi_\rmM(\cZ)$ and $(\frT,\frT_*)$ do not rely on the choice of $\wt\cZ$. 

\begin{figure}
    \centering
    \includegraphics[scale=0.6]{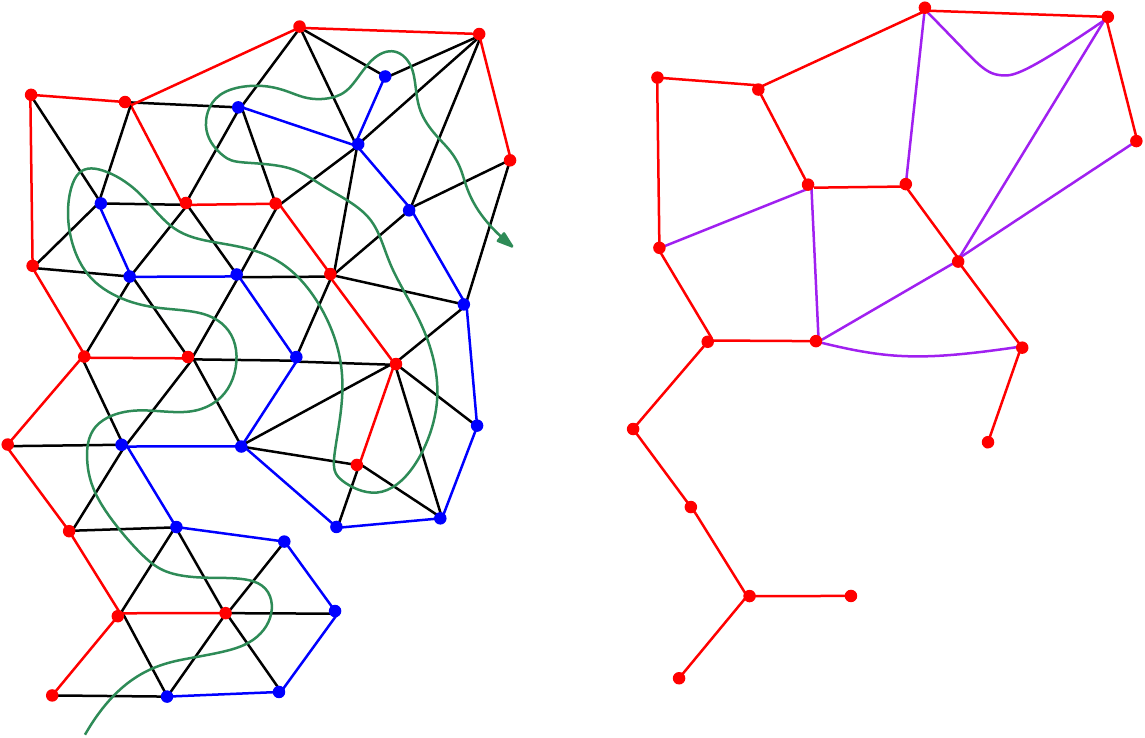}
    \caption{\textbf{Left}: An illustration of the triangulation $T$ and the path $\mathfrak{tr}$ generated from $\Phi_\rmT(\cZ)$, where $\cZ$ is a walk of 49 steps. \textbf{Right}: An illustration of the planar map defined via $\Phi_\rmM(\cZ)$, where the purple edges are those not on the spanning tree $\mathfrak{T}$.}
    \label{fig:PHIM}
\end{figure}

Now we consider two cases where the map $\Phi_\rmT$ further induces  bijections between certain classes of finite-length walks and natural families of tree-decorated planar maps. We start with the walk $\cZ\in\cA_n^\sph$ as defined in~\eqref{eq:Mullin-sph-walk}.
In this particular setting, $\Phi_\rmT(\cZ) = (T,\mathfrak{tr})$, where $T$ is a disk triangulation whose boundary is consisted of two edges $\fre_0,\fre_{2n}$, and we further identify  the edge $\fre_{2n}$ with $\fre_0$ to get a   sphere triangulation $T^{\rm sp}$ with $2n$ faces. {We declare that the root edge of $T^{\sph}$ is $\fre_0$.} Define $\mathfrak{tr}^{\rm sp}$ to be the rooted loop on faces of $T^{\rm sp}$ visiting the faces in the same order as $\mathfrak{tr}$. We define $\Phi_\rmT^{\rm sp}(\cZ) = (T^{\rm sp},\mathfrak{tr}^{\rm sp})$.    We define a planar map $M$ by setting its vertex set equal to the vertices on $\frT$, and the edge set equal to the union of $\frT$ and edges dual to $\frT_*$. We set $e_0$ to be the first edge   of $M$ in clockwise order around the starting vertex of the root edge $\fre_0$ (oriented from the endpoint on $\frT$ to the endpoint on $\frT_*$) of $T$, and we define $\Phi_{\mathrm{M}}^{\rm sp}(\cZ) = (M,e_0,\frT)$. Note that the map $t\mapsto \frv_t$ from the infinite volume setting   defines a bijection between the set of  times when $\cZ$ makes downward jumps and the vertices of $T$ other than the endpoints of  {the root edge}.

 {The following proposition is stated in~\cite[Proposition 2.2]{Tutte}, which essentially follows from the Mullin bijection as presented in e.g.~\cite{She16}.}

\begin{proposition}\label{prop:Mullin-sph}
    For any $n\in\bbN$, the map $\Phi_{\mathrm{M}}^{\rm sp}$ is a bijection between the following sets:
    \begin{itemize}
        \item The set $\cA_n^\sph$ of nearest-neighbor walks defined in terms of~\eqref{eq:Mullin-sph-walk}.
        \item The set of tuples $(M,e_0,\frT)$, where $(M,e_0)$ is a rooted planar map with $n$ edges and $\frT$ is a spanning tree on $M$.
    \end{itemize}
\end{proposition}

\begin{lemma}\label{lem:walk-excursion}
 Let $\ell,n\in\bbN$.   Let $\mathcal{A}_\ell^\disk$ be the set of nearest-neighbor walks defined as in~\eqref{eq:Mullin-disk-walk}, and let $\mathcal{A}_{\ell,2n+\ell}^\disk\subseteq \mathcal{A}_\ell^\disk$ be the set of walks with $2n+\ell-1$ steps. Then $\#\mathcal{A}_{\ell,2n+\ell}^\disk = \frac{\ell}{(n+1)(n+\ell)}\binom{2n+\ell-1}{n}\binom{2n+\ell+1}{n}$.  In particular,  $\#\mathcal{A}_n^\sph=\frac{1}{(n+1)^2}\binom{2n}{n}\binom{2n+2}{n}$ and $\#\mathcal{A}_n^\sph = (1+o_n(1))\pi^{-1}4^{2n+1}n^{-3}$ as $n\to\infty$.
\end{lemma}

\begin{proof}
    By the reflection principle, the number of walks $(\cX_k)_{k=0,...,2n+\ell-1}$ with jumps $\pm 1$ on $\bbZ$ such that $\cX_0=\ell-1$, $\cX_{2n+\ell-1}  = \inf_{0\leq k\leq 2n+\ell}\cX_{k}=0$ is $\binom{2n+\ell-1}{n}-\binom{2n+\ell-1}{n-1} = \frac{\ell}{n+\ell}\binom{2n+\ell-1}{n}$. Therefore
    \eqb
    \begin{split}
    \#\mathcal{A}_{\ell,2n+\ell}^\disk &= \sum_{k=0}^{n}\binom{2n+\ell-1}{2k+\ell-1}\frac{\ell}{k+\ell}\binom{2k+\ell-1}{k}\frac{1}{n-k+1}\binom{2n-2k}{n-k} =  \sum_{k=0}^n \frac{\ell\cdot(2n+\ell-1)!}{k!(k+\ell)!(n-k)!(n-k+1)!} 
    \\ &= \binom{2n+\ell-1}{n}\frac{\ell}{(n+1)(n+\ell)}\sum_{k=0}^n\binom{n+\ell}{n-k}\binom{n+1}{k} = \frac{\ell}{(n+1)(n+\ell)}\binom{2n+\ell-1}{n}\binom{2n+\ell+1}{n}.
    \end{split}
    \eqe
    The claim for $\#\mathcal{A}_n^\sph$ follows since $\#\mathcal{A}_n^\sph = \#\mathcal{A}_{1,2n+1}^\disk$.
\end{proof}

Next we work with the nearest-neighbor walk $\cZ = (\cZ_t)_{t=0,...,n}$ on $\bbZ^2$ such that 
\eqb\label{eq:walk-excursion-disk} \cZ_0 = (\ell,0), \ \  \cZ_n = (0,0); \ \ \cL_t\geq 0 \ \text{and} \ \ \cR_t\geq 0 \ \ \text{for } t=1,...,n,\eqe
where $\ell>0$ is a positive integer. Suppose $\Phi_\rmT(\cZ) = (T,\mathfrak{tr})$. In this case, $T$ is a disk triangulation whose boundary is consisted of $\fre_0,\fre_n$ and $\ell$ edges on $\frT$, where the endpoints of $\fre_0$ and $\fre_n$ on $\frT_*$ are the same.   We add an edge  $e_0$ connecting the endpoints of $\fre_0$ and $\fre_n$ on $\frT$ to get a disk triangulation $T^{\rm disk}$, oriented from the former to the latter.  {We declare that the root edge of $T^\disk$ is the edge $e_0$.} We define $\mathfrak{tr}^{\rm disk}(t) = \mathfrak{tr}(t)$ for $t=1,...,n$ and set $ \mathfrak{tr}^{\rm disk}(0)$ to be the triangular face consisted of $e_0,\fre_0$ and the last expolaration edge $\fre_n$, and we define $\Phi_\rmT^{\rm disk}(\cZ) = (T^{\rm disk},\mathfrak{tr}^{\rm disk})$. We define a planar map $M$ by setting its vertex set equal to the vertices on $\frT$, and the edges of $M$ are the union of edges in $\frT$, the edges  dual to edges $\frT_*$ and $e_0$. We further add $e_0$ to $\frT$ such that $\frT$ is a tree with wired boundary condition.  Define $\Phi_{\rm M}^{\rm disk}(\cZ) = (M,e_0,\frT).$ Observe that in this case,   the map $t\mapsto \frv_t$ from the infinite volume setting also defines a bijection between the set of  times when $\cZ$ makes downward jumps and the vertices of $T^\disk$ other than the endpoints of  {the last edge $\fre_n$}. % {We further define $\frv_0$ to be the vertex on both $\fre_0$ and $\fre_n$, and define $\frv_{n+1}$ to be the endpoint of $\fre_n$ that is not an endpoint of $\fre_0$, and write $v\mapsto \frt_v$ to be the inverse of this bijection.}

 {The following proposition is from~\cite[Proposition 2.5]{Tutte}.}

\begin{proposition}\label{prop:Mullin-bij-disk}
    For any $n\in\bbN$, the map $\Phi_{\mathrm{M}}^{\rm disk}$ is a bijection between the following sets:
    \begin{itemize}
        \item The set of nearest-neighbor walks $\cZ = (\cZ_t)_{t=0,...,2n-\ell-2}$ on $\bbZ^2$ satisfying~\eqref{eq:walk-excursion-disk}.
        \item The set of tuples $(M,e_0,\frT)$, where $(M,e_0)$ is a planar map with disk topology of boundary length $\ell+1$   with $n$ edges and $\frT$ is a spanning tree on $M$ with wired boundary condition where all the boundary edges  are on $\frT$.
    \end{itemize}
\end{proposition}

We further comment that in the above settings, for the planar map $M$ encoded by $\cZ$ via $\Phi_\rmM^\disk,\Phi_\rmM^\sph$ or  the infinite volume version of map $\Phi_\rmM$, the triangulation $T$ can naturally be embedded on the surface $\frM(M)$, and we will also view $\frtr$ as a path on the surface $\frM(M)$. %In the setting where $\cZ$ is a nearest-neighbor walk on $\bbZ^2$ and $T$ (resp.\ $M$) is the triangulation (resp.\ planar map) encoded by $\cZ$ via $\Phi_\rmT$ (resp.\ $\Phi_\rmM$), each face of $T$ has one or two vertices in $\cV M$, and we also view $\frtr$ as a path on $\frM(M)$ where $\frtr(t)$ is taken to be the vertex in $\cV M$ on the corresponding face of $T$. 
% {[*To do: define $\Phi_M$ for walk*]}

\section{The cell configuration for tree weighted maps}\label{sec:3}

In this section, we construct cell configurations $\mr\cH^\cell$ and  $\cH^\cell$ whose associated maps are encoded by a continuous-time  walk $\cZ$ on $\bbZ^2$ and prove that they satisfy the conditions of Theorem~\ref{thm:CellSystemPack} and Theorem~\ref{thm:invariance-uniformization}.  In Section~\ref{subsec:time-sets}, we introduce the time sets $\cT_t^{\rmC,d}$  using the $(Z,\cZ)$  as in~\cite[Section 5]{Tutte}. To address the collapse of multiple edges in circle packing, we will also construct a version of the time sets $\mr\cT_t^{\rmC,d}$, and prove their moment bounds in Section~\ref{subsubsec:3.1.2} and Section~\ref{subsubsec:3.1.3}. In Section~\ref{subsec:cellsystem}, we construct the cell configurations $\mr\cH^\cell$ and  $\cH^\cell$ and prove that they satisfy the conditions in Theorem~\ref{thm:CellSystemPack} and Theorem~\ref{thm:invariance-uniformization}. In Section~\ref{subsec:3.3}, we prove the convergence of the  map and trace $\frtr$ encoded by the walk $\cZ$ under circle packing/Riemann uniformization embedding in the setting of 0-quantum cone. 

Recall the definition of the probability measures $\bbP_\rmE$ (resp.\ $\bbP_\rmQ$) describing the law of the tuple $(h,h^{\rm IG},\eta, \fU)$  as introduced in Definition~\ref{def:P_EP_Q} where $\fU$ is indexed by the 1D uniform dyadic system $\fD$,  and the boundary length process $(Z_t)_{t\in\bbR}$ for $(h,\eta)$. Recall that for $k\in\bbZ$, we write $I(k)$ for the interval in $\fD$ of length $2^k$ containing 0. With slight abuse of notation, for $x\in\bbR$, we will write $I(x)$ for $I(\lceil x\rceil)$. %Let $\fU = (U_I)_{I\in\fD}$ be a collection of i.i.d.\ uniform random variables on $[0,1]$ indexed by intervals in $\fD$ and independent from everything else. 
We write $\Omega_\fU$ (resp.\ $\Omega_\fD$) for the space of all the realizations of $\fU$ (resp.\ $\fD$).  
For $t_0\in\bbR$, recall that $\fD-t_0 = \{I-t_0:I\in \fD\}$, and we set $\fU - t_0 = (\wt U_J)_{J\in \fD-t_0}$ where $\wt U_J = U_{J+t_0}$. For an interval $J\subseteq\bbR$, let $D(J,\bbR^2)$ denote the set of c\`{a}dl\`{a}g real-valued functions  $J\to \bbR^2$.  {The followings  are from ~\cite[Theorem 3.17, Theorem 4.2]{Tutte}.}

\begin{theorem}\label{prop:KMT-gamma-cone}
    Under the probability measure $\bbP_\rmQ$, there exists a process $(\cZ_t)_{t\in\bbR} = (\cL_t,\cR_t)_{t\in\bbR}$ which is right continuous and has left limit and satisfies the following.
    \begin{enumerate}[(i)]
        \item The law of $(\cZ_t)_{t\in\bbR}\in D(\bbR,\bbR^2)$ is a continuous-time nearest neighbor walk on $\bbZ^2$, such that  the set of times $\cJ$ at which $\cZ$ makes a jump is given by a Poisson point process  {on $\bbR$ with unit rate}, and conditioned on $\cJ$ the steps of $\cZ$ are i.i.d.\ and uniform in $\{(1,0),(0,1),(-1,0),(0,-1)\}$. \label{cond:prop:KMT-gamma-cone-(i)}
        \item For any $C_1>1$, there exists $C'>1$ such that for any $m\in\bbN$, 
        $$\bbP_\rmQ\big[\sup_{t\in I(m)}\|Z_t-\cZ_t\|\geq C'm\big]\leq 2^{-C_1m}$$
        where $I(m)$ is the interval in $\fD$ of length $2^m$ containing 0; \label{cond:prop:KMT-gamma-cone-(ii)}
        \item %There exists a constant $c>0$ and functions $F_m:D(\bbR,\bbR^2)\times\Omega_\fU\to D(\bbR,\bbR^2)$ such that if one sets $(\cZ_t^m)_{t\in\bbR}= F_m(Z|_{I(m)},\fU)$, then for $k,m\in\bbN$,
        There exists a measurable function $g$ such that $\cZ = g(h,h^\IG, \fU).$ \label{cond:prop:KMT-gamma-cone-(iii)}
    \end{enumerate}
\end{theorem}

 %{ The above coupling also extends to the setting of 0-quantum cone as below.}
\begin{theorem}\label{thm:KMT-0-cone} 
     Under the probability measure $\bbP_\rmE$, we have the following.
     \begin{enumerate}[(i)]
         \item Let  { $g$} be the same as in Theorem~\ref{prop:KMT-gamma-cone}, and define $\cZ= g(h,h^\IG, \fU)$. Then the walk $\cZ$ %given by
        % $$\cZ|_{I(k)} = \lim_{m\to\infty} \cZ^m|_{I(k)}$$
         is a well-defined continuous time walk in $\bbZ^2$. The process  $(\cZ_t)_{t\in\bbR}\in D(\bbR,\bbR^2)$   makes jumps in $\{(0,1),(1,0),(0,-1),(-1,0)\}$, and makes finitely many jumps during any finite time   interval. Furthermore, $g$ is translation invariant modulo scaling, in the sense that for any $z\in\bbC$ and $C>0$, if we set $\tau_z$ to be the time when the space-filling curve $\eta$ hits $z$, then
         \begin{equation}\label{eq:shift-inv-Z}
             (\cZ_{t+\tau_z}-\cZ_{\tau_z})_{t\in\bbR} = g(h(C\cdot+z)+Q\log C, h^{\IG}(C\cdot+z), \fU-\tau_z).
         \end{equation} \label{cond:thm:KMT-0-cone-(i)} 
         \item For any $C_1>1$, there exists $C',C''>1,$ such that for any $m\in\bbN$, 
        $$ {\bbP_\rmE\big[\sup_{t\in [-2^m,2^m]}|Z_t-\cZ_t|\geq C'm\big]\leq C''2^{-C_1m};}$$   \label{cond:thm:KMT-0-cone-(ii)} 
        % \item There exists a constant $c>0$ such that for $F_m$ same as Theorem~\ref{prop:KMT-gamma-cone} and $(\cZ_t^m)_{t\in\bbR}= F_m(Z|_{I(m)}, \fU)$,
        % $$\bbP_\rmE\big[\cZ^{m+k}|_{ I(m)}\neq \cZ|_{I(m)}\big]\leq 2^{-ck};$$ %\label{cond:thm:KMT-0-cone-(iii)} 
        \item Let $G:D(\bbR,\bbR^2)\times\Omega_\fD\to\{0,1\}$ be a measurable function and let \eqb\label{eq:def-A} A = \{\exists s\in I(0):G(\cZ(\cdot +s)-\cZ(s),\fD-s)=1\}.\eqe
        Suppose $A\in\sigma(\cZ|_{I(k)},I(k))$. Then there exists a constant $C_0>1$ such that for all $k\in\bbN$,
        \begin{equation}\label{eq:KMT-0-cone}
            \bbP_\rmE[A] \leq \bbP_\rmQ[A]^{1/4}2^{C_0k}+2^{-k/C_0}.
        \end{equation} \label{cond:thm:KMT-0-cone-(iv)} 
     \end{enumerate}
\end{theorem}

Without specially mentioned, throughout the rest of the paper we set $\gamma=\sqrt 2$ and $\kappa=8$.

\subsection{The time sets}\label{subsec:time-sets}

To construct our cell configurations using $(h,\eta)$ under the probability measure $\bbP_\rmE$, we introduce the time sets induced from the coupling between the walk $\cZ$ and the boundary length process $Z$ in Theorem~\ref{thm:KMT-0-cone}.  The definitions are mostly the same as ~\cite[Section 5]{Tutte}, except that we will make some modifications to work with the collapsing of the double edges as in the setting of circle packing.

We start by introducing the following   notations for walks. Let $J\subseteq \bbR$ be an interval. 
\begin{definition}\label{def:walk}
    We call a process $\cZ:J\to\bbZ^2$ a continuous-time nearest neighbor walk on $\bbZ^2$, if the trajectories $t\mapsto \cZ_t$ are right-continuous and with a finite number of jumps in any bounded interval, and the jumps of the process take their values in $\{(-1,0),(0,1),(0,-1),(1,0)\}$. We   define $\Delta\cZ_t = \cZ_t - \lim\limits_{s\uparrow t}\cZ_{s}$. Let $\cZ = (\cL,\cR):J\to\bbZ^2$ be a continuous-time  nearest neighbor walk on $\bbZ^2$. We further say $\cZ$ has \emph{rate $\lambda>0$} if the jump times for $\cZ$ is distributed as a Poisson point process with intensity  $\lambda$ times the Lebesgue measure on $J$, {and conditioned on the jump times, the jumps are independent and uniform among the four directions}. Let $\cJ = \{t\in J:\Delta\cZ_t\neq 0\}$ be the set of times when $\cZ$ makes a jump,  and let
   \begin{align}
       & \cJ_{\rmL} = \{t\in J:   \Delta\cL_t  \neq 0\},  \ \ \cJ^-_{\rmL} = \{t\in J:   \Delta\cL_t  =-1\}, \    \cJ^-_{\rmR} = \{t\in J:   \Delta\cR_t =-1\}, \    \cJ^-  = \cJ^-_\rmL\cup\cJ^-_\rmR;\\
       &\cJ_{\rmR} = \{t\in J:   \Delta\cR_t  \neq 0\},  \ \ \cJ^+_{\rmL} = \{t\in J:   \Delta\cL_t  =1\}, \    \cJ^+_{\rmR} = \{t\in J:   \Delta\cR_t =1\}, \      \cJ^+  = \cJ^+_\rmL\cup\cJ^+_\rmR.
   \end{align}
 {In the case $J = [a,b]$ is finite, we also include the time $a$ in $\cJ$.}
\end{definition}

In the case  that $J=\bbR$, we assume
   \begin{equation}\label{eq:walk-liminf-1}
       \liminf_{t\to-\infty} \cL_t= \liminf_{t\to-\infty} \cR_t=  \liminf_{t\to\infty} \cL_t= \liminf_{t\to\infty} \cR_t= -\infty.
   \end{equation}
  
   For $t\in \cJ^-$, define 
   \begin{equation}\label{eq:def-I_t-simple}
       I_{t} = [t^-,t],\  \text{where } \ t^- :=  \left\{
	\begin{aligned}
		\sup\{s\in\cJ^-_\rmL: s<t\}, & \ \ \text{if} \ \ t\in\cJ^-_\rmL \\
		\sup\{s\in\cJ^-_\rmR: s<t\}, & \ \ \text{if} \ \ t\in\cJ^-_\rmR. 
	\end{aligned}
    \right.
   \end{equation}
Notice that $\{I_t:t\in\cJ^-_\rmL\}$ is a collection of intervals with union $\bbR$ and pairwise disjoint interior, and that the same holds with $\cJ^-_\rmR$ instead of $\cJ^-_\rmL$.

Recall that for a nearest-neighbor walk $\bbZ\to\bbZ^2$, using the Mullin bijection, it generates a infinite plane triangulation $T$ and a path $\mathfrak{tr}$ on faces of $T$ via the map $\Phi_\rmT$. In the rest of the paper, we are going to extend the definition of the maps $\Phi_\rmT$ and $\Phi_\rmM$ to continuous-time nearest-neighbor walks $\cZ$ on $\bbZ^2$ satisfying~\eqref{eq:walk-liminf-1} and work with this extension.   {That is, consider an increasing bijection $\phi:\bbZ\to \cJ$ such that $\phi(0)$ is the time in $\cJ$ closest to 0. Then $\wt\cZ:=(\cZ_{\phi(n)})_{n\in\bbZ}$ defines a nearest-neighbor walk, and let $\Phi_\rmT(\wt \cZ) = (T,\wt{\mathfrak{tr}})$. Let $\wt\fre_n$ be the edge in $T\backslash(\frT\cup\frT_*)$ on both $\wt{\mathfrak{tr}}(n)$ and  $\wt{\mathfrak{tr}}(n+1)$. Let $\mathfrak{tr}(t) = \wt{\mathfrak{tr}}(n_t)$, where $n_t = \sup\{n\in\bbZ:\phi(n)\le t\}$, and let $\fre_t = \wt\fre_{n_t}$. Then we define $\Phi_\rmT(\cZ) = (T,\mathfrak{tr})$ and set $\Phi_\rmM(\cZ) = \Phi_\rmM(\wt\cZ)$. The same applies for the setting where the continuous-time nearest-neighbor walk $\bbZ^2$ is defined on a finite interval.}   {For $t\in\cJ$, let $\frv_t$ be the vertex on the triangle $\mathfrak{tr}(t)$ that is not on the edge $\fre_t$.} For the edge $\fre_t$, we  call its endpoints on $\frT$ and $\frT_*$ left and right endpoints of $\fre_t$, respectively.  Let $T _t\subseteq T$ be the map formed by $\{\mathfrak{tr}(s):s\leq t\}$. Note that  {for $t\in\cJ^-$}, $t$ is the first time when $\frv_t$ is a non-boundary vertex of the map $T _t$, and $t\mapsto\frv_t$ gives  a bijection between $\cJ^-$ and $\cV T$.

 {Following ~\cite[Section 5]{Tutte},} we define two graphs with vertex set $\cJ^-$, namely $(\cJ^-,\cE_\rmT)$ and $(\cJ^-,\cE_\rmG)$, where $\cE_\rmT,\cE_\rmG\subseteq \{\{s,u\}:s,u\in \cJ^- \}$ are the collection of edges. The letter ``$\rmT$" (resp. ``$\rmG$") is chosen since it is the first letter of the word ``triangulation" (resp.\ ``geometric"). For $s,u\in\cJ^-$, we say $\{s,u\}\in\cE_{\rm T}$ if and only if the vertices $\mathfrak{v}_s$ and  $\mathfrak{v}_u$ are adjacent in the triangulation $T$ encoded by $\cZ$. For $s,u\in\cJ^-$, we say $\{s,u\}\in\cE_{\rm G}$ if and only if at least one of the {three} following properties hold:
\begin{enumerate}[(i)]
    \item  there are $\wt s\in I_s$ and $\wt u\in I_u$ such that $$ L_{\wt s}=L_{\wt u} = \inf_{t\in[\wt s\wedge\wt u,\wt s\vee\wt u]}L_t;$$ \label{cond:def-cG-(i)}
    \item the property in condition~(\ref{cond:def-cG-(i)}) holds with $R$ and $\rmR$ instead of $L$ and $\rmL$, respectively, or \label{cond:def-cG-(ii)}
    \item $I_s\cap I_u\neq\emptyset$. \label{cond:def-cG-(iii)}
%    \item $I_s\cap I_u\neq\emptyset$.
\end{enumerate}
As explained in~\cite[Section 8.2]{DMS14}, property~(\ref{cond:def-cG-(i)}) or~(\ref{cond:def-cG-(ii)}) is satisfied if and only if $\eta(I_s)$ and $\eta(I_u)$ have non-empty intersection. We say a subset $S\subseteq \cJ^-$ is $\rmT$-connected (resp.\ $\rmG$-connected) if $S$ is connected in the graph  $(\cJ^-,\cE_\rmT)$  (resp.\ $(\cJ^-,\cE_\rmG)$).

 To deal with circle packings, we need the following modification in light of the collapse of double edges.  Recall the map  $\Phi_{\rm simple}$ on walks which erases the $\bbZ_+^2$ cone excursions, and we extend its definition to continuous-time nearest-neighbor walks. To be more precise, let $\mr\cJ$ be $\cJ$ minus all possible intervals  {$(t^-,t^+]\subseteq J$} such that $t^-,t^+\in\cJ$ and \begin{equation}\label{eq:double-edge-Mullin-A} 
     \cL_{t^-}=\cL_{t^+} = \inf_{s\in[t^-,t^+]} \cL_s, \ \ \  \cR_{t^-}=\cR_{t^+} = \inf_{s\in[t^-,t^+]} \cR_s,
 \end{equation}
 and let {$\mr\cZ_t  =(\mr \cL_t,\mr \cR_t) = \cZ_{\mr t}$ where $\mr t = \sup\{\mr s\leq t:\mr s\in \mr\cJ\}$.}  We set $\Phi_{\rm simple}(\cZ) = \mr\cZ$, and 
\begin{align}
       &\mr\cJ_{\rmL} = \{t\in J:   \Delta\mr\cL_t  \neq 0\},   \ \mr\cJ^-_{\rmL} = \{t\in J:   \Delta\mr\cL_t  =-1\}, \    \mr \cJ^-_{\rmR} = \{t\in J:   \Delta\mr\cR_t =-1\}, \    \mr\cJ^-  = \mr\cJ^-_\rmL\cup\mr\cJ^-_\rmR;\label{eq:def-hat-JL} \\
       &\mr\cJ_{\rmR} = \{t\in J:   \Delta\mr\cR_t  \neq 0\}, \ \ \ \mr\cJ^+_{\rmL} = \{t\in J:   \Delta\mr\cL_t  =1\}, \ \  \mr \cJ^+_{\rmR} = \{t\in J:   \Delta\mr\cR_t =1\}, \    \ \ \mr\cJ^+  = \mr\cJ^+_\rmL\cup\mr\cJ^+_\rmR.\label{eq:def-hat-JR}
   \end{align}
Note that by Lemma~\ref{lem:basic-hat}, we have $\Phi_\rmT(\mr\cZ) = \Phi_{\rm simple}\circ \Phi_\rmT(\cZ)$. See also Figure~\ref{fig:doubleedgewalk} for an illustration.

\begin{figure}[t]
    \centering
    \includegraphics[scale=0.7]{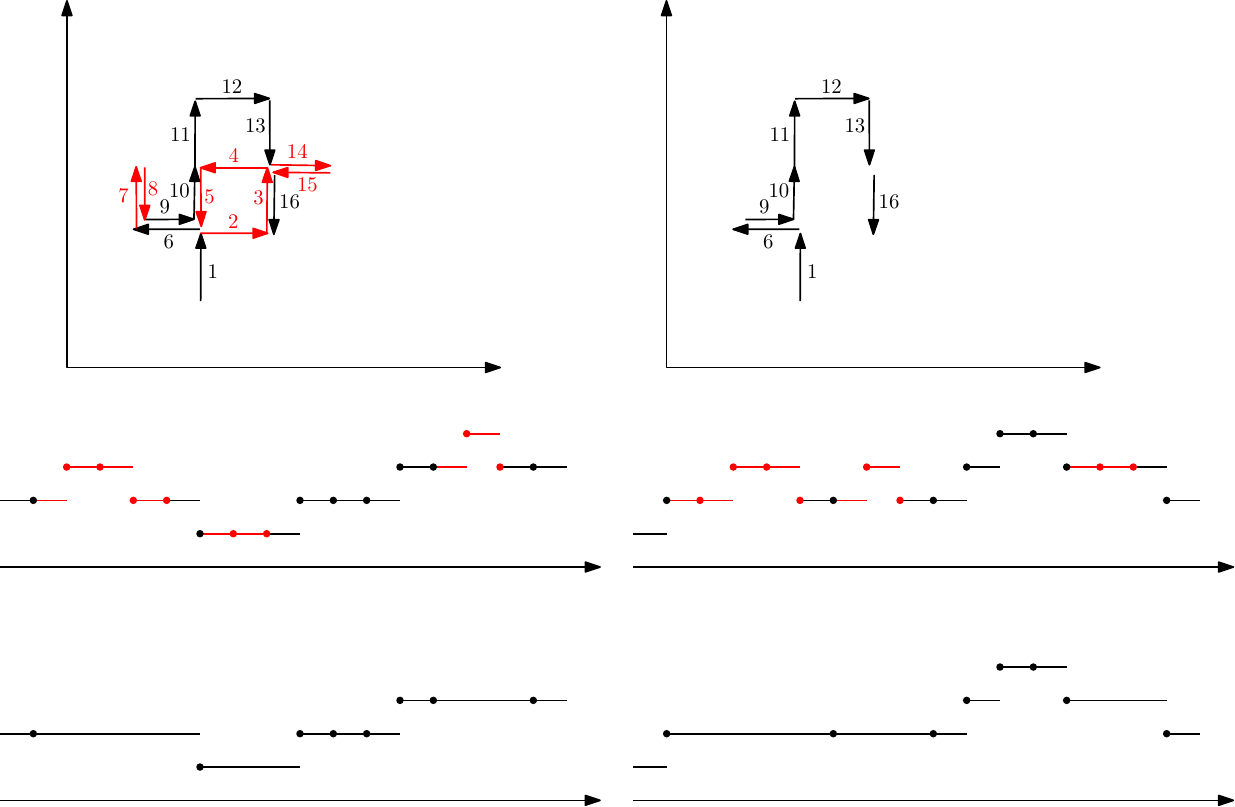}
    \caption{The top left panel illustrates a continuous-time random walk, with the steps of the walks forming a $\bbZ_+^2$ excursion is marked in red. On the top right panel, we remove these steps from the random walk. The two panels in the middle are the coordinates of the walk $\cZ$, with all jump  times in $\cJ$ are represented by the black dots, and the red dots are those to be removed in the definition of $\mr\cJ$. The two panels below  are the coordinates of the walk $\mr\cZ=\Phi_{\rm simple}(\cZ)$ with jump time $\mr\cJ$.}
    \label{fig:doubleedgewalk}
\end{figure}

The map $\Phi_{\rm simple}$ is not local, in the sense that $\mr\cZ|_{[a,b]}$ is not determined by $\cZ|_{[a,b]}$. To circumvent this,  we will frequently use the following simple observation in the remainder of this paper.

\begin{lemma}\label{lem:A2}
     Let $a_0<a<b<b_0$, $J_0=[a_0,b_0]$, and $J\subseteq\bbR$ be an interval such that $J_0\subseteq J$. Suppose $\cZ:J\to\bbZ^2$ and $\cZ_0:J_0\to\bbZ^2$ are continuous-time nearest-neighbor walks on $\bbZ^2$ such that $\cZ|_{J_0}=\cZ_0$. Let $\mr \cZ=\Phi_{\rm simple}(\cZ)$ and $ \mr \cZ_0=\Phi_{\rm simple}(\cZ_0)$.  Then if $\mr\cJ\cap (a_0,a)\neq\emptyset$ and $\mr\cJ\cap (b,b_0)\neq\emptyset$, then $\mr\cZ_0|_{[a,b]} = \mr\cZ|_{[a,b]}$.
\end{lemma}

Different from~\eqref{eq:def-I_t-simple}, for $t\in\mr \cJ$ we define 
\begin{equation}\label{eq:def-I_t-double}
    \mr I_t = [\mr t^-,t], \text{ where } \mr t^- = \sup\{ s\in\mr \cJ^-: s<t\}.
\end{equation}
    
%as inusing $(\mr\cJ^-_\rmL,\mr\cJ^-_\rmR)$ in place of $(\cJ^-_\rmL,\cJ^-_\rmR)$. 
{Note that different from the definition of $I_t$, which is an interval between downsteps of the same coordinate of the walk $\cZ$,  the interval $ \mr I_t$ above records intervals between downsteps of either coordinates of the simplified walk $\mr \cZ$}, and  $\{\mr I_t:t\in\mr \cJ^-\}$ is a collection of intervals with union $\bbR$ and pairwise disjoint interior. We define two graphs  $(\mr\cJ^-,\mr\cE_{\rm T})$ and $(\mr\cJ^-,\mr\cE_{\rm G})$ exactly in the same way as  $(\cJ^-,\cE_{\rm T})$ and $(\cJ^-,\cE_{\rm G})$, with $\mr\cJ^-$ and $(\mr I_t)_{t\in\mr\cJ^-}$ instead of $\cJ^-$ and $(I_t)_{t\in\cJ^-}$, respectively. That is, the vertex set for the two graphs are $\mr\cJ^-$, and for $s,u\in\mr\cJ^-$, $\{s,u\}\in\mr\cE_{\rm T}$ if and only if the vertices for $s$ and $u$ in the triangulation $\mr T = \Phi_{\rm simple}(T)$ encoded by $\mr\cZ$ are connected by an edge, whereas $\{s,u\}\in\mr\cE_{\rm G}$ if and only if at least one of the conditions holds:

\begin{enumerate}[(i)]
    \item there are $\wt s\in \mr I_s$ and $\wt u\in \mr I_u$ such that \eqb\label{eq:def-G-connected-1-1} L_{\wt s}=L_{\wt u} = \inf_{t\in[\wt s\wedge\wt u,\wt s\vee\wt u]}L_t;\eqe \label{cond:def-cG-(i)-1}
    \item the property in condition~(\ref{cond:def-cG-(i)-1}) above holds with $R$ and $\rmR$ instead of $L$ and $\rmL$, respectively.\label{cond:def-cG-(ii)-1}
\end{enumerate}
%~(\ref{cond:def-cG-(i)})-~(\ref{cond:def-cG-(iii)}) as in the definition of $\cE_{\rm G}$ hold for  $(\mr I_s,\mr I_s)$ in place of $(I_s,I_u)$.  %Note that these definitions are equivalent with the corresponding ones where we work with $\mr\cZ$ in place of $\cZ$.

    The following relations are apparent from the definition.
\begin{lemma}\label{lem:hat}
    Let $J=[t_J^-,t_J^+]$ be a finite interval. Suppose $\mr\cJ\cap J=\emptyset$. Then there exists $t^-<t^+$ such that $J\subseteq [t^-,t^+]$ and~\eqref{eq:double-edge-Mullin-A} holds. % Then $J\cap\mr\cJ\neq\emptyset$.
\end{lemma}

\begin{proof}
    We say a finite interval $[t^-,t^+]$ is bad if  ~\eqref{eq:double-edge-Mullin-A} holds (where $t^-,t^+$ are not necessarily in $\cJ$). Observe that if $I_1,I_2$ are bad and $I_1\cap I_2\neq\emptyset$, then $I_1\cup I_2$ is also bad. If $J\cap\cJ=\emptyset$, then $\cZ$ is constant over $J$ and the conclusion holds. Otherwise, let $t_1,...,t_p = J\cap\cJ$. Then for each $j=1,...,p$, $t_j\notin \mr\cJ$ and therefore there exists $t_j^-<t_j\leq t_j^+$ such that~\eqref{eq:double-edge-Mullin-A} holds for $(t^-,t^+)=(t_j^-,t_j^+)$. Since the walk is constant between times in $\cJ$, we may further assume that $t_J^-\geq t_1^-$ and $t_j^-\leq t_{j-1}$ for $j\geq 2$. Then $\cup_{j=1}^p[t_j^-,t_j^+]$ is a bad interval covering $[t_J^-,t_p]$, and since $\cZ$ is constant over $[t_p,t_J^+]$, $[t_p,t_J^+]$ is also a bad interval. Therefore  $[t_p,t_J^+]\cup(\bigcup_{j=1}^p[t_j^-,t_j^+])$ is a bad interval covering $J$, which finishes the proof.  %Now assume on the contrary. Then if $[t^-,t^+]$ is a bad interval containing $t_J^-$, then $t^+<t_J^+$, and the assumption that $\liminf_{t\to-\infty}\cL_t=-\infty$ implies that there is a lower bound on possible values of $t^-$. Therefore one can define the maximal bad interval $[\wt t^-,\wt t^+]$ containing $t^-_J$, and our assumption implies that $t' = \inf\{s>\wt t^+:s\in\cJ\}\in J.$ Moreover, the property where $I_1\cup I_2$ is bad if $I_1,I_2$ are both bad and have nonempty intersection along with the maximality of $\wt t^+$ implies that there does not exist any $s_1,s_2\in\cJ$ such that $s_1<t'$, $s_2\geq t'$ and~\eqref{eq:double-edge-Mullin-A} holds for $(s_1,s_2)$ in place of $(t^-,t^+)$, which further implies that $t'\in J\cap \mr \cJ$. This finishes the proof.
\end{proof}
 
 \begin{lemma}\label{lem:hat-minus}
     Suppose $t_0\in\mr\cJ$. Then for any $a>0$, the time $t_{\rmL,a}:=\inf\{s>t_0:\cL_s\leq \cL_{t_0}-a\}$ is in $\mr\cJ^-_\rmL$ and $t_{\rmR,a}:=\inf\{s>t_0:\cR_s\leq \cR_{t_0}-a\}$ is in $\mr\cJ^-_\rmR$. Furthermore, for every $t_0\in\mr\cJ^-$ with $\mr I_{t_0} = [t_0^-,t_0]$, for each $t\in [t_0^-,t_0)$, $\cL_t\geq \cL_{t_0^-}$ and   $\cL_{t_0}\geq \cL_{t_0^-}-1$, and the same holds $\cR$ in place of $\cL$.
 \end{lemma}
 \begin{proof}
     We prove $t_{\rmL,a}\in \mr\cJ^-_\rmL$ only. Indeed $t_{\rmL,a}\in  \cJ^-_\rmL$. Assume on the contrary $t_{\rmL,a}\notin \mr\cJ^-_\rmL$. Then there exists $t^-<t_{\rmL,a}\leq t^+$ such that~\eqref{eq:double-edge-Mullin} holds.  Then we must have $\cL_{t^-}\leq \cL_{t_{\rmL,a}}$ and $t^-<t_0$, and this contradicts with the fact that $t_0\in\mr\cJ$. The second claim is immediate from the construction of $\mr\cJ$.
 \end{proof}

 {
The following is an immediate consequence of the sewing procedure and Lemma~\ref{lem:hat-minus}.
\begin{lemma}\label{lem:sew}
    For   $t\in\cJ$, let  $v_\rmL$ and $v_\rmR$ be the left and right endpoint of the edge $\fre_t$, respectively. Then $v_\rmL = \frv_{t_\rmL}$ and $v_{\rmR} = \frv_{t_\rmR}$, where
    \eqb\label{eq:lem:sew} t_\rmL = \inf\{s>t:\cL_s = \cL_t-1\}, \ \ \ t_\rmR = \inf\{s>t:\cR_s = \cR_t-1\}.\eqe
    In particular, if $t$ is in $\cJ^-$ (resp.\ $\mr\cJ^-$), then $\{t,t_\rmL\},\{t,t_\rmR\}$ are edges in $ \cE_\rmT$ (resp.\ $ \mr\cE_\rmT$).
   Furthermore, for $t\in\cJ^+_\rmL$, let $v_{\rmL,2}$ be the vertex in the triangle $\mathfrak{tr}(t)$ that is not on the edge $\fre_t$. Then $v_{\rmL,2} = \frv_{t_{\rmL,2}}$ where \eqb\label{eq:lem:sew-1} t_{\rmL,2} = \inf\{s>t:\cL_s = \cL_t-2\}, \ \ \  t_{\rmR,2} = \inf\{s>t:\cR_s = \cR_t-2\}.\eqe The same statement holds for $t\in\cJ^+_\rmR$ and $(\cL,\rmL)$ replaced by $(\cR,\rmR)$. 
\end{lemma}
}

In the rest of this  subsection, we will introduce the time set $\cT_t^{\rmC,d}$  where $t\in\cJ^-$ and $d\in\bbN$ defined in~\cite[Section 5.1]{Tutte}, and introduce its variant  $\mr\cT_t^{\rmC,d}$ defined via  $\{I_t:t\in\mr \cJ^-\}$ and $\mr\cJ^-$. The letter ``$\rmC$"   is chosen since it is the first letter of the word ``continuum".   We will later use $\cT_t^{\rmC,d}$ and  $\mr\cT_t^{\rmC,d}$  to construct our cell configurations in Section~\ref{subsec:cellsystem}.

\begin{proposition}\label{prop:time-set-A-0}
 There exists a.s.\ finite subsets $(\cT_{s}^{\mathrm{C},{d}})_{s\in\cJ^-}$   of $\cJ^-$ such that the following is true. For each $s_0,u_0\in \cJ^-$, we have:
    \begin{enumerate}
        \item $s_0\in \cT_{s_0}^{\mathrm{C},{d}}$.\label{cond:prop:time-set-A-0-1}
        \item $s_0\in \cT_{u_0}^{\mathrm{C},d}$ if and only if $u_0\in \cT_{s_0}^{\rmC,d}$.\label{cond:prop:time-set-A-0-2}
        \item If $\{s_0,u_0\}\in \cE_{\rm{T}}$ then $u_0\in \cT_{s_0}^{\rmC,d}$.\label{cond:prop:time-set-A-0-3}
        \item The set $\cT_{s_0}^{\rmC,d}$ is $\rmG$-connected.  \label{cond:prop:time-set-A-0-4}
        \item Let $k,d\in\bbN$ and $\delta\in(0,1)$. There exists some constant $c$ depending on $d$ and $\delta$ such that 
        $$\bbP_\rmE[\bigcup_{t\in I((1-\delta)k)} \cT_t^{\mathrm C,d}\nsubseteq I(k)]\lesssim 2^{-ck}.$$ \label{cond:prop:time-set-A-0-5}
        \item  For each $d\in\bbN$, there exists some constant $c$ depending only on $d$ such that for large enough $k$, 
      \eqb\bbP_\rmE\big[\exists t\in [-2^{k+1},2^{k+1}]\cap \cJ^-: \#\cT_t^{\rmC,d}>k^{27} \big]\lesssim 2^{-ck}.\eqe \label{cond:prop:time-set-A-0-6}
      \item  {Fix $d\geq 3$. Let $s,s'\in \cJ^-_\rmL$ such that $\{s,s'\}\in \cE_\rmG$ or $s'\in \cT_s^{\rmC,1}$. Then there exists $t_0,...,t_n\in \cT_{s}^{\rmC,d}$ such that $t_0=s$, $t_n=s'$ and for each $j=0,...,n-1$, $\{t_j,t_{j+1}\}\in\cE_\rmT$.} \label{cond:prop:time-set-A-0-7}
    \end{enumerate}
     % The same conclusions hold for $\cT_t^{\rmC,d,\mathrm {large}}$ instead of $\cT_t^{\rmC,d}$. Furthermore, for each $s_0\in\cJ^-$, $ \cT_{s_0}^{\mathrm{C},{d}}\subseteq \cT_{s_0}^{\mathrm{C},{d},\mathrm{large}}$.
\end{proposition}

\begin{proposition}\label{prop:time-set-A}
  There exists a.s.\ finite subsets $(\mr\cT_{s}^{\mathrm{C},{d}})_{s\in\mr\cJ^-}$ of $\mr\cJ^-$ such that the following is true. For each $s_0,u_0\in\mr\cJ^-$, we have:
    \begin{enumerate}
        \item $s_0\in \mr\cT_{s_0}^{\mathrm{C},{d}}$.\label{cond:prop:time-set-A-1}
        \item $s_0\in \mr\cT_{u_0}^{\mathrm{C},d}$ if and only if $u_0\in \mr\cT_{s_0}^{\rmC,d}$.\label{cond:prop:time-set-A-2}
        \item If $\{s_0,u_0\}\in \mr\cE_{\rm{T}}$ then $u_0\in \mr\cT_{s_0}^{\rmC,d}$.\label{cond:prop:time-set-A-3}
        \item  $\mr\cT_{s_0}^{\rmC,d}$ is $\rmG$-connected.  \label{cond:prop:time-set-A-4}
        \item Let $k,d\in\bbN$ and $\delta\in(0,1)$. There exists some constant $c$ depending on $d$ and $\delta$ such that 
        $$\bbP_\rmE[\bigcup_{t\in I((1-\delta)k)} \mr\cT_t^{\mathrm C,d}\nsubseteq I(k)]\lesssim 2^{-ck}.$$ \label{cond:prop:time-set-A-5}
        \item  For each $d\in\bbN$, there exists some constant $c$ depending only on $d$ such that for large enough $k$, 
      \eqb\label{eq:cond:prop:time-set-A-6}\bbP_\rmE\big[\exists t\in [-2^{k+1},2^{k+1}]\cap \mr\cJ^-: \#\mr\cT_t^{\rmC,d}>k^{27} \big]\lesssim 2^{-ck}.\eqe \label{cond:prop:time-set-A-6}
       \item  {Fix $d\geq 3$. Let $s,s'\in\mr\cJ^-$ such that  $\{s,s'\}\in \mr\cE_\rmG$ or $s'\in  \mr\cT_{s}^{\mathrm{C},{1}}$. Then there exists $t_0,...,t_n\in \mr\cT_{s}^{\rmC,d}$ such that $t_0=s$, $t_n=s'$ and for each $j=0,...,n-1$, $\{t_j,t_{j+1}\}\in\mr\cE_\rmT$.} \label{cond:prop:time-set-A-7}
    \end{enumerate}
    %  The same conclusions hold for $\mr\cT_t^{\rmC,d,\mathrm {large}}$ instead of $\mr\cT_t^{\rmC,d}$. Furthermore, for each $s_0\in\mr\cJ^-$, $ \mr\cT_{s_0}^{\mathrm{C},{d}}\subseteq \mr\cT_{s_0}^{\mathrm{C},{d},\mathrm{large}}$.
\end{proposition}

Proposition~\ref{prop:time-set-A-0} mostly follows from ~\cite{Tutte}. The goal of the rest of this subsection is to prove Proposition~\ref{prop:time-set-A}.

\subsubsection{The definition of the time sets}\label{subsubsec:3.1.1}

Let $\cZ=(\cL,\cR)$ and $Z=(L,R)$ be coupled as in Theorem~\ref{prop:KMT-gamma-cone} or Theorem~\ref{thm:KMT-0-cone}, and recall the walk $\mr\cZ=(\mr\cL,\mr\cR)$ defined after~\eqref{eq:double-edge-Mullin-A}. For an interval $[t_0,t_1]$, define $N_\rmL[t_0,t_1]$ to be the maximal 
number of consecutive upsteps for $\cL|_{[t_0,t_1]}$  without any downsteps, i.e.,
\eqb\label{eq:def-jump-N}N_\rmL[t_0,t_1] := \max \{\#(J\cap \cJ_\rmL^+):J\subseteq [t_0,t_1] \ \text{interval}, J\cap \cJ^-_\rmL=\emptyset\}.\eqe
Let $\mr N_\rmL[t_0,t_1]$   be the maximal 
number of consecutive upsteps for $\mr\cL|_{[t_0,t_1]}$  without any downsteps for $\mr\cL$ or $\mr \cR$, i.e.,
\eqb\label{eq:def-jump-hat-N}\mr N_\rmL[t_0,t_1] := \max \{\#(J\cap \mr\cJ_\rmL^+):J\subseteq [t_0,t_1] \ \text{interval}, J\cap \mr\cJ^-=\emptyset\}.\eqe
We set $$\Delta_\rmL[t_0,t_1] = \sup_{t\in [t_0,t_1]}|(L_t-\cL_t)-(L_{t_0}-\cL_{t_0})|+N_\rmL[t_0,t_1]+2.$$

We first introduce the time set $\cT_t^{\rmC,d}$   as in~\cite[Section 5.1]{Tutte}.  {For each $t_0\in\cJ$ and fixed $d\in\bbN$,} we define
\begin{equation}\label{eq:def-TtCd-00}
    \tau_{t_0,\rmL}^{\rmC,d} = \inf\{t\in\cJ^-_\rmL:t>t_0 \ \text{ and } \inf_{s\in [t_0,t]} (L_s-L_{t_0})<-d \cdot\Delta_\rmL[t_0,t] \},
\end{equation}
where the infimum of an empty set is infinity, and 
\begin{equation}\label{eq:def-TtCd-0011}
    \Delta_{t_0,\rmL}^{\rmC,d} = d\cdot \Delta_\rmL[t_0,\tau_{t_0,\rmL}^{\rmC,d}].
\end{equation}
Following~\cite[Section 5.1]{Tutte}, $ \tau_{t_0,\rmL}^{\rmC,d}  $ and $ \Delta_{t_0,\rmL}^{\rmC,d}$ are both a.s.\ finite under $\bbP_\rmE$ and $\bbP_\rmQ$. Now we set
\begin{equation}\label{eq:def-TtCd}
     \cT_{t_0,\rm L}^{\textrm{C},d} = \big\{ {t\in\cJ_\rmL^-}\cap [t_0, \tau_{t_0,\rm L}^{\textrm{C},d}]: \{L_s:s\in I_t\}\cap [L_{t_0}-\Delta_{t_0,\rm L}^{\textrm{C},d},L_{t_0}+\Delta_{t_0,\rm L}^{\textrm{C},d}]\neq\emptyset \big\}.
\end{equation}

Define $N_\rmR[t_0,t_1],\mr N_\rmR[t_0,t_1],\Delta_\rmR[t_0,t_1],\Delta_{t_0,\rmR}^{\rmC,d},\tau_{t_0,\rmR}^{\rmC,d},\cT_{t_0,\rmR}^{\rmC,d}$ in the exact same way by using $(\cR,\mr\cR,R)$ instead of $(\cL,\mr\cL,L)$. For $t_0\in\cJ$ and $s_0\in\cJ^-$, we set
\begin{equation}
    \cT_{t_0}^{\rmC,d,0}=\cT_{t_0,\rmL}^{\rmC,d}\cup \cT_{t_0,\rmR}^{\rmC,d}, \ \ \  \text{ and }  \ \ \ \cT_{s_0}^{\rmC,d} = \bigcup_{t_0\in\cJ:s_0\in  \cT_{t_0}^{\rmC,d,0}} \cT_{t_0}^{\rmC,d,0}.
\end{equation}

\iffalse
To define $\cT_{t_0}^{\rmD,d}$, we define
\begin{equation}
    \tau_{t_0,\rmL}^{\rmD,d} = \inf\{t\in\cJ^-_\rmL:t>t_0 \ \inf_{s\in [t_0,t]} \cL_s-\cL_0<-d \cdot\Delta_\rmL[t_0,t] \},
\end{equation}
also define
$$ \Delta_{t_0,\rmL}^{\rmD,d} = d\cdot \Delta_\rmL[t_0,\tau_{t_0,\rmL}^{\rmD,d}].$$
\fi

% {The Tutte paper also introduced the time set $\cT_{t_0}^{\rmD,d}$, and we omit its definition here.}
Now we define the time set $\mr \cT_{t_0}^{\rmC,d}$, which is slightly different from $\cT_{t_0}^{\rmC,d}$ since we will be working with the circle packing for the triangulation encoded by $\mathring{\cZ}$ rather than the map $\Phi_{\mathrm{M}}(\cZ)$.   We set \eqb\label{eq:def-delta-mr-tTDd} \mr \Delta_\rmL[t_0,t_1] = \sup_{t\in [t_0,t_1]}|(L_t-\cL_t)-(L_{t_0}-\cL_{t_0})|+\mr N_\rmL[t_0,t_1]+2.\eqe  {For each $t_0\in\cJ$ and fixed $d\in\bbN$,} we define
\begin{equation}\label{eq:def-delta-mr-tTDd-1}
   \mathring \tau_{t_0,\rmL}^{\rmC,d} = \inf\{t\in\mr\cJ^-:t>t_0 \ \text{ and } \inf_{s\in [t_0,t]} (L_s-L_{t_0})<-d \cdot\mathring\Delta_\rmL[t_0,t] \},
\end{equation}
where the infimum of an empty set is infinity, and 
\begin{equation}
    \mathring\Delta_{t_0,\rmL}^{\rmC,d} = d\cdot \mathring\Delta_\rmL[t_0,\mr\tau_{t_0,\rmL}^{\rmC,d}].
\end{equation}
We set
\begin{equation}\label{eq:def-TtCd-mr-0}
    \mr \cT_{t_0,\rm L}^{\textrm{C},d} = \big\{ {t\in\mr\cJ^-}\cap [t_0, \mr\tau_{t_0,\rm L}^{\textrm{C},d}]: \{L_s:s\in \mr I_t\}\cap [L_{t_0}-\mr\Delta_{t_0,\rm L}^{\textrm{C},d},L_{t_0}+\mr\Delta_{t_0,\rm L}^{\textrm{C},d}]\neq\emptyset \big\}.
\end{equation}
Define $\mathring \tau_{t_0,\rmR}^{\rmC,d}, \mathring\Delta_{t_0,\rmR}^{\rmC,d},\mr \cT_{t_0,\rm R}^{\textrm{C},d} $ in the exact same way  by using $(\cR,\mr\cR,R)$ instead of $(\cL,\mr\cL,L)$.  For $t_0\in\cJ$ and $s_0\in\mr\cJ^-$, we set
\begin{equation}\label{eq:def-TtCd-mr}
    \mr\cT_{t_0}^{\rmC,d,0}=\mr\cT_{t_0,\rmL}^{\rmC,d}\cup \mr\cT_{t_0,\rmR}^{\rmC,d}, \ \ \  \text{ and }  \ \ \ \mr\cT_{s_0}^{\rmC,d} = \bigcup_{t_0\in\cJ:s_0\in  \mr\cT_{t_0}^{\rmC,d,0}} \mr\cT_{t_0}^{\rmC,d,0}.
\end{equation}

\begin{lemma}\label{lem:time-set-sew}
    For  $t_0\in \cJ$, the following is true. Let $(t_{{\rm L}},t_{\rmR})$ be defined as in~\eqref{eq:lem:sew} and $(t_{ {\rm L,2}},t_{ {\rmR,2}})$ be defined as in~\eqref{eq:lem:sew-1} for $t=t_0$.  Then for $q\in\{\rmL,\rmR\}$, $t_{ {q}},t_{q,2}\in \cT_{t_0,q}^{\mathrm{C},d}$. If $t_q$ (resp.\ $t_{q,2}$) is in $\mr\cJ^-$, then $t_q$ (resp.\ $t_{q,2}$) is also in $\mr\cT_{t_0,q}^{\mathrm{C},d}$.   % let $v_{\rm L}$ and $v_{\rm R}$ be the left and right, respectively, end point of the edge $\mathfrak{e}_{t_0}$. Let $(t_{v_{\rm L}},t_{v_\rmR})$ be defined as in~\eqref{eq:lem:sew},  and $(t_{ {\rm L,2}},t_{ {\rmR,2}})$ be defined as in~\eqref{eq:lem:sew-1}. Then for any $d\in\bbN$, $t_{{\rm L}}\in \cT_{t_0,\rm L}^{\mathrm{C},d}$ and $t_{{\rmR}}\in \cT_{t_0,\rmR}^{\mathrm{C},d}$. If $t_{{\rmL}}$ and $t_{{\rmR}}$ are both in $\mr\cJ^-$, then $t_{{\rmL}}\in \mr\cT_{t_0,\rmL}^{\mathrm{C},d}$ and $t_{{\rmR}}\in \mr\cT_{t_0,\rmR}^{\mathrm{C},d}$. Finally, for $t\in \cJ^+_\rmL$ (resp.\ $t\in \cJ^+_\rmR$), then $t_{ {\rm L,2}}\in \cT_{t_0,\rm L}^{\mathrm{C},d}$ (resp.\ $t_{ {\rm R,2}}\in \cT_{t_0,\rm R}^{\mathrm{C},d}$).
\end{lemma}
% {[*Proof is identical to the draft*]}
\begin{proof}
     %{From Lemma~\ref{lem:sew}, $t_{v_{\rm L}} = \inf\{s>t:\cL_s=\cL_t-1\}$ and $t_{v_{\rm R}} = \inf\{s>t:\cR_s=\cR_t-1\}$}. Therefore 
   We only work with the case where $t_{\rmL,2}\in \mr\cJ^-$. The other cases follows similarly.  For any $t\in[t_0,t_{\rmL,2})$, by the triangle inequality and $\cL_t\geq\cL_{t_0}-1$,
    \eqb\label{eq:pf-sew-time-set-1}  L_t-L_{t_0}\geq \cL_t-\cL_{t_0}-\sup_{u\in[t_0,t]}|L_u-L_{t_0}-\cL_u+\cL_{t_0}|\geq -d\sup_{u\in[t_0,t]}|L_u-L_{t_0}-\cL_u+\cL_{t_0}|-d> -d\cdot \mr \Delta_\rmL[t_0,t]  ,\eqe
    which implies that  {$t_{{\rm L,2}}\leq \mr\tau_{t_0,\rm L}^{\textrm{C},d}$}. This further implies $\mr\Delta_{t_0,\rm L}^{\textrm{C},d}$ is greater or equal to the absolute value of the right hand side of~\eqref{eq:pf-sew-time-set-1}, and $t_{{\rm L,2}}\in \cT_{t_0,\rm L}^{\textrm{C},d}$ since 
    $$|L_{t_{{\rm L,2}}}-L_{t_0}|\leq |\cL_{{t_{\rm L,2}}}-\cL_{t_0}|+\sup_{u\in[t_0,t_{\rmL,2}]}|L_u-L_{t_0}-\cL_u+\cL_{t_0}|\leq \mr\Delta_{t_0,\rm L}^{\textrm{C},d}$$
    where we used that $|\cL_{t_{{\rm L,2}}}-\cL_{t_0}|=2$. The other assertions follows from a similar argument.
\end{proof}

\begin{lemma}\label{lem:def-wh-T-Cd}
    For each $s_0,u_0\in\mr\cJ^-$, the following hold.
    \begin{enumerate}
        \item $s_0\in \mr\cT_{s_0}^{\mathrm{C},{d}}$. \label{cond:lem:def-wh-T-Cd-1}
        \item $s_0\in \mr\cT_{u_0}^{\mathrm{C},d}$ if and only if $u_0\in \mr\cT_{s_0}^{\rmC,d}$. \label{cond:lem:def-wh-T-Cd-2}
        \item If $\{s_0,u_0\}\in \mr\cE_{\rm{T}}$ then $u_0\in \mr\cT_{s_0}^{\rmC,d}$. \label{cond:lem:def-wh-T-Cd-3}
        \item  $\mr\cT_{s_0}^{\rmC,d}$ is $\rmG$-connected. \label{cond:lem:def-wh-T-Cd-4} 
    \end{enumerate}
\end{lemma}

\begin{proof}
    Assertion~\ref{cond:lem:def-wh-T-Cd-1} holds since $s_0\in \mr\cT_{s_0,\rmL}^{\mathrm{C},{d}}$ from the definition~\eqref{eq:def-TtCd-mr-0}, while Assertion~\ref{cond:lem:def-wh-T-Cd-2} follows from the definition~\eqref{eq:def-TtCd-mr}. Assertion~\ref{cond:lem:def-wh-T-Cd-3} is from Lemma~\ref{lem:sew} and Lemma~\ref{lem:time-set-sew} as following. We view the triangulation $\mr T$ encoded by $\mr\cZ$ as a subgraph of the triangulation $T$ encoded by $\cZ$. Let $\frv_{s_0}$ and $\frv_{u_0}$ be the vertices for $s_0$ and $u_0$ in $T$. If $\frv_{s_0}$ and $\frv_{u_0}$ are not both in $\frT$ or $\frT_*$, then we consider  the edge $\fre_{t_0}$ joining $\frv_{s_0}$ and $\frv_{u_0}$, and we have $s_0\in \mr\cT_{t_0,\rmL}^{\mathrm{C},d}$ and $u_0\in \mr\cT_{t_0,\rmR}^{\mathrm{C},d}$. If $\frv_{s_0}$ and $\frv_{u_0}$ are connected by an edge in $\frT$ (resp.\ $\frT_*$) which first appears in the triangle $\mathfrak{tr}(t_0)$, then $s_0,u_0\in \mr\cT_{t_0,\rmL}^{\mathrm{C},d}$ (resp.\  $s_0,u_0\in \mr\cT_{t_0,\rmR}^{\mathrm{C},d}$). The claim then follows from~\eqref{eq:def-TtCd-mr}. 

     To prove Assertion~\ref{cond:lem:def-wh-T-Cd-4}, for consecutive $s<u$ in $\mr\cT_{t_0,\rmL}^{\rmC,d}$, if $\mr I_s\cap \mr I_{u}=\emptyset$, then $L_t> L_{t_0}+\mr\Delta_{t_0,\rmL}^{\rmC,d}$ for $t\in[s,u^-]$ where $\mr I_u = [u^-,u]$. If we set $\wt s = \sup\{t<s:L_t\leq L_{t_0}+\mr\Delta_{t_0,\rmL}^{\rmC,d}\}$ and $\wt u =  \inf\{t>u^-:L_t\leq L_{t_0}+\mr\Delta_{t_0,\rmL}^{\rmC,d}\}$, then $\wt s\in\mr I_s$, $\wt u\in \mr I_u$, and $L_{\wt s}=L_{\wt u} = \inf_{t\in[\wt s, \wt u]}L_t$ and therefore condition (\ref{cond:def-cG-(i)-1}) in the definition of $\mr\cE_{\rmG}$ holds. This implies that  $\mr\cT_{t_0,\rmL}^{\rmC,d}$ is $\rmG$-connected. Similarly $\mr\cT_{t_0,\rmR}^{\rmC,d}$ is G-connected.  {Moreover, if we let $t_0' = \inf\{t\geq t_0:t\in\mr\cJ^-\}$, then $t_0'\in \mr\cT_{t_0,\rmL}^{\rmC,d} \cap \mr\cT_{t_0,\rmR}^{\rmC,d}$, which implies $\mr\cT_{t_0}^{\rmC,d,0}$ is G-connected.}  Finally $\mr\cT_{s_0}^{\rmC,d}$ is $\rmG$-connected since all the sets  $\mr\cT_{t_0}^{\rmC,d,0}$ in the definition of $\mr\cT_{s_0}^{\rmC,d}$ in~\eqref{eq:def-TtCd-mr} are $\rmG$-connected and contain $s_0$. %\nina{*does this proof simplify if we remove (iii) from the definition of G-connected?*}
\end{proof}

Now we define the time set $\mr \cT_t^{\rmD,d}$.  The letter ``$\rmD$"   is chosen since it is the first letter of the word ``discrete". 
For each $t_0\in\cJ$ and fixed $d\in\bbN$, we define
\begin{equation}\label{eq:def:mr-tau-TtDd}
\begin{split}
   \mathring \tau_{t_0,\rmL}^{\rmD,d} &= \inf\{t\in\mr\cJ^-:t>t_0 \ \text{ and } \  \inf_{s\in [t_0,t]}(\cL_s-\cL_{t_0})<-d \cdot\mathring\Delta_\rmL[t_0,t] \}; \ \ \mathring\Delta_{t_0,\rmL}^{\rmD,d} = d\cdot \mathring\Delta_\rmL[t_0,\mr\tau_{t_0,\rmL}^{\rmD,d}]. % \\&=   \inf\{t\in\mr\cJ^-:t>t_0 \ \text{ and } \   \cL_t-\cL_{t_0}<-d \cdot\mathring\Delta_\rmL[t_0,t] \},
   \end{split}
\end{equation}
%  Now we set
% \begin{equation}
%     \mathring\Delta_{t_0,\rmL}^{\rmD,d} = d\cdot \mathring\Delta_\rmL[t_0,\mr\tau_{t_0,\rmL}^{\rmD,d}].
% \end{equation}

 %where the second equation holds since for any $t\in\mr\cJ^-$ and $\mr I_t = [t^-,t]$, $\cL_s\geq \cL_{t^-}$ for any $s\in [t^-,t)$.

We define
\begin{equation}\label{eq:def-mathringTtDd}
    \mr\cT_{t_0,\rm L}^{\textrm{D},d} = \big\{ {t\in\mr\cJ^-}\cap [t_0, \mr\tau_{t_0,\rm L}^{\textrm{D},d}]: \{\cL_s:s\in \mr I_t\}\cap [\cL_{t_0}-\mr\Delta_{t_0,\rm L}^{\textrm{D},d},\cL_{t_0}+\mr\Delta_{t_0,\rm L}^{\textrm{D},d}]\neq\emptyset \big\}.
\end{equation}
For any $s\in  \mr\cT_{t_0,\rm L}^{\textrm{D},d}$ define $u_0^s,u_1^s,...$ iteratively by $u_0^s=s$ and $$u_j^s = \inf\{t>u^s_{j-1}:\cL_t = \cL_{u^s_{j-1}}-1\}.$$ By   Lemma~\ref{lem:hat-minus} and $u_0^s\in \mr\cJ^-$,  $u_j^s\in\mr\cJ^-$ for each $j\in\bbN_0$. Let $\sigma_s = \inf\{j\in\bbN:u_j^s\in \mr\cT_{t_0,\rm L}^{\textrm{D},d} \ \text{or} \ u_j^s\geq  \mathring \tau_{t_0,\rmL}^{\rmD,d}   \}$, and 
\begin{equation*}
    \wh{\mr\cT_{t_0,\rm L}^{\textrm{D},d}} = \bigcup_{s\in  \mr\cT_{t_0,\rm L}^{\textrm{D},d}} \{u_0^s,...,u_{\sigma_s-1}^s\}.
\end{equation*}
Define $\mr\tau_{t_0,\rm R}^{\textrm{D},d},\mathring\Delta_{t_0,\rmR}^{\rmD,d}, \mr\cT_{t_0,\rm R}^{\textrm{D},d},\wh{\mr\cT_{t_0,\rm R}^{\textrm{D},d}}$ in the exact same way by using  $(\cR,\mr\cR,R)$ instead of $(\cL,\mr\cL,L)$. For $t_0\in\cJ$ and $s_0\in\mr\cJ^-$, we set
\begin{equation}\label{eq:def-wh-TtDd}
      \mr\cT_{t_0}^{\rmD,d,0}=\wh{\mr\cT_{t_0,\rmL}^{\rmD,d}}\cup \wh{\mr\cT_{t_0,\rmR}^{\rmD,d}}, \ \ \  \text{ and }  \ \ \ \mr\cT_{s_0}^{\rmD,d} = \bigcup_{t_0\in\cJ:s_0\in  \mr\cT_{t_0}^{\rmD,d,0}} \mr\cT_{t_0}^{\rmD,d,0}.
\end{equation}
We comment that it is straightforward to check from the definitions that for $t_0^+ = \inf\{s> t_0:s\in\mr\cJ^-\}$, $t_0^+\in \mr\cT^{\rmC,d}_{t_0,\rmL}\cap \mr\cT^{\rmC,d}_{t_0,\rmR}\cap \mr\cT^{\rmD,d}_{t_0,\rmL}\cap \mr\cT^{\rmD,d}_{t_0,\rmR}$ since $t_0\in\mr I_{t_0^+}$.

\begin{lemma}\label{lem:lower-bd-tau-hat-AA}
  Let $t_0\in\cJ$. Then  $\mr \tau :=  \mathring \tau_{t_0,\rmL}^{\rmD,d}\in \mr\cT_{t_0,\rmL}^{\rmD,d}$. If $[t_0,\mr \tau)\cap \mr\cJ^-=\emptyset$, then $\mr\cT_{t_0,\rmL}^{\rmD,d}= \wh{\mr\cT_{t_0,\rm L}^{\textrm{D},d}}=\{\mathring \tau\}$. Further assume $[t_0,\mr \tau)\cap \mr\cJ^-\neq\emptyset$. Then for $t\in [t_0, \mr \tau)$, $\cL_t\geq \cL_{t_0}- \mathring\Delta_{t_0,\rmL}^{\rmD,d}$, whereas $ \cL_{t_0}-\mathring\Delta_{t_0,\rmL}^{\rmD,d}-1\leq \cL_{\mr \tau}< \cL_{t_0}-\mathring\Delta_{t_0,\rmL}^{\rmD,d}$. The same holds for $(\cR,\rmR)$ instead of $(\cL,\rmL)$. %Furthermore, $\mathring \tau_{t_0,\rmL}^{\rmD,d}\in \mr\cT_{t,\rmL}^{\rmD,d}$.
\end{lemma}

\begin{proof}
From definition, if $[t_0,\mr \tau)\cap \mr\cJ^-=\emptyset$, then~\eqref{eq:def-mathringTtDd} implies that $\mr\cT_{t_0,\rmL}^{\rmD,d}= \wh{\mr\cT_{t_0,\rm L}^{\textrm{D},d}}=\{\mathring \tau\}$ since $t_0\in \mr I_{\mr \tau}$.
    Now assume $[t_0,\mr \tau)\cap \mr\cJ^-\neq\emptyset$. Let $\sigma = \max\{[t_0,\mr \tau)\cap \mr \cJ^-\}$. Then $\mr I_{\mr \tau} = [\sigma,\mr \tau]$, and the stopping condition in the definition~\eqref{eq:def:mr-tau-TtDd} has not been hit at the time $\sigma$,  which further implies that $\cL_t\geq \cL_{t_0}-d\cdot\mr\Delta_\rmL[t_0,\sigma]\ge \cL_{t_0}-d\cdot\mr\Delta_\rmL[t_0,\mr \tau] = \cL_{t_0}-\mathring\Delta_{t_0,\rmL}^{\rmD,d}$ for $t\in [t_0,\sigma]$. Here we used the fact that the function $t\mapsto d\cdot\mr\Delta_\rmL[t_0,t]$ is non-decreasing. On the other hand, by Lemma~\ref{lem:hat-minus} applied to the interval $\mr I_{\mr \tau}$, we have $\cL_{t}\geq\cL_{\sigma}$ for $t\in [\sigma,\mr \tau)$ whereas $\cL_{\mr \tau}\geq \cL_{\sigma}-1$. Therefore  $\cL_t\geq \cL_{t_0}- \mathring\Delta_{t_0,\rmL}^{\rmD,d}$ for $t\in [t_0,\mr \tau)$ and $\cL_{\mr \tau}\geq \cL_{t_0}- \mathring\Delta_{t_0,\rmL}^{\rmD,d}-1$. On the other hand, the definition~\eqref{eq:def:mr-tau-TtDd} implies that $\cL_{\mr \tau}<\cL_{t_0}- \mathring\Delta_{t_0,\rmL}^{\rmD,d}$. This further implies that $\cL_{\sigma}\in [\cL_{t_0}- \mathring\Delta_{t_0,\rmL}^{\rmD,d},\cL_{t_0}+2]$ and therefore $\mathring \tau
    \in  \mr\cT_{t_0,\rmL}^{\rmD,d}$ by~\eqref{eq:def-mathringTtDd}. This finishes the proof.
     % Suppose for some $t_1\in\mr\cJ^-$ with  $\mr I_{t_1}  = [t_1^-,t_1]$, $t\in [t_1^-,t_1)$. Indeed if $t_1\neq \mr t$, then the definition~\eqref{eq:def:mr-tau-TtDd} implies that we must have $\cL_t\geq \cL_{t_0}- d \cdot\mathring\Delta_\rmL[t_0,t_1]$. For $t_1=\mr t$, by Lemma~\ref{lem:hat-minus}, we have $\cL_t\geq \cL_{t_1^-}$. By our assumption $t_1^-\geq t_0$ then again from the definition~\eqref{eq:def:mr-tau-TtDd} $\cL_t\geq\cL_{t_1^-}\geq \cL_0-d \cdot\mathring\Delta_\rmL[t_0,t_1^-]$. Then it further follows from the definition~\eqref{eq:def:mr-tau-TtDd} that $\cL_{\mr t}< \cL_{t_0}- d \cdot\mathring\Delta_\rmL[t_0,t]$ and we must have $\mr t\in \mr\cT_{t_0,\rmL}^{\rmD,d}$. If $[t_0,\mr t)\cap \mr\cJ^-=\emptyset$, then $t_0\in I_{\mr t_0}$ and it is clear that $\wh{\mr\cT_{t_0,\rmL}^{\rmD,d}} = \mr\cT_{t_0,\rmL}^{\rmD,d}=  \{\mr t_0\} $.
\end{proof}

\begin{lemma}\label{lem:def-wh-T-Dd}
    For each $s_0,u_0\in\mr\cJ^-$, the following hold.
    \begin{enumerate}
        \item $s_0\in \mr\cT_{s_0}^{\mathrm{D},{d}}$. \label{cond:lem:def-wh-T-Dd-1}
        \item $s_0\in \mr\cT_{u_0}^{\mathrm{D},d}$ if and only if $u_0\in \mr\cT_{s_0}^{\rmD,d}$. \label{cond:lem:def-wh-T-Dd-2}
        \item  {For $d\geq 2$}, if $\{s_0,u_0\}\in \mr\cE_{\rm{G}}$ then $u_0\in \mr\cT_{s_0}^{\rmD,d}$. \label{cond:lem:def-wh-T-Dd-3}
       % \item For $d\geq2$, $\cT_{s_0}^{\rmC,d}$ is $G$-connected.  
        \item $\mr\cT_{s_0}^{\rmC,d}\subseteq \mr\cT_{s_0}^{\rmD, d+1}$ and $\mr\cT_{s_0}^{\rmD,d}\subseteq \mr\cT_{s_0}^{\rmC, d+1}$. \label{cond:lem:def-wh-T-Dd-4}
        \item  $\mr\cT_{s_0}^{\mathrm{D},{d}}$ is $\rmT$-connected. \label{cond:lem:def-wh-T-Dd-5}
    \end{enumerate}
   % The same properties hold for $\mr\cT_{s_0}^{\rmD,d,\mathrm {large}}$ and $\mr\cT_{s_0}^{\rmC,d,\mathrm{large}}$, respectively.
\end{lemma}

\begin{proof}
    Assertion~\ref{cond:lem:def-wh-T-Dd-1} and Assertion~\ref{cond:lem:def-wh-T-Dd-2} are clear from the definition~\eqref{eq:def-wh-TtDd}.

      To prove Assertion~\ref{cond:lem:def-wh-T-Dd-3}, without loss of generality assume $s_0<u_0$. If $\mathring I_{u_0} = [s_0,u_0]$, then $\mr\tau_{s_0,\rmL}^{\rmD,d}\geq u_0$ and $u_0\in \mr\cT_{s_0,\rmL}^{\rmD,d}\subseteq\mr\cT_{s_0}^{\rmD,d}$ since $s_0\in\mr I_{u_0}$.  Therefore it  suffices to work on the case when $\mr I_{u_0}$  and $\mr I_{s_0}$ are disjoint. Without loss of generality assume that condition~(\ref{cond:def-cG-(i)-1}) in the definition of $\mr\cE_\rmG$ holds; the case where condition~(\ref{cond:def-cG-(ii)-1}) holds follows from a symmetric argument. Then  there are $\wt s\in \mr I_{s_0}$ and $\wt u\in \mr I_{u_0}$ such that~\eqref{eq:def-G-connected-1-1} holds.  Let $\underline{s} = \sup(\cJ\cap(-\infty,\wt s])$. Then by definition of $\cJ$, $\cL_{\wt s}=\cL_{\underline{s}}$. We claim that $s_0,u_0\in\mr\cT_{\underline{s},\rmL}^{\rmD,d}$. For $t\in (\underline{s}, \wt s]$, $\cL_t=\cL_{\underline{s}}$, and for every $t\in (\wt s,\wt u]$, 
    \begin{equation*}
        \begin{split}
            \cL_t-\cL_{\underline{s}}&=\cL_t-\cL_{\wt s}\geq L_t-L_{\wt s} - \sup_{t'\in [\wt s,t]}|(\cL_{t'}-\cL_{\wt s})-(L_{t'}-L_{\wt s})|\\
            &\geq - \sup_{t'\in [\wt s,t]}\big(|(\cL_{t'}-\cL_{\underline s})-(L_{t'}-L_{\underline s})|+|(\cL_{\wt s}-\cL_{\underline s})-(L_{\wt s}-L_{\underline s})|\big)\\
            &> -2\sup_{t'\in [\underline s,t]}|(\cL_{t'}-\cL_{\underline s})-(L_{t'}-L_{\underline s})|\geq  -2\mr\Delta_{\rmL}[\underline{s},t]+2,
        \end{split}
    \end{equation*}
    where in the second line we used $L_t\geq L_{\wt s}$ on $[\wt s,\wt u]$ and a triangle inequality and in the last line we used the definition~\eqref{eq:def-delta-mr-tTDd}.  By   definition~\eqref{eq:def:mr-tau-TtDd}, this further implies $\mr\tau_{\underline{s},\rmL}^{\rmD,d}\geq\wt u$ and therefore $\mr\tau_{\underline{s},\rmL}^{\rmD,d}\geq u_0$. Since $\underline{s}\in \mr I_{s_0}$, we have $s_0\in \mr\cT_{\underline{s}, \rmL}^{\rmD,d}$. Now using $L_{\wt u}=L_{\wt s}$, by a triangle inequality,
     \begin{equation}\label{eq:pf-def-T-Dd-1}
        \begin{split}
            |\cL_{\wt u}-\cL_{\underline{s}}|&=|\cL_{\wt u}-\cL_{\wt s}|\leq |L_{\wt u}-L_{\wt s}| + |(\cL_{\wt u}-\cL_{\wt s})-(L_{\wt u}-L_{\wt s})|\\
            &\leq  |(\cL_{\wt u}-\cL_{\underline s})-(L_{\wt u}-L_{\underline s})|+|(\cL_{\wt s}-\cL_{\underline s})-(L_{\wt s}-L_{\underline s})|\\
            &\leq 2\sup_{t'\in [\underline s,t]}|(\cL_{t'}-\cL_{\underline s})-(L_{t'}-L_{\underline s})|\leq 2\mr\Delta_\rmL[\underline{s},t].
        \end{split}
    \end{equation}
    Combined with $\mr\tau_{\underline{s},\rmL}^{\rmD,d}\geq \wt u$, the right hand side of~\eqref{eq:pf-def-T-Dd-1} is less or equal to $\mr\Delta_{\underline{s},\rmL}^{\rmD,d}$, which further implies $u_0\in \mr\cT_{\underline{s},\rmL}^{\rmD,d}$. By~\eqref{eq:def-wh-TtDd}  this implies $u_0\in\mr\cT_{s_0}^{\rmD,d}$, which finishes the proof for Assertion~\ref{cond:lem:def-wh-T-Dd-3}.

     To prove Assertion~\ref{cond:lem:def-wh-T-Dd-4}, we note that for any $t_0\in\cJ$, $u_0\in [t_0,\mr\tau_{t_0,\rmL}^{\rmC,d})\cap \mr\cJ^-$ and $t\in [t_0,u_0]$, using the definition of $\mr\tau_{t_0,\rmL}^{\rmC,d}$ as in~\eqref{eq:def-delta-mr-tTDd-1},
     \begin{equation*}
     \begin{split}
    \cL_t-\cL_{t_0} &\geq L_t-L_{t_0} - \sup_{s\in [t_0,t]}|(L_t-\cL_t)-(L_{t_0}-\cL_{t_0})|\\&\geq -d\cdot\mr\Delta_{ \rmL}[t_0,u_0]- \sup_{s\in [t_0,t]}|(L_t-\cL_t)-(L_{t_0}-\cL_{t_0})|\geq -(d+1)\cdot \mr\Delta_\rmL[t_0,u_0] .\end{split}\end{equation*}
    This implies $\mr\tau_{t_0,\rmL}^{\rmD,d+1}\geq\mr\tau_{t_0,\rmL}^{\rmC,d}$, and in particular $$\mr\Delta_{t_0,\rmL}^{\rmD,d+1}\geq \mr\Delta_{t_0,\rmL}^{\rmC,d}+\sup_{s\in [t_0,\mr\tau_{t_0,\rmL}^{\rmC,d}]}|(L_t-\cL_t)-(L_{t_0}-\cL_{t_0})|+\mr N_\rmL[t_0,\mr\tau_{t_0,\rmL}^{\rmC,d}]+1.$$
      Suppose $t\in \mr\cT_{t_0,\rmL}^{\rmC,d}$. Then one can find $t'\in \mr I_t\cap [t_0,\mr\tau_{t_0,\rmL}^{\rmC,d}]$ such that $|L_{t'}-L_{t_0}|\leq \mr\Delta_{t_0,\rmL}^{\rmC,d}$, and
     $$|\cL_{t'}-\cL_{t_0}|\leq |L_{t'}-L_{t_0}|+\sup_{s\in [t_0,\mr\tau_{t_0,\rmL}^{\rmC,d}]}|(L_{t'}-\cL_{t'})-(L_{t_0}-\cL_{t_0})|\leq\mr \Delta_{t_0,\rmL}^{\rmD,d+1}.$$
This further implies $\mr\cT_{t_0,\rmL}^{\rmC,d}\subseteq \mr\cT_{t_0,\rmL}^{\rmD,d+1}\subseteq \wh{\mr\cT_{t_0,\rmL}^{\rmD,d+1}}$ for any $t_0\in\cJ$. Similarly $\mr\cT_{t_0,\rmR}^{\rmC,d}\subseteq \mr\cT_{t_0,\rmR}^{\rmD,d+1}\subseteq \wh{\mr\cT_{t_0,\rmR}^{\rmD,d+1}}$ and further implies the first part of Assertion~\ref{cond:lem:def-wh-T-Cd-4}.
    
    By a symmetric argument,  $\mr\tau_{t_0,\rmL}^{\rmC,d+1}\geq\mr\tau_{t_0,\rmL}^{\rmD,d}$, and   \eqb\label{eq:pf-38712387231} \mr\Delta_{t_0,\rmL}^{\rmC,d+1}\geq \mr\Delta_{t_0,\rmL}^{\rmD,d}+\sup_{s\in [t_0,\mr\tau_{t_0,\rmL}^{\rmD,d}]}|(L_t-\cL_t)-(L_{t_0}-\cL_{t_0})|+\mr N_\rmL[t_0,\mr\tau_{t_0,\rmL}^{\rmD,d}]+1,\eqe
    and also $\mr\cT_{t_0,\rmL}^{\rmD,d}\subseteq \mr\cT_{t_0,\rmL}^{\rmC,d+1}$.
It remains to take care of $t\in \wh{\mr\cT_{t_0,\rmL}^{\rmD,d}}\backslash \mr\cT_{t_0,\rmL}^{\rmD,d}$. If this set is not empty, then we must have $[t_0,\mr\tau_{t_0,\rmL}^{\rmD,d})\cap \mr\cJ^-\neq\emptyset$, otherwise $\wh{\mr\cT_{t_0,\rmL}^{\rmD,d}}= \mr\cT_{t_0,\rmL}^{\rmD,d} = \{\mr\tau_{t_0,\rmL}^{\rmD,d}\}$. Then by Lemma~\ref{lem:lower-bd-tau-hat-AA}, $\cL_s\geq \cL_{t_0}-\mr \Delta_{t_0,\rmL}^{\rmD,d}-1$ for all $s\in [t_0,\mr\tau_{t_0,\rmL}^{\rmD,d}]$. 
Furthermore, there exists $t''\in \mr\cT_{t_0,\rmL}^{\rmD,d}$ such that $t = u^{t''}_j$ for some $j$ as in the definition of $\wh{\mr\cT_{t_0,\rmL}^{\rmD,d}}$. In particular, $\cL_{t''}\geq\cL_t$ and further $t'\in \mr I_{t''}  \cap [t_0,\mr\tau_{t_0,\rmL}^{\rmD,d}]$ such that $|\cL_{t'}-\cL_{t_0}|\leq \mr\Delta_{t_0,\rmL}^{\rmD,d}$. Let $t''' = \inf((\mr\cJ\cap [t',\infty))\leq t''$. Then by Lemma~\ref{lem:hat-minus}, $\cL_{t'}\geq \cL_{t'''}-1$ and $\cL_{t''}\leq\cL_{t'''}+\mr N_\rmL[t',t'']$, which further implies that \eqb\label{eq:TtDdcontainedTtCd+1-pf-1} \cL_{t_0}-\mr\Delta_{t_0,\rmL}^{\rmD,d}-1\leq \cL_t\leq \cL_{t_0}+\mr\Delta_{t_0,\rmL}^{\rmD,d}+ \mr N_\rmL[t_0, \mr\tau_{t_0,\rmL}^{\rmD,d}]+1.\eqe Therefore combined with~\eqref{eq:pf-38712387231}, we have
  $$|L_{t}-L_{t_0}|\leq |\cL_{t}-\cL_{t_0}|+\sup_{s\in [t_0,\mr\tau_{t_0,\rmL}^{\rmD,d}]}|(L_{t'}-\cL_{t'})-(L_{t_0}-\cL_{t_0})|\leq\mr \Delta_{t_0,\rmL}^{\rmC,d+1}.$$
  This verifies that $\wh{\mr\cT_{t_0,\rmL}^{\rmD,d}}\subseteq \mr\cT_{t_0,\rmL}^{\rmC,d+1}$, which further implies the second half of Assertion~\ref{cond:lem:def-wh-T-Dd-4}.

 Below we write $\mr\tau = \mr\tau_{t_0,\rmL}^{\rmD,d}$. For Assertion~\ref{cond:lem:def-wh-T-Dd-5}, we first prove that for each $t_0\in\cJ$, $ \wh{\mr\cT_{s_0,\rmL}^{\rmD,d}}$ is $\rmT$-connected. This is indeed the case where $[t_0,\mr\tau)\cap \mr\cJ^-=\emptyset$, since in this case $ \wh{\mr\cT_{t_0,\rmL}^{\rmD,d}} = \{\mr\tau\}$. Now assume $[t_0,\mr\tau)\cap \mr\cJ^-\neq\emptyset$.  Then  by Lemma~\ref{lem:lower-bd-tau-hat-AA}, $ \mr\tau \in  \wh{\mr\cT_{t_0,\rmL}^{\rmD,d}}$. From our construction of $ \wh{\mr\cT_{s_0,\rmL}^{\rmD,d}}$, Lemma~\ref{lem:hat-minus} and Lemma~\ref{lem:lower-bd-tau-hat-AA}, for any $u\in  \wh{\mr\cT_{t_0,\rmL}^{\rmD,d}}$ such that $u\neq \mr\tau$, we have $\cL_u>\cL_{\mr\tau}$ and therefore  $u^+:=\inf\{s\geq u:\cL_s-\cL_u=-1\}\leq\mr\tau$, $u^+ \in  \wh{\mr\cT_{t_0,\rmL}^{\rmD,d}}$ and $\{u,u^+\}\in\mr\cE_\rmT$ by Lemma~\ref{lem:sew}.  This along with Lemma~\ref{lem:lower-bd-tau-hat-AA} implies that all the elements in $\wh{\mr\cT_{t_0,\rmL}^{\rmD,d}}$ are $\rmT$-connected to $\mr\tau$ and therefore  $\wh{\mr\cT_{t_0,\rmL}^{\rmD,d}}$  is $\rmT$-connected. Similarly  $\wh{\mr\cT_{t_0,\rmR}^{\rmD,d}}$ is $\rmT$-connected, and  $\wh{\mr\cT_{t_0}^{\rmD,d,0}} = \wh{\mr\cT_{t_0,\rmL}^{\rmD,d}}\cup \wh{\mr\cT_{t_0,\rmR}^{\rmD,d}}$ is $\rmT$-connected since $\wh{\mr\cT_{t_0,\rmL}^{\rmD,d}}$ and $\wh{\mr\cT_{t_0,\rmR}^{\rmD,d}}$ both contain $\inf\{s\geq t_0: s\in\mr\cJ^-\}$. This completes the proof of Assertion~\ref{cond:lem:def-wh-T-Dd-5}, since all the sets $\wh{\mr\cT_{t_0}^{\rmD,d,0}} $ in the union  as in the definition~\eqref{eq:def-wh-TtDd} of  $\mr\cT_{s_0}^{\mathrm{D},{d}}$ contain $s_0$.
\end{proof}

\subsubsection{Bounds on the set $\mr\cJ$}\label{subsubsec:3.1.2}

As shown in~\cite[Section 5]{Tutte}, we have the following bounds on the time set $\cJ^-_\rmL$:
\begin{proposition}\label{prop:Tutte5.8}
    For any $a\in (0,1)$, there is a $c\in (0,1)$ such that 
    \begin{equation}
        \bbP_\rmE[ \cJ^-_\rmL\cap [-2^n,(a-1)2^{n}]=\emptyset]\lesssim 2^{-cn}, \ \  \bbP_\rmE[ \cJ^-_\rmL\cap [(1-a)2^n,2^{n}]=\emptyset]\lesssim 2^{-cn}.
    \end{equation}
    The same also holds if we replace $\cJ_\rmL^-$ with $\cJ^-_\rmR$, or replace $\bbP_\rmE$ with $\bbP_\rmQ$.
\end{proposition}

\begin{proof} 
    The claim on $\bbP_\rmE$ follows from~\cite[Lemma 5.8]{Tutte} and symmetry, while the claim on $\bbP_\rmQ$ is immediate since $\cJ^-_\rmL,\cJ^-_\rmR$ are Poisson point processes under $\bbP_\rmQ$.
\end{proof}

The goal here is to prove analogous bounds on the set $\mr\cJ$, which is much more complicated due to the double edge collapsing. Recall the definition of the measure $\bbP_\rmE$ and $\bbP_\rmQ$ as in Definition~\ref{def:P_EP_Q}.

\begin{proposition}\label{lem:gap-hat-J-Q}
     For $\alpha\in(1/2,1)$, for $c\in (0,2\alpha-1)$, we have 
    \eqb\label{eq:lem:gap-hat-J-Q}\bbP_\rmQ\bigg[\exists \text{ interval } J\subseteq [-2^m,2^m], |J|=2^{\alpha m}, \mr\cJ\cap J=\emptyset\bigg]\lesssim 2^{-cm}.\eqe
\end{proposition}
% {[*Check if this proposition is needed*]}

\begin{proof}
   Using Lemma~\ref{lem:walk-excursion}, a union bound and translation invariance, for the simple random walk $(\cX_n,\cY_n)_{n\in\bbZ}$ on $\bbZ^2$, 
\eqb\label{eq:pf-gam-hat-J-Q}\bbP\big[\exists i_1\leq -n, i_2\geq0, s.t., \cX_{i_1}=\cX_{i_2}=\inf_{i_1\leq j\leq i_2}\cX_i,\cY_{i_1}=\cY_{i_2}=\inf_{i_1\leq j\leq i_2}\cY_i\big]\leq cn^{-1}.\eqe
Now for $i\in\bbZ$, let $J_i = [2^{\alpha m-1}i,2^{\alpha m-1}(i+1)]$. Fix $\e\in(0,2-1/\alpha)$. By part~(\ref{cond:prop:KMT-gamma-cone-(i)}) of Theorem~\ref{prop:KMT-gamma-cone},  $\cJ$ is a Poisson point process on $\bbR$, and  $\bbP_\rmQ[\#(J_i\cap \cJ)\leq 2^{(1-\e)\alpha m}]\leq c2^{-c_\e m^2}$.    Then by Lemma~\ref{lem:hat}, ~\eqref{eq:pf-gam-hat-J-Q} and translation invariance for the random walk $(\cX_n,\cY_n)_{n\in\bbZ}$, one can check using part~(\ref{cond:prop:KMT-gamma-cone-(i)}) of Theorem~\ref{prop:KMT-gamma-cone} that $\bbP_\rmQ\big[J_i\cap\mr\cJ=\emptyset\big]\leq c2^{-(1-\e)\alpha m}$. By a further union bound over $J_i$, we have
$$\bbP_\rmQ\big[J_i\cap\mr\cJ=\emptyset,\  \ \text{for some } i \in[ -2^{(1-\alpha)m+1}-2,2^{(1-\alpha)m+1}+2]\cap\bbZ \big]\lesssim 2^{(1-(2-\e)\alpha)m},$$
which further implies~\eqref{eq:lem:gap-hat-J-Q} since every $J\subseteq [-2^m,2^m]$ with $ |J|=2^{\alpha m}$ contains some $J_i$ for $ i \in[ -2^{(1-\alpha)m+1}-2,2^{(1-\alpha)m+1}+2]\cap\bbZ$.
\end{proof}

\begin{proposition}\label{lem:gap-hat-J}
    There exists $\e_0>0$ and $c\in(0,1)$ such that for any $n\in\bbN$, 
    \eqb\label{eq:lem:gap-hat-J}\bbP_\rmE\bigg[\exists \text{ interval } J\subseteq [-2^n,2^n], |J|=2^{(1-\e_0)n}, \mr\cJ\cap J=\emptyset\bigg]\lesssim 2^{-cn}.\eqe
\end{proposition}

The proof of Proposition~\ref{lem:gap-hat-J} under the $\bbP_\rmE$ setting is much more involved since the conclusion in Proposition~\ref{lem:gap-hat-J-Q} is not sufficient for us to use part~(\ref{cond:thm:KMT-0-cone-(iv)}) of Theorem~\ref{thm:KMT-0-cone} to transfer from the $\bbP_\rmQ$ setting. To overcome this, we first use the coupling between $(Z,\cZ)$ as in Theorem~\ref{prop:KMT-gamma-cone} and Theorem~\ref{thm:KMT-0-cone} to see that, if the event in~\eqref{eq:lem:gap-hat-J} does happen, then there exists $[t^-,t^+]$ such that $\eta([t^-,t^+])$ has small boundary length and large quantum area. The probability for this event under $\bbP_\rmQ$ can be controlled via Brownian motion estimates as in Lemma~\ref{lem:BM-e4}, and in Lemma~\ref{lem:gap-hat-J-9-4} we transfer this to the setting of $\bbP_\rmE$ using that the fields $h$ from $\bbP_\rmE$ and $\bbP_\rmQ$ can be coupled such that they differ from $\gamma\log|\cdot|$ within $\bbD$. 

 {We quote the following lemmas, which follows from~\cite[Lemma 4.18, Lemma 5.7]{Tutte} and the scaling in ~\cite[Lemma 4.17]{Tutte}.}
\begin{lemma}\label{lem:L_t-inf-1}
   Let $\delta,a>0$ be deterministic. Then under $\bbP_\rmE$, $(a^{-1/2}L_{at},a^{-1/2}R_{at})_{t\in\bbR} \overset{d}{=}(L_t,R_t)_{t\in\bbR}$.  Furthermore, exists a constant $c>0$ such that 
    $$\bbP_\rmE[\sup_{ {t\in [-2^k,2^k]}} |L_t| \geq k2^{k/2}]\lesssim 2^{-ck^2}, \ \ {\bbP_\rmE[\inf_{t\in {[-2^k,2^k]}}}L_t>-2^{(0.5-\delta)k}]\lesssim 2^{-ck}. $$
    The same holds for $R$ in place of $L$.
\end{lemma}

\begin{lemma}\label{lem:L_t-inf-a1a2}
    For any $a_1>0$ and $a_2\in(0,a_1)$, there is some $c\in(0,1)$ such that for any $m\in\bbN$, 
    $$\bbP_\rmE[\inf_{t\in [-2^{(1+a_1)m},-2^m]}L_t>-2^{(1+a_2)m/2}]\lesssim 2^{-cm}, \ \ \ \ \ \bbP_\rmE[\inf_{t\in [2^m,2^{(1+a_1)m}]}L_t>-2^{(1+a_2)m/2}]\lesssim 2^{-cm}.$$
    The same holds for $R$ instead of $L$. 
\end{lemma}

\begin{lemma}\label{lem:gap-hat-J-9-3}
    Let $0<\wt\e_1<1$. There exists constants $\wt c_1,\wt c_2>0$ depending only on $\wt\e_1$, such that under the probability measure $\bbP_\rmE$ for any $n\in\bbN$, with probability at least $1-\wt c_12^{-\wt c_2n}$, there does not exist $t_1,t_2\in\bbR$ with $|t_1|\leq 2^{(1-\wt\e_1)n}$ and $ |t_2|\geq 2^{n}$ such that~\eqref{eq:double-edge-Mullin-A} holds for $(t_1\wedge t_2,t_1\vee t_2)$ in place of $(t^-,t^+)$.
\end{lemma}

\begin{proof}
Let $m\geq n$. By part~(\ref{cond:thm:KMT-0-cone-(ii)}) of Theorem~\ref{thm:KMT-0-cone}, for some constants $C_1,C_2>1$,  
\eqb \label{eq:double-edge-lem-0-pf-1-00}\bbP_\rmE\big[ \sup_{t\in [-2^n,2^n]} \|\cZ_t - Z_t\|\leq  C_1n  \big]\geq 1-C_22^{-n}.\eqe
Let $\e_3,\e_4 \in (0,\wt\e_1)$ and $\e_3<\e_4$. We  check that for large enough $n$, under the event
\begin{equation}\label{eq:double-edge-lem-0-pf-1}
    \begin{split}
        &\{\sup_{t\in [-2^n,2^n]}  \|\cZ_t - Z_t\|\leq C_1n  \}  \cap \{\inf_{t\in [2^{(1-\wt\e_1)n},2^{n}]} L_t <- 2^{(1-\e_3)n/2} \} \\& \cap  \{  \inf_{t\in [-2^{n},-2^{(1-\wt\e_1)n}]} L_t < -2^{(1-\e_3)n/2}  \} \cap \{ \sup_{t\in [-2^{(1-\wt\e_1)n},2^{(1-\wt\e_1)n}]}\|Z_t\|<2^{(1-\e_4)n/2} \},
    \end{split}
\end{equation}
 for any $t_-,t_+\in\bbR$ such that $|t^-|\leq 2^{(1-\wt\e_1)n}$ and $t^+\geq 2^n$, the condition~\eqref{eq:double-edge-Mullin-A}   does not hold. Denote the four events in~\eqref{eq:double-edge-lem-0-pf-1} by $E_1 ,E_2 ,E_3 ,E_4 $, respectively. Then for large enough $n$, under $E_1 \cap E_4 $, $\cL_{t^-}>-2^{(1-\e_4)n/2+1}$, whereas under $E_1 \cap E_2 \cap E_3$, we have $\inf_{t^-\leq t\leq t^+} \cL_t<-2^{(1-\e_3)n/2-1}$. In particular, condition~\eqref{eq:double-edge-Mullin-A} does not hold. Similarly, under the event~\eqref{eq:double-edge-lem-0-pf-1}, there does not exist $t^-\leq -2^n$ and $|t^+|\leq  2^{(1-\wt\e_1)n}$ such that~\eqref{eq:double-edge-Mullin-A} holds. Therefore it suffices to prove $\bbP_\rmE\big[ E_1 \cap E_2 \cap E_3 \cap E_4 \big]\geq 1-\wt c_12^{-\wt c_2n}$ for some $\wt c_1,\wt c_2>0$. By~\eqref{eq:double-edge-lem-0-pf-1-00}, $\bbP_\rmE[E_1 ]\geq 1-C_22^{-n}$, and  {by Lemma ~\ref{lem:L_t-inf-a1a2}}, $\bbP_{\rmE}[E_2]\geq 1-c_2'2^{-c_2''n}$ and $\bbP_{\rmE}[E_3]\geq 1-c_3'2^{-c_3''n}$, and finally by  {Lemma~\ref{lem:L_t-inf-1}},  $\bbP_\rmE[E_4 ]\geq 1-c_4'2^{-c_4''n}$ for some constants $c_2',c_2'',c_3',c_3'',c_4',c_4''>0$. This completes the proof.
\end{proof}

\begin{lemma}\label{lem:BM-e3}
    Let $(B_t)_{t\geq 0}$ be the standard Brownian motion. There exists some constant $c>0$ such that the following is true. For any $\e\in(0,1)$,  $\bbP\big[\inf_{t\in[0,1]}B_t\geq -\e; B_1<\e \big]\leq c\e^3$.
\end{lemma}
\begin{proof}
    By the reflection principle, we have $\bbP\big[\inf_{t\in[0,1]}B_t\leq -\e; -\e\leq B_1\leq \e \big] = \bbP\big[-3\e\leq B_1\leq -\e \big]$. Therefore $\bbP\big[\inf_{t\in[0,1]}B_t\geq -\e; B_1<\e \big] = \bbP\big[-\e\leq B_1\leq \e \big] - \bbP\big[\inf_{t\in[0,1]}B_t\leq -\e; -\e\leq B_1\leq \e \big]  =   \bbP\big[-\e\leq B_1\leq \e \big]- \bbP\big[-3\e\leq B_1\leq -\e \big] \leq c\e^3$.
\end{proof}

Under the probability measure $\bbP_\rmE$ or $\bbP_\rmQ$, we write $\mr\eta$ for the space-filling $\SLE_\kappa$ curve $\eta$ parameterized by the Lebesgue measure such that $\mr\eta(0)=0$.
\begin{lemma}\label{lem:BM-e4}
    Let $\e_1\in (0,1/4)$ and $\beta>0$ such that $1-3\e_1-2\beta>0$. Then there exists $c_1,c_2>0$, such that under $\bbP_\rmQ$, with probability at least $1-c_12^{-c_2n}$, for any $t_1<t_2$ with $\mu_h(\mr\eta([t_1,t_2]))\geq 2^{-\e_1n}$ and $\mr\eta([t_1,t_2])\subseteq \bbD$, $\nu_h(\partial \mr\eta([t_1,t_2]))\geq 2^{(\beta-\frac{1}{2})n}$.
\end{lemma}

\begin{proof} 
We choose $\lambda>0$ and $p>1-2\beta$ such that $c_2':=3-(2(\lambda+p)+6\beta+3\e_1)>0$.
%We fix $\lambda\in (0,(1-3\e_0-2\beta)/10)$, let $\e_0' = \e_0+\lambda$ and $\beta' = \beta+\frac{\lambda}{2}$. 
By the definition of the boundary length process, \eqb\label{eq:bdryprocess}\nu_h(\partial\eta([t_1,t_2])) = (L_{t_2}-L_{t_1})+(R_{t_2}-R_{t_1})-2\inf_{s\in[t_1,t_2]}(L_s-L_{t_1})-2\inf_{s\in[t_1,t_2]}(R_s-R_{t_1}).\eqe
    Since under $\bbP_\rmQ$, $(L,R)$ evolves as standard 2D Brownian motion, by Lemma~\ref{lem:BM-e3}, the Brownian scaling and a union bound,  
    \begin{equation}\label{eq:lem:BM-e4-11}
        \bbP_\rmQ\big[\, \exists\, t_1,t_2\in [-2^{\lambda n},2^{\lambda n}]\cap 2^{-pn}\bbZ: |t_1-t_2|\geq 2^{-\e_1n-1} \ \text{and }   \nu_h(\partial\eta([t_1,t_2]))\leq 2^{(\beta-\frac{1}{2})n+1}  \big]\lesssim 2^{2(\lambda+p)n}\cdot 2^{6(\beta+\frac{\e_1-1}{2})n}.
    \end{equation}
    Moreover, by standard estimates of the Brownian motion and a union bound, since $p>1-2\beta$, for some constant $c_1'>0$,
      \begin{equation}\label{eq:lem:BM-e4-22}
        \bbP_\rmQ\big[\, \exists\, t_1\in [-2^{\lambda n},2^{\lambda n}]\cap 2^{-pn}\bbZ:  \sup_{s\in[t_1-2^{-pn},t_1+2^{-pn}]}\|Z_s-Z_{t_1}\|\geq 2^{(\beta-\frac{1}{2})n-6} \big]\lesssim 2^{(\lambda+p)n}\cdot 2^{-c_1'n^2}.
    \end{equation}
    Then by~\eqref{eq:lem:BM-e4-11} and~\eqref{eq:lem:BM-e4-22} and the definition of $c_2'$, we have
    \begin{equation*}
         \bbP_\rmQ\big[\, \exists\, t_1,t_2\in [-2^{\lambda n},2^{\lambda n}]: |t_1-t_2|\geq 2^{-\e_1n} \ \text{and }   \nu_h(\partial\eta([t_1,t_2]))\leq 2^{(\beta-\frac{1}{2})n}  \big]\lesssim 2^{-c_2'n}.
    \end{equation*}
     {Finally, by Lemma~\ref{lem:SLE-contain-2}, for some $c_2''>0$, we have}
    \begin{equation*}
        \bbP_\rmQ[\bbD\nsubseteq \eta([-2^{\lambda n},2^{\lambda n}]) ]\lesssim 2^{-c_2''n}.
    \end{equation*}
    Therefore by a union bound, for some $c_1,c_2>0$, under $\bbP_\rmQ$, with probability at least $1-c_12^{-c_2n}$, there cannot be $t_1<t_2$ such that $\eta([t_1,t_2])\subseteq \bbD$, $|t_1-t_2|\geq 2^{-\e_1n}$ and $\nu_h(\partial\eta([t_1,t_2]))\leq 2^{(\beta-\frac{1}{2})n}$.
    The lemma then follows by reparameterizing $\eta$ via Lebesgue measure.
\end{proof}

\begin{lemma}\label{lem:gap-hat-J-9-4}
   There exists $\wt\e_0,\wt\e_2,\wt c_3,\wt c_4>0$ such that the following is true. Under $\bbP_\rmE$, for any $n\in\bbN$, with probability at least $1-\wt c_3 2^{-\wt c_4 n}$, there does not exist $t^-,t^+\in [-2^n,2^n]\backslash[-2^{(1-\wt\e_2)n},2^{(1-\wt\e_2)n}]$ and $t^+-t^-\geq 2^{(1-\wt \e_0)n}$ such that~\eqref{eq:double-edge-Mullin-A} holds.
\end{lemma}

\begin{proof}
    Note that under $\bbP_\rmE$ (resp.\ $\bbP_\rmQ$), we have $h|_\bbD = h^0|_\bbD$ (resp.\ $h|_\bbD = (h^0-\gamma\log|\cdot|)|_\bbD$, where $h^0$ is a whole plane GFF normalized such that its average over $\partial\bbD$ is 0. Let $\theta >0$. Then using this coupling, for $\beta,\e_1,c_1,c_2$  as in Lemma~\ref{lem:BM-e4},   under $\bbP_\rmE$,  with probability at least $1-c_12^{-c_2n}$,  for any $t^-<t^+$ with $\mu_h(\mr\eta([t^-,t^+]))\geq 2^{-\e_1n}$ and $\partial\mr\eta([t^-,t^+])\subseteq \bbD\backslash B(0;2^{-\theta n})$, $\nu_h(\partial \mr\eta([t^-,t^+]))\geq 2^{(\beta-\frac{1+\gamma^2\theta}{2})n}$.  {By Lemma~\ref{lem:SLE-contain-1}, for any $\delta,\e_2\in(0,1)$, there exists some $c_3''>0$, such that
    $$\bbP_\rmE[B(0;2^{-\theta n})\nsubseteq\eta([-2^{-(\theta-\delta)n/\gamma Q},2^{-(\theta-\delta)n/\gamma Q}  ]) ]\lesssim  2^{-c_3''n}; \ \ \ \bbP_\rmE[\eta([-2^{-\e_2n},2^{-\e_2n}])\nsubseteq \bbD]\lesssim  2^{-c_3''n}. $$}
    Pick $\theta,\delta,\e_2<\e_1$ such that $\e_3:=(\theta-\delta)/\gamma Q>\e_2$ and $\e_4:=\beta-\frac{\gamma^2\theta}{2}+\frac{\e_2}{2}>0$.
    We reparameterize $\mr\eta$ by the quantum area to transfer back to $\eta$. By a union bound, for some constant $c_5',c_5''>0$, under the probability measure $\bbP_\rmE$, with probability at least $1-c_5'2^{-c_5''n}$, for any $t^-<t^+$ where $t^-,t^+\in [-2^{-\e_2n},2^{-\e_2n}]\backslash [-2^{-\e_3 n},2^{-\e_3 n}  ]$ with $t^+-t^-> 2^{-\e_1n}$, we have $\nu_h(\partial \eta([t^-,t^+]))\geq 2^{(\beta-\frac{1+\gamma^2\theta}{2})n}$. Further by the definition of the boundary length process as in~\eqref{eq:bdryprocess} together with  {the scaling in Lemma~\ref{lem:L_t-inf-1}}, with probability at least $1-c_5'2^{-c_5''n}$ under $\bbP_\rmE$, for every $t^-<t^+$ with $t^-,t^+\in [-2^n,2^n]\backslash [-2^{(1+\e_2-\e_3)n},2^{(1+\e_2-\e_3)n}]$ and $t^+-t^-\geq 2^{(1+\e_2-\e_1)n}$,
    $$L_{t^-}+L_{t^+}-2\inf_{s\in [t^-,t^+]} L_s + R_{t^-}+R_{t^+}-2\inf_{s\in [t^-,t^+]} R_s\geq 2^{\frac{(1+\e_2)n}{2}}\cdot 2^{(\beta-\frac{1+\gamma^2\theta}{2})n} = 2^{\e_4n}.$$
    Further combined with~\eqref{eq:double-edge-lem-0-pf-1-00}, for large enough $n$, for some $c_6',c_6''>0$, with probability at least $1-c_6'2^{-c_6''n}$ under $\bbP_\rmE$,  for every $t^-<t^+$ with $t^-,t^+\in [-2^n,2^n]\backslash [-2^{(1+\e_2-\e_3)n},2^{(1+\e_2-\e_3)n}]$ and $t^+-t^->2^{(1+\e_2-\e_1)n}$,
     $$\cL_{t^-}+\cL_{t^+}-2\inf_{s\in [t^-,t^+]} \cL_s + \cR_{t^-}+\cR_{t^+}-2\inf_{s\in [t^-,t^+]} \cR_s\geq\frac{1}{4}\cdot 2^{\frac{(1+\e_2)n}{2}}\cdot 2^{(\beta-\frac{1+\gamma^2\theta}{2})n} = 2^{\e_4n-2},$$
     which further implies the conclusion by setting $\wt\e_0=\e_1-\e_2$ and $\wt \e_2=\e_3-\e_2$.
\end{proof}

\begin{proof}[Proof of Proposition~\ref{lem:gap-hat-J}]
  Let $\wt\e_0,\wt\e_2$ be as in Lemma~\ref{lem:gap-hat-J-9-4}, and $F^{2,n}$ be the corresponding event considered in Lemma~\ref{lem:gap-hat-J-9-4}.  We set $\e_0 = \frac{1}{3} ( \wt\e_0\wedge\wt\e_2)$.  
   %In light of Lemma~\ref{lem:hat}, it suffices to prove that for some $c>0$, with $\bbP_\rmE$-probability  at least $1-c^{-1}2^{-cn}$, there does not exist any $t^-<t^+$ such that~\eqref{eq:double-edge-Mullin-A} holds and $[t^-,t^+]\cap [-2^n,2^n]$ has length at least $2^{(1-\e_0)n}$. 
  Let $F^{1,n}$ be the event that there does not exist any $t_1,t_2$ with $|t_1|\leq 2^{(1-\wt\e_2)n}$ and $t_2\in\bbR$ such that $|t_2-t_1|\geq 2^{(1-\e_0) n-1}$ and  ~\eqref{eq:double-edge-Mullin-A} holds for $(t_1\wedge t_2,t_1\vee t_2)$ in place of $(t^-,t^+)$. Since $\e_0<\wt\e_2$,   Lemma~\ref{lem:gap-hat-J-9-3} and a scaling imply that $1-\bbP_\rmE[F^{1,n}]\lesssim 2^{-c''n}$ for some $c''>0$. Then under $F^n:=F^{1,n}\cap F^{2,n}$, there does not exist any $t^-,t^+\in [-2^n,2^n]$ such that $t^+-t^-\geq 2^{(1-\e_0)n-1}$ and~\eqref{eq:double-edge-Mullin-A} holds. Consider the event $\cap_{m=n}^\infty F^m$. If there exists $J\subseteq [-2^n,2^n]$ such that $|J|=2^{(1-\e_0)n}$ and $J\cap \mr \cJ=\emptyset$, then by Lemma~\ref{lem:hat},  there exists $[t^-,t^+]$ such that $J\subseteq [t^-,t^+]$ and~\eqref{eq:double-edge-Mullin-A} holds. Then we have the following two possibilities:
  \begin{itemize}
      \item $|t^-|,|t^+|\leq 2^{n+1}$. Then $t^+-t^-\geq 2^{(1-\e_0)n}\geq 2^{(1-\e_0)(n+1)-1}$, and therefore $F^{n+1}$ does not happen;
      \item $|t^-|\vee|t^+|\geq 2^{n+1}$. Suppose for some $m\geq n+1$, $|t^-|\vee|t^+| \in [2^m,2^{m+1}]$. Then $t^+-t^-\geq 2^{m-1}\geq 2^{(1-\e_0)(m+1)-1}$, and therefore $F^{m+1}$ does not happen.
  \end{itemize}
  Therefore the event in~\eqref{eq:lem:gap-hat-J} does not happen under the event $\cap_{m=n}^\infty F^m$. The proposition then follows from Lemma~\ref{lem:gap-hat-J-9-3}, Lemma~\ref{lem:gap-hat-J-9-4} and a union bound.
\end{proof}

\begin{proposition}\label{lem:gap-hat-J^-}
    There exists $\e_1>0$ and $c\in(0,1)$ such that for any $m\in\bbN$, 
    \eqb\label{eq:lem:gap-hat-J^-}\bbP_\rmE\bigg[\exists \text{ interval } J\subseteq [-2^m,2^m], |J|=2^{(1-\e_1)m}, \mr\cJ^-\cap J=\emptyset\bigg]\lesssim 2^{-cm}.\eqe
    In particular, for any $a_1>a_2>0$, 
\begin{equation}\label{eq:gap-hat-J^--a1a2}
    \bbP_\rmE\Big[\mr\cJ^-\cap [-2^{a_1m},-2^{a_2m}]=\emptyset  \Big]\lesssim 2^{-cm}.
\end{equation}
    Furthermore, there exists $c_0,c>0$, such that 
     \eqb\label{eq:jump-hat-J^-}\bbP_\rmE\big[\mr N_\rmL[-2^m,2^m]\geq c_0m\big]\lesssim 2^{-cm}; \ \ \ \ \bbP_\rmE\big[\mr N_\rmR[-2^m,2^m]\geq c_0m\big]\lesssim 2^{-cm}.  \eqe
     % In particular,
     % \eqb\label{eq:jump-hat-J^--1}\bbP_\rmE\big[\exists s\in \mr\cJ^-, s.t., \mr I_s = [s^-,s]\cap [-2^m,2^m]\neq\emptyset, \text{ and } \max\{\cL_s-\cL_{s^-}, \cR_s-\cR_{s^-}\}\geq c_0m\big]\lesssim 2^{-cm}.\eqe
     The same conclusions hold if we replace $\bbP_\rmE$ by $\bbP_\rmQ$.
\end{proposition}

Recall the operator $\Phi_{\rm simple}$ for the continuous-time walks $\cZ$, $\mr\cZ =\Phi_{\rm simple}(\cZ) $ and we set $\mr\cZ_k = \Phi_{\rm simple}(\cZ|_{I(k)})$ which is measurable with respect to $\cZ|_{I(k)}$. Let $\mr\cJ_k$, $\mr\cJ_k^+$ and $\mr\cJ_k^-$ be defined for the walk $\mr\cZ_k$ in the same way as~\eqref{eq:def-hat-JL} and~\eqref{eq:def-hat-JR}, and $\mr N_{k,\rmL}[t_0,t_1],\mr N_{k,\rmR}[t_0,t_1]$ be defined for the walk $\mr\cZ_k$ in the same way as ~\eqref{eq:def-jump-hat-N} using $\mr\cJ_k$, $\mr\cJ_k^+$ and $\mr\cJ_k^-$. %The following is immediate from definition.

In light of Lemma~\ref{lem:A2}, we first prove the corresponding statement for $\mr\cZ_k$ under the probability measure $\bbP_\rmQ$, and later use Theorem~\ref{thm:KMT-0-cone} to transfer to the setting of $\bbP_\rmE$. For $k,m>0$, we let $E^{\uparrow}_{k,m}$ be the event where there exists consecutive times $t_1<...<t_m$ in $\mr\cJ_k$, such that $\cL_{t_1}\leq \cL_{t_2}\leq...\leq\cL_{t_m}$ and $\cR_{t_1}\leq \cR_{t_2}\leq...\leq\cR_{t_m}$. Note that under the event $(E^{\uparrow}_{k,m})^c$, we have $\mr N_{k,\rmL}[a_k,b_k]\leq m$ and $\mr N_{k,\rmR}[a_k,b_k]\leq m$ where $I(k)=[a_k,b_k]$.

\begin{lemma}\label{lem:number-jump-Q}
 There exists $c'>0$, such that $\bbP_\rmQ[E^{\uparrow}_{k,m}]\leq c'(2^{2.2k-m/20}+2^{-k^2}).$
\end{lemma}

\begin{proof}
Fix $N\in\bbN$. Consider a simple random walk $(\cX_n,\cY_n)_{n=0,...,N}$ on $\bbZ^2$ of $N$ steps. For all $n^-<n^+$ such that \eqb\label{eq:pi/2-excursion-AA}\cX_{n^-}=\cX_{n^+} = \inf_{n^-\leq j\leq n^+}\cX_j, \ \ \cY_{n^-}=\cY_{n^+}=\inf_{n^-\leq j\leq n^+}\cY_j,\eqe  we erase all the times in $(n^-,n^+]$ from $[0,N]\cap\bbZ$ as in the definition of $\Phi_{\rm simple}$ and let $\cN$ be the retained set consisting of $0=\sigma_1<\sigma_2<\cdots$.  Let $\wt E^{\uparrow}_{N,m}$ be the event that there exists at least $m$ times in $\cN$, such that $\cX_{\sigma_1}\leq ... \leq \cX_{\sigma_m}$, $\cY_{\sigma_1}\leq ... \leq \cY_{\sigma_m}$. We prove that for each $N$, \eqb\label{eq:pf-lem-Ex-111112}\bbP[\wt E^{\uparrow}_{N,m}]\leq (\frac{15}{16})^{m-1}.\eqe

Let $1\leq j<m$ and we condition on $(\cX_n,\cY_n)_{0\leq n\leq \sigma_j}$ and $(\sigma_1,...,\sigma_j)$ where $(\cX_{\sigma_i},\cY_{\sigma_i})_{0\leq i\leq j}$ are nondecreasing in both coordinates. Then $(\cX_n,\cY_n)_{\sigma_j\leq n\leq N}$ has the law as a simple random walk started from $(\cX_{\sigma_j},\cY_{\sigma_j})$ and stopped at time $N-\sigma_j$, conditioned on the event $E_{\sigma_1,...,\sigma_j}^N$ that there does not exist any $1\leq i< j$ and $t\in(\sigma_j,N]$ such that $\cL_{\sigma_i} = \cL_{t} = \inf_{\sigma_i\leq u\leq t}\cL_u$ and $\cR_{\sigma_i} = \cR_{t} = \inf_{\sigma_i\leq u\leq t}\cR_u$. The reason is, for any trace of  $(\cX_n,\cY_n)_{n=0,...,N}$ with given  $(\cX_n,\cY_n)_{0\leq n\leq \sigma_j}$ and $(\sigma_1,...,\sigma_j)$, if the above equality does hold for some $i<j$ and $t\in(\sigma_j,N]$, then the time $\sigma_j$ would be erased and cannot remain in $\cN$; if the event $E_{\sigma_1,...,\sigma_j}^N$ holds, then there cannot be $n^-<\sigma_i\leq n^+$ for $i\leq j$ such that~\eqref{eq:pi/2-excursion-AA} holds. Indeed, if this is not the case, then $(\cX_{n^-},\cY_{n^-})$ must be equal to   $(\cX_{\sigma_{i'}},\cY_{\sigma_{i'}})$ with $i'$ being the largest index with $\sigma_{i'}\leq n^-$ since  $(\cX_{\sigma_i},\cY_{\sigma_i})_{0\leq i\leq j}$  are nondecreasing in both coordinates, and this  contradicts with the existing condition on $(\cX_n,\cY_n)_{0\leq n\leq \sigma_j}$ and our event $E_{\sigma_1,...,\sigma_j}^N$. Now we assume that $j\geq 2$, $\sigma_j+2\leq N$ and $\cX_{\sigma_{j}}=\cX_{\sigma_{j-1}}$. Let   $A_j:=\{\cX_{\sigma_j+1}-\cX_{\sigma_j}=-1,\cX_{\sigma_j+2}-\cX_{\sigma_j+1}=1\}$. For a walk started from $(\cX_{\sigma_j},\cY_{\sigma_j})$ that is in the event $E_{\sigma_1,...,\sigma_j}^{N-2}$, if we insert two steps $(-1,0),(1,0)$ at the beginning, then we get a walk that is in $E_{\sigma_1,...,\sigma_j}^{N}\cap A_j$. Therefore $\bbP[A_j\cap E_{\sigma_1,...,\sigma_j}^{N}]\geq \frac{1}{16}\bbP[E_{\sigma_1,...,\sigma_j}^{N-2}]\geq \frac{1}{16}\bbP[E_{\sigma_1,...,\sigma_j}^{N}]$ where the factor $\frac{1}{16}$ is from the insertion map and the last inequality is from the fact that $ E_{\sigma_1,...,\sigma_j}^{N}$ is decreasing over $N$. This gives $\bbP[A_j|E_{\sigma_1,...,\sigma_j}^{N}]\geq \frac{1}{16}$. Furthermore, under the event $A_j$, any excursion starting before $\sigma_j$ that erases $\sigma_j+1$ would also erase $\sigma_j$, while there cannot be any excursion starting at $\sigma_j$ that erase $\sigma_j+1$, and therefore $\sigma_{j+1}=\sigma_j+1$ and $\wt E^{\uparrow}_{N,j+1}$ no longer holds. This proves that \eqb\label{eq:pf-lem-Ex-111}\bbP[\wt E^{\uparrow}_{N,j+1}| \wt E^{\uparrow}_{N,j},(\cX_n,\cY_n)_{0\leq n\leq \sigma_j},(\sigma_1,...,\sigma_j)]\leq\frac{15}{16}.\eqe
  If $\sigma_j+1 = N$ or $j=1$, we can directly work with the event that the first step of the walk $(\cX_n,\cY_n)_{\sigma_j\leq n\leq N}$ is $(-1,0)$ and~\eqref{eq:pf-lem-Ex-111} holds with $3/4$ in place of $15/16$. If $\sigma_j=N$, then the continuation is not possible and the conditional probability is 0. The case where $\cY_{\sigma_{j}}=\cY_{\sigma_{j-1}}$ follows similarly. This proves~\eqref{eq:pf-lem-Ex-111112}. 

  Since $\cJ$ is a  {unit rate} Poisson point process under the measure $\bbP_\rmQ$, $\bbP_\rmQ[\#\{\cJ\cap [-2^k,2^k]\}\geq 2^{1.1k}]\leq 2^{-k^2}$.  Using description of the $\cZ$ as in part~(\ref{cond:prop:KMT-gamma-cone-(i)}) of Theorem~\ref{prop:KMT-gamma-cone}, by a union bound over all the starting and ending positions in $\cJ\cap [-2^k,2^k]\supseteq \cJ\cap I(k)$ and ~\eqref{eq:pf-lem-Ex-111112}, with probability at least $1-(2^{2.2k}\cdot \big(\frac{15}{16}\big)^{m-1}+2^{-k^2})$, uniformly over all possible realizations of $I(k)\subseteq [-2^k,2^k]$, the event $E^{\uparrow}_{k,m}$ does not happen.
 Therefore \begin{equation}
      \bbP_\rmQ[E^{\uparrow}_{k,m}]\leq 2^{2.2k}\cdot \big(\frac{15}{16}\big)^{m-1}+2^{-k^2},
  \end{equation}
which gives the claim.
\end{proof}

\begin{lemma}\label{lem:number-jump-E}
    There exists constants $c_0,c>0$, such that $\bbP_\rmE[E^{\uparrow}_{k,c_0k}]\lesssim 2^{-ck}.$
\end{lemma}
\begin{proof}
   Let $C_0>0$ be the constant from part~(\ref{cond:thm:KMT-0-cone-(iv)}) of Theorem~\ref{thm:KMT-0-cone}. We take $c_0>100(C_0+1)$ and let $m=c_0k$ in Lemma~\ref{lem:number-jump-Q}. Note that if we define the function $G(\cZ,\fD) = \mathds{1}_{E^\uparrow_{k,c_0k}}$, then the event $A$ defined in~\eqref{eq:def-A} agrees with $E^\uparrow_{k,c_0k}$. The claim then follows from Lemma~\ref{lem:number-jump-Q} and Theorem~\ref{thm:KMT-0-cone} applied to the event $A$.
\end{proof}

\begin{proof}[Proof of Proposition~\ref{lem:gap-hat-J^-}]
    We work with the setting of $\bbP_\rmE$ and  the event $\{[-2^{1.5m},2^{1.5m}]\subseteq I(2m)\}$, which has probability $1-2^{-cm}$. By Proposition~\ref{lem:gap-hat-J},   with probability $1-c'2^{-cm}$, $[-2^{1.5m},-2^{m}]\cap \mr\cJ\neq\emptyset$ and $[2^{m},2^{1.5m}]\cap \mr\cJ\neq\emptyset$ and further by Lemma~\ref{lem:A2}, with probability  $1-c'2^{-cm}$, $\mr\cZ|_{[-2^m,2^m]} = \mr\cZ_{2m}|_{[-2^m,2^m]}$, $\mr\cJ\cap [-2^m,2^m] = \mr\cJ_{2m}\cap [-2^m,2^m]$ and $\mr\cJ^-\cap [-2^m,2^m] = \mr\cJ_{2m}^-\cap [-2^m,2^m]$. Now by Lemma~\ref{lem:number-jump-E} and Lemma~\ref{lem:A2}, for some constants $c',c'c_0>0$, with probability $1-c'2^{-cm}$, $\mr N_{2m,\rmL}[-2^m,2^m]\leq c_0m$, $\mr N_{2m,\rmR}[-2^m,2^m]\leq c_0m$ and further $\mr N_{\rmL}[-2^m,2^m]\leq c_0m$, $\mr N_{\rmR}[-2^m,2^m]\leq c_0m$. This verifies~\eqref{eq:jump-hat-J^-}. By Proposition~\ref{lem:gap-hat-J},  for $\e_1\in(0,\e_0)$, with probability $1-c'2^{-cm}$,  each interval $J\subseteq [-2^m,2^m]$ of length $2^{(1-\e_1) m}$ contains at least $2^{(\e_0-\e_1)m/2}$ points in $\mr\cJ$, and under the event $\{\mr N_\rmL[-2^m,2^m],\mr N_\rmR[-2^m,2^m]\leq c_0m\}$, $J$ must contain a point in $\mr\cJ^-$. This implies~\eqref{eq:lem:gap-hat-J^-}. The proof in the setting of $\bbP_\rmQ$ is similar, where one use Proposition~\ref{lem:gap-hat-J-Q} and Lemma~\ref{lem:number-jump-Q}. The relation~\eqref{eq:gap-hat-J^--a1a2} is straightforward from~\eqref{eq:lem:gap-hat-J^-} by replacing $m$ with $a_1m$.
\end{proof}

\begin{corollary}\label{cor:eta-boundary}
    Almost surely, for all large enough $r$, the following holds. If $s\in\mr\cJ^-$ such that $\eta(\mr I_s)\cap B(0;r)\neq\emptyset$, then $\partial\eta(\mr I_s)\cap B(0;r)\neq\emptyset$. Similarly, if $s\in \cJ^-$ such that $\eta(  I_s)\cap B(0;r)\neq\emptyset$, then $\partial\eta( I_s)\cap B(0;r)\neq\emptyset$.
\end{corollary}

\begin{proof}
    We only prove the first claim as the proof for the second is similar. By Proposition~\ref{lem:gap-hat-J^-} and the Borel-Contelli lemma, almost surely the interval $\mr I_s$ containing 0 is finite, and the claim holds if $0\in \mr I_s$. Since $\eta(0)=0$, it is clear that for any $r>0$, if  $\eta(\mr I_s)\cap B(0;r)\neq\emptyset$ and $0\notin \mr I_s$, $\partial\eta(\mr I_s)\cap B(0;r)\neq\emptyset$.
\end{proof}

\subsubsection{Bounds on the time sets}\label{subsubsec:3.1.3}

\begin{lemma}\label{lem:prop:time-set-A-5}
For any $a>0$ and $d>1$, there exists $c>0$  such that 
$$\bbP_\rmE\big[\max\{\mr\tau_{t,\rmL}^{\rmC,d}:t\in [-2^m,2^m]\cap\cJ\}>2^{(1+a)m} \big]\lesssim 2^{-cm}.$$
Furthermore, Assertion~\ref{cond:prop:time-set-A-5} in Proposition~\ref{prop:time-set-A} holds.
\end{lemma}
\begin{proof}
    { The proof of the first claim is identical to  the proof of~\cite[Lemma 5.10]{Tutte}, where we replace ~\cite[Lemma 5.8, Lemma 5.9]{Tutte} there with~\eqref{eq:lem:gap-hat-J^-} and~\eqref{eq:jump-hat-J^-} as in Proposition~\ref{lem:gap-hat-J^-}, respectively. The proof for the second claim is identical to the proof of ~\cite[Proposition 5.6]{Tutte}, where we may assume $\delta<0.1$ and replace all the numbers 0.45, 0.8, 0.85, 0.9 there with $(1-\delta)/2, 1-2\delta, 1-3\delta/2, 1-\delta$, respectively. We omit the details.}
\end{proof}

To bound the cardinality of time sets $\mr\cT_t^{\rmC,d}$,  {we will use the time set $\cT_t^{h,\infty}$ as in ~\cite[Section 5.1.2]{Tutte} and introduce a variant $\mr\cT_t^{h,\infty}$}.
Fix $h\geq 1$. For $t_0\in\cJ$, define 
$$\tau_{t_0,\rmL}^h = \inf\{t>t_0:\cL_t-\cL_{t_0}<-h\} = \inf\{t>t_0: t\in\cJ^- \  \text{and} \ \cL_t-\cL_{t_0}<-h \}.$$
Let \eqb\label{eq:def-Tt-Lh-infty} \cT_{t_0,\rmL}^{h,\infty} = \{t\in \cJ_\rmL^-\cap [t_0,\tau_{t_0,\rmL}^h]:\{\cL_s:s\in I_t\}\cap [\cL_{t_0}-h,\cL_{t_0}+h]\neq\emptyset\}.\eqe
It is clear  that $\tau_{t_0,\rmL}^h\in  \cT_{t_0,\rmL}^{h,\infty} $.  Define $\tau_{t_0,\rmR}^h$ and $\cT_{t_0,\rmR}^{h,\infty}$ in the exact same way by using $\cR$ instead of $\cL$. For $t_0\in\cJ$,$s_0\in\cJ^-$, set
 \eqb\label{eq:def-Tth-simple} \cT_{t_0}^{h,\infty,0} = \cT_{t_0,\rmL}^{h,\infty}\cup \cT_{t_0,\rmR}^{h,\infty}; \ \ \  \cT_{s_0}^{h,\infty} = \bigcup_{t_0\in\cJ:s_0\in \cT_{t_0}^{h,\infty,0}}\cT_{t_0}^{h,\infty,0}.\eqe
 For $m\geq 1$ and $q\in\{\rmL,\rmR\}$, if $t_0\in I(m)$, set $ \cT_{t_0,q}^{h,m}= \cT_{t_0,q}^{h,\infty}\cap I(m)$ and otherwise set $ \cT_{t_0,q}^{h,m}=\emptyset$. 
%set $\cT_{t_0}^{h,m,0} = \cT_{t_0}^{h,\infty,0} $ if $\cT_{t_0}^{h,\infty,0}\subseteq I(m)$ and $\cT_{t_0}^{h,m,0} =\emptyset$ otherwise. }
For $t_0\in\cJ$, $s_0\in\cJ^-$, set
$$\cT_{t_0}^{h,m,0} = \cT_{t_0,\rmL}^{h,m}\cup \cT_{t_0,\rmR}^{h,m}; \ \ \  \cT_{s_0}^{h,m} = \bigcup_{t_0\in\cJ:s_0\in \cT_{t_0}^{h,m,0}}\cT_{t_0}^{h,m,0}.$$
Then by  {~\cite[Lemma 5.5]{Tutte}, we have}
\begin{lemma}\label{lem:T_t^h,m}
   {Let $s_0\in\cJ^-\cap I(m)$. Then} $\cT_{s_0}^{h,m}\subseteq \cT_{s_0}^{h,\infty}$, and  $\cT_{s_0}^{h,m}$ is a measurable function with respect to $\cZ|_{I(m)}$. Furthermore, if $\bigcup_{u\in\cT_{s_0}^{h,\infty}}I_u\subseteq  I(m)$, then $\cT_{s_0}^{h,m}=\cT_{s_0}^{h,\infty}$.
\end{lemma}

Similarly, let $\mr\tau_{t_0,\rmL}^h = \inf\{t\in\mr\cJ^-:  \inf\limits_{s\in [t_0,t]} \cL_s-\cL_{t_0}<-h \}$, and  
$$ {\mr\cT_{t_0,\rmL}^{h,\infty} = \{t\in \mr\cJ^-\cap [t_0,\mr\tau_{t_0,\rmL}^h]:\{\cL_s:s\in \mr I_t\}\cap [\cL_{t_0}-h,\cL_{t_0}+h]\neq\emptyset\}.}$$
We have the following analog of Lemma~\ref{lem:lower-bd-tau-hat-AA}. The proof is identical and we omit it here.
\begin{lemma}\label{lem:lower-bd-tau-hat-AAAA}
      Let $t_0\in\cJ$ and $\mr \tau =  \mathring \tau_{t_0,\rmL}^{h}$. Then $\mr \tau\in \mr \cT_{t_0,\rmL}^{h,\infty}$. If $[t_0,\mr \tau)\cap \mr\cJ^-=\emptyset$, then $\mr\cT_{t_0,\rmL}^{h,\infty} = \{\mr\tau\}$. If $[t_0,\mr \tau)\cap \mr\cJ^-\neq\emptyset$, then for $t\in [t_0, \mr \tau)$, $\cL_t\geq \cL_{t_0}- h$, whereas $\cL_{\mr \tau}< \cL_{t_0}- h$. The same holds for $(\cR,\rmR)$ instead of $(\cL,\rmL)$.
\end{lemma}

%Note that for $u_0\in \mr\cT_{t_0,\rmL}^{h,\infty} $, if $u_0$ is not  the largest element in $\mr\cT_{t_0,\rmL}^{h,\infty}$, we must have $\cL_{u_0}\geq\cL_t-h$.
Define $\mr\tau_{t_0,\rmR}^h$ and $\mr\cT_{t_0,\rmR}^{h,\infty}$ in the exact same way by using $\cR$ instead of $\cL$. Set
$$\mr\cT_{t_0}^{h,\infty,0} = \mr\cT_{t_0,\rmL}^{h,\infty}\cup \mr\cT_{t_0,\rmR}^{h,\infty}.$$
For $s_0\in\mr\cJ^-$, set
$$\mr\cT_{s_0}^{h,\infty} = \bigcup_{t_0\in\cJ:s_0\in \mr\cT_{t_0}^{h,\infty,0}}\mr\cT_{t_0}^{h,\infty,0}.$$
% {
% From the definition of $\cT_{s_0}^{h,m}$ and $\cT_{s_0}^{h,\infty}$, we immediately have
% \begin{lemma}\label{lem:T_t^h,m}
%    \tcr{Let $s_0\in\cJ^-\cap I(m)$. Then} $\cT_{s_0}^{h,m}\subseteq \cT_{s_0}^{h,\infty}$, and  $\cT_{s_0}^{h,m}$ is a measurable function with respect to $\cZ|_{I(m)}$. Furthermore, if $\cT_{s_0}^{h,m}\subseteq  I(m)$, then $\cT_{s_0}^{h,m}=\cT_{s_0}^{h,\infty}$.
% \end{lemma}
% }

The following lemma allows us to bound the time set $\mr\cT_t^{\rmD,d }$ (and  $\mr\cT_t^{\rmC,d }$ by of Assertion~\ref{cond:lem:def-wh-T-Cd-4} of Lemma~\ref{lem:def-wh-T-Dd}) by $\mr \cT_t^{h,\infty}$.
\begin{lemma}\label{lem:T_t^h}
    Let $k\in\bbN$ and $h>0$. Suppose  
    \begin{equation}\label{eq:lem:T_t^h}
        d\,\sup_{t\in [-2^{k+1},2^{k+1}]}|L_t-\cL_t|+2d<\frac{h}{8}, \ \ \  \ d\,\sup_{t\in [-2^{k+1},2^{k+1}]}|R_t-\cR_t|+2d<\frac{h}{8}, \ \ \  {\mr\cJ^-\cap(-2^{k+1},-2^k)\neq\emptyset},
    \end{equation}
   and in addition,
    \begin{equation}\label{eq:lem:T_t^h-1}
       d\cdot \mr N_\rmL[-2^{k+1},2^{k+1}]+1\leq \frac{h}{8}, \ \ d\cdot \mr N_\rmR[-2^{k+1},2^{k+1}]+1\leq \frac{h}{8},  
    \end{equation}
    then for any $t\in\mr\cJ^-$ such that $\mr\cT_t^{\rmD,d }\subseteq I(k)$, it holds that $ \mr\cT_t^{\rmD,d }\subseteq  \mr\cT_t^{h,\infty}$.
\end{lemma}
\begin{proof}
    %We only work on the second claim; the first claim can be proved analogously.
    Assume $\mr\cT_t^{\rmD,d}\subseteq I(k)$. Let $t_0\in\cJ$ such that $t\in \wh{\mr\cT_{t_0,\rmL}^{\rmD,d}}\cup \wh{\mr\cT_{t_0,\rmR}^{\rmD,d}}\subseteq \mr\cT_t^{\rmD,d}$.  Then $ \wh{\mr\cT_{t_0,\rmL}^{\rmD,d}}\subseteq I(k)$ and $\mr\tau_{t_0,\rmL}^{\rmD,d}\in I(k)$, and the same holds for  $\rmR$ in place of $\rmL$.  {Furthermore, since $t_0^+:=\inf\{s\in\mr\cJ^-:s>t_0\}\in \wh{\mr\cT_{t_0,\rmL}^{\rmD,d}}$, under the event $\{\mr\cJ^-\cap(-2^{k+1},-2^k)\neq\emptyset\}$, one must have $t_0\geq -2^{k+1}$.} Therefore $[t_0,\mr\tau_{t_0,\rmL}^{\rmD,d}]\subseteq [-2^{k+1},2^{k+1}]$, and under  conditions~\eqref{eq:lem:T_t^h} and~\eqref{eq:lem:T_t^h-1}, 
    \eqb\label{eq:pf-eq:lem:T_t^h-1} \mr\Delta_{t_0,\rmL}^{\rmD,d} = d\cdot\mr\Delta_{ \rmL}[t_0,\mr\tau_{t_0,\rmL}^{\rmD,d}]=  d\sup_{s\in[t_0,\mr\tau_{t_0,\rmL}^{\rmD,d}]}|(\cL_s-L_s)-(\cL_{t_0}-L_{t_0})|+2d+ d\mr N_\rmL[t_0,\mr\tau_{t_0,\rmL}^{\rmD,d}] <\frac{h}{2}.\eqe
    If $[t_0,\mr\tau_{t_0,\rmL}^{\rmD,d})\cap \mr\cJ^-=\emptyset$, $t_0^+=\mr\tau_{t_0,\rmL}^{\rmD,d}$ and we have $\wh{\mr\cT_{t_0,\rmL}^{\rmD,d}} = \{t_0^+\}\subseteq \mr \cT_{t_0,\rmL}^{h,\infty}$. If $[t_0,\mr\tau_{t_0,\rmL}^{\rmD,d})\cap \mr\cJ^-\neq\emptyset$, then by Lemma~\ref{lem:lower-bd-tau-hat-AA}, we have $\tau_{t_0,\rmL}^h \geq \mr\tau_{t_0,\rmL}^{\rmD,d}$, which implies $\mr\cT_{t_0,\rmL}^{\rmD,d}\subseteq \mr \cT_{t_0,\rmL}^{h,\infty}$ by~\eqref{eq:pf-eq:lem:T_t^h-1} and the definition. For $u\in \wh{\mr\cT_{t_0,\rmL}^{\rmD,d}}\backslash \mr\cT_{t_0,\rmL}^{\rmD,d}$, by ~\eqref{eq:TtDdcontainedTtCd+1-pf-1},~\eqref{eq:lem:T_t^h-1} and~\eqref{eq:pf-eq:lem:T_t^h-1} we also have $u\leq    \tau_{t_0,\rmL}^{h }$ and $|\cL_u-\cL_{t_0}|\leq h$. This verifies $\wh{\mr\cT_{t_0,\rmL}^{\rmD,d}} \subseteq \mr \cT_{t_0,\rmL}^{h,\infty}$. By a symmetric argument, $\wh{\mr\cT_{t_0,\rmR}^{\rmD,d}} \subseteq \mr \cT_{t_0,\rmR}^{h,\infty}$. Now if $s\in\mr\cT_t^{\rmD,d}$, then for some $t_0\in\cJ$, we have $s,t\in \wh{\mr\cT_{t_0,\rmL}^{\rmD,d}}\cup \wh{\mr\cT_{t_0,\rmR}^{\rmD,d}}\subseteq (\mr \cT_{t_0,\rmL}^{h,\infty}\cup\mr \cT_{t_0,\rmR}^{h,\infty})$.  The claim then follows from the definition of $\mr\cT_t^{h,\infty}$.
\end{proof}

We will further use the following comparison between $\mr\cT_t^{h,\infty}$ and $\cT_t^{h,\infty}$. For each $t\in\cJ^-$, define $\cO(t) = \{s\in\mr\cJ^-:\mr I_s\cap I_t\neq\emptyset\}$. 
\begin{lemma}\label{lem:329}
    Let $h>2$, $t_0\in\cJ$. Then for $q\in\{\rmL,\rmR\}$, we have
    \begin{equation}\label{eq:lem329}
        \mr\cT_{t_0,q}^{h,\infty} \subseteq \bigcup_{s\in \cT_{t_0,q}^{h,\infty} }\cO(s).
    \end{equation}
\end{lemma}

\begin{proof}
    Without loss of generality assume $q = \rmL$. Suppose $u\in  \mr\cT_{t_0,q}^{h,\infty}$. If $u = \mr \tau_{t_0,\rmL}^h$, then we must have $\tau_{t_0,\rmL}^h\in \mr I_u$ and $\tau_{t_0,\rmL}^h\in \cT_{t_0,\rmL}^{h,\infty}$. Then $u\in\cO(\tau_{t_0,\rmL}^h)$. If $u<\mr \tau_{t_0,\rmL}^h$, then $\cL_u\geq \cL_{t_0}-h$ and $u<\tau_{t_0,\rmL}^h$. Then there exists $r\in [t_0,u]\cap \mr I_u$, such that $\cL_r\in [\cL_{t_0}-h,\cL_{t_0}+h]$.  Let $u_1 = \inf\{u'\geq r:u'\in\cJ_\rmL^-\}$. Then $r\in I_{u_1}$ and $u_1\in  \cT_{t_0,\rmL}^{h,\infty}$. Therefore $u\in \cO(u_1)$. This verifies~\eqref{eq:lem329}.
\end{proof}

  For each $r\in\bbR$ and $q\in\{\rmL,\rmR\}$, let $u_q^-(r) = \inf\{s\geq r: s\in \cJ^-_q\}$. For $s\in \mr\cJ^-$, let $\cA(s) = \{ u_q^-(r):r\in \mr I_s\cap\mr\cJ, q\in\{\rmL,\rmR\}\}$.
\begin{lemma}\label{lem:330}
  Let $u\in\mr\cJ^-$, $h>2$. Then 
  \begin{equation}\label{eq:lem330}
      \mr\cT_u^{h,\infty}\subseteq \bigcup_{t\in \cA(u)}\bigcup_{s\in \cT_t^{h,\infty}}\cO(s).
  \end{equation}
\end{lemma}

\begin{proof}
We first prove that for $q\in\{\rmL,\rmR\}$ and $t_0\in\cJ$, if $u\in\mr\cT_{t_0,q}^{h,\infty}$ and $[t_0,\mr\tau_{t_0,q}^h)\cap\mr\cJ^-\neq\emptyset$, then $\cA(u)\cap \cT_{t_0,q}^{h,\infty}\neq\emptyset$. Without loss of generality assume $q = \rmL$. By Lemma~\ref{lem:lower-bd-tau-hat-AAAA}, $\mr \tau_{t_0,\rmL}^h =  \tau_{t_0,\rmL}^h\in \cT_{t_0,\rmL}^{h,\infty}$.  If  $u=\mr \tau_{t_0,\rmL}^h$, then the claim holds since $u\in \cT_{t_0,\rmL}^{h,\infty}\cap \cA(u)$. If $u< \mr \tau_{t_0,\rmL}^h$, then we have $\{\cL_r:r\in \mr I_u\cap [t_0,u]\}\cap \{[\cL_{t_0}-h,\cL_{t_0}+h]\}\neq\emptyset$. If $\cL_u \leq \cL_{t_0}+h$, then $u_\rmL^-(u)\in  \cT_{t_0,\rmL}^{h,\infty}\cap \cA(u)$ and the claim continues to hold. Let $\mr I_u = [u^-,u]$. 
Then Lemma~\ref{lem:hat-minus} implies that $\cL_{u^-}\leq \cL_{t_0}+h$, since otherwise $\cL_r> \cL_{t_0}+h$ on $\mr I_u$. 
Let $s$ be the last time in $\mr I_u\cap\mr J$ such that $\cL_s\leq\cL_{t_0}+h$. By Lemma~\ref{lem:hat-minus}, $\cL_r\geq \cL_s$ for $r\in [s,u)$ and we must have $s\geq t_0$, and furthermore $\cL_s\geq \cL_{t_0}+h-1$. Since $h>2$, the time $u_\rmL^-(s)\in  \cT_{t_0,\rmL}^{h,\infty}\cap \cA(u)$. This verifies our claim.

 Now suppose $s\in  \mr\cT_u^{h,\infty}$. Then there exists $t_0\in\cJ$ such that $u,s\in \mr\cT_{t_0}^{h,\infty,0}$. Without loss of generality assume $u\in \mr\cT_{t_0,\rmL}^{h,\infty}$. If $[t_0,\mr\tau_{t_0,\rmL}^h)\cap\mr\cJ^-\neq\emptyset$, then by our claim, there exists $t\in \cA(u)\cap \cT_{t_0}^{h,\infty,0}$, and Lemma~\ref{lem:329} implies that $s\in \cO(w)$ for some $w\in  \cT_{t_0}^{h,\infty,0}$.  The condition that $t,w\in \cT_{t_0}^{h,\infty,0}$ implies that $w\in \cT_t^{h,\infty}$ and therefore $s$ is in the set on the right hand side of~\eqref{eq:lem330}.  If $[t_0,\mr\tau_{t_0,\rmL}^h)\cap\mr\cJ^-=\emptyset$, then $\mr\cT_{t_0,\rmL}^{h,\infty} = \{u\}$ and $t_0\in \mr I_u$. In particular, $u\in \mr\cT_{t_0,\rmR}^{h,\infty}$. Then by an identical argument, if  $[t_0,\mr\tau_{t_0,\rmR}^h)\cap\mr\cJ^-\neq\emptyset$, then  $s$ is in the set on the right hand side of~\eqref{eq:lem330}, while if  $[t_0,\mr\tau_{t_0,\rmR}^h)\cap\mr\cJ^-=\emptyset$, then $\mr\cT_{t_0,\rmL}^{h,\infty} = \mr\cT_{t_0,\rmR}^{h,\infty} = \{u\}$ and thus we must have $s=u$, and~\eqref{eq:lem330} also holds in this case.
\end{proof}

\begin{lemma}\label{lem:331}
    Let $k\in\bbN$. Then there exists some constant $A_1,c>0$ such that 
    \begin{equation}
        \bbP_\rmE\big[\exists\,s\in\cJ^-:I_s\subseteq[-2^k,2^k], \  \#\cO(s)\geq A_1k \big]\lesssim 2^{-ck}.
    \end{equation}
\end{lemma}

\begin{proof}
    Indeed, we have $\#\cO(t)\leq \#(I_t\cap\mr\cJ^-)+2\leq \#(I_t\cap\cJ)+2.$ In the setting of $\bbP_\rmQ$, using the Poissonian description of $\cZ$, by a union bound, for $A>0$, we have
    \begin{equation}
        \bbP_\rmQ\big[\exists\,s\in\cJ^-:I_s\subseteq I(k), \  \#(I_s\cap\cJ)\geq Ak \big] \leq \bbP_\rmQ[\cJ\cap I(k)\geq 2^{1.1k}] + 2^{1.1k}\big(\frac34\big)^{Ak}\lesssim 2^{-ck^2}+2^{(1.1-A/3)k}.
    \end{equation}
    Then by part~(\ref{cond:thm:KMT-0-cone-(iv)}) of Theorem~\ref{thm:KMT-0-cone}, if we pick the constant $A$ sufficiently large,
    \begin{equation}
        \bbP_\rmE\big[\exists\,s\in\cJ^-:I_s\subseteq I(k), \  \#\cO(s)\geq Ak \big] \leq  \bbP_\rmE\big[\exists\,s\in\cJ^-:I_s\subseteq I(k), \  \#(I_s\cap\cJ)\geq (A-1)k \big]\lesssim 2^{-ck}.
    \end{equation}
    Therefore if we set $A_1=2A$, then
    \begin{equation*}
        \begin{split}
           & \bbP_\rmE\big[\exists\,s\in\cJ^-:I_s\subseteq[-2^k,2^k], \   \#\cO(s)\geq A_1k \big]\leq \\& \bbP_\rmE\big[\exists\,s\in\cJ^-:I_s\subseteq I(2k), \  \#\cO(s)\geq 2Ak \big]+ \bbP_\rmE[[-2^k,2^k]\nsubseteq I(2k)] \lesssim 2^{-ck}.
        \end{split}
    \end{equation*}
\end{proof}

Next we bound $\# \mr\cT_t^{\rmC,d}$ via Lemma~\ref{lem:330} and Lemma~\ref{lem:331}.
\begin{lemma}\label{lem:P_QT_0}
   Let $\mathrm{up}(0,\cJ^-) = \inf(\cJ^-\cap [0,\infty))$.  There exists some constant $\wt c>0$ such that $\bbP_\rmQ[\#\cT_{\mathrm{up}(0,\cJ^-)}^{h,\infty}>h^{9}]\lesssim 2^{-\wt ch^{9/8}}$.
\end{lemma}
\begin{proof}
     {This is ~\cite[Lemma 5.32]{Tutte}.}
\end{proof}

In the content below we will set $h=k^2$.
\begin{lemma}\label{lem:P_QT_kk2}
    There exists some constant $c>0$ such that 
    $$\bbP_\rmQ\big[\exists t\in I(k)\cap \cJ^-: \#\cT_t^{k^2,\infty}>k^{18} \big]\lesssim 2^{-ck^2}.$$
\end{lemma}

\begin{proof}
    Let $a>0$ and let $ E_a^{\rm dense}$ be the event where any of the intervals $\{[-2^k+(m-1)a,-2^k+ma]:1\leq m\leq 1+2^{k+1}/a\}$ has at least two points in $\cJ$. Since under $\bbP_\rmQ$, $\cJ$ has the law as a unit rate Poisson point process on $\bbR$, by a union bound we conclude that $\bbP_\rmQ[  E_a^{\rm dense}]\leq 4a^{2}2^{k+2}/a = 4a2^{k+2}$. Now by a union bound and translation invariance in law of $\cZ$ under $\bbP_\rmQ$, from Lemma~\ref{lem:P_QT_0}, 
    $$\bbP_\rmQ\big[\exists t\in I(k)\cap \cJ^-: \cT_t^{k^2,\infty}>k^{18} \big]\leq \bbP_\rmQ[  E_a^{\rm dense}]+\frac{2^{k+2}}{a}\bbP_\rmQ[\#\cT_{\mathrm{up}(0,\cJ^-)}^{k^2,\infty}>k^{18} ]\lesssim 4a2^{k+2}+ \frac{2^{k+2}}{a}\cdot 2^{-\wt c k^{9/4}}$$
    where the implicit constant does not depend on $k$ and $a$. The conclusion then follows by setting $a=2^{-k^2}$.
\end{proof}

Now we transfer to the setting of $\bbP_\rmE$.
\begin{lemma}\label{lem:P_QT_kk2-1}
    There exists some constant $c>0$ such that for large enough $k$,
    \eqb\label{eq:lem:P_QT_kk2-1}\bbP_\rmE\big[\exists t\in [-2^k,2^k]\cap \cJ^-: \#\cT_t^{k^2,\infty}>k^{20} \big]\lesssim 2^{-ck}.\eqe
    %In particular, Assertion~\ref{cond:prop:time-set-A-0-6} in Proposition~\ref{prop:time-set-A-0} holds.
\end{lemma}
\begin{proof}
    By Lemma~\ref{lem:T_t^h,m} and Lemma~\ref{lem:P_QT_kk2}, \eqb\label{eq:lem:P_QT_kk2} \bbP_\rmQ\big[\exists t\in I(k)\cap \cJ^-: \#\cT_t^{k^2,2k}>k^{18}, \cT^{k^2,2k}_t\subseteq I(2k) \big]\lesssim 2^{-ck^2},\eqe
   and the event listed above is measurable with respect to $\cZ|_{I(2k)}$. If we set  $G(\cZ(\cdot),\fD) = 1$ if the event in~\eqref{eq:lem:P_QT_kk2} happens and $G(\cZ(\cdot),\fD) = 0$ otherwise, then since the time sets above does not change if we add a global additive constant to $\cZ$,  
    $$\{\exists s\in I(0):G(\cZ(\cdot+s)-\cZ(s),\fD-s)=1\} = \{G(\cZ(\cdot),\fD)=1\},$$
    and therefore by Theorem~\ref{thm:KMT-0-cone}, 
    $$ {\bbP_\rmE\big[\exists t\in I(k)\cap \cJ^-: \#\cT_t^{k^2,2k}>k^{18}, \cT^{k^2,2k}_t\subseteq I(2k) \big]\lesssim 2^{-ck}.}$$
    By  {following the same proof of ~\cite[Proposition 5.6(iii)]{Tutte}}, 
   \eqb\label{eq:336} \bbP_\rmE\big[\exists t\in I(k)\cap \cJ^-:  \cT^{k^2,\infty}_t\nsubseteq I(\lceil1.5k\rceil) \big]\lesssim 2^{-ck} \eqe
    Furthermore, using Proposition~\ref{prop:Tutte5.8},  
     \begin{equation*}
        \begin{split}
            \bbP_\rmE&[\bigcup_{u\in\cJ^-\cap I(\lceil1.5k\rceil)}I_u\nsubseteq I(2k)]\leq \bbP_\rmE[[-2^{1.7k},2^{1.7k}]\nsubseteq I(2k)] + \bbP_\rmE\big[[-2^{1,7k},-2^{1.6k}]\cap \cJ_\rmL^-]=\emptyset\big] \\&+ \bbP_\rmE\big[[-2^{1.7k},-2^{1.6k}]\cap \cJ_\rmR^-]=\emptyset\big] +   \bbP_\rmE\big[[2^{1.6k},2^{1.7k}]\cap \cJ_\rmL^-]=\emptyset\big] + \bbP_\rmE\big[[2^{1.6k},2^{1.7k}]\cap \cJ_\rmR^-]=\emptyset\big] \lesssim 2^{-ck}.
        \end{split}
    \end{equation*}
    Therefore  by Lemma~\ref{lem:T_t^h,m} and a union bound,  $$\bbP_\rmE\big[\exists t\in I(k)\cap \cJ^-:  \#\cT^{k^2,\infty}_t>k^{18} \big]\lesssim 2^{-ck},$$
    and finally by another union bound,
    $$\bbP_\rmE\big[\exists t\in [-2^k,2^k]\cap \cJ^-: \#\cT_t^{k^2,\infty}>(2k)^{18} \big]\leq \bbP_\rmE\big[\exists t\in I(2k)\cap \cJ^-:  \#\cT^{k^2,\infty}_t>(2k)^{18} \big]+\bbP_\rmE[[-2^k,2^k]\nsubseteq I(2k)] \lesssim 2^{-ck}.$$
    %We then conclude the proof of ~\eqref{eq:lem:P_QT_kk2-1} by Lemma~\ref{lem:T_t^h,m} and a union bound. 
\end{proof}

\begin{lemma}\label{lem:P_QT_kk2-10}
    There exists some constant $c>0$ such that for large enough $k$,
    $$\bbP_\rmE\big[\exists t\in [-2^k,2^k]\cap \mr\cJ^-: \#\mr\cT_t^{k^2,\infty}>k^{26} \big]\lesssim 2^{-ck}.$$
\end{lemma}
\begin{proof}
    We first prove that for  every $k>100$,%
    \eqb\label{pf:lem:P_QT_kk^2-10}
    \begin{split}
    \{\nexists &t\in [-2^k,2^k]\cap \mr\cJ^-: \#\mr\cT_t^{k^2,\infty}>k^{26}\}\supseteq \bigg\{\bigcup_{t\in [-2^k,2^k]\cap \mr\cJ^-} \mr I_t\subseteq [-2^{k+1},2^{k+1}]\bigg\}\cap\{\nexists t\in [-2^{k+4},2^{k+4}]\cap\cJ^-:\#\cT_t^{k^2,\infty}>k^{21} \}\\&\cap\{\nexists t\in [-2^{k+3},2^{k+3}]\cap \mr\cJ^-:\#(\mr I_t\cap \mr\cJ)\geq k^2\}\cap  \{\cJ^-_\rmL\cap [2^{k+2},2^{k+3}]\neq\emptyset,\cJ^-_\rmR \cap[2^{k+2},2^{k+3}]\neq\emptyset  \} \\&\cap \{\nexists t\in \cJ^-\cap [-2^{k+3},2^{k+3}]: \cT_t^{k^2,\infty}\nsubseteq [-2^{2k},2^{2k}]\} \cap \{\nexists s\in\cJ^-\cap [-2^{2k},2^{2k}],\  \#\cO(s)\geq k^2\}.
    \end{split}
    \eqe
Denote the six events on the right  hand side of~\eqref{pf:lem:P_QT_kk^2-10} by $E_1,E_2,E_3,E_4,E_5,E_6$, respectively. Let $u\in \mr\cJ^-\cap [-2^k,2^k]$. Then under $E_1\cap E_4$, we have $\cA(u)\subseteq [-2^{k+3},2^{k+3}]$, and further under $E_5$, for any $t\in\cA(u)$ and $s\in \cT_t^{k^2,\infty}$, $s\in [-2^{2k},2^{2k}]$. Further under $E_3$, $\#\cA(u)\leq 2k^2$, and therefore $\#\mr\cT_u^{k^2,\infty}<k^{26}$ under $\cap_{i=1}^6 E_i$. This completes the proof.

 Below $c>0$ is some constant that may vary line by line.  By Proposition~\ref{lem:gap-hat-J^-}, $\bbP_\rmE[E_1^c]\lesssim 2^{-ck}$. By Lemma~\ref{lem:P_QT_kk2-1}, $\bbP_\rmE[E_2^c]\lesssim 2^{-ck}$. By~\eqref{eq:jump-hat-J^-} and Proposition~\ref{lem:gap-hat-J^-}, $\bbP_\rmE[E_3^c] \lesssim \bbP_\rmE[\mr N_\rmL[-2^{k+4},2^{k+4}]>k^{1.5}]+ \bbP_\rmE[\mr N_\rmR[-2^{k+4},2^{k+4}]>k^{1.5}]+\bbP_\rmE[[-2^{k+4},-2^{k+3}]\cap\mr\cJ^-=\emptyset ] \lesssim 2^{-ck}$. By Proposition~\ref{prop:Tutte5.8} and symmetry, $\bbP_\rmE[E_4^c]\lesssim 2^{-ck}$. By~\eqref{eq:336} applied with $1.1k$ in place of $k$ and $\bbP_\rmE[[-2^{k+4},2^{k+4}]\nsubseteq I(\lceil 1.1k\rceil)]\lesssim 2^{-c(\e_0)k}$, we have $\bbP_\rmE[E_5^c]\lesssim 2^{-ck}$. Finally by Lemma~\ref{lem:331},  $ \bbP_\rmE[\exists s\in\cJ^-\cap [-2^{3k},2^{3k}]: I_s\subseteq[-2^{3k},2^{3k}],\  \#\cO(s)\geq k^2] \lesssim 2^{-ck}$, and  by Proposition~\ref{prop:Tutte5.8}, with $\bbP_\rmE$ probability at least $1-c'2^{-ck}$, $[-2^{3k},-2^{2k}]$ and $[2^{2k},2^{3k}]$ have nonempty intersection with both $\cJ_\rmL^-$ and $\cJ^-_\rmR$, which implies $\bbP_\rmE[E_6^c]\lesssim 2^{-ck}$ by a union bound.

% Suppose the four events on the right hand side of~\eqref{pf:lem:P_QT_kk^2-10} occur, and let $u\in I(k)\cap \mr\cJ^-$. Then the first event and the third event implies that $\cA(u)\subseteq $
% By the first event, $s_j\in I(2k)$, and by the third event, $n\leq k^2$. By the fourth event, $t_j^q\in I(3k)$. By Lemma~\ref{lem:compare-number-hat-T_t^h} and the second event, we have 
% $$\# \mr\cT_{s_0}^{k^2,\infty}\leq \sum_{q\in\{\rmL,\rmR\}}\sum_{j=1}^n\# \cT_{t_j^q}^{k^2,\infty}\leq 2nk^{21}\leq k^{24}.$$
% This verifies~\eqref{pf:lem:P_QT_kk^2-10}.
    
%     By Lemma~\ref{lem:gap-hat-J^-}, 
%     $$\bbP_\rmE[\bigcup_{t\in I(k)\cap\mr\cJ^-} \mr I_t\nsubseteq I(2k)]\leq\bbP_\rmE\big[\mr\cJ^-\cap [-2^{k+1},-2^k]=\emptyset \ \text{or } \mr\cJ^-\cap [2^k,2^{k+1}]=\emptyset\big]\lesssim 2^{-ck}.$$
%     For large enough $k$, the second and third event on the right hand side of~\eqref{pf:lem:P_QT_kk^2-10} each has probability at least $1-c'2^{-ck}$ by Lemma~\ref{lem:P_QT_kk2-1} and~\eqref{eq:jump-hat-J^-}, respectively. {The last event has probability at least $1-c'2^{-ck}$ by Lemma 5.8 of the Tutte draft.} This completes the proof.
\end{proof}

\begin{lemma}\label{lem:prop:time-set-A-6}
       Assertion~\ref{cond:prop:time-set-A-6} in Proposition~\ref{prop:time-set-A} holds.
\end{lemma}

\begin{proof}
    By Assertion~\ref{cond:lem:def-wh-T-Dd-4} of Lemma~\ref{lem:def-wh-T-Dd}, it suffices to prove the claim~\eqref{eq:cond:prop:time-set-A-6} for $\mr\cT_t^{\rmD,d}$ in place of  $\mr\cT_t^{\rmC,d}$. By Lemma~\ref{lem:prop:time-set-A-5} and  Assertion~\ref{cond:lem:def-wh-T-Dd-4} of Lemma~\ref{lem:def-wh-T-Dd} and the fact that $\bbP_\rmE[[-2^{k+1},2^{k+1}]\nsubseteq I(1.5k)]\lesssim 2^{-ck}$, we have  $$\bbP_\rmE[\bigcup_{t\in [-2^{k+1},2^{k+1}]} \mr\cT_t^{\mathrm D,d}\nsubseteq I(2k)]\lesssim 2^{-ck}.$$
    By part~(\ref{cond:thm:KMT-0-cone-(ii)}) of Theorem~\ref{thm:KMT-0-cone},~\eqref{eq:lem:gap-hat-J^-} and ~\eqref{eq:jump-hat-J^-}, with probability at least $1-c'2^{-ck}$ under $\bbP_\rmE$, the conditions in Lemma~\ref{lem:T_t^h} hold for $2k$ in place of $k$ and $h=k^2$. Therefore by Lemma~\ref{lem:T_t^h},   with probability at least $1-c'2^{-ck}$ under $\bbP_\rmE$,  $ \mr\cT_t^{\mathrm D,d}\subseteq  \mr\cT_t^{k^2,\infty}$ for every $t\in [-2^{k+1},2^{k+1}]$. By Lemma~\ref{lem:P_QT_kk2-10},  $$\bbP_\rmE\big[\exists t\in [-2^{k+1},2^{k+1}]\cap \mr\cJ^-: \#\mr\cT_t^{k^2,\infty}>k^{27} \big]\lesssim 2^{-ck}.$$
    This completes the proof.
\end{proof}

\begin{lemma}\label{lem:prop:time-set-A-0-6}
       Assertion~\ref{cond:prop:time-set-A-0-6} in Proposition~\ref{prop:time-set-A-0} holds.
\end{lemma}
\begin{proof}
    The proof is similar to that of Lemma~\ref{lem:prop:time-set-A-6} where one uses Lemma~\ref{lem:P_QT_kk2-1} and  {~\cite[Lemma 5.5]{Tutte}.}
\end{proof}

\begin{proof}[Proof of Proposition~\ref{prop:time-set-A-0}]
    Assertions~\ref{cond:prop:time-set-A-0-1}-\ref{cond:prop:time-set-A-0-4} as well as Assertion \ref{cond:prop:time-set-A-0-7}  {follow from \cite[Proposition 5.1]{Tutte}.} Assertion~\ref{cond:prop:time-set-A-0-5} is  {a consequence of \cite[Proposition 5.6]{Tutte}}, and Assertion~\ref{cond:prop:time-set-A-0-6} follows from Lemma~\ref{lem:prop:time-set-A-0-6}.
\end{proof}

\begin{proof}[Proof of Proposition~\ref{prop:time-set-A}]
    Assertions~\ref{cond:prop:time-set-A-1}-\ref{cond:prop:time-set-A-4} {follow from Lemma~\ref{lem:def-wh-T-Cd} and Lemma~\ref{lem:def-wh-T-Dd}.} Assertion~\ref{cond:prop:time-set-A-5} is proven in Lemma~\ref{lem:prop:time-set-A-5}, while  Assertion~\ref{cond:prop:time-set-A-6} follows from Lemma~\ref{lem:prop:time-set-A-6}. For Assertion \ref{cond:prop:time-set-A-0-7}, if $\{s,s'\}\in\mr\cE_\rmG$ or $s'\in\mr\cT_s^{\rmC,1}$, then by part~\ref{cond:lem:def-wh-T-Dd-3} and part~\ref{cond:lem:def-wh-T-Dd-4} of Lemma~\ref{lem:def-wh-T-Dd}, $s,s'\in \mr\cT_s^{\rmD,d-1}$, and the claim follows from part~\ref{cond:lem:def-wh-T-Dd-5} and part~\ref{cond:lem:def-wh-T-Dd-4} of Lemma~\ref{lem:def-wh-T-Dd}.
\end{proof}

\subsection{The cell configuration}\label{subsec:cellsystem}

Let $\cZ$ be the walk as in Theorem~\ref{thm:KMT-0-cone}, $(T,\frtr) = \Phi_\rmT(\cZ)$ be the triangulation as well as the curve on faces of $T$ encoded by $\cZ$ via the Mullin bijection, and $M = \Phi_{\rmM}(\cZ)$. Let $\mr T = \Phi_\simple(T)$ be the simple triangulation obtained by collapsing the double edges of $T$ and $\mr\frtr$ be the corresponding curve on faces of $\mr T$. 
In this subsection, we {introduce the  cell configuration $\cH^\cell$ in~\cite[Section 5.2]{Tutte} whose associated map is $M$, and}  use the time set $\mr \cT_t^{\rmC,d}$ introduced as in Section~\ref{subsec:time-sets} to construct a cell configuration $\mr\cH^\cell$ whose associated map is $\mr T$. We  further verify that $\mr\cH^\cell$ and $\cH^\cell$ satisfy the conditions in Theorem~\ref{thm:CellSystemPack} and Theorem~\ref{thm:invariance-uniformization}, respectively.  %We start with introducing the  cell configuration $\cH^{\cell}$ as the {the Tutte paper} for the Riemann uniformization setting, and later construct a cell configuration $\mr\cH^\cell$ using the time sets $\mr\cT_t^{\rmC,d}$ as in Proposition~\ref{prop:time-set-A}. 

\subsubsection{The definition of the cell configurations}

Let $\{\cT_t^{\rmC,d}:t\in\cJ^-\}$ be the time sets as constructed in Proposition~\ref{prop:time-set-A-0}. For $s\in\cJ_\rmL^-$ we define
\begin{equation}\label{eq:def-cT-cell}
    \cT_s^\cell:= \bigcup_{u\in\cJ^-:\{u,s\}\in\cE_\rmT \ \text{or} \ u=s  }\cT_u^{\rmC,1}, \ \ \ \cT_s^\lge:= \bigcup_{u\in\cJ^-:\{u,s\}\in\cE_\rmT \ \text{or} \ u=s  }\cT_u^{\rmC,3}.
\end{equation}
For  {$u\in\cJ^-$, we set
\begin{equation}
    \cT_u^{\cell,\inv}:=\{s\in\cJ^-_\rmL: u\in \cT_s^\cell\}, \ \ \ \ \cT_u^{\lge,\inv}:=\{s\in\cJ^-_\rmL: u\in \cT_s^\lge\}.
\end{equation}
}

 {The following follows from \cite[Proposition 5.13]{Tutte}.
\begin{proposition}\label{prop:def-cell-inv}
 In the setting of Theorem~\ref{thm:KMT-0-cone},  there is a cell configuration $\cH^\cell$ whose associated map is $\Phi_{\mathrm{M}}(\cZ)$ with $\cH^\cell = (H_t^\cell)_{t\in\cJ^-_\rmL}$ such that for each $t\in\cJ^-_\rmL$, 
 \begin{equation}\label{eq:cell-area-0}
   \bigcup_{s\in\cT_t^\cell}\partial\eta(I_s)\subseteq H_t^\cell\subseteq \bigcup_{s\in\cT_t^\cell}\eta(I_s); \ \ \ \area(H_t^\cell) = \sum_{s\in \cT_t^\cell}\frac{\area(\eta(I_s))}{2\cdot\# \cT_s^{\cell,\inv}}.
 \end{equation}
There is also a cell configuration $\cH^\lge$ satisfying the exact same conditions with ``\ $\cell$" replaced by ``\ $\lge$". Finally, the cell configurations $\cH^\cell$ and $\cH^\lge$  satisfies all the conditions in Definition~\ref{def:CellConfigConstraint} and conditions~$($\ref{cond:def:CellConfigConstraint-large-(i)}$)$ and~$($\ref{cond:def:CellConfigConstraint-large-(ii)}$)$ in Definition~\ref{def:CellConfigConstraint-large}.
\end{proposition}
}

Next we construct the cell configuration $\mr\cH^\cell$ in the circle packing setting. We set
$$\mr\cT_t^\cell = \mr\cT_t^{\rmC,1}, \ \ \text{and } \ \ \mr\cT_t^\lge = \mr\cT_t^{\rmC,3}. $$
 Let $s\in\mr \cJ^-$ and   $\mr I_s = [s^-,s]$. Let $a_s=\eta(b_s)$, where $b_s$ is the smallest time $u$ in $[s^-,s]$ such that the Euclidean distance $|\eta(u)-\eta(s^-)|$ reaches its maximum over $[s^-,s]$.  We define sets $\{\mr A_{s,t}:s\in \mr\cJ^-,t\in \mr\cT_s^{\cell}\}$ as follows. Consider the ray $L$ starting from $\eta(s^-)$ and passing through $a_s$ of angle $\alpha_0$. Let $n = \# \mr\cT_s^{\cell}$ and  $0=\beta_0<\beta_1<...<\beta_n=2\pi$. For $j=1,...,n$, define the wedge $W_j = \eta(s^-)+\{re^{i\theta}:\alpha_0+\beta_{j-1}\leq\theta\leq \alpha_0+\beta_j, r\geq 0  \}$, and we inductively choose  $\beta_1,...,\beta_{n-1}$ to be the smallest numbers  such that $\area(\eta(\mr I_s)\cap W_j) = \area(\eta(\mr I_s))/n$ for every $j$. For $\mr\cT_s^{\cell} = \{t_1,...,t_n\}$ where $t_1<...<t_n$, let $\mr A_{s,t_j} = (W_j\cap \eta(\mr I_s))\cup \partial\eta(\mr I_s)$. Note that $\mr A_{s,t}$ is connected and compact  since $\eta(\mr I_s)$ is simply connected, and $\cup_{t\in \mr\cT_s^{\cell}}\mr A_{s,t} = \eta(\mr I_s)$.   Define the sets  $\{\mr A_{s,t}^\lge:s\in \mr\cJ^-,t\in \mr\cT_s^{\lge}\}$ in the identical way using $\mr\cT_s^\lge$.  For $t\in \mr\cJ^-$, we define
  \begin{equation}\label{eq:def-H}
     \mr H_t^{\cell} = \bigcup_{s\in \mr\cT_t^{\cell}}\mr A_{s,t}; \ \ \ \mr H_t^{\lge} = \bigcup_{s\in \mr\cT_t^{\lge}} \mr A_{s,t}^\lge.
  \end{equation}
  See Figure~\ref{fig:single-cell} for an illustration. We define the cell configurations $\mr\cH^\cell = \{\mr H_t^\cell:t\in \mr \cJ^-\}$ and $\mr\cH^\lge = \{ \mr H_t^\lge:t\in \mr\cJ^- \}$ where the associated maps are $\mr T$. Note that $\mr\cH^\cell$ and $\mr\cH^\lge$ are measurable with repsect to $(h,\eta,\fU)$.

  \begin{figure}
      \centering
      \includegraphics[scale=0.75]{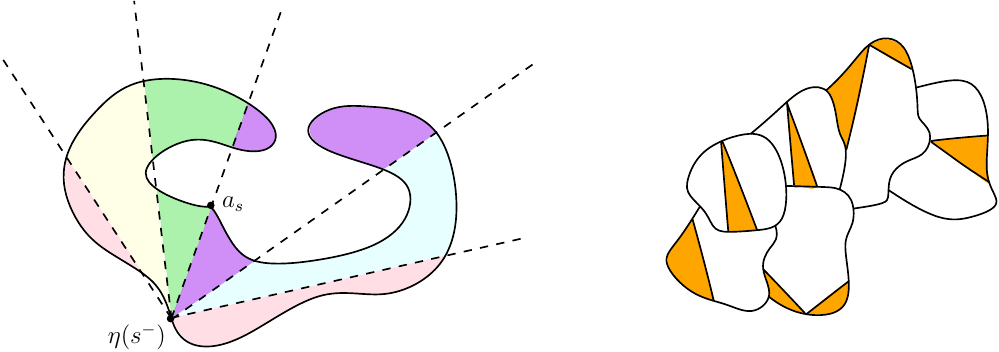}
      \caption{\textbf{Left:} For each $t\in \mr\cT_s^\cell$, the set $\mr A_{s,t}$ is the union of one of the colored regions and $\partial \eta(\mr I_s)$, such that each of $\mr A_{s,t}$ is a connected compact set and each $\mr A_{s,t}$ and $\mr A_{s,t'}$ has equal area. \textbf{Right:} An example of a single cell $\mr H_s^\cell$, which is   the union of the orange regions together with the black curves.}
      \label{fig:single-cell}
  \end{figure}

\begin{proposition}\label{prop:def-cell-pack}
Almost surely, under $\bbP_\rmE$, we have the following. 
\begin{enumerate}[(i)]
    \item The cell configurations $\mr\cH^\cell$ and $\mr\cH^\lge$ are well-defined,  and the map $\mr H_t^\cell\mapsto \mr H_t^\lge$ is an isomorphism between   $\mr \cH^\cell$ and $\mr\cH^\lge$;  
    \item For each $t\in\mr\cJ^-$, the convex hull of   $\mr H_t$ is contained in the convex hull of $\mr H_t^\lge$;
    \item For any $t\in\mr\cJ^-$, we have    
         \eqb\label{eq:cell-area}\bigcup_{s\in \mr\cT_t^\cell} \partial \eta(\mr I_s)\subseteq \mr H_t^\cell\subseteq \bigcup_{s\in \mr\cT_t^\cell} \eta(\mr I_s); \ \ \ \ \area(\mr H_t^\cell) = \sum_{s\in\mr \cT_t^\cell} \frac{\area(\eta(\mr I_s))}{\#\mr\cT_s^\cell}, \eqe
         and the same holds if we replace ``\ $\cell$" by ``\ $\lge$".
\end{enumerate}
\end{proposition}

\begin{proof}
 The claim (iii)  follows from the definition of $\mr H_t^\cell$ and $\mr A_{s,t}$. By the Borel-Contelli lemma applied to Assertion~\ref{cond:prop:time-set-A-6} of Proposition~\ref{prop:time-set-A}, almost surely under $\bbP_\rmE$, all the time sets are finite, so each cell is a compact set. Furthermore, for each $s$,  $\mr\cH^\cell$ and $\mr\cH^\lge$ are locally finite.
    By Assertion~\ref{cond:prop:time-set-A-3} of Proposition~\ref{prop:time-set-A} and the definition~\eqref{eq:def-H}, we have $\mr H_s\cap \mr H_t\neq\emptyset$ if $\{s,t\}\in\mr\cE_\rmT$. Furthermore by~\eqref{eq:cell-area}, if $\mr H_t^\cell$ intersects $\eta(\mr I_s)$, then for some $s'\in \mr\cJ^-$ such that $\eta(\mr I_s)\cap \eta(\mr I_{s'})\neq\emptyset$, $s'\in \mr\cT_t^\cell$ and hence $t\in \mr\cT_{s'}^\cell$. This proves the local finiteness for $\mr\cH^\cell$, and the local finiteness for    $\mr\cH^\lge$ follows similarly. Assertion~\ref{cond:prop:time-set-A-4} of Proposition~\ref{prop:time-set-A} as well as~\eqref{eq:def-H}, for each $t\in\mr\cJ^-$, $\mr H_t^\cell$ is connected. The same properties holds if we replace ``\ $\cell$" by ``\ $\lge$". This verifies the claim (i). The claim (ii) is immediate since $\mr\cT_t^\cell\subseteq\mr\cT_t^\lge$ and the the convex hull of $\mr A_{s,t}$ and $\mr A_{s,t}^\lge$ agrees with the convex hull of $\eta(\mr I_s)$.
\end{proof}

Next we check the conditions in Definition~\ref{def:CellConfigConstraint} for the cell configuration $\mr\cH^\cell$ and $\mr\cH^\lge$. We will use the following  definition of dyadic systems and uniform dyadic systems as in~\cite{GMSinvariance}.  For a square $S$, we write $|S|$ for the side length of $S$. A \emph{dyadic child} of $S$ is one of the 4 squares with half side length of $S$ whose corners include one corner of $S$ and the center of $S$. A \emph{dyadic parent} of $S$ is one of the four squares with $S$ as a dyadic
child. Then each square has 4 dyadic children and 4 dyadic parents. A dyadic descendant (resp.\ dyadic ancestor) of $S$ is a square which can be obtained from $S$ by iteratively choosing dyadic children (resp.\ parents) finitely many times. A \emph{dyadic system} is a collection $\fD^\bbC$ of closed squares with their sides parallel to the $x$-axis or $y$-axis with the following properties:
\begin{enumerate} 
    \item If $S\in\fD^\bbC$, then each of the four dyadic children of $S$ is in $\fD^\bbC$.
    \item If $S\in\fD^\bbC$, then exactly one of the dyadic parents of $S$ is in $\fD^\bbC$.
    \item Any two squares in $\fD^\bbC$ have a common dyadic ancestor.
\end{enumerate}
Let $s$ be uniformly sampled from $[0,1]$, and conditioned on $s$, we uniformly randomly sample $w$ in $[0,2^s]\times[0,2^s]$. Let $S_0 = [0,2^s]\times[0,2^s]-w$. For $k\in\bbN$, inductively let $S_k$ be sampled uniformly from the four dyadic parents of $S_{k-1}$.  The \emph{uniform dyadic system} $\fD^\bbC$ on $\bbC$ is defined by the set of all dyadic descendants of $S_k$ for each $k\in\bbN$.

\begin{lemma}
    There exists  a measurable function $\mr\phi$ such that $\mr\cH^\cell = \mr\phi(h,h^\IG,\fU)$ and for every $z\in\bbC$ and $C>0$, a.s.\  
    \begin{equation}\label{eq:lem:ergodicity-200}
        C^{-1}(\mr\cH^\cell-z) = \mr\phi(h(C\cdot+z)+Q\log C, h^\IG(C\cdot+z) ,\fU-\tau_z),
    \end{equation}
 {where $\tau_z$ is the time when the space-filling curve $\eta$ generated by $h^\IG$ and parameterized by the LQG measure with respect to $h$ hits $z$.} Let $\fD^\bbC$ be an independent uniform dyadic system on $\bbC$, and let $\wh S_m$ be the largest square $S$ in $\fD^\bbC$ containing 0 such that \eqb\label{eq:lem:ergodicity-200202}\sum_{H\in\mr\cH^\cell} \frac{\area(H\cap S)}{\area(H)}\leq m.\eqe Then for every fixed $m>0$, there exists a measurable function $\psi_{m,n}$ taking values in $\fD^\bbC$, such that 
    \eqb\label{eq:lem:ergodicity-200201}\lim_{n\to\infty}\bbP_\rmE\big[\wh S_m = \psi_{m,n}(h|_{B(0;n)},h^{\IG}|_{B(0;n)},\fU,\fD^\bbC) \big]=1.\eqe 
    % There are measurable functions $\mr\phi_n$ taking their value in the set of cell configurations such that for any fixed $a>1$,
    % \begin{equation}\label{eq:lem:ergodicity-2}
    %     \lim_{n\to\infty} \bbP_\rmE\Big[\mr\cH^\cell|_{B(0;a)} = \mr \phi_n\big(h|_{B(0;n)},h^\IG|_{B(0;n)}, \fU\big)|_{B(0;a)} \Big]=1, 
    % \end{equation}
    % where the restriction of a cell configuration to $B(0;a)$ means the collection of cells intersecting $B(0;a)$, along with information about the associated map restricted to these cells. 
    The same property holds for $\mr\cH^\lge$ instead of $\mr\cH^\cell$ as we define $\wh S_m$ using $\mr\cH^\lge$ in place of $\mr\cH^\cell$.
\end{lemma}

\begin{proof}
We only deal with the $\mr\cH^\cell$ case as the proof for $\mr\cH^\lge$ is identical.  The boundary length process for $(h(C\cdot+z)+Q\log C, h^\IG(C\cdot+z))$ is $(Z_{t+\tau_z}-Z_{\tau_z})_{t\in\bbR}$ and by relation~\eqref{eq:shift-inv-Z} as in part~(\ref{cond:thm:KMT-0-cone-(i)}) of Theorem~\ref{thm:KMT-0-cone}, the corresponding discrete-time walk for $(h(C\cdot+z)+Q\log C, h^\IG(C\cdot+z) ,\fU-\tau_z)$ is $(\cZ_{t+\tau_z}-\cZ_{\tau_z})_{t\in\bbR}$.  The relation~\eqref{eq:lem:ergodicity-200} is a consequence of the fact that the time sets are invariant if we add a constant to the walk $\cZ$ and our construction of the cells commutes with translation and scaling. 

Now let $\cF_n = \sigma(h|_{B(0;n)}, h^\IG|_{B(0;n)}, \fU,\fD^\bbC)$. Let $S_j\in\fD^\bbC$ be the a.s.\ unique square containing 0 of side length within $[2^j,2^{j+1})$. Then there exists a finite integer $K_m$ which is measurable with respect to $\sigma(h,h^\IG,\fU,\fD^\bbC)$ such that $\wh S_m = S_{K_m}$. For each integer $j$, let $p_{n,j} = \bbE[\mathds{1}_{K_m=j}|\cF_n]$. Then by the martingale convergence theorem, almost surely, for each $j$, $p_{n,j}\to \mathds{1}_{K_m=j}$. Let $K_{m,n} = j$ if $p_{n,j}>2/3$ for some $j$ and $K_{m,n}=0$ otherwise. Since $\sum_j p_{n,j} = 1$, there is at most one $j$ such that $p_{n,j}>2/3$ so $K_{m,n}$ is well-defined and measurable with respect to $\cF_n$. Therefore combined with $p_{n,j}\to \mathds{1}_{K_m=j}$, almost surely, for all sufficiently large $n$, $K_{m,n}=K_m$. Therefore the claim~\eqref{eq:lem:ergodicity-200201} holds if we set $\psi_{m,n}(h|_{B(0;n)},h^{\IG}|_{B(0;n)},\fU,\fD^\bbC) = S_{K_{m,n}}$. The proof for $\mr\cH^\lge$ is identical.

\end{proof}

\begin{lemma}\label{lem:def-cell-pack-AA}
    The cell configurations $\mr\cH^\cell$ and $\mr\cH^\lge$ satisfy all the conditions in Definition~\ref{def:CellConfigConstraint}.
\end{lemma}
\begin{proof}
 {We only prove the claim for $\mr\cH^\cell$ as the proof for $\mr\cH^\lge$  is identical. Let $\psi_m$ be the measurable function such that $\wh S_m = \psi_m(h,h^\IG,\fU,\fD^\bbC)$. Then by~\eqref{eq:lem:ergodicity-200} and~\eqref{eq:lem:ergodicity-200202},  $C^{-1}(\wh S_m-z) = \psi_m(h(C\cdot+z)+Q\log C, h^\IG(C\cdot+z) ,\fU-\tau_z,C^{-1}(\fD^\bbC-z))$ for $z\in \wh S_m$. Then the same argument as in the proof of \cite[Lemma 5.21]{Tutte} gives the translation invariant modulo scaling property as in Condition~\ref{cond:def:CellConfigConstraint-1} of Definition~\ref{def:CellConfigConstraint}.  The ergodicity modulo scaling in Condition~\ref{cond:def:CellConfigConstraint-2} of Definition~\ref{def:CellConfigConstraint} follows from the argument of \cite[Lemma 5.22]{Tutte} using~\eqref{eq:lem:ergodicity-200}. The rotation invariance in law property in  Condition~\ref{cond:def:CellConfigConstraint-3} is a consequence of the rotation invariance in law of $(h,\eta)$ and the rotation invariance of the underlying construction.}
  %  {The proof for Conditions~\ref{cond:def:CellConfigConstraint-1} and~\ref{cond:def:CellConfigConstraint-2} of Definition~\ref{def:CellConfigConstraint}  are identical to Lemmas 5.21 and 5.22 of Tutte, respectively. Condition~\ref{cond:def:CellConfigConstraint-3} is a consequence of the rotation invariance in law of $(h,\eta)$.}
\end{proof}

\begin{proposition}\label{prop:moment-bound-1}
    The cell configurations $\mr\cH^\cell$ and $\mr\cH^\lge$ satisfy the moment bound~\eqref{eq:thm-inv-circle}.
\end{proposition}

\begin{proposition}\label{prop:moment-bound-2}
    The cell configurations $\cH^\cell$ and $\cH^\lge$ satisfy the moment bound~\eqref{eq:thm-inv-circle-1}.
\end{proposition}

The following is  {\cite[Lemma 5.24]{Tutte}}, which essentially follows from~\cite[Lemma 2.9]{GMSinvariance}.
\begin{lemma}\label{lem:GMS2.9}
    Let $\cH$ be a cell configuration satisfying the conditions in Definition~\ref{def:CellConfigConstraint} and the moment bound~\eqref{eq:thm-inv-circle} or~\eqref{eq:thm-inv-circle-1}. Then for each $\e>0$, almost surely for large enough $r$, $\diam(H)<\e r$ for each $H\in \cH([-r,r]^2)$.
\end{lemma}

The proof of Proposition~\ref{prop:moment-bound-1} and~\ref{prop:moment-bound-2} are deferred to Section~\ref{subsubsec:moment-bound-cell}. %Now we apply the moment bounds to verify condition~(\ref{cond:def:CellConfigConstraint-large-(iii)}) in Definition~\ref{def:CellConfigConstraint-large}. 
\begin{lemma}\label{lem:line-connectivity}
    Condition~$($\ref{cond:def:CellConfigConstraint-large-(iii)}$)$ in Definition~\ref{def:CellConfigConstraint-large} holds for $(\cH^\cell,\cH^\lge)$ and $(\mr\cH^\cell,\mr\cH^\lge)$. 
\end{lemma}

\begin{proof}
   We only prove the claim for $(\mr\cH^\cell,\mr\cH^\lge)$ as the proof for the other one  {follows from the same argument as in \cite[Lemma 5.26]{Tutte}}. Fix $\e\in(0,1)$.  By Proposition~\ref{prop:moment-bound-1} and Lemma~\ref{lem:GMS2.9}, where we set the conductance to be equal to 1 for all the edges in $\mr\cH^\cell$ and $\mr\cH^\lge$,   almost surely for large enough $r>1$, for all $H\in\mr\cH^\cell([-r,r]^2)$  and $H\in\mr\cH^\lge([-r,r]^2)$ we have $\diam(H)\leq \e r/3$. In particular, by~\eqref{eq:cell-area}, for all $s\in\mr\cJ^-$ such that $\eta(\mr I_s)\cap [-r,r]^2\neq\emptyset$ we have $\diam(\eta(\mr I_s))\leq \e r/3$.

   Let $\ell\subseteq B(0;r)$  with length at least $\e r$, and suppose $\mr H_{t}^\cell\cap\ell\neq\emptyset$ and $\mr H_{t'}^\cell\cap \ell\neq \emptyset$. Below we assume $r$ is sufficiently large. Then  $\ell$ intersects the outer boundaries of $\mr H_{t}^\cell$ and $\mr H_{t'}^\cell$ and therefore $\ell\cap \mr H_{t}^\lge\neq\emptyset$ and  $\ell\cap \mr H_{t'}^\lge\neq\emptyset$ by our construction. From our definition of $\mr H_t^\cell$ and $\mr H_{t'}^\cell$, there exists $s\in  \mr\cT_t^\cell=\mr\cT_t^{\rmC,1}$ and $s'\in \mr\cT_{t'}^\cell=\mr\cT_{t'}^{\rmC,1}$ such that $\eta(\mr I_s)\cap \ell\neq\emptyset$ and $\eta(\mr I_{s'})\cap\ell\neq\emptyset$. Then $\partial\eta(\mr I_s)\cap \ell\neq\emptyset$ and $\partial\eta(\mr I_{s'})\cap\ell\neq\emptyset$.  Using the line connectivity property for $\{\eta(\mr I_s):s\in\cJ^-\}$ as similar in the proof of~\cite[Proposition 3.1]{GMS21tutte}, one can find $s_1,...,s_n\in \mr\cJ^-$ such that $s_1=s$ and $s_n=s'$, and for $j=1,...,n-1$, $\{s_j,s_{j+1}\}\in\mr\cE_\rmG$ and $\eta(\mr I_{s_j})\cap\ell\neq\emptyset$.    By Assertion~\ref{cond:prop:time-set-A-7} of Proposition~\ref{prop:time-set-A}, for $j=1,...,n-1$ there further exist $s_{j,1},...,s_{j,n_j}\in \mr\cT_{s_j}^{\lge}=\mr\cT_{s_j}^{\rmC,3}$ with $s_{j,1}=s_j$ and $s_{j,n_j}=s_{j+1}$, such that for $b=1,...,n_j-1$, $\{s_{j,b},s_{j,b+1}\}\in \mr\cE_\rmT$. Further using~\eqref{eq:cell-area} and Assertion~\ref{cond:prop:time-set-A-2} of Proposition~\ref{prop:time-set-A}, we have $s_j\in \mr\cT_{s_{j,b}}^\lge$ and therefore $\partial\eta(\mr I_{s_j})\subseteq \mr H_{s_{j,b}}^\lge$. Furthermore, since $\eta(\mr I_{s_j})\cap\ell\neq\emptyset$, $\diam(\eta(\mr I_{s_j}))<\e r/3$ and $\ell$ has length at least $\e r$, we see that $\partial\eta(\mr I_{s_j})\cap\ell\neq\emptyset$ and thus $\mr H_{s_{j,b}}^\lge\cap\ell\neq\emptyset$.    Therefore the collection of the cells $\mr H_{s_{j,b}}^\cell$ form a path connecting $\mr H_s^\cell$ and $\mr H_{s'}^\cell$ in $\mr \cH^\cell$ and each   $\mr H_{s_{j,b}}^\lge$ intersects $\ell$. Furthermore, by another application of Assertion~\ref{cond:prop:time-set-A-2} and  Assertion~\ref{cond:prop:time-set-A-7} of Proposition~\ref{prop:time-set-A}, $t\in  \mr\cT_s^\cell=\mr\cT_s^{\rmC,1}$, and there exists $t_1,...,t_m$ where $t_m=s$ and $t_1=t$ such that $\{t_j,t_{j+1}\}\in\mr\cE_\rmT$ for $j=1,...,m-1$ and $t_j \in \mr\cT_s^{\rmC,3}$. Therefore $\mr H_{t_{j}}^\cell$ form a path connecting $\mr H_s^\cell$ and $\mr H_{t}^\cell$ in $\mr \cH^\cell$ and each   $\mr H_{t_{j}}^\lge$ intersects $\ell$. By the same argument, the same holds for $\mr H_{s'}^\cell$ and $\mr H_{t'}^\cell$.  This concludes the proof. 
\end{proof}

\begin{proposition}\label{prop:property-CH}
    The cell configurations $\mr\cH^\cell$ and $\mr\cH^\lge$ satisfy all the conditions in Theorem~\ref{thm:CellSystemPack}, while the cell configurations $\cH^\cell$ and $\cH^\lge$ satisfy all the conditions in Theorem~\ref{thm:invariance-uniformization}.
\end{proposition}

\begin{proof}
    The first claim follows from Proposition~\ref{prop:def-cell-pack}, Lemma~\ref{lem:def-cell-pack-AA}, Proposition~\ref{prop:moment-bound-1} and Lemma~\ref{lem:line-connectivity}, while the second claim follows from  Proposition~\ref{prop:def-cell-inv}, Proposition~\ref{prop:moment-bound-2} and Lemma~\ref{lem:line-connectivity}.
\end{proof}

\subsubsection{Cell configuration moment bounds}\label{subsubsec:moment-bound-cell}

The goal here is to prove Proposition~\ref{prop:moment-bound-1} and~\ref{prop:moment-bound-2}. Throughout this subsection we fix $p>0$.

\begin{lemma}\label{lem:diam^2/area-E}
    For all sufficiently large $c_1,c_2>0$,  there exists $c'$ depending on $c_1,c_2,p$ such that 
    \begin{equation}
        \bbP_\rmE\bigg[\exists \text{ interval } J\subseteq [-M^{c_1},M^{c_1}], \  |J|\geq M^{-c_2},  \frac{\diam(\eta(J))^2}{\area(\eta(J))}>M  \bigg]\leq  
        c' M^{-p}.
    \end{equation}
\end{lemma}

\begin{proof}
    {Following \cite[Lemma 5.40]{Tutte}}, for all sufficiently large $c_1,c_2>0$, there is a constant $c_0'$ depending on $c_1,c_2,p$ such that  
    \begin{equation}\label{eq:pf-diam2/area-1}
          \bbP_\rmE\bigg[\exists \text{ interval } J\subseteq [-M^{c_1}, M^{c_1}], \ M^{-c_2}\leq |J|\leq 1, \frac{\diam(\eta(J))^2}{\area(\eta(J))}>M \bigg]\leq c_0'M^{-p}.
    \end{equation}
    Since  for each set $A$, $\frac{\diam(A)^2}{\area(A)}$ is invariant under dilation, using the scaling of $\eta$ in Lemma~\ref{lem:Rb}, we have 
     \begin{equation}\label{eq:pf-diam2/area-2}
          \bbP_\rmE\bigg[\exists \text{ interval } J\subseteq [-M^{c_1+c_2}, M^{c_1+c_2}], \ 1\leq |J|\leq M^{c_2}, \frac{\diam(\eta(J))^2}{\area(\eta(J))}>M \bigg]\leq c_0'M^{-p}.
    \end{equation}
     Therefore by a union bound, 
    \begin{equation}\label{eq:pf-diam2/area-3}
          \bbP_\rmE\bigg[\exists \text{ interval } J\subseteq [-M^{c_1}, M^{c_1}], \ M^{-c_2} \leq |J|\leq M^{c_2}, \frac{\diam(\eta(J))^2}{\area(\eta(J))}>M \bigg]\leq 2c_0'M^{-p}.
    \end{equation}
 Now repeating this trick of scaling and union bound, for any $j\in \bbN$,
    \begin{equation}\label{eq:pf-diam2/area-4}
          \bbP_\rmE\bigg[\exists \text{ interval } J\subseteq [-M^{c_1}, M^{c_1}], \ M^{-c_2} \leq |J|\leq M^{(2^j-1)c_2}, \frac{\diam(\eta(J))^2}{\area(\eta(J))}>M \bigg]\leq 2^jc_0'M^{-p},
    \end{equation}
    and the conclusion follows by fixing some $j$ such that $(2^j-1)c_2\geq c_1$ and setting $c' = 2^j c_0'$.
\end{proof}

%{The following is from identical argument as Lemma 5.36 in the draft.}
\begin{lemma}\label{lem:gap-J-min}
   Let $\wt c, c_1>0$.  There exists $c',c_2>0$, such that for all $M>1$,
    \begin{equation}\label{lem:gap-J-min-A}
        \bbP_\rmE\big[\exists\ t\in \cJ^-: I_t\subseteq I(\lceil c_1\log M\rceil) \text{ and } |I_t|\notin  (M^{-c_2},c_2\log M)  \big]\leq c'M^{-\wt c}.
    \end{equation}
   % The same holds for $\cJ^-_\rmL$ or $\cJ^-_\rmR$ instead of $\cJ^-$.  
    {For $t\in\cJ^-$, let $\wt I_t = [\wt t^-,t]$ where $\wt t^- = \sup\{s<t:s\in\cJ^-\}$. Then there exists $\wt c',\wt c_2>0$, such that 
      \begin{equation}\label{lem:gap-J-min-B}
        \bbP_\rmE\big[\exists\ t\in \cJ^-: \wt I_t\subseteq I(\lceil c_1\log M\rceil) \text{ and } |\wt I_t|\notin  (M^{-\wt c_2},\wt c_2\log M)  \big]\leq \wt c'M^{-\wt c}.
    \end{equation}}
\end{lemma}

\begin{proof}
    The first claim~\eqref{lem:gap-J-min-A} follows from  \cite[Lemma 5.36]{Tutte}, while the second claim~\eqref{lem:gap-J-min-B} follows analogously, where one first use the setting where $\cJ^-$ is a rate 0.5 Poisson point process under $\bbP_\rmQ$ and further use part~(\ref{cond:thm:KMT-0-cone-(iv)}) of Theorem~\ref{thm:KMT-0-cone} to transfer to the setting of $\bbP_\rmE$.
\end{proof}

\begin{proof}[Proof of Proposition~\ref{prop:moment-bound-1}]
   We only prove~\eqref{eq:thm-inv-circle} for $\mr\cH^\cell$; the proof for $\mr \cH^\lge$ is identical. Throughout the proof, $c>0$ is a constant that may vary line by line. We write $\wh\cell(0)$ for the a.s.\ unique time $t_0$ in $\mr\cJ^-$   such that $0\in \mr H_{t_0}^{\rm cell}$. Then following~\eqref{eq:cell-area},
    \eqb\label{eq:pf-prop:cell-moment-bound-000}\begin{split}
    \frac{\diam(\mr H_0^{\rm cell})^2}{\area(\mr H_0^{\rm cell})}&\leq \frac{\big(\sum_{s\in \mr\cT_{\wh\cell(0)}^{\cell}} \diam(\eta(\mr I_s))\big)^2}{\sum_{s\in \mr\cT_{\wh\cell(0)}^{\cell}}\area(\eta(\mr I_s))/\# \mr\cT_{s}^{\cell}}\leq \#\mr\cT_{\wh\cell(0)}^{\cell}\cdot \max_{s\in \mr\cT_{\wh\cell(0)}^{\cell}}\# \mr\cT_{s}^{\cell}\cdot\frac{\sum_{s\in \mr\cT_{\wh\cell(0)}^{\cell}} \diam(\eta(\mr I_s))^2}{\sum_{s\in \mr\cT_{\wh\cell(0)}^{\cell}}\area(\eta(\mr I_s))}\\
    &\leq \#\mr\cT_{\wh\cell(0)}^{\cell}\cdot \max_{s\in \mr\cT_{\wh\cell(0)}^{\cell}}\# \mr\cT_{s}^{\cell}\cdot \max_{s\in \mr\cT_{\wh\cell(0)}^{\cell}}\frac{\diam(\eta(\mr I_s))^2}{\area(\eta(\mr I_s))}.
    \end{split}\eqe
    We further write $s_0$ for the a.s.\ unique time in $\mr\cJ^-$ such that $0\in \mr I_{s_0}$. From Lemma~\ref{lem:gap-hat-J^-}, we have $\bbP_\rmE[\mr I_{s_0}\nsubseteq I(k/4) ]\lesssim 2^{-ck}$. By~\eqref{eq:cell-area}, we have $s_0\in\mr\cT_{\wh\cell(0)}^\cell$, and by Assertion~\ref{cond:prop:time-set-A-2} of Proposition~\ref{prop:time-set-A}, $\wh\cell(0)\in \mr\cT_{s_0}^{\cell}$. We write $\wt\cT:=\cup_{t\in \mr\cT_{s_0}^{\cell}} \mr\cT_{t}^{\cell}$, then
    \begin{equation}\label{eq:pf-prop:cell-moment-bound-0}
         \frac{\diam(\mr H_0^{\rm cell})^2}{\area(\mr H_0^{\rm cell})}\leq \big(\max_{s\in\wt\cT}\# \mr\cT_{s}^{\cell}\big)^2\cdot \max_{s\in \wt\cT}\bigg(\frac{\diam(\eta(\mr I_s))^2}{\area(\eta(\mr I_s))}\bigg).
    \end{equation}
    Applying Assertion~\ref{cond:prop:time-set-A-5} of Proposition~\ref{prop:time-set-A} twice and a union bound, we have  
    $$\bbP_\rmE[\wt\cT\nsubseteq I(k/2)]\lesssim 2^{-ck}.$$
    Further combined with Lemma~\ref{lem:gap-hat-J^-}, we have
     \eqb\label{eq:pf-prop:cell-moment-bound-1}\bbP_\rmE[\bigcup_{s\in\wt \cT}\mr I_s\nsubseteq I(k)]\lesssim 2^{-ck}.\eqe
    By~\eqref{eq:pf-prop:cell-moment-bound-1}, we may choose $c_1>0$ sufficiently large, such that for some $c'>0$,
    $$\bbP_\rmE\big[\bigcup_{s\in\wt \cT}\mr I_s\nsubseteq I(\lceil c_1\log M\rceil) \big]\leq c'M^{-p}.$$
    Combined with ~\eqref{lem:gap-J-min-B} and $\mr\cJ^-\subseteq \cJ^-$, by a union bound, for some $c_2>0$,
    $$\bbP_\rmE[\bigcup_{s\in\wt \cT}\mr I_s\subseteq I(\lceil c_1\log M \rceil),\ \min_{s\in\wt\cT}|\mr I_s|\geq M^{-c_2}]\geq1- 2c'M^{-p}.$$
    Further by Lemma~\ref{lem:diam^2/area-E} and a union bound, we conclude that  for some constant $c'$,
    $$\bbP_\rmE\big[\max_{s\in \wt\cT}\bigg(\frac{\diam(\eta(\mr I_s))^2}{\area(\eta(\mr I_s))}\bigg)>M\big]\leq c'M^{-p}.$$
    This verifies that for each $p'\in(0,p)$, $$\bbE_\rmE\bigg[\max_{s\in \wt\cT}\bigg(\frac{\diam(\eta(\mr I_s))^2}{\area(\eta(\mr I_s))}\bigg)^{p'}\bigg]<\infty.$$
    By~\eqref{eq:pf-prop:cell-moment-bound-1} along with Assertion~\ref{cond:prop:time-set-A-6} of Proposition~\ref{prop:time-set-A}, for large enough $k$, we have
    $$\bbP_\rmE[\max_{s\in\wt\cT}\# \mr\cT_{s}^{\cell}>k^{28}]\lesssim 2^{-ck},$$
    which implies $\bbE_\rmE\big[\big(\max_{s\in\wt\cT}\# \mr\cT_{s}^{\cell}\big)^p\big]<\infty$. Since $p>0$ is arbitrary, by~\eqref{eq:pf-prop:cell-moment-bound-000}, we infer that $\bbE_\rmE[\big( \frac{\diam(\mr H_0^{\rm cell})^2}{\area(\mr H_0^{\rm cell})}\big)^p]<\infty$ for every $p>0$. Further using $\deg(\mr H_0^{\rm cell})\leq \#\mr\cT_{\wh\cell(0)}^{\cell}\leq \max_{s\in\wt\cT}\# \mr\cT_{s}^{\cell}$ we have $\bbE_\rmE[\deg(\mr H_0^{\rm cell})^p]<\infty$. This completes the proof.
\end{proof}
For $t\in\cJ^-$, we define 
$$\wt \cT_t^{\rmC,d} = \cup_{u\in \cT^{\rmC,d}_t}\cT_u^{\rmC,d}.$$
By iterating once, it is straightforward to check that Assertion~\ref{cond:prop:time-set-A-0-5} of Proposition~\ref{prop:time-set-A-0}  holds for $\wt\cT_t^{\rmC,d}$ in place of $\cT_t^{\rmC,d}$, and by \cite[Lemma 5.29]{Tutte}, $\cT_s^\cell\subseteq \wt \cT_s^{\rmC,1}$ and $\cT_s^{\cell,\inv}\subseteq \wt \cT_s^{\rmC,1}$ for any $s\in\cJ^-$.

We will also use the following basic result.
\begin{lemma}\label{lem:Euler}
    Let $T$ be a disk triangulation without loops, and $\frv$ be a vertex on $T$. Let $f(\frv)$ be the number of  faces of $T$ incident to $\frv$ (excluding the outer face), and $d(\frv)$ be the number of vertices incident to $\frv$. Then $f(\frv)\leq 2d(\frv)-2$.
\end{lemma}
\begin{proof}
    Consider the planar map $T_\frv$ formed by the  faces of $T$ that are incident to $\frv$ (excluding the outer face). By Euler's formula, $2\#\cE T_\frv = 2\# \cF T_\frv+2\# \cV T_\frv-4$. There are exactly $f(\frv)$ faces of $T$ incident to $\frv$ (excluding the outer face), and since $T$ has no loops, the remaining faces of $T_\frv$ has degree at least 2. Therefore $2\#\cE T_\frv\geq 3f(\frv) + 2(\# \cF T_\frv-f(\frv))$. This gives $f(\frv)\leq 2\# \cV T_\frv-4$, and the claim follows since $\# \cV T_\frv = d(\frv)+1$.
\end{proof}

\begin{proof}[Proof of Proposition~\ref{prop:moment-bound-2}]
 We only prove~\eqref{eq:thm-inv-circle-1} for $\cH^\cell$; the proof for $ \cH^\lge$ is identical. By \cite[Lemma 5.29]{Tutte}, we have $\cT_s^\cell\subseteq \wt\cT_s^{\rmC,1}$ and $\cT_s^{\cell,\inv}\subseteq \wt\cT_s^{\rmC,1}$ for each $s\in \cJ^-$. Then by~\eqref{eq:cell-area-0}, for each $t\in\cJ^-_\rmL$,
  \eqb\label{eq:pf-prop:cell-moment-bound-000-A}\begin{split}
    \frac{\diam( H_t^{\rm cell})^2}{\area( H_t^{\rm cell})}&\leq \frac{2\big(\sum_{s\in \cT_{t}^{\cell}} \diam(\eta( I_s))\big)^2}{\sum_{s\in \cT_{t}^{\cell}}\area(\eta( I_s))/\# \cT_{s}^{\cell,\inv}}\leq 4\#\cT_{t}^{\cell}\cdot \max_{s\in \cT_{t}^{\cell}}\#  \cT_{s}^{\cell,\inv}\cdot\frac{\sum_{s\in  \cT_{t}^{\cell}} \diam(\eta( I_s))^2}{\sum_{s\in \cT_{t}^{\cell}}\area(\eta( I_s))}\\
    &\leq 4\#\wt\cT_{t}^{\rmC,1}\cdot \max_{s\in \wt\cT_{t}^{\rmC,1}}\# \wt\cT_{s}^{\rmC,1}\cdot \max_{s\in \wt\cT_{t}^{\rmC,1}}\frac{\diam(\eta( I_s))^2}{\area(\eta( I_s))}\leq 4\big( \max_{s\in \wt\cT_{t}^{\rmC,1}}\# \wt\cT_{s}^{\rmC,1}\big)^2 \cdot \max_{s\in \wt\cT_{t}^{\rmC,1}}\frac{\diam(\eta( I_s))^2}{\area(\eta( I_s))}
    \end{split}\eqe
    since $t\in\wt \cT^{\rmC,1}_t$.
 %{As shown in Lemma 5.14 of Tutte, we have $\deg(H_t^\cell)\leq \#\wt\cT_t^{\rmC,1}$.}
 
 Recall the definition of $d_0^f$ and $\cH_0^{\rm face}$ before Theorem~\ref{thm:invariance-uniformization}.   We write $\cell(0)$ for the a.s.\ unique time $t_0$ in $\cJ^-_\rmL$ such that $0\in H_{t_0}^\cell$. Now suppose $t\in \cJ^-_\rmL$ such that $H_t^\cell\in \cH_0^{\rm face}$. Then there exists $t'\in \cJ^-_\rmR$ such that $\{t',\cell(0)\},\{t',t\}\in \cE_\rmT$. By Assertion~\ref{cond:prop:time-set-A-0-3} of Proposition~\ref{prop:time-set-A-0}, we have $t\in \wt \cT_{\cell(0)}^{\rmC,1}$.  We write $s_{0,\rmL}\in\cJ_\rmL^-$ and $s_{0,\rmR}\in\cJ^-_\rmR$ for the a.s.\ unique time  such that $0\in I_{s_{0,\rmL}}\cap I_{s_{0,\rmR}}$. Then we have $\{s_{0,\rmL},s_{0,\rmR}\}\cap \cT_{ \cell(0)}^\cell\neq\emptyset$. Following Assertion~\ref{cond:prop:time-set-A-0-3} of Proposition~\ref{prop:time-set-A-0}, we have $s\in \wt\cT_t^{\rmC,1}$ if and only if $t\in \wt\cT_s^{\rmC,1}$, which further implies that $\cell(0)\in \wt\cT_{s_{0,\rmL}}^{\rmC,1}\cup \wt\cT_{s_{0,\rmR}}^{\rmC,1}$. Now by Proposition~\ref{prop:Tutte5.8}, Proposition~\ref{lem:gap-hat-J^-} holds for $\cJ^-_\rmL$ or $\cJ^-_\rmR$ in place of $\mr\cJ^-$, and therefore we have $\bbP_\rmE[\{s_{0,\rmL},s_{0,\rmR}\}\nsubseteq I(k)]\lesssim 2^{-ck}$. Let
 $$\wt\cT_1 =\bigcup_{t_0\in \wt\cT_{s_{0,\rmL}}^{\rmC,1}\cup \wt\cT_{s_{0,\rmR}}^{\rmC,1}}  \wt\cT_{t_0}^{\rmC,1} \ \ \text{, } \ \ \  \wt\cT_2 =\bigcup_{t\in \wt\cT_1} \cT_t^{\rmC,1} \ \ \text{ ,  } \ \ \ \wt\cT_3 = \bigcup_{u\in \wt\cT_2} \cT_u^{\rmC,1}  \ \ \text{ and  } \ \ \  \wt\cT_4 = \bigcup_{s\in \wt\cT_3} \cT_s^{\rmC,1} .  $$
 It follows that for each pair $(s,t)$ such that $s\in \wt \cT_t^{\rmC,1}$ and $t\in \wt\cT_{\cell(0)}^{\rmC,1}$, $s\in\wt \cT_3$.
 By iteratively applying Assertion~\ref{cond:prop:time-set-A-0-5} of Proposition~\ref{prop:time-set-A-0}, we have $\bbP_\rmE[\wt\cT_3\nsubseteq I(6k)]\lesssim 2^{-ck}$ and $\bbP_\rmE[\wt\cT_4\nsubseteq I(8k)]\lesssim 2^{-ck}$, and furthermore
 $$\bbP_\rmE\bigg[  \bigcup_{s\in\wt\cT_3} I_s\nsubseteq I(7k)   \bigg]\lesssim 2^{-ck} $$
 for some $c>0$. Following the same proof as Proposition~\ref{prop:moment-bound-1}, using Lemma~\ref{lem:diam^2/area-E} and Lemma~\ref{lem:gap-J-min}, we conclude that for every $p'\in (0,p)$,
  $$\bbE_\rmE\bigg[\max_{s\in \wt\cT_3}\bigg(\frac{\diam(\eta(  I_s))^2}{\area(\eta( I_s))}\bigg)^{p'}\bigg]<\infty.$$
  Using $\bbP_\rmE[\wt\cT_4\nsubseteq I(8k)]\lesssim 2^{-ck}$ and Assertion~\ref{cond:prop:time-set-A-0-6} of Proposition~\ref{prop:time-set-A-0}, we conclude that 
  \eqb\label{eq:degreeeeeee-new}\bbP_\rmE\big[\max_{s\in \wt\cT_3} \#\wt\cT_s^{\rmC,1}\geq k^{60}\big]\lesssim 2^{-ck},\eqe
  and therefore $\bbE_\rmE\big[ \big(\max_{s\in \wt\cT_3} \#\wt\cT_s^{\rmC,1}\big)^p \big]<\infty$. 

  Now consider a time $t\in\cJ^-$ and let $\frv_t$ be its vertex in $T$ with degree $d$. Since there are exactly $d$ triangular faces incident to $\frv_t$ and $T$ has no loops, by Lemma~\ref{lem:Euler} and Assertion~\ref{cond:prop:time-set-A-0-3} of Proposition~\ref{prop:time-set-A-0}, $d\leq 2\#\cT_t^{\rmC,1}$.
  %Let $T_t$ be the planar map formed by the union of all the faces incident to $\frv_t$, and $\cV T_t,\cE T_t,\cF T_t$ be its vertex, edge and face sets (with the unbounded face counted), respectively. By Euler's formula, $\# \cE T_t = \#\cF T_t+\#\cV T_t-2$.  we have $2\#\cE T_t\geq 3d+ 2(\#\cF T_t - d)$ since the faces in $\#\cF T_t$ not incident to $\frv_t$ has degree at least 2. Thus $d\leq 2\# \cV T_t-4$. On the other hand, all the times $t'\in\cJ^-$ for vertices of $T_t$ are contained within $\cT_t^{\rmC,1}$ by  Assertion~\ref{cond:prop:time-set-A-0-3} of Proposition~\ref{prop:time-set-A-0}. 
  Therefore for each cell $H_t\in \cH$, since its degree in the primal map $M$ is bounded by its degree in the triangulation $T$, we have $\deg(H_t)\leq 2\#\cT_t^{\rmC,1}$ and~\eqref{eq:degreeeeeee-new} implies that $\bbE\big[\max_{H\in\cH^{\mathrm{face}}_0}(\deg(H))^p\big]<\infty$.
  Since $p>0$ is arbitrary, we have verified that 
  $$\bbE\Big[ \max_{H\in \cH^{\mathrm{face}}_0}\Big(\frac{\diam(H)^2}{\area(H)}\big(\deg(H)\big)^q\Big)\Big]<\infty $$
  for any $q>0$. Finally, for any $t^f\in\cJ^-_\rmR$ such that $\{t^f,\cell(0)\}\in\cE_\rmT$, we have $t^f\in\wt\cT_3$ and $\deg(t^f)\leq 2\# \cT_{t^f}^{\rmC,1}$. Then by an identical argument as above, $\bbE_\rmE[(d_0^f)^q]<\infty$ for any $q>0$. This finishes the proof.
\end{proof}

\subsection{Convergence of tree-weighted maps on 0-quantum cone}\label{subsec:3.3}

In this subsection, we prove that under the probability measure $\bbP_\rmE$, the mating-of-trees embedding under 0-quantum cone generated using $(h,\eta,Z)$ is close to the circle packing of $\Phi_\rmT(\cZ)$ and the Riemann uniformization embedding of $M=\Phi_\rmM(\cZ)$. Recall that $\Phi_\rmT(\cZ) = (T,\mathfrak{tr})$ and $\Phi_{\rm simple}\circ \Phi_\rmT(\cZ) = (\mr T,\mr \frtr)$, where $T$ is the infinite plane triangulation generated from $\cZ$ using Mullin bijection, $\mathfrak{tr}:\bbR\to \cF T$ is the path on faces of $T$, $\mr T = \Phi_{\rm simple}(T)$ is the version of $T$ with double edges collapsed.  We will also view a circle packing for $\mr T$ as a circle packing for $T$, where the disks for the vertices enclosed by multiple edges between vertices $v_1$ and $v_2$ will be the tangent point of the disks for $v_1$ and $v_2$. % and   $\mr{\mathfrak{tr}}:\bbR\to \cF \mr T$ is the path on faces of $\mr T$. 
With slight abuse of notation, once we fix an embedding of $T$ or $M$, we will also view $\mathfrak{tr}$ as a continuous path $\bbR\to\bbC$ such that for each $t\in\bbR$, $\mathfrak{tr}(t)$ is a point on the corresponding embedded face, where the faces are possibly degenerate in case of circle packing for $T$. For   the Riemann uniformization, the triangulation $T$ can be naturally embedded onto $\frM(M)$ by embedding the vertices of $\cV T\backslash\cV M$ at the center of the corresponding faces, and we also write $\frtr$ for the image of the space-filling path on $T$ under a conformal map $\varphi:\frM(M)\to\bbC$. Note that some points and edges of $T$ under this embedding of $T$ onto $\frM(M)$ might be degenerate in this case since we are collapsing faces of degree 1 and 2 of $M$ when defining $\frM(M)$.  Recall that $s\mapsto \frv_s$ is a bijection between $\cJ_\rmL^-$ and $\cV M$ and also a bijection  between $\cJ_\rmR^-$ and $\cF M$, and we have map $t\mapsto \fre_t$ from $\cJ$ to $\cE T$. %  and   $s\mapsto \mr\frv_s$ is a  bijection between $\mr\cJ^-$ and $\cV \mr T$.
For $s\in\cJ^-_\rmL$, we  identify $\frv_s$ with the corresponding vertex for $\frv_s$ on the surface $\frM(M)$, and for $s\in\cJ^-_\rmR$, we identify $\frv_s$ with the center of its corresponding face, which is denoted by $\mathfrak{f}_s$, on $\frM(M)$. If the face $\mathfrak{f}_s$ has degree 1, then it  shrinks to its adjacent vertex on  $\frM(M)$, and we identify $\frv_s$ with that vertex. If  $\mathfrak{f}_s$ has degree 2, then it shrinks to an edge on $\frM(M)$ and we identify $\frv_s$ with the midpoint of that edge.  Given a circle packing for $  T$, we write $o(\frv_s)$ for the center of the disk for the vertex $\frv_s$. 

% We recall the topology on curves viewed modulo time parameterization. %as in~\cite{GMSinvariance}.
% Let $\beta_1:[0,T_{\beta_1}]\to\bbC$ and $\beta_2:[0,T_{\beta_2}]\to\bbC$ be continuous curves defined on possibly different time intervals. We define 
% \begin{equation}\label{eq:curve-metric-0}
%     \mathbb{d}^{\rm CMP}(\beta_1,\beta_2):=\inf_\phi \sup_{t\in[0,T_{\beta_1}]}|\beta_1(t)-\beta_2(\phi(t))|,
% \end{equation}
% where the infimum is over all increasing homeomorphisms $\phi:[0,T_{\beta_1}]\to [0,T_{\beta_2}]$ and ``CMP'' stands for ``curves modulo parameterization". 

%We can also consider the extended version of circle packing for $T$, where the faces removed in the construction of $\mr T=\Phi_{\rm simple}(T)$ will be embedded as a single  point or a straight line segment in the packing. 
For $s\in\cJ$, let $T_s\subseteq T$ be the map formed by the faces $\{\mathfrak{tr}(u):u\leq s \}$, i.e., the sewing procedure stopped at time $s$.   Let $s_{1,\rmL}<s_{2,\rmL}<...$ and $s_{1,\rmR}<s_{2,\rmR}<...$ be the set of times in $\cJ^-$ such that the  vertices $\frv_{s_{1,\rmL}},\frv_{s_{2,\rmL}},...$ and $\frv_{s_{1,\rmR}},\frv_{s_{2,\rmR}},...$ form the left and right boundaries of $T_s$, respectively. As a consequence of Lemma~\ref{lem:sew}, $s_{k,\rmL} = \inf\{u>s:\cL_u-\cL_s=-k\}$.  For $s,t\in \cJ$ where $s<t$,  we let $T_{s,t}$ be the disk triangulation formed by $\{{\mathfrak{tr}}(u):s\leq u\leq t\}\backslash\{{\mathfrak{tr}}(s) \}$, and let $\frv_{s_{1,\rmL}},\frv_{s_{2,\rmL}},...,\frv_{s_{i_1,\rmL}}$ and $\frv_{s_{1,\rmR}},\frv_{s_{2,\rmR}},...,\frv_{s_{i_2,\rmR}}$ be the boundary vertices of $T_{s,t}$ that are on the boundary of $T_s$, and let $\frv_{t_{1,\rmL}},\frv_{t_{2,\rmL}},...,\frv_{t_{j_1,\rmL}}$ and $\frv_{t_{1,\rmR}},\frv_{t_{2,\rmR}},...,\frv_{t_{j_2,\rmR}}$ be the boundary vertices of $T_{s,t}$ that are on the boundary of $T_t$ but not on the boundary of $T_s$. We allow that $j_1,j_2$ to be zero, and set $s_{0,q}=s$, $t_{0,q}=t$ for $q\in\{\rmL,\rmR\}$ below.
%and call the vertex on the left (resp.\ right) boundary of both $T_s$ and $T_t$ the leftmost (resp.\ rightmost) vertex of $T_{s,t}$. %Similarly, for $s,t\in \mr\cJ$,  let $\mr T_{s,t}$ be the disk triangulation formed by $\{\mr{\mathfrak{tr}}(u):s\leq u\leq t\}\backslash\{\mr{\mathfrak{tr}}(s) \}$, and let the leftmost and rightmost vertex of $\mr T_{s,t}$ to be the same as the leftmost and rightmost vertex of  $T_{s,t}$. 
Note that  by Lemma~\ref{lem:hat-minus}, if $s\in \mr\cJ$, then the vertices on the boundary of $T_s$ are also vertices in $\mr T$. %Similarly, 
For $s<t$, we call the point $z_{s,t,\rmL}^-=\eta( u_{s,t,\rmL}^-)$ and $z_{s,t,\rmR}^-=\eta( u_{s,t,\rmR}^-)$ the leftmost and rightmost point of $\eta([s,t])$, where $u_{s,t,\rmL}^-$ (resp.\ $u_{s,t,\rmR}^-$) is the time in $[s,t]$ when $L|_{[s,t]}$ (resp.\ $R|_{[s,t]}$) hits its running infimum. 

\begin{figure}
    \centering
    \includegraphics[scale=0.85]{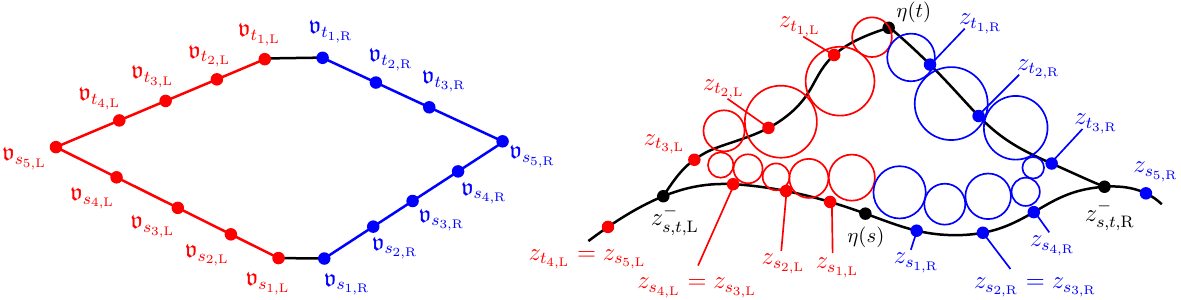}
    \caption{An illustration of part~(\ref{cond:thm:conv-0-cone-pack-Hcell-(iii)}) of Theorem~\ref{thm:conv-0-cone-pack-Hcell} where $(i_1,i_2,j_1,j_2) = (5,5,4,3)$. The goal of this part is to show that when restricted to boundary vertices in $T_{s,t}$, each disk in the packing $\cP$ will be close to a reference point, which are red and blue solid dots. This will eventually allow us to conclude that, as we join centers of these red and blue boundary disks, we get a curve that is close to $\partial\eta([s,t])$, and therefore we can apply Theorem~\ref{thm:He-Schramm-extension-1} below to cut and transfer to the finite area settings.}
    \label{fig:curvemetricclose}
\end{figure}

\begin{theorem}\label{thm:conv-0-cone-pack-Hcell}
    Under the probability measure $\bbP_\rmE$, the following is true. There exists a circle packing $\cP$ for $T$ such that almost surely, the following is true.
    \begin{enumerate}[(i)]
        \item The maximal diameter of disks intersecting $B(0;r)$ is   $o(r)$ as $r\to\infty$; \label{cond:thm:conv-0-cone-pack-Hcell-(i)}
        \item We have \begin{align}
 &  \lim_{r\to\infty} \frac{1}{r}\sup\{ |\eta(t)-o(\frv_t)|: t\in \cJ^-,\eta(t)\in B(0;r)\} = 0; \label{eq:E-pack} \\ & 
   \lim_{r\to\infty} \frac{1}{r}\sup\{ |\eta(t)-{\mathfrak{tr}}(t)|:\eta(t)\in B(0;r)\} = 0,\label{eq:E-curve-pack}
\end{align}  \label{cond:thm:conv-0-cone-pack-Hcell-(ii)}
where the limit~\eqref{eq:E-curve-pack} is uniform over all choices of the space-filling curve ${\mathfrak{tr}}$ such that ${\mathfrak{tr}}(t)$ is within the corresponding  embedded face. 
\item For every $\e>0$, a.s.\ for each large enough $r>0$, we have the following. For every $s,t\in\mr\cJ$ such that $\eta([s,t])\subseteq B(0;r)$, in the disk triangulation $  T_{s,t}$, there exists (not necessarily distinct) points $z_{s_{1,\rmL}},z_{s_{2,\rmL}},...,z_{s_{i_1,\rmL}}$ on the left boundary of $\eta((-\infty,s]))$ aligned from right to left, $z_{s_{1,\rmR}},z_{s_{2,\rmR}},...,z_{s_{i_2,\rmR}}$ on the right boundary of $\eta((-\infty,s]))$ aligned from left to right,  $z_{t_{1,\rmL}},z_{t_{2,\rmL}},...,z_{t_{j_1,\rmL}}$ on the left boundary of $\eta((-\infty,t]))$ aligned from right to left, $z_{t_{1,\rmR}},z_{t_{2,\rmR}},...,z_{t_{j_2,\rmR}}$ on the right boundary of $\eta((-\infty,t]))$ aligned from left to right, such that for $x\in\{s,t\}$, $q\in\{\rmL,\rmR\}$, the distance between each $z_{x_{k,q}}$ and the center of the disk for $\frv_{x_{k,q}}$ is less than $\e r$, and the diameter of the segment of $\partial\eta((-\infty,x])$ between $z_{x_{k,q}}$ and $z_{x_{k-1,q}}$ is less than $\e r$, where we define $z_{x_{0,q}} = \eta(x)$. Furthermore, the segment of $\partial\eta((-\infty,s])$ between $z_{s_{i_1,\rmL}}$  (resp.\ $z_{s_{i_2,\rmR}}$) and $z_{s,t,\rmL}^-$  (resp.\ $z_{s,t,\rmR}^-$) has diameter less than $\e r$, and the segment of $\partial\eta((-\infty,t])$ between $z_{t_{j_1,\rmL}}$  (resp.\ $z_{t_{j_2,\rmR}}$) and $z_{s,t,\rmL}^-$  (resp.\ $z_{s,t,\rmR}^-$)  has diameter less than $\e r$.  \label{cond:thm:conv-0-cone-pack-Hcell-(iii)} 
%distance between the $z_{s_{i_1,\rmL}}$ (resp.\ $z_{s_{i_2,\rmR}}$) and the leftmost (resp.\ rightmost) point of $\eta([s,t])$ is less than $\e r$, and the same holds for $z_{t_{j_1,\rmL}}$ and $z_{t_{j_2,\rmR}}$ in place of $z_{s_{i_1,\rmL}}$ and $z_{s_{i_2,\rmR}}$.  %formed by $\{\mr{\mathfrak{tr}}(u):s\leq u\leq t\}\backslash\{\mr{\mathfrak{tr}}(s) \} $.  the distance~\eqref{eq:curve-metric-0} between the curve joining the centers of the disks for vertices on $\partial \mr T_{s,t}$ in $\cP$ and $\partial\eta([s,t])$ is less than $\e r$, and the distance between the center of the disk for the leftmost (resp.\ rightmost) vertex of $\mr T_{s,t}$ and the leftmost (resp.\ rightmost) point of $\eta([s,t])$ is less than $\e r$.
\item There is an absolute constant $C>0$, such that in the setting of claim (iii), the quantum length of the boundary segment of $\eta((-\infty,x])$ from $\eta(x)$ to $z_{x_{k,q}}$ is in $[k-C\log r,k+C\log r]$. Furthermore, if $z_{x_{k,q}}$ is on $\partial\eta((-\infty,s]))\backslash\partial\eta([s,t])$, then the quantum length of the segment of  $\eta((-\infty,s])$ from $z_{x_{k,q}}$ to $z_{s,t,q}^-$ is less than $C\log r$. \label{cond:thm:conv-0-cone-pack-Hcell-(iv)}
    \end{enumerate}
%and  The same conclusions hold for the extended version of the circle packing for the triangulation $T$.
%Furthermore,  we have the following. 
\end{theorem}
%In particular, the claims~(\ref{cond:thm:conv-0-cone-pack-Hcell-(i)}), ~(\ref{cond:thm:conv-0-cone-pack-Hcell-(iii)}) above along with the continuity of $\partial\eta([s,t])$ imply that the curve distance~\eqref{eq:curve-metric-0} between the embedded boundary of $T_{s,t}$ and $\partial\eta([s,t])$ has order $o(r)$.

The reason for claim~(\ref{cond:thm:conv-0-cone-pack-Hcell-(iii)}) is that, as in Condition~(\ref{cond:thm:He-Schramm-extension-1-(iv)}) of Theorem~\ref{thm:He-Schramm-extension-1}, to transfer the above convergence of circle packing to the other settings, we will need the convergence of the boundary of the above circle packing for $T_{s,t}$ when viewed as curves as well. See Figure~\ref{fig:curvemetricclose} for an illustration. Note that as explained in Section~\ref{subsubsec:MatedCRT}, here we continue to multiply the actual quantum length by a half.

\begin{theorem}\label{thm:conv-0-cone-unif-Hcell}
    Under the probability measure $\bbP_\rmE$, the following is true.  There exists a conformal map $\varphi:\frM(M)\to \bbC$ such that the following is true.
    \begin{enumerate}
        \item Almost surely, \begin{align}
 &  \lim_{r\to\infty} \frac{1}{r}\sup\{ |\eta(t)-\varphi(\frv_t)|: t\in \cJ_\rmL^-,\eta(t)\in B(0;r)\} = 0; \label{eq:E-unif} \\  
 &  \lim_{r\to\infty} \frac{1}{r}\sup\{\diam(\varphi(\mathfrak{f}_t)): t\in\cJ_\rmR^-, \eta(t)\in B(0;r)\} =0; \label{eq:E-edge}   \\ & 
   \lim_{r\to\infty} \frac{1}{r}\sup\{ |\eta(t)-{\mathfrak{tr}}(t)|:\eta(t)\in B(0;r)\} = 0,\label{eq:E-curve-unif}
\end{align}
where the limit~\eqref{eq:E-curve-unif} is uniform over all choices of the space-filling curve ${\mathfrak{tr}}$ such that ${\mathfrak{tr}}(t)$ is within the corresponding  embedded face.  {Furthermore, for each $\e>0$, almost surely for sufficiently large $r>0$, if the embedded face $\varphi(\mathfrak{f}_{f})$ intersects $B(0;r)$, then $\diam(\varphi(\mathfrak{f}_{f}))<\e r$.}
\item Claims $($\ref{cond:thm:conv-0-cone-pack-Hcell-(iii)}$)$ and $($\ref{cond:thm:conv-0-cone-pack-Hcell-(iv)}$)$ of Theorem~\ref{thm:conv-0-cone-pack-Hcell} with $\cJ$ in place of $\mr\cJ$ and ``the center of the disk for $\frv_{x_{k,q}}$" replaced by ``$\varphi(\frv_{x_{k,q}})$" are true.
    \end{enumerate}
    %Furthermore, for every $\e>0$, a.s.\ for each large enough $r>0$ we have the following. For every $s,t\in\mr\cJ$ such that $\eta([s,t])\subseteq B(0;r)$, consider the disk triangulation $ T_{s,t}$ formed by $\{{\mathfrak{tr}}(u):s\leq u\leq t\}\backslash\{{\mathfrak{tr}}(s) \} $ and the surface $M(T_{s,t})\subseteq M(T)$. Then  the distance~\eqref{eq:curve-metric-0} between $\varphi(\partial M(T_{s,t}))$ and $\partial\eta([s,t])$ is less than $\e r$.
\end{theorem}

\begin{figure}
    \centering
    \includegraphics[scale=0.8]{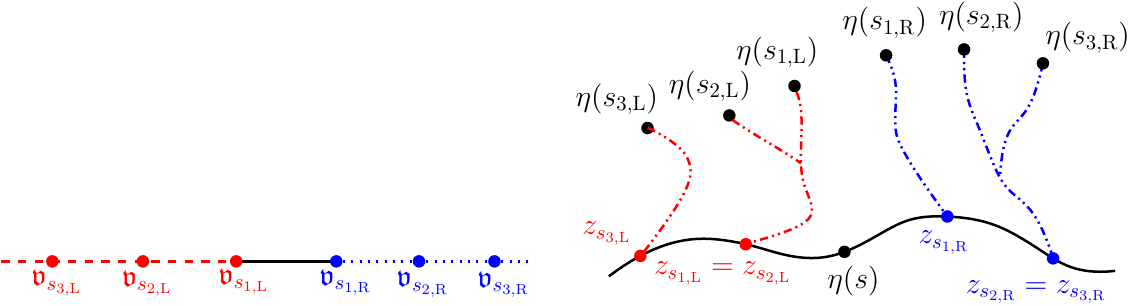}
    \caption{An illustration of Lemma~\ref{lem:construction-embed-cell-0}. \textbf{Left:} The map $T_s\subseteq T$ generated by the sewing procedure stopped at time $s$ and its boundary vertices. \textbf{Right:} The merging points $z_{s,k,\rmL}$ and $z_{s,j,\rmR}$ defined in Lemma~\ref{lem:construction-embed-cell-0}.% We will prove that  %and the dashed segments connecting $z_{s,k,\rmL}$ with $\eta(s_{k,\rmL})$ or $z_{s,k,\rmR}$ with $\eta(s_{k,\rmR})$ are parts of the left or right boundaries of $\eta((-\infty,s_{k,\rmL}])$ or $\eta((-\infty,s_{k,\rmR}])$. We are going to show that, each of these dashed segments are contained within the corresponding $\mr H^{\rm cell}_{s_{k,\rmL}}$ or $\mr H^{\rm cell}_{s_{k,\rmR}}$.
    }
    \label{fig:lem-def-cell-embed-001}
\end{figure}

Below we construct the points $z_{s_{k,\rmL}}$ in Theorem~\ref{thm:conv-0-cone-pack-Hcell} and Theorem~\ref{thm:conv-0-cone-unif-Hcell}.   For each $k\in\bbN$, let $\wh s_{k,\rmL}\in [s,s_{k,\rmL}]$ be the time when $L$ hits its infimum during $[s,s_{k,\rmL}]$, and let $z_{s_{k,\rmL}} = \eta(\wh s_{k,\rmL})$. From the boundary length process description,  $z_{s_{k,\rmL}}$ is the leftmost point on $\eta((-\infty,s])\cap\eta([s,s_{k,\rmL}])$, and  the boundary segment of $\eta((-\infty,s])$ from $\eta(s)$ to $z_{s_{k,\rmL}}$ has quantum length equal to $L_s - \inf_{u\in[s,s_{k,\rmL}]} L_u$. Define $z_{s_{k,\rmR}}$ analogously using $(\rmR,R)$ in place of $(\rmL,L)$.

\begin{lemma}\label{lem:construction-embed-cell-0}
  Let $s\in\cJ$.   For each $k\in\bbN$, let $z_{s_{k,\rmL}}$ be  defined as above.  Then we have
    \begin{enumerate}[(i)]
        \item  $z_{s_{1,\rmL}}, z_{s_{2,\rmL}},...$ are aligned from right to left on the left boundary of $\eta((-\infty,s])$ (some points may duplicate). 
        \item  The segment of the left boundary of $\eta((-\infty,s])$ between $z_{s_{k-1,\rmL}}$ and $z_{s_{k,\rmL}}$    is contained within the convex hull of $ H_{\wt s_{k,\rmL}}^{\mathrm{cell}}$, where $\wt s_{k,\rmL}\in\cJ_\rmL^-$ is any time such that $\wt s_{k,\rmL} = s_{k,\rmL}$ or $\{\wt s_{k,\rmL},s_{k,\rmL}\}\in\cE_\rmT$ and we write $z_{s_{0,\rmL}}=\eta(s)$. \label{cond:lem:construction-embed-cell-0-(ii)}
     \item There exists an absolute constant $C>0$ such that almost surely,  for $n$ large enough, the following is true. Let $\ell_{s,k,\rmL}$ be quantum length of the boundary segment of $\eta((-\infty,s])$ from $\eta(s)$ to $z_{s_{k,\rmL}}$.   If $s,s_{k,\rmL}\in [-2^n,2^n]$, then $|\ell_{s,k,\rmL}-k|\leq Cn$.  \label{cond:lem:construction-embed-cell-0-(iii)} %If $s\in\mr\cJ$, then the times $s_{k,\rmL}$  are also in $\mr\cJ^-$, the vertices $\frv_{s_{k,\rmL}}$ are also vertices of $\mr T$, and $z_{s,k,\rmL}\in \mr H_{s_{k,\rmL}}^{\mathrm{cell}}$.  
    \end{enumerate}
  The analogous statement is true with $\rmR$ instead of $\rmL$ and ``right" instead of ``left".  
\end{lemma}

\begin{proof}
 The  claim (i) is immediate from construction of the space-filling SLE curve in~\cite[Section 4.2]{ig4}.
 
 For the claim (ii), first assume $\inf_{u\in[s,s_{k,\rmL}]}L_u<L_s$ choose the largest integer $k'<k$ such that $\inf_{u\in[s,s_{k',\rmL}]}L_u>\inf_{u\in[s,s_{k,\rmL}]}L_u$ %. If $\inf_{u\in[s,s_{k,\rmL}]}L_u=L_s$, we set $k'=0$
 where recall that $s_{0,\rmL}=s$. Then $\inf_{u\in[s,s_{k'+1,\rmL}]}L_u=\inf_{u\in[s,s_{k,\rmL}]}L_u$,   $\inf_{u\in [s_{k',\rmL},s_{k,\rmL}]}L_u< \inf_{u\in[s,s_{k',\rmL}]}L_u$ and $\inf_{u\in [s_{k',\rmL},s_{k,\rmL}]}L_u = \inf_{u\in [s_{k',\rmL},s_{k'+1,\rmL}]}L_u$.
 For $t\in [s_{k',\rmL},s_{k'+1,\rmL}]$, $\cL_t\geq \cL_{s_{k',\rmL}}-1$, and therefore $$L_{t}\geq  L_{s_{k',\rmL}}-\sup_{u\in [s_{k',\rmL},t]}|(L_{s_{k',\rmL}} - \cL_{s_{k',\rmL}}) - (L_u-\cL_u)|-1> L_{s_{k',\rmL}} -\Delta_\rmL[s_{k',\rmL},t],$$
 which implies that $s_{k'+1,\rmL}< \tau_{s_{k',\rmL},\rmL}^{\rmC,1}$ and hence  $s_{k,\rmL}< \tau_{s_{k',\rmL},\rmL}^{\rmC,1}$. In particular, $L_{s_{k',\rmL}} -  \Delta_{s_{k',\rmL},\rmL}^{\rmC,1}< \inf_{u\in[s,s_{k,\rmL}]}L_u \leq L_{s_{k',\rmL}}$. Suppose $t> s_{k',\rmL}$, $t\in\cJ^-_\rmL$ and $w\in I_t$ such that $\eta(w)$ is in the interior of segment of $\partial\eta((-\infty,s])$ between $z_{k',\rmL}$ and $z_{k,\rmL}$. Then $\inf_{u\in[s,s_{k,\rmL}]}L_u\leq L_w\leq L_{s_{k',\rmL}}$, $w<s_{k,\rmL}$, which implies that $t\leq \tau_{s_{k',\rmL},\rmL}^{\rmC,1}$ and thus $t\in \cT_{s_{k',\rmL},\rmL}^{\rmC,1}$. On the other hand, since $\cL_{s_{k,\rmL}}<\cL_{s_{k',\rmL}}$, 
 \eqb\label{eq:pf-close-0900} L_{s_{k,\rmL}}\leq  L_{s_{k',\rmL}}+\sup_{u\in [s_{k',\rmL},s_{k,\rmL}]}|(L_{s_{k',\rmL}} - \cL_{s_{k',\rmL}}) - (L_u-\cL_u)| \leq  L_{s_{k',\rmL}}+ \Delta_\rmL[s_{k',\rmL},s_{k,\rmL}]\leq L_{s_{k',\rmL}}+\Delta_{s_{k',\rmL},\rmL}^{\rmC,1} \eqe
 where we used $s_{k'+1,\rmL}< \tau_{s_{k',\rmL},\rmL}^{\rmC,1}$ in the last step. Thus $s_{k,\rmL}\in \cT_{s_{k',\rmL},\rmL}^{\rmC,1}$ and we further have $t\in \cT_{s_{k,\rmL}}^{\rmC,1}\subseteq \cT_{\wt s_{k,\rmL}}^\cell$ and $\partial\eta(I_t)\subseteq H_{\wt s_{k,\rmL}}^\cell$.  This proves claim (ii). If $\inf_{u\in[s,s_{k,\rmL}]}L_u=L_s$, then clearly $s_{k,\rmL}<\tau_{s,\rmL}^{\rmC,1}$ and $L_{s_{k,\rmL}}\geq L_s$.    Replacing $s_{k',\rmL}$ in~\eqref{eq:pf-close-0900} with $s$ then gives $s_{k,\rmL}\in \cT_{s,\rmL}^{\rmC,1}$.  Let $s'' = \inf\{u\geq s:u\in\cJ^-_\rmL\}$. Then $s''\in \cT_{s,\rmL}^{\rmC,1}$  and therefore $s''\in \cT_{s_{k,\rmL}}^{\rmC,1}$ and $\partial\eta(I_{s''})\subseteq H_{\wt s_{k,\rmL}}^\cell$. Then claim (ii) for this case also follows since $z_{s_{k,\rmL}} = z_{s_{k-1,\rmL}} = \eta(s)$ and $s\in I_{s''}$.

 For the claim (iii), by claim~(\ref{cond:thm:KMT-0-cone-(ii)}) of Theorem~\ref{thm:KMT-0-cone}, \cite[Lemma 5.9, Lemma 5.10]{Tutte} and the Borel-Contelli lemma, there is some constant $C>1$ such that for every large enough $n$, $\sup_{u\in [-2^{2n},2^{2n}]}|L_u-\cL_u|\leq Cn$. We work in this setting. By our definition, we have $$\ell_{s,k,\rmL} = L_s - \inf_{u\in [s,s_{k,\rmL}]} L_u, \ \ \inf_{u\in [s,s_{k,\rmL}]} \cL_u  = \cL_s-k.$$
 Then $$|\ell_{s,k,\rmL}-k|\leq \sup_{u\in [s,s_{k,\rmL}]}|(L_s-\cL_s) - (L_u-\cL_u)|\leq 2Cn$$
 since $[s,s_{k,\rmL}]\subseteq [-2^n,2^n]$. This finishes the proof of the first part.

 % Now let $k'<k$ be the largest number such that $z_{s_{k',\rmL}}$ is different from $z_{s_{k,\rmL}}$. Let $u = s_{k',\rmL}$ and $t_{k,\rmL}^+$ be the last time when $\eta$ hits $z_{s_{k,\rmL}}$. Similar to the proof of claim (ii), where we work with the time $u$ instead of $s$,  we have $s_{k,\rmL}\leq t_{k,\rmL}^+<\tau_{u,\rmL}^{\rmC,1}$, $L_{s_{k,\rmL}}\geq L_{t_{k,\rmL}^+}$ and \eqb \label{eq:pf:construction-embed-cell-0-A0}0\geq L_{t_{k,\rmL}^+}-L_u\geq -\Delta_{u,\rmL}^{\rmC,1}.\eqe  By our assumption, $u\in [-2^n,2^n]$, $\tau_{u,\rmL}^{\rmC,1}\leq 2^{2n}$, and therefore by the definition~\eqref{eq:def-TtCd-0011},  $\Delta_{u,\rmL}^{\rmC,1}\leq 4Cn+2$, which further implies that
 % $$\cL_{s_{k,\rmL}} - \cL_u\geq L_{s_{k,\rmL}}-L_u - |\cL_{s_{k,\rmL}} -L_{s_{k,\rmL}} |-|\cL_u-L_u|\geq -\Delta_{u,\rmL}^{\rmC,1}-4Cn\geq -6Cn-2.$$
 % Since for each $j$, $\cL_{s_{j,\rmL}} = \cL_s-j$, we have $\cL_u=\cL_s-k'$ and $\cL_{s_{k,\rmL}}=\cL_s-k$, so the above implies
 % \begin{equation}\label{eq:pf:construction-embed-cell-0-A}
 %     0<k -k'\leq 6Cn+2.
 % \end{equation}
 % Now combining~\eqref{eq:pf:construction-embed-cell-0-A0} and~\eqref{eq:pf:construction-embed-cell-0-A}, we have
 % \begin{equation}
 %     |L_s - L_{t_{k,\rmL}^+} - k|\leq |L_{t_{k,\rmL}^+}-L_u|+|L_s-L_u-\cL_s+\cL_u|+|\cL_s-\cL_u-k|\leq \Delta_{u,\rmL}^{\rmC,1}+2Cn+ (6Cn+2)\leq 20Cn.
 % \end{equation}
 % The claim (iii) then follows since $L_s-L_{t_{k,\rmL}^+} $ is the quantum length of the segment of the left boundary of $\eta((-\infty,s))$ from $\eta(s)$ to $z_{s_{k,\rmL}}$.
 \end{proof}

\begin{proof}[Proof of Theorem~\ref{thm:conv-0-cone-pack-Hcell}]
Below we assume $r$ is sufficiently large.    By Proposition~\ref{prop:property-CH}, the conditions in Theorem~\ref{thm:CellSystemPack} holds for  $(\mr\cH^\cell,\mr \cH^\lge)$, and claim (\ref{cond:thm:conv-0-cone-pack-Hcell-(i)}) follows from Theorem~\ref{thm:CellSystemPack}. Note that for $t\in\cJ^-$, $\frtr(t)$ is in the convex hull of the union of the disk for $\frv_t$ and its neighboring disks, so~\eqref{eq:E-pack} is a consequence of claim (\ref{cond:thm:conv-0-cone-pack-Hcell-(i)}), ~\eqref{eq:E-curve-pack} and a triangle inequality. Now we prove~\eqref{eq:E-curve-pack}. Let $t\in\bbR$ such that $\eta(t)\in B(0;r)$. Let $t_0 = \sup\{u\leq t:u\in\mr\cJ\}$,   $t_\rmL^+$ be the left  endpoint of $\fre_{t_0}$, and $ t^+ = \inf\{u> t_0: u\in\mr\cJ^-\}$. Following Lemma~\ref{lem:time-set-sew}, we have $t_\rmL^+\in \mr\cT_{t_0,\rmL}^{\rmC,1}$ and by~\eqref{eq:def-TtCd-mr-0}, $t^+\in  \mr\cT_{t_0,\rmL}^{\rmC,1}$, and therefore by~\eqref{eq:def-TtCd-mr} we have $t^+\in \mr\cT_{t_\rmL^+}^{\rmC,1}$. Since $t\in\mr I_{t^+}$, by Corollary~\ref{cor:eta-boundary}, $\partial \eta(\mr I_{t^+})\cap B(0;r)\neq\emptyset$. By~\eqref{eq:cell-area}, we have $\partial\eta(\mr I_{t^+})\subseteq \mr H_{t_\rmL^+}^\cell$ and in particular  $\mr H_{t_\rmL^+}^\cell\cap B(0;r)\neq\emptyset$. Furthermore, by Lemma~\ref{lem:GMS2.9}, for large enough $r$, we have $\diam(\mr H_{t_\rmL^+}^\cell)<0.1\e r$ and thus $|\eta(t)-c(\mr H_{t_\rmL^+}^\cell)|\leq 0.1\e r$. Therefore by Theorem~\ref{thm:CellSystemPack},  {up to a reflection of $\cP$, for each $t$ such that $\eta(t)\in B(0;r)$,} $|\eta(t)-o(\frv_{t_\rmL^+})|<0.2\e r$ for $r$ large enough. By a translation, we may further assume that the disk for the cell $\mr H_0^\cell$ is centered at 0, which does not affect the pervious convergence. On the other hand, $\frtr(t)$ is in the convex hull of the union of the disk for $\frv_{t_\rmL^+}$ and its neighboring disks. This along with claim (\ref{cond:thm:conv-0-cone-pack-Hcell-(i)}) applied to $B(0;3r)$ and a triangle inequality gives~\eqref{eq:E-curve-pack}, which completes the proof for (\ref{cond:thm:conv-0-cone-pack-Hcell-(ii)})  {up to a possible reflection of $\cP$}.

    To prove claim (\ref{cond:thm:conv-0-cone-pack-Hcell-(iii)}), let $s,t\in\mr\cJ$ such that $\eta([s,t])\subseteq B(0;r)$.   By~\eqref{eq:E-curve-pack} and claim  (\ref{cond:thm:conv-0-cone-pack-Hcell-(i)}), all the disks for the vertices of $T_{s,t}$ are contained within $B(0;1.5 r)$ if we assume $r$ is large enough. Let $\frv=\frv_{x_{j,q}}$ be a boundary vertex of $T_{s,t}$ that is on an edge $\fre_u$ where $u\in [s,t]$. Let $\wt s\in\cJ^-_\rmL$ such that $\frv_{\wt s}$ is also on the edge $\fre_u$. Following the identical argument as in the previous paragraph, using Lemma~\ref{lem:time-set-sew}, ~\eqref{eq:def-TtCd}, for  $u' = \inf\{u_1\geq u: u_1\in\cJ^-\}$, we have $u',\wt s\in \cT_u^{\rmC,1,0}$ and therefore $u'\in \cT_{\wt s}^{\rmC,1}$.  Then by~\eqref{eq:cell-area-0},  $\partial\eta(I_{u'})\subseteq H_{\wt s}^\cell$, and by Corollary~\ref{cor:eta-boundary}, $H_{\wt s}^\cell\cap B(0;r)\neq\emptyset$.
    By Lemma~\ref{lem:GMS2.9} applied to $\cH^\cell$, for $r$ large enough $\diam (H_{\wt s}^\cell)<0.1\e r$. By claim~(\ref{cond:lem:construction-embed-cell-0-(ii)}) of Lemma~\ref{lem:construction-embed-cell-0}, the  corresponding segment of $\partial\eta((-\infty,x])$ between $z_{x_{j,q}}$ and $z_{x_{j-1,q}}$ is contained within the convex hull of $H_{\wt s}^\cell$ and has diameter less than $0.1\e r$. Then $|\eta(u) - z_{x_{j,q}}|\leq \diam(H_{\wt s}^\cell)<0.1\e r$. By~\eqref{eq:E-curve-pack}, $|\eta(u)-\frtr(u)|\leq 0.1\e r$ for $r$ large enough, while by claim (\ref{cond:thm:conv-0-cone-pack-Hcell-(i)}) applied to $B(0;3r)$, since $\frtr(u)$ is within the convex hull of the union of the disk for $\frv$ and its neighboring disks,   $|\frtr(u)-o(\frv)|\leq 0.1\e r$. Thus $|z_{x_{j,q}}-o(\frv)|\leq 0.4\e r$. This verifies the first part of claim (\ref{cond:thm:conv-0-cone-pack-Hcell-(iii)}) up to a global reflection of $\cP$.

    %while by~\eqref{eq:E-curve-pack} along with claim (i), we have $|\eta(u) - o(\frv)|<0.1\e r$. This together with a triangle inequality verify the first part of claim (iii).
 To prove the last sentence of claim  (\ref{cond:thm:conv-0-cone-pack-Hcell-(iii)}), again by claim~(\ref{cond:lem:construction-embed-cell-0-(ii)}) of Lemma~\ref{lem:construction-embed-cell-0}, the segment of the left boundary of $\eta((-\infty,s])$ from $\eta(s)$ to $z_{s_{1,\rmL}}$ is contained in the convex hull of $H_{{s_{1,\rmL}}}^\cell$ and therefore has diameter less than $\e r$ for all large enough $r$, which implies claim  (\ref{cond:thm:conv-0-cone-pack-Hcell-(iii)}) for $z_{s_{i_1,\rmL}}$ if $i_1=1$, since in this case $t<s_{1,\rmL}$. If $i_1\geq2$, $s_{i_1-1,\rmL}\in [s,t]$. Then $z_{s_{i_1-1,\rmL}}$, the leftmost point of $\eta([s,t])$ and $z_{s_{i_1,\rmL}}$  are on the left boundary of $\eta((-\infty,s]))$ and ordered from right to left, while by the same argument, using claim (\ref{cond:lem:construction-embed-cell-0-(ii)}) of Lemma~\ref{lem:construction-embed-cell-0},  this segment of $\partial\eta((-\infty,s]))$ between $z_{s_{i_1-1,\rmL}}$ and $z_{s_{i_1,\rmL}}$ is contained within the convex hull of some cell $H^\cell\in\cH^\cell$ that intersects $B(0;r)$ and therefore has diameter less than $\e r$ by Lemma~\ref{lem:GMS2.9}. This gives the desired closeness between $z_{s_{i_1,\rmL}}$ and the leftmost point of $\eta([s,t])$. Now we prove the analogous result for $z_{t_{j_1,\rmL}}$. Assume $\inf_{u\in [s,s_{i_1,\rmL}]} L_u< L_s$ and let $k$ be the largest integer such that $\inf_{u\in [s,s_{k,\rmL}]} L_u> \inf_{u\in [s,s_{i_1,\rmL}]}L_u$.    Then the same argument as in the proof of Lemma~\ref{lem:construction-embed-cell-0} implies that $L_u\geq L_{s_{k,\rmL}} - \Delta_{s_{k,\rmL},\rmL}^{\rmC,1}$ for $u\in [s_{k,\rmL},s_{i_1,\rmL}]$, $s_{i_1,\rmL}<\tau_{s_{k,\rmL},\rmL}^{\rmC,1}$ and $s_{i_1,\rmL}\in \cT_{s_{k,\rmL},\rmL}^{\rmC,1}$. Furthermore, $t\in [s_{k,\rmL},s_{i_1,\rmL}]$ and $\inf_{u\in [s_{k,\rmL},s_{i_1,\rmL}]} L_u<\inf_{u\in [s,s_{k,\rmL}]}L_u$ and therefore $-\Delta_{s_{k,\rmL},\rmL}^{\rmC,1}<\inf_{u\in [s,t]}L_u-L_{s_{k,\rmL}}\leq 0$. On the other hand, by the sewing procedure, $\cL_{t_{j_1,\rmL}}\leq \cL_{s_{k,\rmL}}$ and $t_{j_1,\rmL}\in [s_{k,\rmL},s_{i_1,\rmL}]$, and similar to~\eqref{eq:pf-close-0900}, we have $L_{t_{j_1,\rmL}}\leq L_{s_{k,\rmL}}+\Delta_{s_{k,\rmL},\rmL}^{\rmC,1}$. Observe that $\eta([s_{k,\rmL},s_{i_1,\rmL}])$ covers  the segment of $\partial\eta((-\infty,t])$ between $z_{s,t,\rmL}^-$ and $z_{t_{j_1,\rmL}}$. Suppose $u\in \cJ^-_\rmL$ such that for some $w\in I_u\cap[s_{k,\rmL},s_{i_1,\rmL}]$,   $\eta(w)$ is on this segment.  Then  $L_w$ is between $\inf_{u\in [s,t]}L_u$ and $\inf_{u\in [t,t_{j_1,\rmL}]}L_u$, where both are values are in $[L_{s_{k,\rmL}}-\Delta_{s_{k,\rmL},\rmL}^{\rmC,1},L_{s_{k,\rmL}}+\Delta_{s_{k,\rmL},\rmL}^{\rmC,1}]$. Furthermore, $u\leq s_{i_1,\rmL} \leq \tau_{s_{k,\rmL},\rmL}^{\rmC,1}$. Thus $u\in \cT_{s_{k,\rmL},\rmL}^{\rmC,1}$, which combined with $s_{i_1,\rmL}\in \cT_{s_{k,\rmL},\rmL}^{\rmC,1}$ gives $u\in \cT_{s_{i_1,\rmL}}^{\rmC,1}$. By~\eqref{eq:cell-area-0}, $\partial\eta(I_u)\subseteq H_{s_{i_1,\rmL}}^\cell$, and since $u$ is arbitrary, this further implies that the segment of $\partial\eta((-\infty,t])$ between $z_{s,t,\rmL}^-$ and $z_{t_{j_1,\rmL}}$ is within the convex hull of  $H_{s_{i_1,\rmL}}^\cell$ and hence has diameter less than $\e r$ if we assume $r$ is large enough. The case where  $\inf_{u\in [s,s_{i_1,\rmL}]} L_u= L_s$ can be handled analogously, where the same proof above works with $s$ in place of $s_{k,\rmL}$.

 The same applies for $z_{s_{i_2,\rmR}}$, $z_{t_{j_2,\rmR}}$ and the rightmost point of $\eta([s,t])$. This verifies the claim  (\ref{cond:thm:conv-0-cone-pack-Hcell-(iii)}) {up to a possible reflection of $\cP$}. {Now suppose with $\bbP_\rmE$-probability, the claim (\ref{cond:thm:conv-0-cone-pack-Hcell-(ii)}) and (\ref{cond:thm:conv-0-cone-pack-Hcell-(iii)})   hold for the reflection of $\cP$. Assume $p>0$. By the scale invariance in Lemma~\ref{lem:Rb},  we may choose some constant $c>0$ small enough, then with  $\bbP_\rmE$-probability $1-p/2$, for some large $r>0$, $B(0;cr)\subseteq \eta([s_n,t_n])\subseteq B(0;c^{-1}r)$ where $s_n = \inf(\mr\cJ\cap [-2^n,\infty))$ and  $t_n = \inf(\mr\cJ\cap [2^n,\infty))$. We fix $r$ large enough and work with the intersection between two events, which has positive probability. Let $\gamma_{s_n,t_n}$ be an $o(r)$-perturbation of the curve formed by the points $z_{x_{i,q}}$ and segments of $\partial((-\infty,s_n]))\cup\partial((-\infty,t_n]))$, such that it is a simple curve and forms an  embedding of $\partial T_{s_n,t_n}$ with the same orientation as on the planar map $T$. Let $\gamma_{s_n,t_n}'$ be the embedded curve for $\partial T_{s_n,t_n}$  formed by joining the centers of adjacent disks in $\cP$ for vertices on $\partial T_{s_n,t_n}$ with line segments, which also preserves the orientation.   The curves  $\gamma_{s_n,t_n}$ and  $\wt\gamma'_{s_n,t_n}$ have distance at most $4\e r$ when viewed as curve modulo time parameterizations, where $\wt\gamma'_{s_n,t_n}$ is   the reflection of $\gamma'_{s_n,t_n}$ , and using $B(0;cr)\subseteq \eta([s_n,t_n])$, the two curves must have the same winding number around 0.  Since we have centered the disk for $\mr H_0^\cell$ at 0 in $\cP$, $\gamma_{s_n,t_n}$ and  $\gamma'_{s_n,t_n}$ must have the same winding number around 0. This is a contradiction and thus $p=0$. This finishes the proof for claims (\ref{cond:thm:conv-0-cone-pack-Hcell-(i)})-(\ref{cond:thm:conv-0-cone-pack-Hcell-(iii)}).
}

    To prove the claim  (\ref{cond:thm:conv-0-cone-pack-Hcell-(iv)}), by~\eqref{eq:SLE-contain-ball-1}, ~\eqref{eq:SLE-contained-in-ball-1} and the Borel-Contelli lemma, a.s.\ for each $\delta>0$, for large enough $n$, $B(0;2^{(\frac{1}{\gamma Q}-\delta)n})\subseteq \eta([-2^n,2^n])\subseteq B(0;2^{(\frac{1}{\gamma Q}+\delta)n})$. One can further infer that for some $\beta>0$, a.s.\ for large enough $r>0$, $B(0;4r)\subseteq \eta([-r^\beta,r^\beta])\subseteq B(0;r^{1+\delta})$ and therefore if $\eta([s,t])\subseteq B(0;r)$ then $s,t\in [-2^m,2^m]$ where $m\leq 10\beta\log r$. Since $s_{i_1-1,\rmL}\leq t$ and $t_{j,\rmL}\leq s_{i_1+1,\rmL}$ for $j=1,...,j_1$,  {by \cite[Lemma 5.10]{Tutte}, part~(\ref{cond:thm:KMT-0-cone-(ii)}) of Theorem~\ref{thm:KMT-0-cone} and Borel-Contelli lemma, for some constant $C_1>0$,} almost surely for large enough $m$, $\sup_{u\in [-2^{2m},2^{2m}]}|L_u-\cL_u|\leq Cm $, and for every $u\in\cJ\cap[-2^m,2^m]$, $\inf\{u'\geq u: \cL_{u'}<\cL_u-3\}<2^{2m}$,  $\inf\{u'\geq u: \cR_{u'}<\cR_u-3\}<2^{2m}$. Then for $r$ large enough, all the times $s,t$ and $x_{j,q}$ are within $[-2^{2m},2^{2m}]$ with $m\leq 10\beta\log r$. We work with this setting.  %  The claim (\ref{cond:thm:conv-0-cone-pack-Hcell-(iv)}) for the $x=s$ case then follows from the claim~(\ref{cond:lem:construction-embed-cell-0-(iii)}) of Lemma~\ref{lem:construction-embed-cell-0}. For the  $x=t$ case, by the claim (\ref{cond:thm:conv-0-cone-pack-Hcell-(iii)}), the points $z_{t_{j,q}}$ are in $B(0;1.5r)$, and since $z_{t_{j,q}}$ is not in the interior of $\eta((-\infty, t_{j,q}]))$, we must have $t_{j,q}\leq t'$, and therefore 
    The first part of claim (\ref{cond:thm:conv-0-cone-pack-Hcell-(iv)}) then follows from the claim~(\ref{cond:lem:construction-embed-cell-0-(iii)}) of Lemma~\ref{lem:construction-embed-cell-0}. For the last claim, suppose $q=\rmL$ and $z_{x_{k,\rmL}}$ is on $\partial\eta((-\infty,s])\backslash\partial\eta([s,t])$. Then the segment of $\partial\eta((-\infty,s])$ between $z_{x_{k,\rmL}}$ and $z_{s,t,\rmL}^-$ is contained with in the segment of $\partial\eta((-\infty,s])$ between $z_{s_{i_1,\rmL}}$ and $z_{s,t,\rmL}^-$ so we may assume $x=s,k=i_1$. Now the quantum length of the segment of $\partial\eta((-\infty,s])$ between $\eta(s)$ and $z_{s,t,\rmL}^-$ is equal to $L_s-\inf_{u\in[s,t]}L_u$, whereas $i_1 = 1+\cL_s-\inf_{u\in [s,t]}\cL_u$. Then 
    $$|L_s-\inf_{u\in[s,t]}L_u - (1+\cL_s-\inf_{u\in [s,t]}\cL_u)|\leq 1+\sup_{u\in[s,t]}|L_s-L_u-(\cL_s-\cL_u)|\leq 1+6Cm.$$
    The last claim then follows from the first part of (\ref{cond:thm:conv-0-cone-pack-Hcell-(iv)}) as the quantum length of the  segment of $\partial\eta((-\infty,s])$ between $\eta(s)$ and $z_{s_{i_1,\rmL}}$ is in $[i_1-C\log r,i_1+C\log r]$.
    \end{proof}

\begin{proof}[Proof of Theorem~\ref{thm:conv-0-cone-unif-Hcell}]
    By Proposition~\ref{prop:property-CH}, the conditions in Theorem~\ref{thm:invariance-uniformization} holds for  $(\cH^\cell, \cH^\lge)$. By identical arguments as in the proof of Theorem~\ref{thm:conv-0-cone-pack-Hcell}, a.s.\ for large enough $r$, for $t\in\cJ^-_\rmL$ such that $\eta(t)\in B(0;r)$, we have $H_t^\cell\cap B(0;r)\neq\emptyset$. Similarly, for   $t\in\cJ^-_\rmR$ such that $\eta(t)\in B(0;r)$ and $\{s,t\}\in\cE_\rmT$, $s\in \cJ^-_\rmL$, by~\eqref{eq:def-cT-cell} along with Assertion~\ref{cond:prop:time-set-A-0-3} of Proposition~\ref{prop:time-set-A-0}, $t\in \cT_s^\cell$, and by~\eqref{eq:cell-area-0}, $\partial\eta(I_t)\subseteq H_s^\cell$ and $H_s^\cell\cap B(0;r)\neq\emptyset$. The claims~\eqref{eq:E-unif} and~\eqref{eq:E-edge} as well as the final claim in the first part then follow from Theorem~\ref{thm:invariance-uniformization}. The proof for~\eqref{eq:E-curve-unif} is the same as that for~\eqref{eq:E-curve-pack}, where  $t\in \bbR$ and $\eta(t)\in B(0;r)$, for $t_0 = \sup\{u\leq t: u\in \cJ\}$ and $t_\rmL^+$ being the left endpoint of $\fre_{t_0}$, we have $|\eta(t)-\varphi(\frv_{t_\rmL^+})|<0.2\e r$ up to a reflection and $\frtr(t)$ is within the union of all embedded faces adjacent to $\frv_{t_\rmL^+}$ under $\varphi$.  The proof of the second assertion is identical to the proof of claim~(\ref{cond:thm:conv-0-cone-pack-Hcell-(iii)}) of Theorem~\ref{thm:conv-0-cone-pack-Hcell}, where we use the same orientation argument to exclude the case where the reflection of the embedding via $\varphi$ is close to $\cH^\cell$.
\end{proof}

\section{Convergence of circle packing for triangulations from tree-weighted maps}\label{subsec:pack-pf}

In this section, we prove the convergence of circle packing for triangulations from tree-weighted maps, namely Theorems~\ref{thm:pack-disk}-\ref{thm:pack-cone}. %We will combine the methods from {mated-CRT} and {Section 6 of the Tutte draft.}  
We start by gathering some results on the convergence of circle packing maps to conformal maps in~\cite{MatedCRT} in Section~\ref{subsec:He-Schramm}. In Section~\ref{subsec:pack-iid}, we work with the triangulations encoded by i.i.d.\ walks using the transform in Theorem~\ref{thm:He-Schramm-extension-1} and ~\cite[Section 6.1]{Tutte}. In Section~\ref{subsec:pack-disk-pf}-\ref{subsec:pack-cone-pf}, we use local absolute continuity as well as Theorem~\ref{thm:He-Schramm-extension-1}-\ref{thm:He-Schramm-disk} to transfer to other cases and prove Theorems~\ref{thm:pack-disk}-\ref{thm:pack-cone}.

% Recall the operator $\wh \cZ=\Phi_{\rm simple}(\cZ)$ defined for continuous-time nearest-neighbor walks $\cZ:I\to \bbZ^2$ where $I\subseteq \bbR$ is an interval, and $\mr\cJ$ is the time when $\mr\cZ$ makes a jump. The map $\Phi_{\rm simple}$ is not local, in the sense that $\mr\cZ|_{[a,b]}$ is not determined by $\cZ|_{[a,b]}$. To circumvent this,  we will frequently use the following simple observation.

% \begin{lemma}\label{lem:A2}
%      Let $a_0<a<b<b_0$, $I_0=[a_0,b_0]$, and $I\subseteq\bbR$ be an interval such that $I_0\subseteq I$. Suppose $\cZ:I\to\bbZ^2$ and $\cZ_0:I_0\to\bbZ^2$ are continuous-time nearest-neighbor walks on $\bbZ^2$ such that $\cZ|_{I_0}=\cZ_0$. Let $\mr \cZ=\Phi_{\rm simple}(\cZ)$ and $ \mr \cZ_0=\Phi_{\rm simple}(\cZ_0)$.  Then if $\mr\cJ\cap (a_0,a)\neq\emptyset$ and $\mr\cJ\cap (b,b_0)\neq\emptyset$, then $\mr\cZ_0|_{[a,b]} = \mr\cZ|_{[a,b]}$.
% \end{lemma}

In Section~\ref{subsubsec:Mullin-finite}, we defined the map $\Phi_\rmT(\cZ) = (T,\mathfrak{tr})$ for finite discrete nearest-neighbor walks on $\bbZ^2$, and we will also extend it to the setting  for  continuous-time nearest-neighbor walks on $\cZ:I\to\bbZ^2$ during a finite interval in an identical way   as discussed after~\eqref{eq:def-I_t-simple}. As in the infinite volume case, we parameterize $\mathfrak{tr}:I\to \cF T$ such that $\mathfrak{tr}$ moves to the next faces only when $\cZ$ makes a jump. When $T$ is embedded on $\bbC$, we will also view $\mathfrak{tr}:I\to\bbC$ as a path where $\mathfrak{tr}(t)$ is an arbitrary point on the corresponding embedded face. We will continue to view circle packings for $\mr T=\Phi_{\rm simple}(T)$ as circle packing for the triangulation $T$ with potential multiple edges. For  triangulations $T,T'$, where $T$ is a submap of $T'$, we will often compare circle packing of $T$ with circle packing of $T'$ restricted to vertices in $T$; when doing this, we will always assume that $\Phi_{\rm simple}(T)$ is a submap of $\Phi_{\rm simple}(T')$.

\subsection{Convergence of circle packing maps to conformal maps}\label{subsec:He-Schramm}

In this subsection we gather some result on the convergence of circle packing maps to conformal maps.
An \emph{interstice} of a circle packing $\cP$ is a connected component of the complement of $\cup_{D\in \cP} D$. The \emph{support}, $\supp(\cP)$, of $\cP$ is defined to be the union of the disks in $\cP$ and all bounded interstices of $\cP$. 
Let $\{T^n\}_{n\geq1}$ be simple  disk triangulations, and $\cP^n, \wt\cP^n$ be circle packings of $T^n$ in domains $D$ and $\wt D$, respectively. This induces a natural  {(non-unique)} homeomorphism $f^n$ between the supports of $\cP^n$ and $\wt\cP^n$, sending each disk in $\cP^n$ to its corresponding disk in $\wt\cP^n$. It was conjectured by Thurston that if the maximal diameter of disks in $\cP^n$ and $\wt\cP^n$ converge to 0, then the circle packing map $\{f^n\}$ approximates a conformal map $f$ from $D$  to $\wt D$ as $n\to\infty$. This was proved by Rodin and Sullivan for the hexagonal lattice case %\nina{*does this paper consider any triangulation with uniformly bounded degrees or only the hexagonal lattice? Or do you mean that the proof carries through to the uniform bounded degree case?*} {[*It was commented in ~\cite{HeSchramm} that the proof by~\cite{RodinSullivan87} carries to uniform bounded degree case *]}, 
and later extended to general simple disk triangulations by He and Schramm~\cite{HeSchramm}. For a circle packing $\cP$, we define $\partial\cP$ for the union of the disks that are on the boundary. We will use an extension of~\cite[Theorem 1.1]{HeSchramm}, where we allow some tiny portion of disks in $\cP^n$ and $\wt\cP^n$ to  lie outside $D$ and $\wt D$. We will use the  following definition on the uniform convergence of functions $\{f^n\}$ defined on different domains.
\begin{definition}\label{def:unif-convergence}
    Let  $D,D^n\subseteq\bbC$ be Jordan domains such that the Hausdorff distance between $\ol {D^n}$ and $\ol D$ tends to 0 as $n\to\infty$.   Let $f^n:\ol {D^n} {\to\bbC}$ and $f:\ol D\to\bbC$.   We say $f^n$ converges to $f$ uniformly, if for each $\e>0$, there exists $\delta>0$ and $N>0$, such that for each $n>N$, $z'\in \ol {D^n}$ and $z\in \ol D$ with $|z'-z|<\delta$, we have $|f^n(z')-f(z)|<\e$.
\end{definition}

We also recall the topology on curves viewed modulo time parameterization. %as in~\cite{GMSinvariance}.
Let $\beta_1:[0,T_{\beta_1}]\to\bbC$ and $\beta_2:[0,T_{\beta_2}]\to\bbC$ be continuous curves defined on possibly different time intervals. We define 
\begin{equation}\label{eq:curve-metric-0}
    \mathbb{d}^{\rm CMP}(\beta_1,\beta_2):=\inf_\phi \sup_{t\in[0,T_{\beta_1}]}|\beta_1(t)-\beta_2(\phi(t))|,
\end{equation}
where the infimum is over all increasing homeomorphisms $\phi:[0,T_{\beta_1}]\to [0,T_{\beta_2}]$ and ``CMP'' stands for ``curves modulo parameterization". To compare circle packings of maps and its submap, we also introduce the \emph{circle embedding} $\cP = \cup_{v\in\cV  T} D_v$  of a disk triangulation $T$, where $D_v$ are   closed disks with disjoint interiors, such that if $\{v_1,v_2\}\in\cE T$ then $D_{v_1}$ and $D_{v_2}$ are tangent, while if $D_{v_1}$ and $D_{v_2}$ are tangent and $v_1,v_2$ are not both boundary vertices then $\{v_1,v_2\}\in\cE T$. In other words, a circle embedding is a circle packing except that we also allow additional tangencies of disks for boundary vertices. We further connect centers of  disks $D_{v_1}$ and $D_{v_2}$ where $\{v_1,v_2\}\in\cE T$ by line segments. This gives an embedding of $T$, and we require that this embedding preserves the orientation of $T$.   Note that if $T\subseteq T_1$ are triangulations, and $T$ is a disk triangulation, then the restriction of a circle packing for $T_1$ to disks for vertices in $T$ gives a circle embedding of $T$ but not necessarily a circle packing of $T$. For a circle embedding of a simple disk triangulation $T$, we also define its \emph{carrier} to be $\cup_{f\in\cF T}\Delta_f$, where the union is over all triangular faces of $T$ and $\Delta_f$ is the triangle formed by the centers of the disks for vertices adjacent to $f$.

The following is from~\cite[Theorem 1.12, Theorem 1.14]{MatedCRT}.

\begin{theorem}\label{thm:He-Schramm-extension-1}
   Let $D,\wt D\subsetneq \bbC$ be    {bounded Jordan} domains. Let $\{T^n\}_{n\geq1}$ be simple disk triangulations, $\cP^n, \wt\cP^n$ be circle embeddings of $T^n$, and $f^n$ be homeomorphisms from the carrier of $\cP^n$ to that of $\wt\cP^n$, sending  each $D_v\cap \mathrm{carrier}(\cP^n)$ where  $v\in\cV T^n$ to $\wt D_v\cap \mathrm{carrier}(\wt \cP^n)$. Consider the following conditions:
   \begin{enumerate}[(I)]
    \item The maximal spherical diameter of the disks in $\cP^n$ converges to 0 as $n\to\infty$; \label{cond:thm:He-Schramm-extension-1-(i)}
       \item The distance between  the union of boundary disks of $\cP^n$ and $D$  converges to 0 as $n\to\infty$, in the sense that
       $$\lim_{n\to\infty}\max_{{D^n\in\partial\cP^n}}\inf_{z\in D^n}d(z,\partial D) = 0,$$
       %\nina{*Could you write ``$D^n\in\partial\cP^n$'' instead in the subscript of max and define $\partial\cP^n$ as a subset of $\cP^n$? That'll make this easier to read*}
       and the same holds for $\wt\cP^n$ and $\wt D$ instead of $\cP^n$ and $D$. Moreover, the maximal spherical diameter of the disks in $\wt\cP^n$ intersecting $\bbC\backslash \ol{\wt D}$ converges to 0 as $n\to\infty$; \label{cond:thm:He-Schramm-extension-1-(ii)} %\nina{*where is this used?*} {[*This is used to show that the maximal diameter of the disks in $\wt\cP^n$ converges to 0*]} %The Hausdorff distance between the union of boundary disks of $\cP^n$ (resp.\ $\wt\cP^n$) and $D$ (resp.\ $\wt D$) converges to 0 as $n\to\infty$; \tcr{[*To do: double check if need to use curve modulo parameterization distance here instead of Hausdorff distance*]}
       \item For some $p_0\in D$, for some $N_0>0$, the sequence $\{f^n(p_0):n\geq N_0\}$ is compact in $\wt D$; \label{cond:thm:He-Schramm-extension-1-(iii)}
       \item Let $\gamma_n$ (resp.\ $\wt \gamma_n$) be the loop formed by connecting the centers of adjacent boundary disks of $\cP^n$ (resp.\ $\wt\cP^n$) using straight line segments, i.e., $\gamma_n$ (resp.\ $\wt \gamma_n$) is the boundary of the carrier of $\cP^n$ (resp.\ $\wt\cP^n$). Then the distance~\eqref{eq:curve-metric-0} between $\gamma_n$ and $\partial D$ (resp.\ $\wt\gamma_n$ and $\wt\partial D$) tends to 0 as $n\to\infty$. \label{cond:thm:He-Schramm-extension-1-(iv)}
   \end{enumerate}
 Suppose Conditions (I)-(III) holds. Then Conditions (I) and (III) hold for $(\wt\cP^n,\wt D)$ in place of $(\cP^n,D)$, and there is   a subsequence of  $\{f^n\}$ converging to a conformal map $f:D\to\wt D$ uniformly on compact subsets of $D$. If we further assume Condition (IV) holds,    then the above convergence is  uniform  {in the sense of Definition~\ref{def:unif-convergence}.} %\nina{*should define what this means*}.
\end{theorem}

In Theorem~\ref{thm:He-Schramm-extension-1}, we required that $D,\wt D\neq\bbC$. The following result from~\cite[Theorem 1.16]{MatedCRT} is a version of Theorem~\ref{thm:He-Schramm-extension-1} when $D=\wt D=\bbC$.

\begin{theorem}\label{thm:He-Schramm-sphere}
     Let $\{T^n\}_{n\geq1}$ be  disk triangulations, $\cP^n, \wt\cP^n$ be circle embeddings of $T^n$ in $\bbC$, and $f^n$ be homeomorphisms from the carrier of $\cP^n$ to that of $\wt\cP^n$, sending  each $D_v\cap \mathrm{carrier}(\cP^n)$ where  $v\in\cV T^n$ to $\wt D_v\cap \mathrm{carrier}(\wt \cP^n)$. Assume the following hold:
     \begin{enumerate}[(I)]
         \item The maximal spherical diameter of disks as well as the bounded interstices  in $\cP^n$ tends to 0 as $n\to\infty$;\label{cond:thm:He-Schramm-sphere-(i)}
         \item For every $R>0$, for large enough $n$, $B(0;R)\subseteq \mathrm{carrier}(\cP^n)$; \label{cond:thm:He-Schramm-sphere-(ii)}
         \item For each $n$, $f^n(0)=0$ and $f^n(1)=1$. \label{cond:thm:He-Schramm-sphere-(iii)} 
     \end{enumerate}
     Then (I) and (II) above hold for $\wt\cP^n$ in place of $\cP^n$, and $f^n$ converges uniformly to the identity in the spherical metric, in the sense that for any $\e>0$, there exists some $N>0$, such that for  every $n>N$ and $z\in\mathrm{carrier}(\cP^n)$, $d_{\sph}(f^n(z),z)<\e$.
\end{theorem}
The next result ~\cite[Theorem 1.17]{MatedCRT} shows that in some sense, circle packings of disk triangulations will not change too much if we add a few vertices. See Figure~\ref{fig:pack-disk-1} for an illustration for the setup.
\begin{figure}
    \centering
    \includegraphics[scale=0.78]{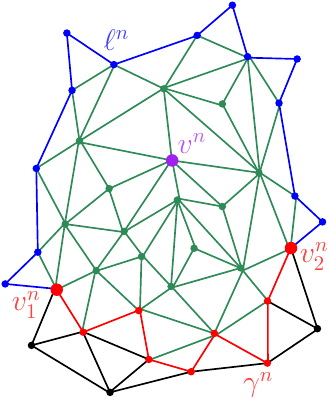}
    \caption{An illustration of the triangulation $\wt T^n$ in Theorem~\ref{thm:He-Schramm-disk}.  The map $T^n$ is formed by the red path $\gamma^n$, the blue path $\ell^n$ and the green vertices and edges between $\gamma^n$ and $\ell^n$.}
    \label{fig:pack-disk-1}
\end{figure}

\begin{theorem}\label{thm:He-Schramm-disk}
    Let $\wt T^n$ be disk triangulations, $v_1^n$ and $v_2^n$ be boundary vertices of $\wt T^n$, and $\ell^n$ be the clockwise boundary arc of $\partial\wt T^n$ from $v_1^n$ to $v_2^n$. Let $\gamma^n$ be a simple path on $\wt T^n$ that intersects $\ell^n$ only at $v_1^n$ and $v_2^n$. Let $ T^n\subseteq \wt T^n$ be the disk triangulation formed by $\ell^n$, $\gamma^n$ and vertices between $\ell^n$ and $\gamma^n$.  Let $v^n$ be an interior vertex of $ T^n$, and $\wt\cP^n$ and $\cP^n$ be circle packing of $\wt T^n$ and $T^n$ in $\bbD$ centered at $v^n$. Further assume  the following:
    \begin{enumerate}[(i)]
        \item The maximal diameter of disks in $\cP^n$ tends to 0 as $n\to\infty$;
        \item The diameter of the union of the disks in $\cP^n$ for vertices in $\gamma^n$ tends to 0 as $n\to\infty$.
    \end{enumerate}
    Let $\cP_1^n$ be the restriction of $\wt\cP^n$ to the disks for  vertices in $ T^n$, and $f^n$ be a homeomorphism between the  {carriers} of $\cP^n$ and $\cP_1^n$   sending  each $D_v\cap \mathrm{carrier}(\cP^n)$ where  $v\in\cV T^n$ to $\wt D_v\cap \mathrm{carrier}( \cP_1^n)$.  Then the maximal diameter of disks in $\wt\cP^n$ as well as the diameter of the union of disks in $\wt\cP^n$ for vertices in $\gamma^n\cup(\wt T^n\backslash T^n)$ tends to 0 as $n\to\infty$, and there exists $\theta_n\in [0,2\pi)$ such that $e^{i\theta_n} f^n$ converges uniformly to the identity in the sense of Definition~\ref{def:unif-convergence}.
\end{theorem}

We further gather some results on conformal maps, which can be viewed as the analog of Theorem~\ref{thm:He-Schramm-extension-1}, Theorem~\ref{thm:He-Schramm-disk} and Theorem~\ref{thm:He-Schramm-sphere} in the continuum setting. These are \cite[Propositions 3.1, 3.4, 3.5]{MatedCRT} and will be used in Section~\ref{sec:unif-pf} for the Riemann uniformization setting.

\begin{proposition}\label{prop:boundary-perturb}
    Let $(D^n)_{n\geq1}$,  $(\wt D^n)_{n\geq1}$, $D$ and $\wt D$ be bounded Jordan domains containing 0, and $f^n:D^n\to\wt D^n$ be conformal maps sending 0 to 0 such that $(f^n)'(0)>0$. Suppose $\partial D^n$ converges to $\partial D$ and $\partial \wt D^n$ converges to $\partial \wt D$ in the metric~\eqref{eq:curve-metric-0}. Then $\{f^n\}_{n\geq1}$ is equicontinuous, and $f^n$ converges uniformly to a conformal map $f:D\to\wt D$ with $f'(0)>0$ in the sense of Definition~\ref{def:unif-convergence}.
\end{proposition}

\begin{proposition}\label{prop:unif-disk}
    Let $\theta_n>0$ be a sequence tending to 0, $D^n\subseteq \bbD$ be a Jordan domain,  $\varphi^n:\ol \bbD\to \ol {D^n}$ be a conformal map fixing 0 such that $(\varphi^n)'(0)>0$, and $\varphi^n(e^{i\theta})\in\partial\bbD$ for $\theta\in [\theta_n,2\pi]$. Then the diameter of $\bbD\backslash D^n$ tends to 0, and $\varphi^n$ converges uniformly on $\ol \bbD$ to the identity function. 
\end{proposition}

\begin{proposition}\label{prop:unif-sph}
    Let $D^n,\wt D^n$ be simply connected domains containing 0 and 1, $\varphi^n:D^n\to\wt D^n$ be conformal maps fixing 0 and 1. Further assume that for any $R>0$,  $B(0;R)\subseteq D^n$ for large enough $n$. Then for each $R>0$, for large enough $n$, $B(0;R)\subseteq \wt D^n$. Moreover,  $\varphi^n$ converges uniformly to the identity function, in the sense that for any $\e>0$, there exists $N>0$ such that for every $n>N$ and $z\in D^n$, $d_\sph(\varphi^n(z),z)<\e$.
\end{proposition}

\subsection{The i.i.d.\ walk case}\label{subsec:pack-iid}

Let $I = [a_I,b_I]\subseteq\bbR$ be a bounded and possibly random interval of positive length. For each $n\in\bbN$ let
\begin{itemize}
    \item $\cZ^n:I\to\bbR^2$ be a continuous-time nearest-neighbor walk on $\bbZ^2$, $\mr\cZ^n = \Phi_{\rm simple}(\cZ^n)$, $\mr \cJ^n$ be the set of times when $\mr \cZ^n$ makes a jump, and
    \item $Z^n:I\to\bbR^2$ be a continuous process with values in $\bbR^2$ such that conditioned on $I$, $Z^n$ is absolutely continuous with respect to the boundary length process restricted to $I$ under one of the measures $\bbP_\rmQ$ or $\bbP_\rmE$. % or is absolutely continuous with respect to an uncorrelated Brownian cone excursion in the first quadrant from $(1,0)$ to $(0,0)$.
\end{itemize}

\begin{definition}\label{def:CPtoMOT}
    We say that \emph{circle packing converges to mating of trees for $(\cZ^n,Z^n)_{n\in\bbN}$} if the following is true.    {Let $a,b\in (0,1)$ with $a+b<1$, 
    let $I_{a,b}\subseteq  I$ be the closed interval such that the left (resp.\ right) endpoint of $I_{a,b}$ has distance $a|I|$ (resp.\ $b|I|$) from the left (resp.\ right) endpoint of $I$. Let $I'=[t_{a,b}^-,t_{a,b}^+]$, where $t_{a,b}^-$ (resp.\ $t_{a,b}^+$) is the smallest (resp.\ largest) time in $\mr\cJ^{n}\cap I_{a,b}$. Let $t_{\rm mid}$ be the midpoint of $I_{a,b}$, and $t_{\rm mid}':=\inf\{u\geq t_{\rm mid}:u\in\mr\cJ^{n,-}\}$.}  Consider the SLE-decorated LQG surface encoded by $Z^n|_{I}$ and restricted to the time $I'$ and embed it as $(\bbD,h_{n,I'},\eta_{n,I'})$ such that $\eta_{n,I'}(t_{\rm mid})=0$ and $\eta_n$ starts from 1 (See Section~\ref{subsubsec:MatedCRT}). Consider the disk triangulation $ T_{n,I'}$ and the path ${\mathfrak{tr}}_{n,I'}$ on $\cF T_{n,I'}$ from  $\Phi_\rmT(\cZ^n|_{I'})$.   
    Consider the univalent circle packing $\cP_{n,I'}$ for $ T_{n,I'}$ in $\bbD$ centered at the vertex $\frv_{t_{\rm mid}'}$ such that the disk for the left endpoint of the root edge (i.e., the edge corresponding to the smallest time in $I'$ in the sewing procedure) is tangent to $\partial\bbD$ at 1.
     Then for any deterministic $a,b$, we have:
    \begin{enumerate} 
    \item with probability tending to 1, the map $\Phi_\rmT(\cZ^n|_{I'})$ has at least three vertices on the boundary, $\big||I'|- |I_{a,b}|\big|\to 0$ and  $|t_{\rm mid}'- t_{\rm mid}|\to0$ in probability as $n\to\infty$;\label{cond:def:CPtoMOT-0}
        \item the  maximal diameter of disks in $\cP_{n,I'}$ tends to 0 in probability as $n\to\infty$;\label{cond:def:CPtoMOT-1}
        \item we have  $$\sup_{t\in I'}|\eta_{n,I'}(t)-{\mathfrak{tr}}_{n,I'}(t)| \to 0 \text{ in probability as } n\to\infty,$$ where the convergence is uniform over all choices of $\mathfrak{tr}$ such that ${\mathfrak{tr}}_{n,I'}(t)$ is on the corresponding embedded face; 
        \label{cond:def:CPtoMOT-2}
        \item let $\nu_{n,I'}$ be the measure on $\partial\bbD$ that assigns mass $\nu(n)$ to each of the tangent point of a disk in $\cP_{n,I'}$ and $\partial \bbD$ where $\nu(n)>0$ is a constant. Then the Prokhorov distance between $\nu_{n,I'}$ and the LQG boundary length measure $\nu_{h_{n,I'}}$ tends to 0 in probability as $n\to\infty$.\label{cond:def:CPtoMOT-3}
    \end{enumerate}
    %the distance between the disk for the leftmost (resp.\ rightmost) vertex of $\mr T_{n,I'}$ and the leftmost (resp.\ rightmost) point of $\eta(I')$ tends to 0 in probability as $n\to\infty$, and 
\end{definition}
Note that $\cP_{n,I'}$ is measurable with respect to $(\cZ^n,I,a,b)$. Suppose $\cP_{n,I'}$ is   well-defined and $t_{\rm mid},t_{\rm mid}' \in I'$.
We  write $\wt F^n(\cZ^n,a,b) = (t_{\rm mid}',|I'|,\delta_{n,I'},{\mathfrak{tr}}_{n,I'},\nu_{n,I'})$, where $\delta_{n,I'}$ is the maximal diameter of disks in $\cP_{n,I'}$ and ${\mathfrak{tr}}_{n,I'}$ is chosen to   be continuous and %space-filling on each embedded face and 
measurable with respect to $\cZ^n|_{I'}$. Let $\wt F^{\infty}(Z^n,\cZ^n,a,b) = (t_{\rm mid}, |I_{a,b}|,0,\eta_{n,I'},\nu_{h_{n,I'}})$. If $\cP_{n,I'}$ is  not well-defined or any of $t_{\rm mid}$, $t_{\rm mid}'$ is not in $I'$, let $\wt F^{\infty}(Z^n,\cZ^n,a,b) = (-1,-1,0,0,0)$ and $\wt F^n(\cZ^n,a,b) = (-1,0,0,0,0)$. Set $\wt F_0^{\infty}(Z^n, a,b) = (t_{\rm mid}, |I_{a,b}|,0,\eta_{n,I_{a,b}},\nu_{h_{n,I_{a,b}}})$, where $\eta_{n,I_{a,b}},\nu_{h_{n,I_{a,b}}}$ are defined in the same way as $\eta_{n,I'},\nu_{h_{n,I'}}$ using $I_{a,b}$ in place of $I'$. Let $\mathfrak{d}$ be the metric on $\bbR^3\times C(I,\ol\bbD)\times \mathrm{Measure}(\partial\bbD)$ using the supremum norm over continuous functions and Prokhorov distances on measures.
We further define   $\mathrm{err}(\cZ^n,Z^n,a,b)= \mathfrak{d}(\wt F^n(\cZ^n,a,b),\wt F^{\infty}(Z^n,\cZ^n,a,b))$ where the traces ${\mathfrak{tr}}_{n,I'},\eta_{n,I'},\eta_{n,I_{a,b}}$  are set to be constant on $I\backslash I'$ or $I\backslash I_{a,b}$. % we declare that $\mathrm{err}(\cZ^n,Z^n,a,b)=1$ whenever the circle packing $\cP_{n,I'}$ is not well-defined. 
%where $t_{\rm mid}$ is the mid-point of the interval $I$. The walk $\cZ^n$ encodes a disk triangulation $T_n$ and a path $\mathfrak{tr}_n$ on $\cF\cT_n$ through the map $\Phi_\rmT$.  
%Let $\mr\cZ^n = \Phi_{\rm simple}(\cZ^n)$.  % We parameterize $\mathfrak{tr}_n:I\to \cF T_n$ such that  
The reason we use the subintervals $I'$ instead of the whole interval $I$ is that, as explained before Lemma~\ref{lem:A2}, the operator $\Phi_{\rm simple}$ is non-local, and for our choice of $I'$, $\Phi_{\rm simple}(T_{n,I'})$ can be viewed as a submap of the $\Phi_{\rm simple}$ version of the triangulation encoded by $\cZ^n|_I$. For a dyadic system $\fD$ and $n\in\bbN$, we define $2^n\fD = \{2^n I:I\in\fD\}$, and for $\fU = (U_I:I\in\fD)$ indexed by the dyadic system $\fD$, we let $2^n\fU = (U_I^n:I\in 2^n\fD)$ where $U_I^n= U_{2^{-n}I}$.  {The following is from \cite[Lemma 6.3]{Tutte}.}

\begin{lemma}\label{lem:def-cZ^n}
     Consider the setting above under one of the probability measures $\bbP_\rmQ$ or $\bbP_\rmE$. For each $n\in\bbN$ there is a process $\cZ^n:\bbR\to\bbR^2$ such that 
     \begin{equation}
         ((2^{n/2}Z_{2^{-n}t})_{t\in\bbR}, 2^n\fU,(\cZ^n_{2^{-n}t})_{t\in\bbR}) \overset{d}{=} ((Z_t)_{t\in\bbR},  \fU,(\cZ_t)_{t\in\bbR}).
     \end{equation}
     Furthermore, $\cZ^n$ is measurable with respect to $(Z, \fU)$.
\end{lemma}

\begin{proposition}\label{prop:cZ^n-pack-E}
    Let $\cZ^n$ be defined as in Lemma~\ref{lem:def-cZ^n}. Let $a',b'\in [0,1)$ such that $a'+b'\leq 0.8$ and $I_{a',b'}(0)\subseteq \ol {I(0)}$ be the closed interval whose left (resp.\ right) endpoint has distance $a'$ (resp.\ $b'$) from the left (resp.\ right) endpoint of $I(0)$. Under the probability measure $\bbP_\rmE$, $\|(2^{-n/2}\cZ^n-Z)|_{I_{a',b'}(0)}\|_\infty$ converges to 0 in probability, and circle packing converges to mating of trees for $(\cZ^n|_{I_{a',b'}(0)},Z|_{I_{a',b'}(0)})$ with $\nu(n)=2^{-n/2}$ in Condition~\ref{cond:def:CPtoMOT-3}. Furthermore, for each compact set $K\subseteq\{(a,b,a',b'):a,b>0,a+b<1,a',b'\geq0, a'+b'\leq 0.8\}$ and $\e>0$, we have \eqb\label{eq:410uniform}\sup_{(a,b,a',b')\in K}\bbP_\rmE\big[\mathrm{err}(\cZ^n|_{I_{a',b'}(0)},Z^n|_{I_{a',b'}(0)}, a,b)>\e\big]\to 0 \ \ \text{as } n\to\infty.\eqe
\end{proposition}
The reason for introducing $I_{a',b'}(0)$ here is that, as in Lemma~\ref{lem:cZ^n-pack-Q-n} below, to transfer the convergence along the sequence $2^n$ to the general $n$, we need to truncate $I(0)$ and work with its subintervals as well.

\begin{proof}
    The first claim is immediate from Lemma~\ref{lem:def-cZ^n} and Theorem~\ref{thm:KMT-0-cone}. For the second claim, let  $\Phi_\rmT(\cZ^n) = ( T^n,{\mathfrak{tr}^n})$, and let $\eta^n$ be the space-filling curve generated via the boundary length process  $(2^{n/2}Z_{2^{-n}t})_{t\in\bbR}$.    Let $\cJ^{n,-}$ and $\mr\cJ^{n,-}$ be the downward jump times defined in the same way as $\cJ^-$ and $\mr\cJ^-$ for $\cZ^n$, and $t\mapsto \frv_t^n$ be the bijection between $\cJ^{n,-}$ and $\cV T^n$.  Note that from the scaling in Lemma~\ref{lem:Rb}, $\eta^n(t) = \fR_{2^{-n}}^{-1}\eta(2^{-n}t)$, and the downward jump times for $(\cZ^n_{2^{-n}t})_{t\in\bbR}$ is equal to $2^n\cJ^{n,-}$.  By Lemma~\ref{lem:def-cZ^n} and Theorem~\ref{thm:conv-0-cone-pack-Hcell} applied to $r = \diam(\eta^n([-2^{n+1},2^{n+1}])) = \fR_{2^{-n}}^{-1}\diam(\eta([-2,2]))$, as we scale by  $\fR_{2^{-n}}$, there is a circle packing for $  T^n$ such that   in probability as $n\to\infty$,
    \begin{align}
 &  \lim_{n\to\infty} \frac{1}{ \diam(\eta([-2,2])) }\sup\{ |\eta(t)-o(\frv_t^n)|: t\in \cJ^{n,-},t\in [-2,2]\} = 0;  \label{eq:E-pack-0-0}  \\ & 
   \lim_{n\to\infty} \frac{1}{ \diam(\eta([-2,2]))}\sup\{ |\eta(t)-{\mathfrak{tr}^n}(t)|: t\in [-2,2]\} = 0,  \label{eq:E-pack-0-1}
\end{align}
where the limit~\eqref{eq:E-pack-0-1} is uniform over all choices of the space-filling curve $\mathfrak{tr}^n$. This along with part (\ref{cond:thm:conv-0-cone-pack-Hcell-(i)}) of Theorem~\ref{thm:conv-0-cone-pack-Hcell} implies that the maximal diameter of disks for vertices incident to the faces for times in $[-2,2]$ tends to 0 in probability. Fix a compact set $K$ as in the statement and $(a,b,a',b')\in K$.  Also choose any deterministic sequence $(a'_n,b'_n,a_n,b_n)\in K$ such that $(a'_n,b'_n,a_n,b_n)\to (a',b',a,b)$.  Let $I = I_{a',b'}(0)$ and the interval $I_{a,b}$ for $I$ be as in Definition~\ref{def:CPtoMOT} with its middlepoint $t_{\rm mid}$. Define $I_{a_n',b_n'}(0)$ and $I_{a_n,b_n}$ analogously. Let $I_n'$ be the interval $I'$ defined for $I = I_{a_n',b_n'}(0)$ in Definition~\ref{def:CPtoMOT}. By Proposition~\ref{lem:gap-hat-J^-} and the scaling from Lemma~\ref{lem:def-cZ^n},  with probability tending to 1 as $n\to\infty$,   $\mr\cJ^{n,-}$ intersects both components of $I_{a_n',b_n'}\backslash I_{a_n,b_n}$, and the length of the interval $|I_n'|\to |I_{a,b}|$   in probability.  By    Lemma~\ref{lem:A2},  $ \Phi_{\rm simple}(\cZ^n)|_{I_{a_n,b_n}} =   \Phi_{\rm simple}(\cZ^n|_{I_{a_n',b_n'}(0)})|_{I_{a_n,b_n}}$, and therefore  the simple version  $\Phi_{\rm simple}(T_{n,I'_n})$ as in Definition~\ref{def:CPtoMOT} can be viewed as a submap of $\Phi_{\rm simple}(T^n)$. 

Let $t_{\rm mid}' =  \inf\{u\geq t_{\rm mid}:u\in \mr\cJ^{n,-}\}$. Then as $n\to\infty$, $t_{\rm mid}'\in \mr\cJ^{n,-} $ with probability tending to 1, and $t_{\rm mid}'\to t_{\rm mid}$ in probability by  Lemma~\ref{lem:A2} and Proposition~\ref{lem:gap-hat-J^-}. The convergence in~\eqref{eq:E-pack-0-1} and $t_{\rm mid}'\to t_{\rm mid}$ implies that with probability tending to 1, $\frv_{t_{\rm mid}'}$ is an interior vertex of $T_{n,I_n'}$.    Let $D=\eta(I_{a,b})$, $\wt D=\bbD$, $\cP^n$ is the circle packing for $T^n$ restricted to vertices in  $T_{n,I_n'}$  and thus a circle embedding for $T_{n,I_n'}$, and $\wt\cP^n$ is a   circle packing in $\bbD$ for $ T_{n,I_n'}$ centered at the vertex $\frv_{t_{\rm mid}'}$.   Let $f^{n,0}$ be a homeomorphism between the carriers of $\cP^n$ and $\wt\cP^n$ that corresponds disks and send centers of disks to the  corresponding centers.   %{centered at} $\frv_{t_{\rm mid}'}^n$, where $t_{\rm mid}' = \inf\{u\geq t_{\rm mid}:u\in \mr\cJ^{n,-}\}$. 
  Note that $\eta(I_{a,b})$ is a Jordan domain by the construction of the space-filling curve $\eta$ via GFF flow line as in~\cite{ig4}. Below we apply Theorem~\ref{thm:He-Schramm-extension-1} for $\cP^n$ and $\wt\cP^n$.

Condition~(\ref{cond:thm:He-Schramm-extension-1-(i)}) is true since the maximal diameter of disks for vertices $\frv_t^n$ where $t\in [-2,2]\cap  \cJ^{n,-}$ and the disks tangent to them tends to 0 in probability. By part (\ref{cond:thm:conv-0-cone-pack-Hcell-(iii)}) of Theorem~\ref{thm:conv-0-cone-pack-Hcell}, the curve distance~\eqref{eq:curve-metric-0} between $\partial \mathrm{carrier}(\cP^n)$ and $\partial\eta(I_n')$ tends to 0 in probability. Using $|I_n'|\to |I_{a,b}|$ in probability and the continuity of the curve $\eta$,   the two components of $\eta(I_{a,b}\backslash I_n')\cup\eta(I_n'\backslash I_{a,b})$ are contained within a small neighborhood of the two endpoints of $\eta(I_{a,b})$ and thus the curve distance~\eqref{eq:curve-metric-0} between $\partial\eta(I_n')$ and  $\partial\eta(I_{a,b})$  tends to 0. Thus   the curve distance~\eqref{eq:curve-metric-0} between $\partial \mathrm{carrier}(\cP^n)$ and $\partial\eta(I_{a,b})$ tends to 0 in probability as $n\to\infty$ as well. This along with Condition~(\ref{cond:thm:He-Schramm-extension-1-(i)}) implies both Condition~(\ref{cond:thm:He-Schramm-extension-1-(ii)}) and Condition~(\ref{cond:thm:He-Schramm-extension-1-(iv)}) for $\cP^n$.  Moreover, by continuity of $\eta$ and~\eqref{eq:E-pack-0-0}, $|\eta(t_{\rm mid}) - o(\frv_{t_{\rm mid}'})|\leq |\eta(t_{\rm mid})-\eta(t'_{\rm mid})|+|\eta(t'_{\rm mid})-o(\frv_{t_{\rm mid}'})|\to 0$.  Then Conditions~(\ref{cond:thm:He-Schramm-extension-1-(i)}), (\ref{cond:thm:He-Schramm-extension-1-(ii)}) and  (\ref{cond:thm:He-Schramm-extension-1-(iv)}) remain true if we translate $\cP^n$ such that the disk for $\frv_{t_{\rm mid}'}$ is now centered at $\eta(t_{\rm mid})$. We write $g^n(z) = f^{n,0}(z - \eta(t_{\rm mid})+o(\frv_{t_{\rm mid}'})) $.    Condition~(\ref{cond:thm:He-Schramm-extension-1-(iii)}) then holds for $p_0=\eta(t_{\rm mid})$ after the translation of $\cP^n$ since $g^n(p_0)=0$ by our choice of embedding of $\wt\cP^n$.  Condition~(\ref{cond:thm:He-Schramm-extension-1-(ii)}) for $\wt\cP^n$ is immediate since the circle packing is in $\bbD$.

Now by Theorem~\ref{thm:He-Schramm-extension-1}, since Conditions~(\ref{cond:thm:He-Schramm-extension-1-(i)})-(\ref{cond:thm:He-Schramm-extension-1-(iii)}) hold, the maximal diameter of disks in $\cP_{n,I_n'} =  \wt\cP^n$ tends to 0 in probability, which further implies  Condition~(\ref{cond:thm:He-Schramm-extension-1-(iv)}) for $ \wt\cP^n$ since the packing is in $\partial \bbD$. Further by Theorem~\ref{thm:He-Schramm-extension-1}, up to a rotation, the maps $g^n$ converges uniformly to a conformal map $\varphi_{I_{a,b}}:\eta(I_{a,b})\to\bbD$ sending $\eta(t_{\rm mid})$ to 0, while by Proposition~\ref{prop:boundary-perturb}, up to a rotation, the conformal map $\varphi_{I_n'}:\eta(I_n')\to\bbD$ sending $\eta(t_{\rm mid})$ to 0 converges uniformly to $\varphi_{I_{a,b}}$ in probability. The uniformity of the limit~\eqref{eq:E-pack-0-1} over the trace $\frtr$ and the choice that $\eta_{n,I_n'}$ starts from 1 further imply that the tangent point between $\partial\bbD$ and the disk in $\cP_{n,I_n'}$ for the left endpoint of the root edge of $T_{n,I_n'}$ is in $o_n(1)$ distance from 1 in probability.  This along with ~\eqref{eq:E-pack-0-0} and ~\eqref{eq:E-pack-0-1} verifies  Conditions~\ref{cond:def:CPtoMOT-0}-\ref{cond:def:CPtoMOT-2} in Definition~\ref{def:CPtoMOT} for $\cP_{n,I_n'}$ and $(I_{a_n,b_n},I_n')$, since recentering the packing $\cP_{n,I_n'}$ at the corresponding midpoint vertex for $I_{a_n,b_n}$ instead of $\frv_{t_{\rm mid}'}$ and rotating $\cP_{n,I_n'}$ such that the disk for the vertex to the left of the root edge is tangent to $\partial\bbD$ at 1 only change the error by $o_n(1)$. For  Condition~\ref{cond:def:CPtoMOT-3}, note that by ~\eqref{eq:Rb-bound}, $2^{-n/2}\log(\fR_{2^{-n}}^{-1}\diam(\eta([-2,2])))\to 0$ in probability as $n\to\infty$. Let $I_n' = [s_n,t_n]$, and $z_{s_n,t_n,q}^-$, $z_{x_{k,q}}^n$ where $q\in\{\rmL,\rmR\}$, $x\in\{s,t\}$ be the points defined in Theorem~\ref{thm:conv-0-cone-pack-Hcell} for $[s,t]=[s_n,t_n]$. We further move the points  $z_{x_{k,q}}^n$ that are not on $\partial\eta(I_n')$ to $z_{s_n,t_n,q}^-$. 
%and if we further move all the points $x_{i,q}$ that is not on $\partial\eta([s,t])$ to $z_{s,t,q}^-$, then part~(\ref{cond:thm:conv-0-cone-pack-Hcell-(iv)}) of Theorem~\ref{thm:conv-0-cone-pack-Hcell} remains true by part~(\ref{cond:thm:conv-0-cone-pack-Hcell-(iii)}) of Theorem~\ref{thm:conv-0-cone-pack-Hcell} and the continuity of the boundary length process. 
Conditions~\ref{cond:def:CPtoMOT-3} in Definition~\ref{def:CPtoMOT} then follows by  the uniform convergence of map $g^n$ and the conformal map $\varphi_{I_n'}$ to the conformal map $\varphi_{I_{a,b}}$ and the bound in part~(\ref{cond:thm:conv-0-cone-pack-Hcell-(iv)}) of Theorem~\ref{thm:conv-0-cone-pack-Hcell}. This gives that for any $\e>0$, $\bbP_\rmE\big[\mathrm{err}(\cZ^n|_{I_{a_n',b_n'}(0)},Z^n|_{I_{a_n',b_n'}(0)},a_n,b_n)>\e\big]\to 0$, which implies the desired convergence of circle packing to mating of trees by setting $(a_n',b_n',a_n,b_n)=(a',b',a,b)$ and also~\eqref{eq:410uniform}, since otherwise by compactness there will be deterministic $(a_n',b_n',a_n,b_n)$ and $\e_0,\delta_0>0$ such that $\bbP_\rmE\big[\mathrm{err}(\cZ^n|_{I_{a_n',b_n'}(0)},Z^n|_{I_{a_n',b_n'}(0)}, a_n,b_n)>\e_0\big]\to \delta_0$, which is a contradiction.
\end{proof}

Now we transfer from Proposition~\ref{prop:cZ^n-pack-E} to the setting of $\bbP_\rmQ$. Recall the definition of rate $\lambda$ continuous-time nearest-neighbor walk on $\bbZ^2$ from Definition~\ref{def:walk}.
\begin{proposition}\label{prop:cZ^n-pack-Q}
    For each $n\in\bbN$ there is a probability measure on pairs $(Z^n,\cZ^n)$ such that the following holds: 
    \begin{enumerate}[(i)]
        \item $(\cZ_t^n)_{t\in [0,1]}$ is a continuous-time nearest-neighbor walk on $\bbZ^2$ with rate $2^n$ such that $\cZ_0^n=0$;
        \item $(Z^n_t)_{t\in[0,1]}$ has the law of a variance $\frac12$ uncorrelated 2D Brownian motion restricted to time $[0,1]$;
        \item Circle packing converges to mating of trees for $(\cZ^n,Z^n)_{n\in\bbN}$ with $\nu(n)=2^{-n/2}$ in Condition~\ref{cond:def:CPtoMOT-3}, and for each compact set $K$ in Proposition~\ref{prop:cZ^n-pack-E} and any $\e>0$, %with $[0,1]$ and $[a',1-b']$ in place of $I(0)$ and $I_{a',b'}(0)$ converges to 0 uniformly in probability as $n\to\infty$ in the sense of~\eqref{eq:410uniform};  \
        \begin{equation}
           \sup_{(a,b,a',b')\in K} \bbP\big[\mathrm{err}(\cZ^n|_{[a',1-b']},Z^n|_{[a',1-b']},a,b)>\e \big]\to 0 \text{ as } n\to\infty.
        \end{equation}
        \label{cond:prop:cZ^n-pack-Q-(iii)}
        \item $\|Z^n-2^{-n/2}\cZ^n\|_\infty$ converges to 0 in probability as $n\to\infty$.
    \end{enumerate}
\end{proposition}

\begin{proof}
    {In~\cite[Section 6.1]{Tutte}, an analogous statement (\cite[Proposition 6.2]{Tutte} there) was proven based on the analogous result of Proposition~\ref{prop:cZ^n-pack-E} (\cite[Proposition 6.4]{Tutte} there) for random walk converging to Brownian motion on $(Z^n,\cZ^n)$. Note that in our definition of circle packing converges to mating of trees in Definition~\ref{def:CPtoMOT}, the events considered there is measurable with respect to $(Z^n,\cZ^n)$ and is invariant if we shift the domain of $(Z^n,\cZ^n)$. Therefore the claim follows from the identical proof of ~\cite[Proposition 6.2]{Tutte} where we replace ``random walk converges to Brownian motion" by ``circle packing converges to mating of trees".}
\end{proof}

Our Proposition~\ref{prop:cZ^n-pack-Q} above is stated for walks of rate $2^n$. Next we transfer to the walks of rate $n$.

\begin{lemma}\label{lem:cZ^n-pack-Q-n}
  Let $t_0>0$. There is a coupling of $Z,\cZ^1,\cZ^2,...$ such that the following holds: 
    \begin{enumerate}[(i)]
        \item For each $n$, $(\cZ_t^n)_{t\in [0,t_0]}$ is a continuous-time nearest-neighbor walk on $\bbZ^2$ with rate $n$ such that $\cZ_0^n=0$;
        \item $(Z_t)_{t\in[0,t_0]}$ has the law of a variance $\frac12$ uncorrelated 2D Brownian motion restricted to time $[0,t_0]$;
          \item $\|Z-n^{-1/2}\cZ^n\|_\infty$ converges to 0 in probability as $n\to\infty$;
        \item Circle packing converges to mating of trees for $(\cZ^n,Z)_{n\in\bbN}$ with $\nu(n)=n^{-1/2}$ in Condition~\ref{cond:def:CPtoMOT-3}, and for each compact set $K\subseteq\{(a,b,a',b'):a,b>0,a+b<1,a',b'\geq0, a'+b'\leq 0.3\}$ and $\e>0$, %with $[0,1]$ and $[a',1-b']$ in place of $I(0)$ and $I_{a',b'}(0)$ converges to 0 uniformly in probability as $n\to\infty$ in the sense of~\eqref{eq:410uniform};  \
        \begin{equation}\label{eq:conv-dist-check}
           \sup_{(a,b,a',b')\in K} \bbP\big[\mathrm{err}(\cZ^n|_{[a't_0,(1-b')t_0]},Z|_{[a't_0,(1-b')t_0]},a,b)>\e \big]\to 0 \text{ as } n\to\infty.
        \end{equation}\label{cond:lem:cZ^n-pack-Q-n-(iv)}
       % and the corresponding error in Proposition~\ref{prop:cZ^n-pack-E} with $[0,t_0]$ and $[a't_0,(1-b')t_0]$ in place of $I(0)$ and $I_{a',b'}(0)$ converges to 0 uniformly in probability as $n\to\infty$  in the sense of~\eqref{eq:410uniform}.
    \end{enumerate}
\end{lemma}
\begin{proof}
 Consider the setting in Proposition~\ref{prop:cZ^n-pack-Q},  except that we suppose the processes for different $n$ are defined on the same probability space and that they have been coupled together so that $Z^n = Z^1$ for all $n\in\bbN$. Let $\check{Z},\check{\cZ^1},\check{\cZ^2},...,$ be equal to the processes $Z,\cZ^1,\cZ^2,...$ in  Proposition~\ref{prop:cZ^n-pack-Q} under these assumptions.  For each sufficiently large $n$, choose $m=m(n)$ such that $2^{m-1}<nt_0\leq 2^m$. Let $\wh\cZ^n_{t} = \check{\cZ}^m_{tn2^{-m}}$ and $Z_{t}^n = n^{-1/2}2^{m/2}\check{Z}_{tn2^{-m}}$ for $t\in [0,t_0]$. Note that each $Z_t^n$ has the law as a  variance $\frac12$ uncorrelated 2D Brownian motion. We further fix one Brownian motion $Z$, and conditioned on  $Z$, let $\cZ^n$ be sampled from the conditional law of $\wh\cZ^n$ given $Z^n=Z$.  Then claims (i)-(iii) follows directly from our construction, while the claim (iv) follows from the uniformity of the convergence in part (\ref{cond:prop:cZ^n-pack-Q-(iii)}) of Proposition~\ref{prop:cZ^n-pack-Q} as we set $(a',b')=(0,1-nt_02^{-m})$ since $nt_02^{-m}\in [1/2,1]$.
\end{proof}

\subsection{The disk topology case}\label{subsec:pack-disk-pf}

In this subsection we prove Theorem~\ref{thm:pack-disk}. Our proof is organized as follows. We will first work with walks $\wt\cZ^n$ with Poissonian jumps.  We begin with the setup of Lemma~\ref{lem:cZ^n-pack-Q-n} and transfer to the setting where the process there $Z$ is the process $Z$ in Theorem~\ref{thm:pack-disk} restricted to $[\e,A-\e]$. We then use  Theorem~\ref{thm:He-Schramm-disk} to  prove that in the circle packing for the whole disk triangulation generated by $\wt\cZ^n$, the portion of disks for times outside $[\e,A-\e]$ is little. This implies the convergence of the space-filling curve $\wt\frtr^n$ for $\wt\cZ^n$ as well as the closeness between the circle packing and the mating of trees embedding, in the setting where the walk has Poissonian jumps (Lemma~\ref{lem:disk-pf-2}). Finally, in Lemma~\ref{lem:disk-pf-3}, we prove that the times at which $\wt\cZ^n$ jumps are approximately evenly distributed, which allows us to reparameterize $\wt\frtr^n$ and transfer back to the setting of Theorem~\ref{thm:pack-disk}, and further prove the measure convergence in Theorem~\ref{thm:pack-disk}.

\begin{lemma}\label{lem:disk-pf-0}
      There is a coupling of $(Z,\wt\cZ^1,\wt\cZ^2,...)$ such that
    \begin{enumerate}[(i)]
    \item $Z:[0,A]\to\bbR_+^2$ has the law of  {$\frac{1}{2}$ times} a Brownian cone excursion  {of covariance~\eqref{eq:variance}} from $(2,0)$ to $(0,0)$ for random duration $A>0$ where $A$ has density $\rho(a) = a^{-3}e^{-1/a}$; \label{cond:lem:disk-pf-0-(i)}
        \item for each $n\in\bbN$, conditioned on $A=a$, $\wt\cZ^n:[0,a]\to\bbZ^2$ is  a rate $n$ continuous-time nearest neighbor walk on $\bbZ^2$ started from $(\ceil{ {\sqrt n}-1 },0)$ and ended at $(0,0)$, conditioned on  $\wt\cZ^n_t\in\bbR_+^2$ for every $t\in [0,a]$;  \label{cond:lem:disk-pf-0-(ii)}
        \item {the total variation distance between the law of  $\Phi_\rmT^\disk(\wt\cZ^n)=(\wt T^{\disk,n},\wt\frtr^{\disk,n})$  and the  law of  $(T^{\disk,n},\frtr^{\disk,n})$ in Theorem~\ref{thm:pack-disk} for $\ell=\ceil{\sqrt n}$ tends to 0 as $n\to\infty$ up to a reparameterization of the trace $\wt\frtr^{\disk,n}$};  \label{cond:lem:disk-pf-0-(iii)}
        \item  {for each deterministic $\e\in(0,1/4)$,} circle packing converges to mating of trees for  {$(\wt\cZ^n|_{[A\e,A(1-\e)]},Z|_{[A\e,A(1-\e)]})_{n\in\bbN}$}  with $\nu(n)=n^{-1/2}$ in Condition~\ref{cond:def:CPtoMOT-3};  \label{cond:lem:disk-pf-0-(iv)}
        \item $\|Z-n^{-1/2}\wt\cZ^n \|_\infty \to 0 $ in probability as $n\to\infty$ and  \label{cond:lem:disk-pf-0-(v)}
        \item let $B^n = \{t\in [0,A]: \wt\cZ_t^n-\lim_{u\uparrow t}\wt\cZ_u^n \neq 0 \}$ denote the set of times at which $\wt\cZ^n$ jumps, conditioned on $A$ and $\#B^n$, the elements of $B^n$ have the law of $\#B^n$ uniform and independent times in $[0,A]$. {Furthermore, $\#B^n/n$ converges to $A$ in probability as $n\to\infty$.}  \label{cond:lem:disk-pf-0-(vi)}
    \end{enumerate}
\end{lemma}

\begin{proof}
   A version of these claims for a fixed $\e\in(0,0.1)$ and $A$ replaced by a deterministic $a$ can be shown via identical arguments as in  the proof of  \cite[Lemma 6.28]{Tutte}  draft using Lemma~\ref{lem:cZ^n-pack-Q-n} and reweighting the law of $Z|_{[\e a,(1-\e)a]}$,  where $Z$ and $\wt Z^n$ are the processes in claim (i) and claim (ii) conditioned on the fix duration $A=a$, $\|(Z- n^{-1/2}\wt \cZ^n)|_{[a\e,a(1-\e)]}\|_\infty\to 0$ in probability,  and  circle packing converges to mating of trees for  $(\wt\cZ^n|_{[a\e,a(1-\e)]},Z|_{[a\e,a(1-\e)]})_{n\in\bbN}$.
    Conditioned on $A = a$, for $m\in\bbN$, we have $$\bbP[\#B^n = 2m+\ell-1|A=a] = \frac{4^{-2m}\#\cA^\disk_{\ell,2m+\ell}\bbP[\xi_{an} = 2m+\ell-1]}{\sum_{j=0}^\infty 4^{-2j}\#\cA^\disk_{\ell,2j+\ell}\bbP[\xi_{an} = 2j+\ell-1]},$$
    where $\xi_{an}$ is a Poisson$(an)$ random variable. On the other hand, for the walk $\cZ^{\disk,n}$ in Theorem~\ref{thm:pack-disk}, the probability that its duration is equal to $2m+\ell-1$ is given by $$\frac{4^{-2m}\#\cA^\disk_{\ell,2m+\ell}}{\sum_{j=0}^\infty 4^{-2j}\#\cA^\disk_{\ell,2j+\ell}}.$$
    As computed in the proof of \cite[Lemma 6.29]{Tutte}, using Lemma~\ref{lem:walk-excursion}, for $a$ in a compact set $K\subseteq (0,\infty)$, $\bbP[|\#B^n-an|>C\sqrt{an}|A=a] = o_C(1)$ as $C\to\infty$ where $o_C(1)$ is uniform over $a\in K$ and all $n$, and the total variation distance between the law of the total number of jumps of the walk $\wt\cZ^n$ as we further integrate over the density $\rho(a)\,da$ of $A$ and  the law of the total number of jumps of the walk $\cZ^{\disk,n}$ tends to 0 as $n\to\infty$. Since given the number of jumps and $A=a$, the times when $\wt\cZ^n$ jumps are i.i.d.\ uniform in $[0,a]$ and the steps  by $\wt\cZ^n$ are uniform in $\cA_\ell^\disk$, we see that claim (iii) and the first part of claim (vi) hold, while the second part of claim (vi) follows by further sending the above compact set $K\uparrow(0,\infty)$. 
    Claim (iv) follow by constructing the corresponding coupling $(\wt\cZ^n|_{[a\e,a(1-\e)]},Z|_{[a\e,a(1-\e)]})_{n\in\bbN}$ with $\e=\e(n)$ tending to 0 sufficiently slowly, and claim (v) follows from the convergence $\|(Z- n^{-1/2}\wt \cZ^n)|_{[a\e,a(1-\e)]}\|_\infty\to 0$ for fixed $\e>0$ and the tightness of $n^{-1/2}\cZ^n$ near its endpoints \cite[Lemma 6.26]{Tutte}. 
\end{proof}

Let $t\mapsto  \frv_t^n$ be the bijection between the downward jump times for $\wt\cZ^n$ to vertices of $\wt T^{\disk,n}$ that are not on the last exploration edge as  described in Section~\ref{subsubsec:Mullin-finite}. 

\begin{figure}
    \centering
    \includegraphics[scale=0.8]{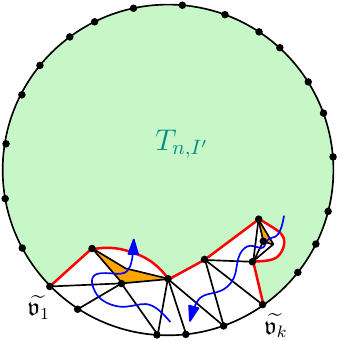}
    \caption{An illustration of the proof of Lemma~\ref{lem:disk-pf-2}.  The disk triangulation $\wt T^{\disk,n}$ where the blue curve is the path $\wt\frtr^{\disk,n}$ on faces. The orange faces are $\wt\frtr^{\disk,n}(A\e)$ and $\wt\frtr^{\disk,n}(A(1-\e))$. The disk triangulation $T_{n,I'}$ is shown in green where its interior vertices and edges are not drawn. The interface $\wt\gamma_{n,I'}$ is shown in red.  }
    \label{fig:disk-pf-1}
\end{figure}

\begin{lemma}\label{lem:disk-pf-2}
    Let $(Z,\wt\cZ^1,\wt\cZ^2,...)$ and $\wt T^{\disk,n}$ be as in Lemma~\ref{lem:disk-pf-0}. Let $\wt \frtr^n:[0,A]\to \cF \wt T^{\disk,n}$ be the path on faces of $\wt T^{\disk,n}$.  Let $\frv^{\mathrm{mid},n}$ be the vertex on $\wt T^{\disk,n}$ for the first downward jump time of  {$\Phi_{\rm simple}(\wt\cZ^n)$} after $A/2$.   Then with probability tending to 1, $\frv^{\mathrm{mid},n}$ is not a boundary vertex. Further consider a univalent  circle packing $\wt\cP^{\disk,n}$ in $\bbD$ for $\wt T^{\disk,n}$ centered at $\frv^{\mathrm{mid},n}$. Then up to a rotation, we have the following:
    \begin{enumerate}[(i)]
        \item The maximal diameter of disks in $\wt\cP^{\disk,n}$ tends to 0 in probability as $n\to\infty$;\label{cond:lem:disk-pf-2-(i)}
        \item Let $\wt\frtr^{\disk,n}:[0,A]\to\ol\bbD$ be the path on faces of $\wt T^{\disk,n}$ under the embedding $\wt\cP^{\disk,n}$. Then %Then up to a rotation,
        \eqb\label{eq:lem:disk-pf-2}\sup_{t\in [0,A]} |\eta^\disk(t)-\wt\frtr^{\disk,n}(t)|\to 0 \ \ \ \text{in probability as } n\to\infty. \eqe 
        and the limit is uniform over all choices of the trace $\wt\frtr^{\disk,n}$ such that $\wt\frtr^{\disk,n}(t)$ is on the embedded face. 
        \label{cond:lem:disk-pf-2-(ii)}
        \item  Let $\wt\nu^{\disk,n}$ be the measure on $\partial\bbD$ that assigns mass $n^{-1/2}$ to each tangent point of disks in $\wt\cP^{\disk,n}$ and $\partial\bbD$. The Prohokov distance between $\wt\nu^{\disk,n}$ and $\nu_{h^\disk}$ tends to 0 in probability as $n\to\infty$. \label{cond:lem:disk-pf-2-(iii)}
    \end{enumerate} 
 %   if we send $\e\to 0$ sufficiently slower than $n\to\infty$, the diameter of the union of the disks for vertices that are adjacent to faces in $\cF \wt T^{\disk,n}\backslash\{\wt\frtr^{\disk,n}(t):t\in [\e,A-\e]\}$   tends to 0 in probability.
\end{lemma}

\begin{proof}
    %The claim that $\frv^{\mathrm{mid},n}$ is not on the boundary is a direct consequence of part (\ref{cond:lem:disk-pf-0-(iv)}) of Lemma~\ref{lem:disk-pf-0}, since the conditions in Definition~\ref{def:CPtoMOT} guarantees that with probability tending to 1,  there is a univalent circle packing for a submap of $\wt T^{\disk,n}$ in $\bbD$ such that the disk for $\frv^{\mathrm{mid},n}$ is away from $\partial \bbD$. Now for the second claim, 
    Let $\delta\in(0,0.1)$ be small. Fix $\delta'\in(0,1)$. We pick deterministic constants $0<\e_3 <\e_1<\e<\delta$ sufficiently small (depending on $\delta$), such that as $\delta\to0$, with probability tending to 1, we have
    \begin{align*}
       & %\inf_{s\in [\e_2,\e_1]} L_s < \inf_{s\in [0,\e_2]}L_s-\e_3; \ \  
       \inf_{s\in [A\e_1,A\e]} L_s < \inf_{s\in [0,A\e_1]}L_s-\e_3; \ \ \inf_{s\in [A\e,A\delta]} L_s< \inf_{s\in [0,A\e]}L_s-\e_3; \ \ \inf_{s\in [0,A\delta]} L_s> 1-\delta'; \\ &  %\inf_{s\in [0,A-\e_2]} L_s < \inf_{s\in [0,A-\e_1]}L_s-\e_3; \ \  
       \inf_{s\in [0,A(1-\e_1)]} L_s < \inf_{s\in [0,A(1-\e)]}L_s-\e_3; \ \ \inf_{s\in [0,A(1-\e)]} L_s < \inf_{s\in [0,A(1-\delta)]}L_s-\e_3; \ \ \inf_{s\in [0,A(1-\delta)]} L_s< \delta'.
    \end{align*}
    Then using part~(\ref{cond:lem:disk-pf-0-(v)}) from Lemma~\ref{lem:disk-pf-0}, with probability tending to 1 as $n\to\infty$ sufficiently faster than $\delta\to0$, 
    \begin{align}
       & %\inf_{s\in [\e_2,\e_1]} \wt\cL_s^n < \inf_{s\in [0,\e_2]}\wt\cL_s^n; \ \ 
       \inf_{s\in [A\e_1,A\e]} \wt\cL_s^n < \inf_{s\in [0,A\e_1]}\wt\cL_s^n; \ \ \inf_{s\in [A\e,A\delta]} \wt\cL^n_s< \inf_{s\in [0,A\e]}\wt\cL_s^n-\e_3\sqrt n/2; \ \ \inf_{s\in [0,A\delta]} \wt\cL_s^n> (1-2\delta')n^{1/2}; \label{eq:lem:disk-pf-1-1} \\ & % \inf_{s\in [0,A-\e_2]} \wt\cL_s^n < \inf_{s\in [0,A-\e_1]}\wt\cL_s^n; \ \  
       \inf_{s\in [0,A(1-\e_1)]} \wt\cL_s^n < \inf_{s\in [0,A(1-\e)]}\wt\cL_s^n-\e_3\sqrt n/2; \ \ \inf_{s\in [0,A(1-\e)]} \wt\cL_s^n < \inf_{s\in [0,A(1-\delta)]}\wt\cL_s^n<2\delta' n^{1/2}.\label{eq:lem:disk-pf-1-2} 
    \end{align}
 We work in this setting.   Let $\mr B^n$ be the  $B^n$ minus all the intervals $(t^-,t^+]$ such that~\eqref{eq:double-edge-Mullin} holds for $(\wt\cL^n,\wt\cR^n)$ in place of $(\cL,\cR)$.    %Consider the walk $\mr\cZ^n = \Phi_{\rm simple}(\wt\cZ^n)$ and the jump times $\mr\cJ^n$ for $\mr\cZ^n$. 
    Then~\eqref{eq:lem:disk-pf-1-1} and~\eqref{eq:lem:disk-pf-1-2} along with Lemma~\ref{lem:hat-minus} implies that $\mr B^n$ has non-empty intersection with all of   $[A\e_1,A\e]$, $[A\e,A\delta]$, $[A(1-\delta),A(1-\e)]$ and $[A(1-\e),A(1-\e_1)]$. Let $t_\e^- = \inf\{s\geq A\e:s\in\mr B^n\}$, $t_\e^+ = \sup\{s\leq A(1-\e): s\in\mr B^n\}$ and $I'=[t_\e^-,t_\e^+]$. Then $[A\delta,A(1-\delta)]\subseteq I'$. Let $T_{n,I'}$ be the map formed by the faces $\{\wt\frtr^{\disk,n}(t):t\in I'\}\backslash\{\wt\frtr^{\disk,n}(t_\e^-)\}$.  By Lemma~\ref{lem:A2},  the interval $I'$ and the triangulation $T_{n,I'}$  here agrees with the interval $I'$ and the disk triangulation $T_{n,I'}$ in Definition~\ref{def:CPtoMOT} defined for the interval $I_{a,b}=[A\e,A(1-\e)]$ and $I=[A\e_1,A(1-\e_1)]$, and the triangulation $\Phi_{\rm simple}(T_{n,I'})$  is a submap of the triangulation $\Phi_{\rm simple}(\wt T^{\disk,n})$. In particular, by Condition~\ref{cond:def:CPtoMOT-0}-~\ref{cond:def:CPtoMOT-2} of Definition~\ref{def:CPtoMOT}, with probability tending to 1 as $n\to\infty$, $\frv^{\mathrm{mid},n}$ is an interior vertex of $\Phi_{\rm simple}(T_{n,I'})$ and thus an interior vertex of $\Phi_{\rm simple}(\wt T^{\disk,n})$. 
   
    Consider the path $\wt\gamma_{n,I'}$ on $\partial T_{n,I'}$ formed by the vertices $\wt\frv_1,...,\wt\frv_k$, such that $\wt\frv_1$ and $\wt\frv_k$ are also on the boundary of $\wt T^{\disk,n}$ while the vertices on  $\partial T_{n,I'}$ other than $\wt\frv_2,...,\wt\frv_{k-1}$ are all on the boundary of $\wt T^{\disk,n}$. In other words, $\wt\gamma_{n,I'}$ is the interface between $T_{n,I'}$ and the map formed by $(\cF \wt T^{\disk,n}\backslash\{\wt\frtr^{\disk,n}(t):t\in I' \}) \cup\{\wt\frtr^{\disk,n}(t_\e^-)\}$, and $\cF \wt T^{\disk,n}\backslash\{\wt\frtr^{\disk,n}(t):t\in [A\e,A(1-\e)]\}\subseteq(\cF \wt T^{\disk,n}\backslash\{\wt\frtr^{\disk,n}(t):t\in I' \}) \cup\{\wt\frtr^{\disk,n}(t_\e^-)\}$. See   Figure~\ref{fig:disk-pf-1} for an illustration. Assume that $\wt\frv_1 = \frv^n_{s_1'}$ and  $\wt\frv_k = \frv^n_{s_2'}$ such that $s_1'<s_2'$.   Then the last inequalities of~\eqref{eq:lem:disk-pf-1-1} and~\eqref{eq:lem:disk-pf-1-2} along with the sewing procedure imply that $s_1'\leq A\delta$ while $s_2'\geq A(1-\delta)$. Let $s_0'$ be the first time when $\wt\cL^{\disk,n}$ is below $0.5n^{1/2}$, which is in $[A\delta,A(1-\delta)]$ with probability tending to 1 by our choice of $\delta$. In this case, the vertex $\wt\frv_0 = \frv^n_{s_0'}$ is on $\partial T_{n,I'}$ and not on the path $\wt\gamma_{n,I'}$.  Consider an embedding of $(h^\disk,\eta^\disk)$ in $\bbD$ such that $\eta^\disk(A/2)=0$, and a conformal map $\varphi_{I'}:\eta^\disk(I')\to \bbD$ fixing 0. By continuity of $\eta^\disk$ and $[A\delta,A(1-\delta)]\subseteq I'$, it follows from Proposition~\ref{prop:boundary-perturb} that up to a rotation, $\varphi_{I'}$ converges uniformly to the identity as $\delta\to 0$ in probability.  Then as $\delta\to0$, the diameter of $\varphi(\eta^\disk(I'\backslash [A\delta,A(1-\delta)] ))$ tends to 0, and further by part (\ref{cond:lem:disk-pf-0-(iv)}) of Lemma~\ref{lem:disk-pf-0}, if we send $\delta\to 0$ sufficiently slower than $n\to\infty$, in a univalent circle packing $\cP_{n,I'}$ of $T_{n,I'}$ in $\bbD$ centered at $\frv^{\mathrm{mid},n}$, the maximal diameter of disks as well as the distance between the disks for $\wt\frv_1$ and $\wt\frv_k$ tends to 0 in probability as $n\to\infty$, while the distance between the disks for $\wt\frv_1$ and  $\wt\frv_0$ is bounded away from 0 in probability.  Therefore the diameter of the union of the disks for $\wt\frv_1,...,\wt\frv_k$ also tends to 0 in probability since they cannot fill up $\partial \bbD$ due to the positive distance between the disks for $\wt\frv_1$ and  $\wt\frv_0$. This verifies the conditions in Theorem~\ref{thm:He-Schramm-disk} holds for $\cP_{n,I'}$ and $\wt\cP^{\disk,n}$ as we send $\delta\to 0 $ sufficiently slower than $n\to\infty$. %Claim (i) and claim (ii) then follow from 
    By Theorem~\ref{thm:He-Schramm-disk} and  part (\ref{cond:lem:disk-pf-0-(iv)}) of Lemma~\ref{lem:disk-pf-0}, claim (i) holds, and we may further rotate the packing $\wt\cP^{\disk,n}$ such that $\sup_{t\in I'} |\eta^\disk(t)-\wt\frtr^{\disk,n}(t)|\to 0$ in probability as $n\to\infty$. Furthermore,  
    the disks for vertices on $\wt\gamma_{n,I'}\cup(\wt T^{\disk,n}\backslash\{T_{n,I'}\})$ are within an $o_n(1)$ neighborhood of 1 and the diameter of their union tends to 0 in probability as $n\to\infty$. This along with the continuity of $\eta^\disk$ gives claim (ii).  For    claim (iii), by part (\ref{cond:lem:disk-pf-0-(v)}) of Lemma~\ref{lem:disk-pf-0}, if we send $n\to\infty$ then $\delta\to0$,  $$\max\Big\{\sup_{t\in[0,A\delta]}|n^{-1/2}\wt\cL_t -1 |,    \sup_{t\in[0,A\delta]}|n^{-1/2}\wt\cR_t|,  \sup_{t\in[A(1-\delta),A]}|n^{-1/2}\wt\cL_t| , \sup_{t\in[A(1-\delta),A]}|n^{-1/2}\wt\cR_t|\Big\} \to 0 \text{ in probability}, $$
    and the same holds for $(L_t,R_t)$ in place of $(n^{-1/2}\wt\cL_t,n^{-1/2}\wt\cR_t)$. In other words, further using $s_1'\le A\delta$ and $s_2'\geq A(1-\delta)$, the $\wt\nu^{\disk,n}$ measure of tangent points for boundary disks of $\wt\cP^{\disk,n}$ for vertices not in $\partial T_{n,I'}\backslash \wt\gamma_{n,I'}$ and the $\nu_{h^\disk}$ boundary length measure of $\eta^\disk([0,A]\backslash I')$ tend  to 0 in probability, and the $\nu_{n,I'}$-measure as in Condition~\ref{cond:def:CPtoMOT-3} of Definition~\ref{def:CPtoMOT} of the boundary disks for vertices on $\wt\gamma_{n,I'}$ in the packing $\cP_{n,I'}$ also tends to 0 in probability. This together with part (\ref{cond:lem:disk-pf-0-(iv)}) of Lemma~\ref{lem:disk-pf-0}, Condition~\ref{cond:def:CPtoMOT-3} of Definition~\ref{def:CPtoMOT} and Theorem~\ref{thm:He-Schramm-disk} gives the claim (iii). 
    %The lemma then follows from the identical proof of Lemma~\ref{lem:disk-subdomain-A} where one uses Lemma~\ref{lem:He-Schramm} to compare the circle packing $\cP_{n,I'}$ for $T_{n,I'}$ and the circle packing $\wt\cP^{\disk,n}$. See also the middle and right panel of Figure~\ref{fig:disk-pf-1}.
\end{proof}

To transfer the convergence in Lemma~\ref{lem:disk-pf-2} to the convergence in Theorem~\ref{thm:pack-disk}, in light of part (\ref{cond:lem:disk-pf-0-(iii)}) of Lemma~\ref{lem:disk-pf-0}, it remains to reparameterize the trace $\wt\frtr^{\disk,n}$ such that it agrees with the trace $\frtr^{\disk,n}$ such that with high probability. For a vertex $v$ in $\wt T^{\disk,n}$, we write $\wt\frt_v^n$ for its preimage under the map $t\mapsto \frv_t^n$, except that if $v$ is incident to  {last exploration edge} then we write $\frt_v^n = A$. We write $\wt b^n = \# B^n$ where $B^n$ is the set of times when $\wt\cZ^n$ jumps so that $\wt T^{\disk,n}$ has $\wt b^n+1$ faces, and write $\wt a^n = \frac{\wt b^n+\ceil{\sqrt n}-1}{2}+2$ such that $\wt T^{\disk,n}$ has $\wt a^n$ vertices. 

\begin{lemma}\label{lem:disk-pf-3}
    Let $\wt t^n_1\leq\wt t^n_2\leq...\leq \wt t^n_{\wt b^n}$ be the re-ordering of the set of times in $B^n$ from the smallest to largest. Then 
    \eqb\label{eq:lem:disk-pf-3-1}\sup_{1\leq j\leq \wt b^n} |\wt t^n_j - \frac{Aj}{\wt b^n}|\to 0 \text{ in probability as } n\to\infty.\eqe
    Likewise, let $\wt s^n_1\leq\wt s^n_2\leq...\leq \wt s^n_{\wt a^n}$ be the re-ordering of the set of times $\frt^n_v$. Then
     \eqb\label{eq:lem:disk-pf-3-2}\sup_{1\leq j\leq \wt a^n} |\wt s^n_j - \frac{Aj}{\wt a^n}|\to 0 \text{ in probability as } n\to\infty.\eqe
\end{lemma}

\begin{proof}
    The convergence in~\eqref{eq:lem:disk-pf-3-1} is a direct consequence of part (\ref{cond:lem:disk-pf-0-(vi)}) of Lemma~\ref{lem:disk-pf-0}. For~\eqref{eq:lem:disk-pf-3-2}, observe that for $1\leq j\leq \wt b^n$, the number of downward jump times between $[0,\wt t^n_j]$ is equal to 
    $\frac{1}{2}({j+\ceil{\sqrt{n}}-1 -\wt\cL^n_{\wt t^n_j}-\wt \cR^n_{\wt t^n_j}})$. By part (\ref{cond:lem:disk-pf-0-(v)}) of Lemma~\ref{lem:disk-pf-0}, with probability tending to 1 as $n\to\infty$, $\sup_{s\in [0,A]}|\wt\cL_s^n|\leq n^{2/3}$ and $\sup_{s\in [0,A]}|\wt\cR_s^n|\leq n^{2/3}$ and therefore  with probability tending to 1 as $n\to\infty$, for every $j=1,...,\wt b^n$, $\wt s^n_{\lfloor j/2-n^{3/4}\rfloor}\leq \wt t_j^n\leq \wt s^n_{\lfloor j/2+n^{3/4}\rfloor}$ (we set $\wt s^n_{-k}=0$ for $k\geq0$ and $\wt s^n_{k}=A$ for $k>\wt a^n$). This along with~\eqref{eq:lem:disk-pf-3-1} implies~\eqref{eq:lem:disk-pf-3-2}.
\end{proof}

\begin{proof}[Proof of Theorem~\ref{thm:pack-disk}]
    The first three claims follows from parts (\ref{cond:lem:disk-pf-0-(iii)}) and (\ref{cond:lem:disk-pf-0-(v)}) of Lemma~\ref{lem:disk-pf-0}, parts (\ref{cond:lem:disk-pf-2-(i)}) and (\ref{cond:lem:disk-pf-2-(ii)}) of Lemma~\ref{lem:disk-pf-2}, where we repameterize the walk $\wt\cZ^n$ the curve $\wt\frtr^{\disk,n}$ in~\eqref{eq:lem:disk-pf-2} using counting measure on faces of $\cF\wt\cT^{\disk,n}$ and apply~\eqref{eq:lem:disk-pf-3-1} along with the continuity of the curve $\eta^\disk$ and the process $Z$. The convergence of the measure $\nu^{\disk,n}$ follows from  part (\ref{cond:lem:disk-pf-0-(iii)}) of Lemma~\ref{lem:disk-pf-0} and part (\ref{cond:lem:disk-pf-2-(iii)}) of Lemma~\ref{lem:disk-pf-2}.  {Note that we used the fact that the process $Z$ encodes the same LQG surface as $Z^\disk$ since we scaled the definition of the LQG length measure by a factor $\frac{1}{2}$ in Section~\ref{subsubsec:MatedCRT}, and the measure $\nu_{h^{\disk}}$ in Theorem~\ref{thm:pack-disk} is 2 times the measure $\nu_{h^\disk}$ in Lemma~\ref{lem:disk-pf-2}.}  To prove the convergence of the measure $\mu^{\disk,n}$, let $\wt\mu^{\disk,n}$ be defined analogously using the circle packing $\wt\cP^{\disk,n}$. Let $\wt\mu^{\disk,n}_{\MOT}$ be the measure on $\ol \bbD$ that assigns mass $\frac{1}{n}$ to each point $\eta^\disk(\frt_v^n)$. Then parts (\ref{cond:lem:disk-pf-2-(i)}) and (\ref{cond:lem:disk-pf-2-(ii)}) of Lemma~\ref{lem:disk-pf-2} implies that the Prokhorov distance between  $\wt\mu^{\disk,n}$ and $\wt\mu^{\disk,n}_{\MOT}$ tends to 0 in probability as $n\to\infty$, while part (\ref{cond:lem:disk-pf-0-(vi)}) of Lemma~\ref{lem:disk-pf-0} and~\eqref{eq:lem:disk-pf-3-2} implies that the Prokhorov distance between $2\wt\mu^{\disk,n}_{\MOT}$ and $\mu_{h^\disk}$ tends to 0 in probability as $n\to\infty$. Therefore the convergence of the measure $\mu^{\disk,n}$ further follows from part (\ref{cond:lem:disk-pf-0-(iii)}) of Lemma~\ref{lem:disk-pf-0}.
\end{proof}

\subsection{The sphere topology case}\label{subsec:pack-sph-pf}
 
In this subsection we prove Theorem~\ref{thm:pack-sph}. Recall the bijection $\Phi_{\mathrm{M}}^\sph$ defined in Proposition~\ref{prop:Mullin-sph} and the operator $\Phi_\rmT^\sph$ defined for walks $\cZ^\sph\in \mathcal{A}_n^\sph$. For $(T^\sph,\frtr^\sph) = \Phi_{\rmT}^\sph(\cZ)$, recall that to define the map $\Phi_{\rm simple}(T^\sph)$, we assume that for each multiple edge $(e_1,...,e_k)$ connecting vertices $v_1$ and $v_2$, there is a connected component of $T^\sph\backslash (e_1\cup...\cup e_k)$ containing at least $3n/2$ faces, and we contract the other connected components. On the walk $(\cZ_k)_{k=0,...,2n}$ side, we require that for each pair $t^-<t^+$ such that 
\eqb\label{eq:sph-double-edge} \cL_{t^-}=\cL_{t^+} = \inf_{t^-\leq s\leq t^+}\cL_s, \ \ \cR_{t^-}=\cR_{t^+} = \inf_{t^-\leq s\leq t^+}\cR_s,\eqe
the interval $[t^-,t^+]$ has length either at least $3n/2$ or at most $n/2$. In the latter case, we erase the walk on $[t^-,t^+]$ by setting  $\cZ_s = \cZ_{t^-}$ for $s\in [t^-,t^+]$, while in the former case, we erase the walk on $[0,2n]\backslash [t^-,t^+]$ by setting $\cZ_s = \cZ_{t^+}$ for $s\in [0,2n]\backslash [t^-,t^+]$. %Also recall that as discussed in Section~\ref{subsubsec:MatedCRT}, according to~\cite{MS-mot-finite}, for $a>0$, a Brownian cone excursion $Z=(L_t,R_t)_{t\in[0,a]}$   on $\bbR_+^2$ of duration $a$ generates a space-filling $\SLE_8$ curve decorated $\sqrt{2}$-quantum sphere of area  $a$. 

The proof of Theorem~\ref{thm:pack-sph} is in a similar fashion as that of Theorem~\ref{thm:pack-disk}. We start with the rate $2n$ walks $\wt\cZ^n:[0,1]\to\bbZ^2$ with Poissonian jumps, and use local absolute continuity to transfer from the setting of Lemma~\ref{lem:cZ^n-pack-Q-n} to the case where $ \wt\cZ^n$ is conditioned to start and end at 0 and stay nonnegative. This is done in  Lemma~\ref{lem:sph-pack-pf-1}, where we construct a coupling $(\wt\cZ^n,Z)$ with $Z$ being the Brownian excursion and  circle packing converges to mating of trees for $(\wt\cZ^n|_{[\delta ,(1-\delta)]},Z|_{[\delta,(1-\delta)]})$ for each fixed $\delta$. Next we use Theorem~\ref{thm:He-Schramm-sphere} to extend the above convergence of circle packing to the whole interval $[0,1]$, which is done in Lemma~\ref{lem:sph-pack-pf-4}. This further implies Theorem~\ref{thm:pack-sph} by further reparameterizing $\wt\cZ^n$ to have fixed jumps.  %We further introduce a random area $A\sim\mathrm{Unif}([1/2,1])$. In Lemma~\ref{lem:sph-pack-pf-2}, we prove that the total variation distance between the walk $\cZ^n$ with $a$ replaced by random $A$ and the walk $\wt\cZ^n$ which makes $2\lfloor An/2\rfloor$ jumps tends to 0 as $n\to\infty$. Then using similar methods as in Section~\ref{subsec:pack-disk-pf}, in Lemma~\ref{lem:sph-pack-pf-4}, we use Theorem~\ref{thm:He-Schramm-sphere} to show that Theorem~\ref{thm:pack-sph} holds for the triangulation encoded by $\wt\cZ^n$, and further in Lemma~\ref{lem:sph-pack-pf-6} we prove that Theorem~\ref{thm:pack-sph} holds $\lfloor An/2\rfloor$ in place of $n$. Since $A\in [1/2,1]$, further using the Brownian scaling, the convergence in  Lemma~\ref{lem:sph-pack-pf-6} implies the convergence in Theorem~\ref{thm:pack-sph}.

\begin{lemma}\label{lem:sph-double-edge-1}
Let   $(\cZ_t^{\sph,n})_{t=0,...,2n}=(\cL_t^{\sph,n},\cR_t^{\sph,n})_{t=0,...,2n}$ be a uniform element chosen from $\mathcal{A}_n^\sph$. Then for some constant $c>0$, with probability at least $1-cn^{-1/4}$, there does not exist any  $[t^-,t^+]$ satisfying~\eqref{eq:sph-double-edge} such that $t^+-t^-\in [n^{3/4},2n-n^{3/4}]$. Furthermore, with probability at least $1-c2^{-c^{-1}n^{1/4}}$, there does not exist any $0\leq j\leq 2n-n^{3/4}$, such that $\cL^{\sph,n}_j\leq \cL_{j+1}^{\sph,n}\leq...\leq \cL^{\sph,n}_{j+\lfloor n^{3/4}\rfloor}$.
\end{lemma}

\begin{proof}
    Let $0\leq t^-<t^+\leq 2n$ where $t^+-t^-\in [n^{3/4},2n-n^{3/4}]$. Then by Lemma~\ref{lem:walk-excursion}, the probability that~\eqref{eq:sph-double-edge} holds for $t^-,t^+$ is equal to $$\frac{\#\mathcal A^\sph_{(t^+-t^-)/2}\#\mathcal A^\sph_{n-(t^+-t^-)/2}}{\#\mathcal{A}^\sph_n} = C(1+o_n(1))\frac{n^3}{(t^+-t^-)^3(n-(t^+-t^-)/2)^3}\leq 400C(1+o_n(1)) n^{-9/4}.$$
    The first claim then follows by summing over all possible  $0\leq t^-<t^+\leq 2n$ such that $t^+-t^-\in [n^{3/4},2n-n^{3/4}]$. For the second claim, by a union bound, the probability for the event in the claim is at least $1-2n\cdot3^{n^{3/4}}4^{2n+1-n^{3/4}}(\#\cA_n^\sph)^{-1}\geq 1-c2^{-c^{-1}n^{1/4}}$.
\end{proof}

\begin{lemma}\label{lem:417}
    Let $(\cZ_t^{\sph,n})_{t=0,...,2n}$ be a uniform element chosen from $\mathcal{A}_n^\sph$. Let $\wt B^{n}$ be the set of   $2n$ independent points on $[0,1]$ that are sampled   from  the Lebesgue measure and independent from $\cZ^{\sph,n}$,   and for $t\in[0,1]$, let $\wt\cZ^n_t = \cZ_{\#([0,t]\cap \wt B^n)}^{\sph,n}$. Let $\check{\cZ^n}$ be  an unconditioned rate $2n$ continuous-time nearest-neighbor walk on $\bbZ^2$ started from 0.  Fix $a,b>0$ such that $a+b<1$. Then the law of $(\wt\cZ^n_{t+a}-\wt\cZ^n_a)_{0\leq t\leq 1-b-a}$ is absolutely continuous with respect to the law of $(\check{\cZ^n_t})_{0\leq t\leq 1-b-a}$  and the Radon-Nikodym derivative (depending on $a,b$) is bounded uniformly in $n$.
\end{lemma}

\begin{proof}
    Let $(\cX_j,\cY_j)_{j\geq0}$ be a standard simple random random walk on $\bbZ^2$, and write $$q_m(x,y) = \bbP\big[(\cX_m,\cY_m) = y, \cX_j,\cY_j\geq 0\text{ for }j=0,...,m\big|(\cX_0,\cY_0)=x\big].$$
    For the walk $(\check{\cZ^n_t})_{0\leq t\leq 1}$, consider the event $E_n$ that $\check{\cZ^n_0}=\check{\cZ^n_1} = 0$, $\check{\cZ^n_t}\in[0,\infty)^2$ for $t\in [0,1]$ and there are $2n$ jumps in total.  Let $p_\lambda(j) = \frac{\lambda^j}{j!}e^{-\lambda}$. By Lemma~\ref{lem:walk-excursion}, we have $\bbP[E_n]=p_{2n}(2n)q_{2n}(0,0) = p_{2n}(2n)4^{-2n}\#\cA_n^\sph = c(1+o_n(1))n^{-7/2}$. Then conditioned on $E_n$,  $(\check{\cZ^n_t})_{0\leq t\leq 1}$  has the same law as $(\wt\cZ^n_t)_{t\in [0,1]}$. Then for any trace $w:[0,1-b-a]\to \bbZ^2$ with $w(1-b-a)=y\in\bbZ^2$ and $r$ jumps, the Radon-Nikodym derivative $X_{a,b}$ between the laws of  $(\wt\cZ^n_{t+a}-\wt\cZ^n_a)_{0\leq t\leq 1-b-a}$ and the unconditional law of  $(\check{\cZ^n_t})_{0\leq t\leq 1-b-a}$ is
    \begin{equation}\label{eq:pf-417-1}
        X_{a,b}(w) = \frac{1}{p_{2n}(2n)q_{2n}(0,0)}\Big(\sum_{ i+j = 2n-r} p_{2na}(i)p_{2nb}(j) \cdot\big(\sum_{x\in\bbZ^2}q_i(0,x)q_j(x+y,0)\mathds{1}_{w(t)+x\in[0,\infty)^2 \forall t\in [0,1-a-b]} \big)\Big).
    \end{equation}
    We have $\sum_{ i+j = 2n-r} p_{2na}(i)p_{2nb}(j) = p_{2n(a+b)}(2n-r)\leq c_{a,b}n^{-1/2}$.
    By Cauchy-Schwarz inequality, 
    $$ \Big(\sum_{x\in\bbZ^2}q_i(0,x)q_j(x+y,0)\Big)^2\leq \Big(\sum_{x\in\bbZ^2}q_i(0,x)^2\Big)\Big(\sum_{x\in\bbZ^2}q_j(x+y,0)^2\Big) = q_{2i}(0,0)q_{2j}(0,0).$$
    We have $p_{2an}(i)\leq c e^{-c'n}$ if $i<an$ and $p_{2bn}(j)\leq c e^{-c'n}$ if $j<bn$   for some constant $c>0$. For $i\geq an$ and $j\geq bn$, we have $q_{2i}(0,0)q_{2j}(0,0)\leq cn^{-6}$. Therefore the Radon-Nikodym derivative $X_{a,b}(w)$ is bounded from above by some constant $c$ depending only on $a,b$. This finishes the proof.
\end{proof}

\begin{lemma}\label{lem:sph-pack-pf-1}
   There is a coupling of $(\wt\cZ^n,Z)$ such that 
    \begin{enumerate}[(i)]
        \item $\wt\cZ^n:[0,1]\to\bbZ^2$ has the same law as the process $\wt\cZ^n$ in Lemma~\ref{lem:417};\label{cond:lem:sph-pack-pf-0-(i)}
         \item $Z:[0,1]\to\bbR_+^2$ has the law of $\frac{1}{2}$ times a Brownian cone excursion of duration $1$ in the first quadrant from $(0,0)$ to $(0,0)$ of covariance~\eqref{eq:variance};\label{cond:lem:sph-pack-pf-0-(ii)}
        \item for any $\delta\in(0,0.01)$, circle packing converges to mating of trees for $(\wt\cZ^n|_{[\delta ,(1-\delta)]},Z|_{[\delta,(1-\delta)]})$ with $\nu(n)=(2n)^{-1/2}$; \label{cond:lem:sph-pack-pf-0-(iii)}
        \item $\|(Z-(2n)^{-1/2}\wt\cZ^n)\|_\infty\to0 $ in probability as $n\to\infty$. \label{cond:lem:sph-pack-pf-0-(iv)}
    \end{enumerate}
\end{lemma}

\begin{proof}
   Let $\wt \cZ^n$ and $\cZ^{\sph,n}$ be the processes in Lemma~\ref{lem:417}. By classical results (e.g.~\cite{DIM77}), there exists a coupling $(\cZ^{\sph,n},Z)_{n\in\bbN}$ where $\cZ^{\sph,n}$ is  uniform from $\cA_n^\sph$ and $Z$ is the Brownian excursion as in claim (ii) such that $\sup_{t\in[0,1]}|(Z_t-(2n)^{-1/2}\cZ^{\sph,n}_{\lfloor 2nt \rfloor})|\to 0$ in probability, which further implies that $\|(Z-(2n)^{-1/2}\wt\cZ^n)\|_\infty\to0 $ in probability as $n\to\infty$.  Let $(\check{Z}, \check{\cZ^1}, \check{\cZ^2}, \check{\cZ^3},\cdots)$ be the processes $(Z,\cZ^2,\cZ^4,\cZ^6,\cdots)$ as in Lemma~\ref{lem:cZ^n-pack-Q-n}. Fix $\delta>0$.  Let $(Z^n,\cZ^n)$ be $(\check Z,\check{\cZ^n})$ whose law is weighted by the Radon-Nikodym derivative in Lemma~\ref{lem:417} such that $(\cZ^n_t-\cZ_{\delta}^n)_{t\in[\delta,1-\delta]}$ has the same law as $(\wt\cZ^n_t-\wt\cZ_{\delta}^n)_{t\in[\delta,1-\delta]}$. We may further couple $(Z^n,\cZ^n,\wt\cZ^n)_{n\in\bbN}$  such that $(\wt\cZ^n_t-\wt\cZ_{\delta}^n)_{t\in[\delta,1-\delta]} = (\cZ^n_t-\cZ_{\delta}^n)_{t\in[\delta,1-\delta]}$, and by Lemma~\ref{lem:cZ^n-pack-Q-n} and Lemma~\ref{lem:417}, %for each deterministic $\delta\in (0,1)$, 
   $\|\big((2n)^{-1/2}\wt\cZ^n-Z^n - ((2n)^{-1/2}\wt\cZ^n_\delta-Z^n_\delta )\big)|_{[\delta,1-\delta]}\|_\infty\to 0$ in probability. % and circle packing converges to mating of trees for $(\wt\cZ^n|_{[\delta,1-\delta]},Z^n|_{[\delta,1-\delta]})$. 
   
    By Lemma~\ref{lem:417}, for each deterministic $\delta\in(0,1)$, the law of $Z^n|_{[\delta,1-\delta]}$ is absolutely continuous with respect to that of $\check{Z}|_{[\delta,1-\delta]}$ with Radon-Nikodym derivative bounded by a constant $C_{\delta}$, and the claim on $\|\big((2n)^{-1/2}\wt\cZ^n-Z^n - ((2n)^{-1/2}\wt\cZ^n_\delta-Z^n_\delta )\big)|_{[\delta,1-\delta]}\|_\infty\to 0$   and  $\|(Z-(2n)^{-1/2}\wt\cZ^n)\|_\infty\to0 $ in probability imply that $(Z_{t}^n-Z_\delta^n)_{t\in [\delta,1-\delta]}$ converges to $(Z_{t}-Z_\delta)_{t\in [\delta,1-\delta]}$ in probability under the uniform topology.  Now consider the metric $\mathfrak{d}$, the functions  $\wt F^n$, $\wt F^\infty$ and $\wt F_0^\infty$ defined below Definition~\ref{def:CPtoMOT}.  Then the bounded Radon-Nikodym derivative and the convergence of  $(Z_{t}^n-Z_\delta^n)_{t\in [\delta,1-\delta]}$  to $(Z_{t}-Z_\delta)_{t\in [\delta,1-\delta]}$ in probability implies that for each deterministic $a,b>0$ with $a+b<1$, \eqb\label{eq:pf-418-00}
    \mathfrak{d}\big(\wt F_0^\infty(Z^n|_{[\delta,1-\delta]},a,b), \wt F_0^\infty(Z|_{[\delta,1-\delta]},a,b)\big)\to 0 \  \text{ in probability as }n\to\infty. \eqe To see this, by Lusin's theorem, for any $\zeta>0$, $(Z-Z_\delta)|_{[\delta,1-\delta]}\mapsto \wt F_0^\infty(Z |_{[\delta,1-\delta]},a,b)$ is continuous on a compact set $K_\zeta\subseteq C([\delta,1-\delta],\bbR^2)$ such that $\mu(K_\zeta^c)<\zeta$, where $\mu$ is the law of the unconditioned Brownian motion $(\check{Z}-\check{Z}_\delta)|_{[\delta,1-\delta]}$ and the SLE curve $\eta_{n,[\delta,1-\delta]}$ constructed by $Z^n|_{[\delta,1-\delta]}$ is defined to be constant over $[0,\delta]$ and $[1-\delta,1]$. Here we used the fact  that $Z|_{[\delta,1-\delta]}$ and $(Z-Z_\delta)|_{[\delta,1-\delta]}$ encodes the same quantum surface.  For any $\vartheta>0$,  there exists $\vartheta_1>0$ such that if $\|((Z^n-Z^n_\delta)-(Z-Z_\delta))|_{[\delta,1-\delta]}\|_\infty<\vartheta_1$ and  both $(Z-Z_\delta)|_{[\delta,1-\delta]},(Z^n-Z^n_\delta)|_{[\delta,1-\delta]}$ are in $K_\zeta$, then $\mathfrak{d}\big(\wt F_0^\infty(Z^n|_{[\delta,1-\delta]},a,b),\wt F_0^\infty(Z|_{[\delta,1-\delta]},a,b)\big)<\vartheta$.  Then  
    \begin{equation}\label{eq:418-pf-1}
    \begin{split}
              &\bbP\big[\mathfrak{d}\big(\wt F_0^\infty(Z^n|_{[\delta,1-\delta]},a,b),\wt F_0^\infty(Z|_{[\delta,1-\delta]},a,b)\big)\geq\vartheta \big]
              \leq \bbP\big[\|((Z^n-Z^n_\delta)-(Z-Z_\delta))|_{[\delta,1-\delta]}\|_\infty\geq\vartheta_1\big] \\&+ \bbP[(Z^n-Z^n_\delta)|_{[\delta,1-\delta]}\notin K_\zeta]+\bbP[(Z-Z_\delta)|_{[\delta,1-\delta]}\notin K_\zeta].
    \end{split}
    \end{equation} 
    For fixed $\zeta$, the first term on the right hand side of~\eqref{eq:418-pf-1} tends to 0, the second   tends to 0 as $\zeta\to 0$ uniformly over $n$ by the boundedness of Radon-Nikodym derivative from Lemma~\ref{lem:417}, and the third term tends to 0 as $\zeta\to 0$ since the law of $(Z-Z_\delta)|_{[\delta,1-\delta]}$ is absolutely continuous with respect to the unconditioned version. This gives the desired convergence in probability.

    Now using the convergence of circle packing to mating of trees for $(\check{\cZ}^n,\check{Z})$ as in~\eqref{eq:conv-dist-check}, we have 
    $$\mathfrak{d}\big(\wt F^\infty(\check{Z}|_{[\delta,1-\delta]},\check{\cZ^n}|_{[\delta,1-\delta]},a,b),\wt F^n(\check{\cZ^n}|_{[\delta,1-\delta]},a,b)\big)\to 0 \text{ in probability as }n\to\infty. $$
    Further using the continuity of $\eta$ and the convergence of $|I'|\to |I_{a,b}|$ as in Condition~\ref{cond:def:CPtoMOT-0},  $\partial\eta(I')$ is close to $\partial\eta(I_{a,b})$ for large $n$ under the metric~\eqref{eq:curve-metric-0}, and therefore one can further use Proposition~\ref{prop:boundary-perturb} to prove that
     $$\mathfrak{d}\big(\wt F^\infty(\check{Z}|_{[\delta,1-\delta]},\check{\cZ^n}|_{[\delta,1-\delta]},a,b),\wt F^\infty_0(\check{Z}|_{[\delta,1-\delta]}, a,b)\big)\to 0 \text{ in probability as }n\to\infty. $$
     %If we choose $\e(n)\to 0$ sufficiently slowly as $n\to\infty$, then 
     By Lemma~\ref{lem:417}, 
     \begin{equation}\label{eq:pf-418-001}
         \mathfrak{d}\big(\wt F^\infty_0( {Z}^n|_{[\delta,1-\delta]}, a,b), \wt F^n(\wt{\cZ}^n|_{[\delta,1-\delta]},a,b)\big)\to 0 \ \text{ in probability as }n\to\infty. 
     \end{equation}
     Therefore by~\eqref{eq:pf-418-00} and~\eqref{eq:pf-418-001}, 
     \begin{equation}\label{eq:pf-418-002}
         \mathfrak{d}\big(\wt F^\infty_0( {Z}|_{[\delta,1-\delta]}, a,b), \wt F^n(\wt{\cZ}^n|_{[\delta,1-\delta]},a,b)\big)\to 0 \ \text{ in probability as }n\to\infty. 
     \end{equation}
     Further using the continuity of $\eta^\sph$ and the convergence of the interval length coordinate for the $\wt F$-function and the inclusion $I'\subseteq I_{a,b}$ in Condition~\ref{cond:def:CPtoMOT-0} in Definition~\ref{def:CPtoMOT}, one can again use Proposition~\ref{prop:boundary-perturb} to check that $\mathfrak{d}\big(\wt F^\infty_0( {Z}|_{[\delta,1-\delta]}, a,b), \wt F^\infty( {Z}|_{[\delta,1-\delta]},{\wt\cZ^n}|_{[\delta,1-\delta]}, a,b)\big)\to 0$ in probability as $n\to\infty$. This implies that $$ \mathfrak{d}\big(\wt F^\infty( {Z}|_{[\delta,1-\delta]},{\wt\cZ^n}|_{[\delta,1-\delta]},a,b), \wt F^n(\wt{\cZ}^n|_{[\delta,1-\delta]},a,b)\big)\to 0 \ \text{ in probability as }n\to\infty. $$ This gives claim (iii) since the vanishing disk diameter in the packing and  the convergence of one space-filling trace $\frtr$ imply the convergence of all space-filling traces.
\end{proof}

Now we let $(\wt \cZ^n,  Z)$ be coupled as in Lemma~\ref{lem:sph-pack-pf-1}. Let $(\wh\bbC,  h^\sph,  \eta^\sph)$ be the space-filling $\SLE_8$ decorated unit area  quantum sphere generated by $ Z$ and embedded such that $\eta^\sph(1/2)=0$, $\eta^\sph(1/4)=1$ and $ \eta^\sph(0)=\infty$. Let $\wt B^n$ be the set of times at which $\wt \cZ^n$ jumps. Let $\Phi_\rmT^\sph(\wt\cZ^n) = (  T^{\sph,n},\wt\frtr^{\sph,n})$. Let $[\wt t^-_n,\wt t^+_n]$ be the intersection over all $[t^-,t^+]$  where $t^-<1/4$, $t^+>3/4$ such that
\eqb\label{eq:sph-double-edge-wt} \wt\cL^n_{t^-}=\wt\cL^n_{t^+} = \inf_{t^-\leq s\leq t^+}\wt\cL^n_s, \ \ \wt\cR^n_{t^-}=\wt\cR^n_{t^+} = \inf_{t^-\leq s\leq t^+}\wt\cR^n_s, \ \ \ t^-,t^+\in\wt B^n\cup\{0,1\} \eqe 
holds, and let $\mr B^n$ be $\wt B^n\cap (\wt t^-_n,\wt t_n^+]$ minus all possible $(t^-,t^+]\subsetneq  (\wt t^-_n,\wt t_n^+]$ such that~\eqref{eq:sph-double-edge-wt} holds. Let $ \wt B^{n,-}$ be the set of times in $  \wt B^n$ when $\wt \cL^n$ or $\wt\cR^n$ makes a downward jump, and let $\mr B^{n,-} = \wt B^{n,-}\cap \mr B^n$. Let $t\mapsto \frv_t^n$ be the map $B^{n,-}\to \cV   T^{\sph,n}$ as in Section~\ref{subsubsec:Mullin-finite}. Note that by our construction of $\wt \cZ^n$ in Lemma~\ref{lem:417}, $T^{\sph,n}$ has the same law as the triangulation $T^{\sph,n}$ from Theorem~\ref{thm:pack-sph}.

%The following is a straightforward consequence of Lemma~\ref{lem:sph-double-edge-1}.
\begin{lemma}\label{lem:sph-double-edge-2}
     With probability tending to 1 as $n\to\infty$, $[n^{-1/5},1-n^{-1/5}]\subseteq (\wt t^-_n,\wt t_n^+)$, and gaps between consecutive times in $\mr B^{n,-}$ are no larger than $n^{-1/5}$.
\end{lemma}

\begin{proof}
    Since $\wt B^n$ are $2n$ i.i.d.\ uniform points on $[0,1]$, $\bbP\big[\# \wt B^n\cap [0,n^{-1/5}) <2n^{3/4}, \ \ \# \wt B^n\cap (1-n^{-1/5},1] <2n^{3/4}\big]\to 0$ as $n\to\infty$, and the first claim follows from Lemma~\ref{lem:sph-double-edge-1}. For the second claim, by~\eqref{eq:pf-lem-Ex-111112} and a union bound, for  the standard nearest neighbhor walk $(\cX_j,\cY_j)_{0\leq j\leq 2n}$ started from 0, with   probability at least  $1-Cn^4e^{-cn^{0.01}}$, for any $0\leq i_1\leq i_2\leq 2n$, there does not exist $n^{0.01}$ consecutive upsteps during $[i_1,i_2]$ after we restrict to $(\cX_j,\cY_j)_{i_1\leq j\leq i_2}$ and perform the $\frac{\pi}{2}$-cone excursion erasure as in ~\eqref{eq:pi/2-excursion-AA}. Further using $\bbP[(\cX,\cY)\in \cA^{\sph}_n] = c_1(1+o_n(1))n^{-3}$ by Lemma~\ref{lem:walk-excursion}, the above claim about  consecutive upsteps  continues to hold with probability at least $1-Cn^7e^{-cn^{0.01}}$ as we condition on $(\cX,\cY)\in \cA^{\sph}_n$. The last claim then holds by Lemma~\ref{lem:sph-double-edge-1} since with probability tending to 1 as $n\to\infty$, each interval  $[x_1,x_2]\subseteq [0,1]$ of length $n^{-1/5}$ contains at least $n^{0.79}$ points in $\wt B^n$ and at least $n^{0.03}$ points in $\mr B^n$.
\end{proof}

 In the setting of Lemma~\ref{lem:sph-double-edge-2}, observe that if $s_1,s_2\in \mr B^n$, then $\Phi_{\rm simple}$ version image of the triangulation from $\Phi_\rmT(\wt\cZ^n|_{[s_1,s_2]})$ can be viewed as a submap of the $\Phi_{\rm simple}$ image of the triangulation $  T^{\sph,n}$ from $\Phi_{\rmT}^\sph(\wt\cZ^n)$. Recall that we work with circle packings for $ T^{\sph,n}$ by looking into the circle packings $\cP$ for $\Phi_{\rm simple}(T^{\sph,n})$ in the sphere $\wh\bbC$ and declaring that the disks for vertices in $\cV T^{\sph,n}\backslash\cV \Phi_{\rm simple}( T^{\sph,n})$ are tangent points of disks in $\cP$. Let  $t\mapsto \frv_t^n$ be the bijection between  $\wt B^{n,-}$ and vertices of $  T^{\sph,n}$ not on the root edge as described in Section~\ref{subsubsec:Mullin-finite}.

\begin{lemma}\label{lem:sph-pack-pf-4}
Let $\frv^{0,n} = \frv^n_{t^{0,n} }$, $\frv^{\mathrm{mid},n} = \frv^n_{t^{\mathrm{mid},n} }$ and     $\frv^{1/4,n} = \frv^n_{t^{1/4,n} }$, where $t^{0,n} = \inf\mr B^{n,-}$, $t^{\mathrm{mid},n}  = \inf\{s\geq 1/2:s\in \mr B^{n,-}\}$ and $t^{1/4,n}  = \inf\{s\geq 1/4:s\in \mr B^{n,-}\}$. Let $ \cP^{\sph,n}$ be a circle packing in $\wh\bbC$ for $ T^{\sph,n}$ such that the disks for  $\frv^{\mathrm{mid},n}$ and $\frv^{1/4,n}$ are centered at  0 and 1, respectively. Let $\wt\frtr^{\sph,n}:[0,1]\to \cF \wt T^{\sph,n}$ be the space-filling path on the faces of $ T^{\sph,n}$, and suppose $\infty$ is contained within the interstice between disks for the face $\mathfrak{f}^{n,0}$ on $\Phi_{\rm simple}(T^{\sph,n})$ containing $\frtr^{\sph,n}(t^{0,n})$.  Then we have the following:
 \begin{enumerate}[(i)]
        \item With probability tending to 1 as $n\to\infty$, $\Phi_{\rm simple}(T^{\sph,n})$ is non-empty and $ \cP^{\sph,n}$ is well-defined, and the maximal spherical diameter of disks in $ \cP^{\sph,n}$ tends to 0 in probability as $n\to\infty$;
        \item Let $\wt\frtr^{\sph,n}:[0,1]\to \wh\bbC$ be the path on faces of $ T^{\sph,n}$ under the embedding $ \cP^{\sph,n}$. Then
        \eqb\label{eq:lem:sph-pack-pf-4}\sup_{t\in [0,1]} d_{\sph}( \eta^\sph(t),\wt\frtr^{\sph,n}(t))\to 0 \ \ \ \text{in probability as } n\to\infty, \eqe
        where the convergence is uniform over all choices of the space-filling path $\wt\frtr^{\sph,n}$.
        %\item  Let $\wt\nu^{\disk,n}$ be the measure on $\partial\bbD$ that assigns mass $n^{-1/2}$ to each tangent point of disks in $\wt\cP^{\disk,n}$ and $\partial\bbD$. The Prohokov distance between $\wt\nu^{\disk,n}$ and $\nu_{h^\disk}$ tends to 0 in probability as $n\to\infty$.
    \end{enumerate}
%if we send $\e\to 0$ sufficiently slower than $n\to\infty$, the spherical diameter of the union of the disks for vertices that are adjacent to faces in $\cF \wt T^{\sph,n}\backslash\{\wt\frtr^{\sph,n}(t):t\in [\e,A-\e]\}$   tends to 0 in probability.
\end{lemma}

\begin{proof}
  The claim that  $\Phi_{\rm simple}(T^{\sph,n})$ is non-empty and $ \cP^{\sph,n}$ is well-defined with probability tending to 1 follows from Lemma~\ref{lem:sph-double-edge-2}. Let $\e>0$ be small and $n>\e^{-10}$. By the continuity of the curve $\eta^\sph$, as $\e\to0$, the spherical diameter of $\wh\bbC\backslash\eta^\sph([3\e,1-3\e])$ tends to 0 in probability. Let $t_\e^- = \inf\{s\geq \e:s\in \mr B^n\}$, $t_\e^+ = \sup\{s\leq 1-\e:s\in \mr B^n\}$ and $I_\e'=[t_\e^-,t_\e^+]$.  Let $  T_{n,I_\e'}$ be the map formed by the faces $\{\wt\frtr^{\sph,n}(t):t\in I'\}\backslash\{\wt\frtr^{\sph,n}(t_\e^-)\}$. Let $t_1^{0,n} = \inf\{u>t^{0,n}:u\in\mr B^{n,-}\}$.  By Lemma~\ref{lem:A2}, with probability tending to 1 as $n\to\infty$, $\mr B^{n,-}$ intersects both $[\e/2,\e]$ and $[1-\e,1-\e/2]$, and $t_1^{0,n}<\e/4$.  By Lemma~\ref{lem:sph-double-edge-2},  with probability tending to 1 as $n\to\infty$ for fixed $\e$,   the interval $I'_\e$ and the triangulation $  T_{n,I'_\e}$  here agree with the interval $I'$ and the disk triangulation $T_{n,I'}$ in Definition~\ref{def:CPtoMOT} defined for  the interval $I_{a,b}=[\e,1-\e]$, $I=[\e/2,1-\e/2]$ and $\wt\cZ^n|_I$, and the triangulation $\Phi_{\rm simple}(  T_{n,I_\e'})$  is a submap of the triangulation $\Phi_{\rm simple}(  T^{\sph,n})$ and does not contain $\mathfrak{f}^{n,0}$. Consider a conformal map $\varphi_{\e}: \eta^\sph(I')\to \bbD$ sending $\eta^\sph(1/2)=0$ to 0, and let $g_\e(z) =\varphi_\e(z)/\varphi_\e(1)$.   By endpoint continuity of $\eta^\sph$ and Lemma~\ref{lem:sph-double-edge-2},   for each $R>0$, with probability tending to 1 as $(\e,n)\to(0,\infty)$, $B(0;R)\subseteq \eta(I_\e')$, and therefore by Proposition~\ref{prop:unif-sph}, $\sup_{z\in \eta(I_\e')} d_\sph(g_\e(z),z)\to 0$  and $\varphi_\e(1)\to 0$ in probability as $\e,n\to(0,\infty)$.

    By part (\ref{cond:lem:sph-pack-pf-0-(iii)}) of Lemma~\ref{lem:sph-pack-pf-1}, for fixed $\e\in(0,0.01)$ and the circle packing $\cP_{n,I'_\e}$ of $T_{n,I_\e'}$ as in Definition~\ref{def:CPtoMOT} and its embedded trace $\frtr_{n,I_\e'}$, the maximal diameter of disks in  $\cP_{n,I'_\e}$ tends to 0, and up to a rotation, $\sup_{t\in I_\e'}|\frtr_{n,I_\e'}(t) - \varphi_\e(\eta^\sph(t))|\to 0$ in probability. In particular, $B(0;1/2)\subseteq \mathrm{carrier}(\cP_{n,I'_\e})$. We may further choose $\e = \e_n\downarrow 0$ sufficiently slow, such that $\Phi_{\rm simple}(  T_{n,I_{\e_n'}})$  is a submap of the triangulation $\Phi_{\rm simple}(  T^{\sph,n})$ and does not contain $\mathfrak{f}^{n,0}$, and as we further dilate  by $\varphi_{\e_n}(1)^{-1}$, the maximal Euclidean diameter of disks in  $\varphi_{\e_n}(1)^{-1}\cP_{n,I'_{\e_n}}$ tends to 0, and up to a rotation, $\sup_{t\in I_{\e_n}'}|\varphi_{\e_n}(1)^{-1}\frtr_{n,I_{\e_n}'}(t) - g_{\e_n}(\eta^\sph(t))|\to 0$ in probability as $n\to\infty$. We may further perform an affine transform  to obtain a circle packing $\cP^n$ of $T_{n,I_{\e'_n}}$ such that the disks for the vertices $\frv^{\mathrm{mid},n}$ and $\frv^{1/4,n} $ are centered at 0 and 1, respectively, and the maximal spherical diameter of disks as well as bounded interstices as well as the above convergence continue to hold in spherical metric. Since $B(0;1/2)\subseteq \mathrm{carrier}(\cP_{n,I'_\e})$ and $\varphi_{n_\e}(1)\to 0$ in probability, for each $r>0$, with probability tending to 1, $B(0;r)\subseteq \mathrm{carrier}(\cP^n)$. Let $\wt\cP^n$ be the restriction of the packing $\cP^{\sph,n}$ to vertices in $T_{n,I_{\e_n}'}$, which is a circle embedding for $T_{n,I_{\e_n}'}$ in $\bbC$ since $\Phi_{\rm simple}(T_{n,I_{\e_n}'})$ does not contain the face $\mathfrak{f}^{n,0}$. %By Lemma~\ref{lem:sph-double-edge-2}, with probability tending to 1, 
    By our choice of $\cP^{\sph,n}$, we may further fix a homeomorphism $f^n$ between carriers of $\cP^n$ and $\wt\cP^n$ that corresponds disks and $f^n(0)=0,f^n(1)=1$. Then $(\cP^n,\wt\cP^n)$ satisfies the constraints in Theorem~\ref{thm:He-Schramm-sphere}, the maximal spherical diameter of disks in $\wt\cP^n$ tends to 0 in probability as $n\to\infty$, and $f^n$ converges uniformly to the identity function in probability under the spherical metric. Furthermore, the conclusion that for every $r>0$,  $B(0;r)\subseteq \mathrm{carrier}(\wt\cP^n)$ for large enough $n$ implies that the disks for the vertices in $\cV T^{\sph,n}\backslash\cV T_{n,I_{\e_n}'}$ are located in an $o_n(1)$ spherical neighborhood of $\infty$ in probability as $n\to\infty$. This together with the endpoint continuity of $\eta^\sph$ gives claim (i) and claim (ii).
    %Then as $\e\to 0$, $|\varphi_{n,I'}(1)|\to 0$. Let $\cP_{n,I'}$ be the univalent circle packing in $\bbD$ for $T_{n,I'}$ as in Definition~\ref{def:CPtoMOT} for $(\wt\cZ_{n,I'},Z|_{I'})$, and we scale $\cP_{n,I'}$ to get $\wt\cP_{n,I'}$ such that the disk for $\frv^{1/4,n}$ is centered at 1. Therefore by the convergence of circle packing to mating of trees for $(\wt\cZ^n|_{[A\e,A(1-\e)]},\wt Z|_{[A\e,A(1-\e)]})$ from Lemma~\ref{lem:sph-pack-pf-2}, as we send $\e\to\infty$ sufficiently slower than $n\to\infty$, the spherical diameter of the union of the boundary disks as well as the maximal spherical diameter of disks  in $\wt\cP_{n,I'}$   tends to 0 in probability. This verifies the conditions of Theorem~\ref{thm:He-Schramm-sphere} for the circle packing $\wt\cP_{n,I'}$ and the restriction of the circle packing $\wt\cP^{\sph,n}$ restricted to the vertices in  $T_{n,I'}$. The claim then follows from Theorem~\ref{thm:He-Schramm-sphere}. 
\end{proof}

% \begin{lemma}\label{lem:sph-pack-pf-4}
%     In the setting of Lemma~\ref{lem:sph-pack-pf-3}, we have the following:
%     \begin{enumerate}[(i)]
%         \item The maximal spherical diameter of disks in $\wt\cP^{\sph,n}$ tends to 0 in probability as $n\to\infty$;
%         \item Let $\wt\frtr^{\sph,n}:[0,A]\to \wh\bbC$ be the path on faces of $\wt T^{\sph,n}$ under the embedding $\wt\cP^{\sph,n}$. Then
%         \eqb\label{eq:lem:sph-pack-pf-4}\sup_{t\in [0,A]} d_{\sph}(\wt\eta^\sph(t),\wt\frtr^{\sph,n}(t))\to 0 \ \ \ \text{in probability as } n\to\infty. \eqe
%         %\item  Let $\wt\nu^{\disk,n}$ be the measure on $\partial\bbD$ that assigns mass $n^{-1/2}$ to each tangent point of disks in $\wt\cP^{\disk,n}$ and $\partial\bbD$. The Prohokov distance between $\wt\nu^{\disk,n}$ and $\nu_{h^\disk}$ tends to 0 in probability as $n\to\infty$.
%     \end{enumerate}
% \end{lemma}

% \begin{proof}
%     Let $\wt\cP_{n,I'}$ be the circle packing as in the proof of Lemma~\ref{lem:sph-pack-pf-3}. Using identical arguments as in the proof of Theorem~\ref{thm:circle-QS},  as we send $\e\to 0$ sufficiently slower than $n\to\infty$, for any homeomorphism $f^n$ of $\wh\bbC$ that takes disks in  $\wt\cP_{n,I'}$ to the corresponding disks in $\wt\cP^{\sph,n}$, $f^n$ converges uniformly to the identity on every $B(0;R)$. The claims then follows from Lemma~\ref{lem:sph-pack-pf-3}.
% \end{proof}

Let $v\mapsto \frt_v^n$ be the inverse of the map $t\mapsto \frv^n_t$, and for the two vertices on the root edge of $\wt T^{\sph,n}$, we set $\frt_v^n=1$.

\begin{lemma}\label{lem:sph-pack-pf-5}
    Let $\wt t^n_1\leq\wt t^n_2\leq...\leq \wt t^n_{2n}$ be the re-ordering of the set of times in $\wt B^n$ from the smallest to largest. Then 
    \eqb\label{eq:lem:sph-pf-5-1}\sup_{1\leq j\leq  2n} |\wt t^n_j - \frac{j}{2n}|\to 0 \text{ in probability as } n\to\infty.\eqe
    Likewise, let $\wt s^n_1\leq\wt s^n_2\leq...\leq \wt s^n_{n+2}$ be the re-ordering of the set of times $\frt_v^n$. Then
     \eqb\label{eq:lem:sph-pf-5-2}\sup_{1\leq j\leq n+2} |\wt s^n_j - \frac{j}{ n+2}|\to 0 \text{ in probability as } n\to\infty.\eqe
\end{lemma}
\begin{proof}
   The first part follows directly since the points in $\wt B^n$ are i.i.d.\ uniform on $[0,1]$. The proof for the second part is identical to the proof of Lemma~\ref{lem:disk-pf-3} using the description of the law of $\wt\cZ^n$.
\end{proof}

\iffalse
\begin{lemma}\label{lem:sph-pack-pf-6}
Let $A$ and $\wt Z$ be as in Lemma~\ref{lem:sph-pack-pf-2},  $(\wt h^\sph,\wt \eta^\sph)$ be the $\SLE_8$-decorated quantum sphere encoded by $\wt Z$, and $\wt a^n=\lfloor\frac{An}{2}\rfloor$. Let $(\cZ^{\sph,A,n}_t)_{t=0,...,2\wt a^n}$ be a walk chosen uniformly from the set $\mathcal{A}_{\wt a^n}^\sph$. Let $\Phi_\rmT^\sph(\cZ^{\sph,A,n}) = (T^{\sph,A,n},\frtr^{\sph,A,n})$, where we extend $\frtr^{\sph,A,n}$ such that it is defined on $[0,2\wt a^n]$. Then there exists a coupling of $((\cZ^{\sph,A,n})_{n\in\bbN},Z)$ such that there is a circle packing $\cP^{\sph,A,n}$ for $T^{\sph,A,n}$ in $\wh\bbC$ such that the following holds:
\begin{enumerate}
    \item The maximal spherical diameter of disks in $\cP^{\sph,A,n}$ tends to 0 in probability as $n\to\infty$;
    \item $\sup_{t\in [0,A]}|{n}^{-1/2}\cZ_{\lfloor 2\wt a^nt/A\rfloor}^{\sph,A,n}-\wt Z_t|$ converges to 0 in probability as $n\to\infty$;
    \item If we view $\frtr^{\sph,A,n}:[0,A]\to\wh\bbC$ as a path on the embedded faces of $T^{\sph,A,n}$, then
    $$\sup_{t\in[0,A]}d_{\sph}(\wt\eta^\sph(t),\frtr^{\sph,A,n}(2\wt a^nt/A))\to 0 \ \ \text{in probability as }n\to\infty; $$
    \item Let $\mu^{\sph,A,n}$ be the measure that assigns mass $\frac{1}{n}$ to the center of each disk, with duplication counted. Then the Prokhorov distance between  $2\mu^{\sph,A,n}$ and $\mu_{\wt h^\sph}$ tends to 0 in probability as $n\to\infty$. 
\end{enumerate}
\end{lemma}
\fi

\begin{proof}[Proof of Theorem~\ref{thm:pack-sph}]
    The first three claims follows from part (\ref{cond:lem:sph-pack-pf-0-(iv)}) of Lemma~\ref{lem:sph-pack-pf-1}, Lemma~\ref{lem:sph-pack-pf-4}, Lemma~\ref{lem:sph-pack-pf-5} and the continuity of $ \eta^\sph$ and $ Z$ by reparametrizing the walk $\wt\cZ^n$ and the  trace $\wt\frtr^{\sph,n}$ in Lemma~\ref{lem:sph-pack-pf-4} according to the number of jumps made by $\wt\cZ^n$. For the last claim, it follows from Lemma~\ref{lem:sph-pack-pf-5} that if we consider the measure $\wt\mu^{\sph,n}_\MOT$ where we assign mass $\frac{1}{n}$ to each point $\eta^\sph(\frt_v)$ in Lemma~\ref{lem:sph-pack-pf-5}, then the Prokhorov distance between $\wt\mu^{\sph,n}_\MOT$ and $\mu^{\sph,n}$ tends to 0 in probability as $n\to\infty$, while Lemma~\ref{lem:sph-pack-pf-5} along with the continuity of $\eta^\sph$ imply that the Prokhorov distance between $\wt\mu^{\sph,n}_\MOT$ and $\mu_{ h^\sph}$ tends to 0 in probability as $n\to\infty$. This verifies the last claim.
\end{proof}

% \begin{proof}[Proof of Theorem~\ref{thm:pack-sph}]
%     We start with the setting of Lemma~\ref{lem:sph-pack-pf-6}. Let $(Z_t)_{t\in[0,1]} = (A^{-1/2}\wt Z_{At})_{[0,1]}$, $(\wh\bbC,h^\sph,\eta^\sph)$ be the $\SLE_8$-decorated quantum sphere embedded as described in the statement. By our choice of embedding, for the quantum sphere generated by $Z$, we have $h^\sph = \wt h^\sph -\frac{1}{\sqrt 2}\log A$, $\mu_{h^\sph} = A^{-1}\mu_{\wt h^\sph}$ and $\eta^\sph(t) = \wt\eta^\sph(At)$, so $(h^\sph,\eta^\sph)$ is an $\SLE_8$-decorated  unit area quantum sphere as desired. By further removing the scaling of the LQG length measure as in Section~\ref{subsubsec:MatedCRT}, Lemma~\ref{lem:sph-pack-pf-6} implies that Theorem~\ref{thm:pack-sph} holds if one replaces every $n$ in the statement with $\lfloor \frac{An}{2}\rfloor$ where $A$ is a uniform random variable in $[1/2,1]$ independent from everything else. Since $1/2\leq A\leq 1$, we see that the convergence for $\lfloor \frac{An}{2}\rfloor$ in place of $n$ readily implies the convergence for $n$.
% \end{proof}

\subsection{The whole-plane topology case}\label{subsec:pack-cone-pf}

The goal of this subsection is to prove Theorem~\ref{thm:pack-cone}. We first prove the parabolicity in Lemma~\ref{lem:pf-pack-cone-1}. In Lemma~\ref{lem:pf-pack-cone-2}, we use Lemma~\ref{lem:cZ^n-pack-Q-n}, Theorem~\ref{thm:He-Schramm-sphere} along with the scale invariance of the quantum cone to prove  Theorem~\ref{thm:pack-cone} under the setting where $\cZ$ is a rate $n$ continuous-time nearest neighbor walk in $\bbZ^2$. Finally we prove Theorem~\ref{thm:pack-cone} by a reparameterization. 

\begin{lemma}\label{lem:pf-pack-cone-1}
    Let  $T$ be as in Theorem~\ref{thm:pack-cone}. Then $\Phi_{\rm simple}(T)$ is a.s.\ not empty. The simple random walk on $\Phi_{\rm simple}(T)$ is a.s.\ recurrent, and $T$ is circle packing parabolic. 
\end{lemma}

\begin{proof}
  The first claim is a consequence of the Mullin bijection description of $T$ along with Lemma~\ref{lem:walk-excursion} and Proposition~\ref{lem:gap-hat-J-Q} and the Borel-Contelli lemma.   Let $  T^n$ be a disk triangulation defined as follows. Let  $\Phi_\rmT(\cZ|_{[-n,n]}) = (  T^n, \wh{\frtr}^n)$. We uniformly randomly pick a face of $ T^n$, and we uniformly randomly pick a vertex $ {v}^n$ on this face and  declare that $ {v}^n$ is the root of $\wh T^n$. % Let $\check{v}^n$ be a vertex of $\wh T^n$ chosen uniformly. 
  We also uniformly randomly pick a vertex $v$ on the face $\frtr(0)$ of $T$.  By~\cite[Lemma 3.3]{GRRiemann}, the rooted disk triangulations $(  T^n, v^n)$ converges in distribution to $(T,v)$ in Benjamini-Schramm topology, and by~\cite[Lemma 3.4]{GRRiemann}, $\bbP[\deg(v^n)\geq M]\leq 2(3/4)^{M/4}$ for $M\in\bbN$ for every $n$. Now independently for each multiple edge $(e_1,...,e_k)$ connecting to vertices $v_1$ and $v_2$ on $ T^n$, we further uniformly randomly keep one of these edges and remove the remaining $k-1$ of those, which further turns $ T^n$ to a simple graph $\wh G^n$. Let $\wh{G}$ be the simple graph generated by $T$ in the same way. Then the convergence above of $(\wh G^n,\wh v^n)$ to $(\wh G,v)$ continues to hold, and $\bbP[\deg(v)\geq M]\leq 2(3/4)^{M/4}$ on $\wh G$ since $\wh{G}$ is obtained by removing only edges but not vertices from $T$. For a vertex $v_0$ of $T^n$, let $d(v_0)$ be its degree in $\wh G^n$ and $f(v_0)$ be the number of faces in $T^n$ incident to $v_0$ (excluding the outer face). Then Lemma~\ref{lem:Euler} implies that $f(v_0)\leq 2d(v_0)$. On the other hand, each triangular face of $T^n$ incident to $v_0$ uses two edges incident to $v_0$, and therefore $d(v_0)\leq 2f(v_0)$. This further implies that, if we sample a vertex $\check{v}^n$ on $\wh G^n$ with probability proportional to its degree, then the law of  $(\wh G^n,\wh v^n)$ and $(\wh G^n,\check v^n)$ are mutually absolutely continuous with Radon-Nikodym derivative in $[1/4,4]$. Furthermore, the law of $(\wh G^n,\check v^n)$ is tight and  any subsequential limiting rooted graph continues to have exponential tail for the degree of the root vertex, and its law is absolutely continuous with respect to $(\wh G,v)$ with Radon-Nikodym derivative in $[1/4,4]$. 
  Now~\cite[Theorem 1.1]{GNrecurrence} implies that the simple random walk on $\wh G$ is a.s.\ recurrent, and since $\Phi_{\rm simple}(T)$ can be viewed as a subgraph of $\wh G$, the simple  random walk on $\Phi_{\rm simple}(T)$ is a.s.\ recurrent. The last claim follows from~\cite[Theorem 1.1]{HeSchramm-hy-pa}.
\end{proof}

\begin{lemma}\label{lem:pf-pack-cone-2}
    Let $(\wt \cZ^n_t)_{t\in\bbR}$ be a rate $n$ continuous-time nearest-neighbor walk on $\bbZ^2$ such that $\wt \cZ^n_0=0$. Then there is a coupling for $(Z,(\wt\cZ^n)_{n\in\bbN})$ such that Assertions ~\ref{result:thm:pack-cone-1}-\ref{result:thm:pack-cone-3} in Theorem~\ref{thm:pack-cone} holds for $\wt\cZ^n$ in place of $\cZ^n$. Furthermore, for each $R>0$, with probability tending to 1 as $n\to\infty$ then $K\to\infty$, there does not exist any time $t\in \cJ^{n,-}$ such that $|t|\geq K$ and the disk for the vertex $\frv_{t}$ intersects $B(0;R)$.
\end{lemma}

\begin{proof}
    Let $\Phi_\rmT(\wt\cZ^n) = (\wt T^n,\wt\frtr^n)$. %Then $(\wt T^n,\wt\frtr^n)$ and $(T,\frtr)$ agrees in law up to a reparameterization of $\wt\frtr^n$, and we can couple them such that $\wt T^n$ and $T$ are isomorphic. 
    Since the convergence in Lemma~\ref{lem:cZ^n-pack-Q-n} is invariant if we shift and translate the walks, we see that the conclusions in  Lemma~\ref{lem:cZ^n-pack-Q-n} hold for  $(Z,(\wt\cZ^n)_{n\in\bbN})$ for the interval $[-2,2]$ in place of $[0,a]$. Let $\mr\cJ^{n,-}$ be the time set $\mr\cJ^-$ defined for the walk $\wt\cZ^n$ as before Lemma~\ref{lem:A2}.   By Proposition~\ref{lem:gap-hat-J^-} and scaling, for some $\alpha\in (0,1/2)$, with probability tending to 1 as $n\to\infty$, every interval $J\subseteq [-2,2]$ of length at least $n^{-\alpha}$ contains at least one element of $\mr\cJ^{n,-}$.  Let $t_n^- = \inf\{s\geq -1:s\in \mr\cJ^{n}\}$, $t_n^+=\sup\{s\leq 1:s\in \mr\cJ^{n}\}$ and $I_n'=[t_n^-,t_n^+]$. 

   Let $j\in\bbN$. By Lemma~\ref{lem:cZ^n-pack-Q-n}, circle packing converges to mating of trees for $(\wt\cZ^{jn}|_{[-2,2]},Z|_{[-2,2]})$, and let $(\wt T^{n,j},\wt\frtr^{n,j}) = \Phi_\rmT(\wt\cZ^{jn}|_{I_{jn}'})$. Let $\frv^{n,j,0}$ (resp.\ $\frv^{n,j,1}$) be the vertex for the time $\inf(\mr\cJ^{jn,-}\cap\bbR_+)$ (resp.\ $\inf(\mr\cJ^{jn,-}\cap [1/j,\infty))$) in the map $\wt T^{n,j}$, $\cP^{n,j}$ be a circle packing in $\bbD$ for $\wt T^{n,j}$ centered at $\frv^{n,j,0}$ as in Definition~\ref{def:CPtoMOT}.    Let $\varphi_{n,j}:\eta(I_{jn}')\to\bbD$ (resp.\  $\varphi:\eta([-1,1])\to\bbD$) be the conformal map fixing 0 such that $\varphi_{n,j}'(0)>0$ (resp.\ $\varphi'(0)>0$).  Then it follows from Proposition~\ref{lem:gap-hat-J^-}, Definition~\ref{def:CPtoMOT} along with Proposition~\ref{prop:boundary-perturb}, up to a rotation of $\cP^{n,j}$, there exists $\e_{n,j}>0$ such that with probability tending to 1 as $n\to\infty$, the following holds:
   \begin{enumerate}[(i)]
      \item for each fixed $j$, $\e_{n,j}\to 0$ as $n\to\infty$;
       \item  the maximal diameter of disks in $\cP^{n,j}$ is less than $\e_{n,j}$;
       \item  for all choices of the embedded space-filling curve $\wt\frtr^{n,j}$ and $t\in I_{jn}'$,
       \begin{equation}\label{eq:CONE-1-new}
           |\varphi_{n,j}(\eta(t))-\wt\frtr^{n,j}(t)|\leq \e_{n,j} \text{ and } |\varphi_{n,j}(\eta(t))-\varphi(\eta(t))|<\e_{n,j}, % \text{ for $|t|<1/2$}
       \end{equation}
       \item $\sup_{|t|\leq 2}|(jn)^{-1/2}\wt\cZ^{jn}_t - Z_t|\to 0$ \text{in probability} as $n\to\infty$.
   \end{enumerate}
%   Note that the conditions above implies that the center of the disk for the vertex $\frv^{n,j,1}$ tends to $\varphi(\eta(1/j))$ in probability as $n\to\infty$ for fixed $j$.
    Now let $R^{n,j}>0$ and $\theta^{n,j}\in\bbR$, such that in the circle packing  $\wt\cP^{n,j}:= R^{n,j}e^{i\theta^{n,j}}\cP^{n,j}$,    the disks for $\frv^{n,j,0}$ and $\frv^{n,j,1}$ are centered at 0 and 1, respectively. Then following Proposition~\ref{lem:gap-hat-J^-} and the claims above, % for some $\e_{n,j}'$ such that $\e_{n,j}'\to 0$ in probability as $n\to\infty$ for each fixed $j$, 
    \eqb\label{pf:lem:pf-pack-cone-2-N1} |(R^{n,j})^{-1}e^{-i\theta^{n,j}} - \varphi(\eta(1/j))|\to 0 \text{ in probability as }n\to\infty \text{ for fixed }j.  \eqe  Consider a circle packing $\wt\cP^{jn}$ for $\wt T^{jn}$ in $\bbC$ such that the disks for $\frv^{n,j,0}$ and $\frv^{n,j,1}$ are centered at 0 and 1, respectively. Note that by our definition of $I_n'$,  $\Phi_{\rm simple}(\wt T^{n,j})$ can be viewed as a submap of $\Phi_{\rm simple}(\wt T^{jn})$, and let $\check{\cP}^{jn}$ be the restriction of $\wt\cP^{jn}$ to the disks for vertices in $\Phi_{\rm simple}(\wt T^{n,j})$, which can be understood as a circle embedding for  $\Phi_{\rm simple}(\wt T^{n,j})$. Consider a homeomorphism $f^{n,j}$ between carriers of $\wt\cP^{n,j}$ and $\check{\cP}^{jn}$ that corresponds disks  and fixes 0 and 1. Choose a deterministic sequence $j=j(n)\to\infty$ sufficiently slow, such that  $j^{1/2}\sup_{|t|\leq 2}|(jn)^{-1/2}\wt\cZ^{jn}_t - Z_t|\to 0$ \text{in probability} as $n\to\infty$, and\eqb\label{eq:CONE-2-new}R_{n,j}\e_{n,j} \to 0, \ \frac{|(R^{n,j})^{-1}e^{-i\theta^{n,j}} - \varphi(\eta(1/j))|}{|\varphi(\eta(1/j))|}\to 0 \text{ for $j=j(n)$ in probability as }n\to\infty. \eqe
    By the continuity of the curve $\eta$, as $j\to\infty$, $\varphi(\eta(1/j))\to 0$ in probability, which along with~\eqref{pf:lem:pf-pack-cone-2-N1}  further imply that for every $R>0$, with probability tending to 1 as $n\to\infty$, $B(0;R)\subseteq \mathrm{carrier}(\wt\cP^{n,j})$. Therefore % as we send $j\to\infty$ sufficiently slower than $n\to\infty$, 
    the conditions in Theorem~\ref{thm:He-Schramm-sphere} holds for the circle embeddings $\wt\cP^{n,j}$ and  $\check\cP^{jn}$ as well as the map $f^{n,j}$. By Theorem~\ref{thm:He-Schramm-sphere}, for each $R>0$, with probability tending to 1 as $n\to\infty$,  $B(0;R)$ is within the carrier of $\check\cP^{jn}$, $\{f^{n,j}\}$ converges uniformly to the identity in probability under the spherical metric, i.e., we may assume for each $z$ in the carrier of $\wt\cP^{n,j}$, $d_{\sph}(f^{n,j}(z),z)\leq \wt\e_{n,j}$ where $\wt \e_{n,j}\to0 $ in probability as $n\to\infty$ for $j=j(n)$. %sufficiently slower than $n\to\infty$.

   Now we fix $k\in\bbN$. Using the scale invariance of $\eta$ in law as in Lemma~\ref{lem:Rb}, it is not hard to argue that $\{|\eta(1/j)|^{-1}\diam(\eta([-k/j,k/j])):j\in\bbN\}$ is tight. Then using $\varphi(z) =\varphi'(0)z+o(|z|)$ as $z\to 0$, we have
   \eqb\label{eq:CONE-3-new}\sup_{|t|\leq \frac{k}{j}}\big|\frac{\varphi(\eta(t))}{\varphi(\eta(1/j))} - \frac{ \eta(t)}{ \eta(1/j)} \big|\to 0 \ \ \text{ in probability as }j\to \infty.\eqe
% By sending $j\to\infty$ sufficiently slower than $n\to\infty$, it follows that in the circle packing $\wt\cP^{jn}$, the maximal diameter of disks for times in $[-k/j,k/j]$ tends to 0 and $\sup_{t\in [-k/j,k/j]} | \wt\frtr^{jn}(t)-e^{i\theta^{n,j}} R^{n,j}\varphi(\eta(t))|$ tends to 0 in probability. Combined with~\eqref{pf:lem:pf-pack-cone-2-N1} and $\varphi(\eta(t)) = (\varphi'(0)+\delta_j)\eta(t)$ for some $\delta_j\to0$ and $|t|\leq k/j$, we can pick some deterministic sequence $j(n)$ which tend to $\infty$ as $n\to\infty$, such that 
Further gathering~\eqref{eq:CONE-1-new}, ~\eqref{eq:CONE-2-new},~\eqref{eq:CONE-3-new} and the convergence of the map $f^{n,j}$, the maximal diameter of disks for times in $[-k/j,k/j]$ tends to 0 and
   \begin{equation}\label{pf:lem:pf-pack-cone-2-N2}
       \sup_{t\in [-\frac{k}{j(n)},\frac{k}{j(n)}]}|\wt\frtr^{n\cdot j(n)}(t) -  \frac{\eta(t)}{\eta(1/j(n))}| \to 0 \ \ \ \text{in probability as } n\to\infty.
   \end{equation}
  By the same argument, for fixed $R>0$ and each  disk $\check D_{{v}}^{jn}$ in $\check\cP^{jn}$  intersects $B(0;R)$, by Theorem~\ref{thm:He-Schramm-sphere}, its diameter tends to 0 in probability,  and  its corresponding disk $D_{{v}}^{jn}$ in $\wt\cP^{n,j}$ is contained in $B(0;2R)$ with probability tending to 1 as $n\to\infty$. If we write  $\check{o}_{jn}(v)$ for the centers of  $\check D_{{v}}^{jn}$ and $t_v^{jn}$ for the corresponding time in $\cJ^{jn,-}$, then using \eqb\label{eq:428-pf-a}\sup_{|z|\leq R|\eta(1/j)|}\big|\frac{\varphi(z)}{\varphi(\eta(1/j))} - \frac{ z}{ \eta(1/j)} \big|\to 0  \text{ in probability as }j\to \infty,\eqe we also have \eqb\label{eq:CONE-4-new}|\check{o}_{jn}(v) - \frac{\eta(t_v^{jn})}{\eta(1/j)}|\to 0\eqe in probability as $n\to\infty$ uniformly over all choices of $v$ such that $\check D_{{v}}^{jn}$  intersects $B(0;R)$. To see this, let $o_{jn}^0(v)$ be the center of the disk for $v$ in $\cP^{n,j}$. Then
 \begin{equation*}
     \begin{split}
         |\check{o}_{jn}(v) - \frac{\eta(t_v^{jn})}{\eta(1/j)}| &\leq |\check{o}_{jn}(v) - R^{n,j}e^{i\theta^{n,j}} o_{jn}^0(v)|+|(R^{n,j}e^{i\theta^{n,j}} - \varphi(\eta(1/j))^{-1}) o_{jn}^0(v)|\\&+ |\varphi(\eta(1/j))^{-1}({o}_{jn}^0(v)-\varphi(\eta(t_n^{jn})))|+ |\frac{\varphi(\eta(t_n^{jn}))}{\varphi(\eta(1/j))} -\frac{ \eta(t_v^{jn})}{\eta(1/j)}|,
     \end{split}
 \end{equation*}
 and the four terms on the right hand side all  converge to 0 in probability following~\eqref{eq:CONE-1-new},~\eqref{eq:CONE-2-new} and ~\eqref{eq:428-pf-a} and the closeness between $\wt\cP^{n,j}$ and $\check{\cP}^{jn}$.

    Let   $(h^{1/j},\eta^{1/j})$ be defined by~\eqref{eq:hb-etab} for $b=1/j$, which is encoded by $\wh Z^{j,n}:=(j^{1/2}Z_{t/j})_{t\in\bbR}$. Let $\wh \cZ_t^{j,n} = \wt\cZ_{t/j}^{jn}$.  Then $\wh Z^{j,n}\overset{d}{=}Z$ and $\wh \cZ_t^{j,n}$ is a rate $n$ continuous-time nearest-neighbor walk on $\bbZ^2$. Furthermore, 
    $$\sup_{t\in [-k,k]}|\wh Z^{j,n}_t- n^{-1/2}\wh \cZ_t^{j,n}| = j^{1/2}\sup_{t\in [-k/j,k/j]}|Z_t -(nj)^{-1/2} \wt\cZ_t^{jn}|,$$
    which tends to 0 in probability as $n\to\infty$ by our choice of $j(n)\to\infty$. Observe that $\Phi_\rmT(\wh \cZ_t^{j,n}) = (\wt T^{jn},\wh\frtr^{j,n})$ where $\wh\frtr^{j,n}(t) = \wt\frtr^{jn}(t/j)$. Therefore~\eqref{pf:lem:pf-pack-cone-2-N2} implies that under the circle packing $\wt\cP^{jn}$ for $\wt T^{jn}$,
    \begin{equation}\label{pf:lem:pf-pack-cone-2-N3}
       \sup_{t\in [-k,k]}|\wh\frtr^{j,n}(t) -  \eta^{1/j}(1)^{-1}\cdot\eta^{1/j}(t)| \to 0 \ \ \ \text{in probability as } n\to\infty,
   \end{equation}
and the maximal diameter for disks for the vertices for the times in $[-k,k]$ tends to 0 in probability as $n\to\infty$. We further scale the circle packing $\wt\cP^{jn}$ by $\eta^{1/j}(1)$, and this finishes the proof of Assertions ~\ref{result:thm:pack-cone-2}-\ref{result:thm:pack-cone-3} since $k$ is arbitrary and for every $j$, $\wh Z^{j,n}\overset{d}{=} Z$ and $ \wh \cZ^{j,n}\overset{d}{=}\wt\cZ^n$. To be more precise, we may first sample a common two-sided Brownian motion $Z$ and resample $\wt\cZ^n$ according to the conditional law of $\wh \cZ^{j(n),n}$ given $\wh Z^{j(n),n} = Z$. Assertion ~\ref{result:thm:pack-cone-1} and the final claim follow from~\eqref{eq:CONE-4-new}.  %Moreover, for each given $R,\vartheta>0$, we may pick some deterministic $k_R<\infty$ such that $B(0;2R)\subseteq \eta([-k_R,k_R])$ with probability $1-\vartheta/3$. Assertion~\ref{result:thm:pack-cone-3} implies that for sufficiently large $n$, with probability at least $1-\vartheta/3$, 
  \end{proof}

We will also need the following   the analog of Lemma~\ref{lem:disk-pf-3} and Lemma~\ref{lem:sph-pack-pf-5}, which is straightforward since the times at which $\wt\cZ^n$ jumps is a Poisson point process on $\bbR$.
\begin{lemma}\label{lem:pf-pack-cone-3}
    Let $...<\wt t^n_{-2}<\wt t_{-1}^n<0$ and $0<\wt t_1^n<\wt t_2^n<...$ be the times at which $\wt\cZ^n$ jumps, and $...<\wt s^n_{-2}<\wt s_{-1}^n<0$ and $0<\wt s_1^n<\wt s_2^n<...$ be the times at which $\wt\cL^n$ or $\wt\cR^n$ makes a downward jump. Let $\wt t_0^n=\wt s_0^n=0$. Then for every fixed $k\in\bbN$,
    $$\max\big\{\sup_{-kn\leq j\leq kn}|\wt t_j^n-\frac{j}{n}|,\sup_{-kn\leq j\leq kn}|\wt s_j^n-\frac{2j}{n}|\big\} \to 0\text{ in probability as } n\to\infty.$$
\end{lemma}

\begin{proof}[Proof of Theorem~\ref{thm:pack-cone}]
    The claim that $T$ is a.s.\ circle packing parabolic follows from Lemma~\ref{lem:pf-pack-cone-1}. The convergence of the circle packing as in Assertions~\ref{result:thm:pack-cone-1}-\ref{result:thm:pack-cone-3} follows the analogous results from the walk $\wt\cZ^n$ in Lemma~\ref{lem:pf-pack-cone-2} and reparameterizing the walk $\wt\cZ^n_t$ and the curve $\wt\frtr^n$ by the number of jumps made by $\wt\cZ^n_t$, along with Lemma~\ref{lem:pf-pack-cone-3} and the continuity of $\eta$ and $Z$. For Assertion~\ref{result:thm:pack-cone-4}, by Lemma~\ref{lem:pf-pack-cone-2}, if we define the measure $\wt\mu^n_\MOT$ on $\bbC$ by assigning mass $\frac{1}{n}$ to each of $\eta(\wt s_j^n)$, then for each $k$, the Prokhorov distance between $\wt\mu^n_\MOT|_{B(0;R)}$ and $\mu^n|_{B(0;R)}$ tends to 0 in probability as $n\to\infty$, while by Lemma~\ref{lem:pf-pack-cone-3}, the  Prokhorov distance between $2\wt\mu^n_\MOT|_{B(0;R)}$ and $\mu_h|_{B(0;R)}$ tends to 0 in probability. This further implies the measure convergence in Assertion~\ref{result:thm:pack-cone-4}.  
\end{proof}

\section{Convergence of   tree-weighted maps under Riemann uniformization embedding}\label{sec:unif-pf}

In this section, we prove the convergence of tree-weighted maps  under the Riemann uniformization embedding. Much of the arguments are parallel to   Section~\ref{subsec:pack-pf} and we will be relatively brief here. %We start with the following results 

In Section~\ref{subsec:iid-unif}, we prove a version of Theorem~\ref{thm:unif-disk} for planar maps encoded by walks of finite steps. In Sections~\ref{subsec:unif-disk}-\ref{subsec:unif-cone}, we further apply Propositions~\ref{prop:boundary-perturb}-\ref{prop:unif-sph} to conclude the proof of Theorems~\ref{thm:unif-disk}-\ref{thm:unif-cone}.

\subsection{The i.i.d.\ walk case}\label{subsec:iid-unif}

  Let $I$ and $(\cZ^n,Z^n)_{n\in\bbN}$  be defined as in the paragraph above Definition~\ref{def:CPtoMOT}. Let $I'\subseteq I$. Let $\cJ^{n,-}_\rmL$ (resp.\ $\cJ^{n,-}_\rmR$) be the set of times when $\cZ^n$ makes a downward jump in its first (resp.\ second) coordinate.  Recall the triangulation $\Phi_\rmT(\cZ^n|_{I'}) = (T_{n,I'},\frtr_{n,I'})$ and the planar map $\Phi_\rmM(\cZ^n|_{I'}) = (M_{n,I'},\frT_{n,I'})$ encoded by $\cZ^n$ as in  Section~\ref{subsubsec:Mullin-finite}. We define a curve  $\check\frtr_{n,I'}$ as follows. Consider an embedding of $T_{n,I'}$. If the triangular face $\frtr_{n,I'}(t)$ on $T_{n,I'}$ has two vertices on the tree $\frT_{n,I'}$, then we put  $\check\frtr_{n,I'}(t)$ to be on the edge connecting these two vertices, whereas if $\frtr_{n,I'}(t)$ has only one vertex on $\frT_{n,I'}$, we put $\check\frtr_{n,I'}(t)$ to be at that vertex. We may further assume that $\check\frtr_{n,I'}:I'\to\bbC$ is continuous. In other words, $\check\frtr_{n,I'}$ is  a contour on the embedded tree $\frT_{n,I'}$ induced by $\frtr_{n,I'}$. Let $t_{\rm mid}$ be the midpoint of $I'$ and $t_{n,\rm mid} = \inf\{u\geq t_{\rm mid}:u\in \cJ^{n,-}_\rmL\}$. Let $\frv_{I'}^{\mathrm{mid},n}$ be the vertex for $t_{n,\rm mid}$ on the triangulation $T_{n,I'}$.

  Recall from Figure~\ref{fig:PHIM} that the map $M_{n,{I'}}$ might not have disk topology. If $t_{n,\rm mid} \in {I'}$ and $\frv_{I'}^{\mathrm{mid},n}$ is an interior vertex of $M_{n,{I'}}$,   let $\wh M_{n,{I'}}$ be the largest submap of $M_{n,{I'}}$ with disk topology and $\frv_{I'}^{\mathrm{mid},n}$ as an interior vertex. In other words, under an embedding of $M_{n,{I'}}$,   $\wh M_{n,{I'}}$ is the closure of the connected component of the interior of $M_{n,{I'}}$ (where we fill all the non-outer faces) containing $\frv_{I'}^{\mathrm{mid},n}$. If  $\frv_{I'}^{\mathrm{mid},n}$ is not well-defined or on the boundary of $M_{n,{I'}}$, we set $\wh M_{n,{I'}}=\emptyset$. For $t\in {I'}$, let $\wh t_{I'}^{n,\min}$ (resp.\ $\wh t_{I'}^{n,\max}$) for the smallest (resp.\ largest) time in ${I'}$ when $\check \frtr_{n,{I'}}(t)$ is on the map $\wh M_{n,{I'}}$. We further set $\wh\frtr_{n,{I'}}(t) = \check\frtr_{n,{I'}}(\wh t^{n,\min})$ for $t\leq \wh t_{I'}^{n,\min}$ and  $\wh\frtr_{n,{I'}}(t) = \check\frtr_{n,{I'}}(t^-)$ where $t^-$ is largest time less or equal to $t$ such that $ \check\frtr_{n,{I'}}(t^-)$ is on the map $\wh M_{n,{I'}}$. Then under an embedding of $M_{n,{I'}}$, $\wh\frtr_{n,{I'}}$ can be viewed as a continuous path moving along edges of $\wh M_{n,{I'}}$. The reason we work with the map  $\wh M_{n,{I'}}$ rather than the triangulation $T_{n,I'}$ is, for triangles adjacent to the right boundary of $T^n$, its geometry is non-local and not completely determined by $\cZ^n|_{I'}$. Note that the left boundary of $T_{n,I'}$ is simple and on $M_{n,I'}$, while we are removing components of $M_{n,I'}$ that are attached to $\wh M_{n,I'}$ at a single vertex to get $\wh M_{n,I'}$. Then provided that $\wh t_{I'}^{n,\min}<\wh t_{I'}^{n,\max}$, the clockwise boundary arc of $\wh M_{n,I'}$ from the vertex $\check\frtr_{n,I'}(\wh t^{n,\min})$ to $\check\frtr_{n,I'}(\wh t^{n,\max})$ agrees with the segment of the left boundary arc of $T_{n,I'}$ between the two vertices.
  
 % As discussed at the end of Section~\ref{subsubsec:Mullin-finite}, for $\Phi_\rmT(\cZ^n) = (T^n,\frtr^n)$ and $\Phi_\rmM(\cZ^n) = (M_{n,I},\frT)$, we can view $\frtr^n:I\to \frM(M_{n,I})$ as a path on the Riemann surface $\frM(M)$ via linear interpolation. Following the sewing procedure, the boundary of $T^n$ consists of a path in $\frT$ (resp.\ $\frT_*$) , which we call the left (resp.\ right)  boundary of $T^n$, and also the two starting and ending edges $\fre_t$. Note that $M_{n,I}$ is not necessarily a planar map of disk topology. Set $t_{I}^{n,-}$ (resp.\ $t_{I}^{n,+}$) to be the largest time before (resp.\ smallest time after) $t_{\rm mid}$ such that the triangular face $\frtr^n(t)$   simultaneously has vertices on the left and right boundaries of $T^n$, and let $\wh I^n = [t_{I}^{n,-},t_{I}^{n,+}]$.  Then provided $t_{I}^{n,-}<t_{I}^{n,+}$, the map  $M_{n,I}$ encoded by $\Phi_\rmM(\cZ^n|_{\wh I^n})$ has disk topology. We consider the path $\frtr_{n,I}:\wh I^n\to \frM(M_{n,I})$   generated from $\Phi_\rmT(\cZ^n|_{\wh I^n})$ as a path on the Riemann surface $\frM(M_{n,I})$, and continue to also write $\frtr_{n,I}$ for its image under a conformal map $\frM(M_{n,I})\to\bbC$.

The following is the analog of Definition~\ref{def:CPtoMOT}.
 \begin{definition}\label{def:uniftoMOT}
   Consider the setting above. We say that \emph{Riemann uniformization converges to mating of trees for $(\cZ^n,Z^n)_{n\in\bbN}$} if the following is
true. Let $a,b\geq 0$ be deterministic such that $a+b<1$, and $I':=I_{a,b}\subseteq I$ be an interval such that the left (resp.\ right) endpoint of $I'$ has distance $a|I|$ (resp.\ $b|I|$) from the left (resp.\ right) endpoint of $I$.  Consider the SLE-decorated LQG surface encoded by $Z^n|_{I'}$ and embed it as $(\bbD,h_{n,I'},\eta_{n,I'})$ such that $\eta_{n, I'}(t_{\rm mid})=0$ and $\eta_{n,I'}$ starts from 1. Let $t_{n,\rm mid}$ be as above. Consider planar map $\wh M_{n,I'}$ encoded by $\Phi_\rmM(\cZ^n|_{I'})$ and the trace $\wh\frtr_{n,I'}$ as above  and $\frM (\wh M_{n,I'})$ be the Riemann surface associated with $\wh M_{n,I'}$. Then with probability tending to 1 as $n\to\infty$, $\wh M_{n,I'}$ is not empty, and for the conformal map $\varphi_{n,I'}:\frM( \wh M_{n,I'})\to \bbD$ sending the vertex $\frv_{I'}^{\mathrm{mid},n}$ to 0 and the   vertex $\check\frtr_{n,I'}(\wh t^{n,\min})$ on $\wh M_{n,I'}$  to 1,  the following holds:
\begin{enumerate} 
\item $|t_{n,\rm mid}-t_{\rm mid}|\to 0$ as $n\to\infty$ in probability; \label{cond:def:uniftoMOT-(i)}
    \item the maximal diameter of   faces of $ \frM(\wh M_{n,I'})$ under the map $\varphi_{n,I'}$ tends to 0 in probability; \label{cond:def:uniftoMOT-(ii)}
    \item we have  $$\sup_{t\in I'}|\eta_{n, I'}(t)-\wh\frtr_{n,I'}(t)| \to 0 \text{ in probability as } n\to\infty;$$ \label{cond:def:uniftoMOT-(iii)}
    \item %let $t_{a,b}^-$ (resp.\ $t_{a,b}^+$) be the smallest (resp.\ largest) time in $I'$ such that $\check \frtr_{n,I'}(t)$ is on an edge on the left boundary of $T_{n,I'}$ which is also a boundary edge of $\wh M_{n,I'}$. 
    we have $|\wh t_{I'}^{n,\max}-\wh t_{I'}^{n,\min}|\to |I'|$ as $n\to\infty$ in probability; %for each deterministic $\lambda\in(0,1)$, with probability tending to 1, the following is true. For any $t$ such that $|t-t_{\rm mid}|<(1-\lambda)|I'|/2$, if $\check \frtr_{n,I'}(t)$ is on an edge on the left boundary of $T_{n,I'}$, then this edge is also a boundary edge of $\wh M_{n,I'}$; 
    \label{cond:def:uniftoMOT-(v)}
    \item let $\nu_{n,I',\rmL}$ be a measure on $\partial\bbD$ that assigns $\nu(n)$ to the image of each edge   of $\frM (\wh M_{n,I'})$ which is also on the left boundary of $T_{n,I'}$ where $\nu(n)>0$ is some constant. Then the Prokhorov distance between $\nu_{n,I',\rmL}$ and  $\nu_{h_{n, I'},\rmL}$, which is the restriction of the LQG boundary length measure $\nu_{h_{n, I'},\rmL}$ to the left boundary of $\eta_{n, I'}$,  tends to 0 in probability as $n\to\infty$. \label{cond:def:uniftoMOT-(iv)}
\end{enumerate}
\end{definition}

Following Definition~\ref{def:CPtoMOT}, in the above setting, if    $\frv_{I'}^{\mathrm{mid},n}$ is an interior vertex of $M_{n,{I'}}$ and all of the embeddings and objects in Definition~\ref{def:uniftoMOT} are well-defined,  define $\wt G^n(\cZ^n,a,b) = (t_{n,\rm mid},|\wh t_{I'}^{n,\max}-\wh t_{I'}^{n,\min}|,\delta_{n,I'},\wh\frtr_{n,I'},\nu_{n,I',\rmL})$, where $\delta_{n,I'}$ is the maximal diameter of embedded faces under $\varphi_{n,I'}$ and $\nu_{n,I'}$ is chosen in a measurable way. If this is not true, then we set $\wt G^n(\cZ^n,a,b) = (-1,-1,0,0,0)$. Let $\wt G^\infty(Z^n,a,b) = (t_{\rm mid},|I'|,0,\eta_{n,I'},\nu_{h_{n, I'},\rmL})$. Recall the metric $\mathfrak{d}$ on   $\bbR^3\times C(I,\ol\bbD)\times \mathrm{Measure}(\partial\bbD)$ using the supremum norm over continuous functions and Prokhorov distances on measures. We further define   $\mathrm{err}(\cZ^n,Z^n,a,b)= \mathfrak{d}(\wt G^n(\cZ^n,a,b),\wt G^{\infty}(Z^n, a,b))$ where the traces $\wh{\mathfrak{tr}}_{n,I'},\eta_{n,I'}$  are set to be constant on $I\backslash I'$.

\begin{proposition}\label{prop:cZ^n-unif-E}
    Let $\cZ^n = (\cL^n,\cR^n)$ be defined as in Lemma~\ref{lem:def-cZ^n}. %Let $a,b\geq 0$ be deterministic such that $a+b<1$. 
   % Let $I\subseteq I(0)$ such that the left (resp.\ right) endpoint of $I$ has distance $a$ (resp.\ $b$) from the left (resp.\ right) endpoint of $I(0)$. 
    Under  $\bbP_\rmE$, $\|(2^{-n/2}\cZ^n-Z)|_{I(0)}\|_\infty$ converges to 0 in probability, and Riemann uniformization converges to mating of trees for  $(\cZ^n|_{I(0)},Z|_{I(0)})$ with $\nu(n)=2^{-n/2}$ in Condition~\ref{cond:def:uniftoMOT-(iv)}. Furthermore, for each compact set $K\subseteq \{(a,b):a,b\geq0,a+b<1\}$ and $\e>0$, we have 
    \begin{equation}\label{eq:q-q-prop:cZ^n-unif-E}
        \sup_{(a,b)\in K}\bbP_\rmE\big[\mathrm{err}(\cZ^n|_{I(0)},Z|_{I(0)},a,b)\geq\e \big] \to 0 \text{ as } n\to\infty.
    \end{equation}
\end{proposition}

\begin{proof}
    We proceed as in the proof of Proposition~\ref{prop:cZ^n-pack-E}. Fix   $(a_n,b_n)\in K$ such that $(a_n,b_n)\to (a,b)$. Let $I' = I_{a,b}$ and $I_n' = I_{a_n,b_n}$ be the interval defined in Definition~\ref{def:uniftoMOT} for $I = I(0)$. Below we verify the convergence in Definition~\ref{def:uniftoMOT} with $I'_n$ in place of $I'$.  Let $t_{\rm mid}$ be the midpoint of $I'$,   $(T_{n,I_n'},\mathfrak{tr}_{n,I_n'}) = \Phi_\rmT(\cZ^n|_{I_n' })$, and $\Phi_\rmM(\cZ^n) = (M^n,e_0^n,\frT^n)$. Let $M_{n,I_n'}$, $t_{n,\rm mid}$, $\frv_{I_n'}^{\mathrm{mid},n}$ and $\wh M_{n,I_n'}$ be defined as above Definition~\ref{def:uniftoMOT} for  $\cZ^n|_{I_n'}$. %We write $\cJ_\rmL^{n,-}$ (resp.\ $\cJ_\rmR^{n,-}$) for the set of times when $\cL^n$ (resp.\ $\cR^n$) makes a downward jump.  
    For $t\in \cJ_\rmL^{n,-}$, we also view $\frv_t^n$ as a point for the corresponding vertex on $\frM(M^n)$. For $t\in \cJ_\rmR^{n,-}$, we write $\mathfrak{f}_t^n$ for the face on $\frM(M^n)$ corresponding to the time $t$, and write $\frv_t^n$ for the center of the face  $\mathfrak{f}_t^n$  on $\frM(M^n)$. By Lemma~\ref{lem:def-cZ^n} and Theorem~\ref{thm:conv-0-cone-unif-Hcell} applied to $r = \diam(\eta^n([-2^{n+1},2^{n+1}])) = \fR_{2^{-n}}^{-1}\diam(\eta([-2,2]))$, as we scale by  $\fR_{2^{-n}}$, there is a conformal map  $\varphi^n:\frM(M^n)\to \bbC$ such that   in probability as $n\to\infty$,
    \begin{align}
    & \lim_{n\to\infty}  \sup\{|\eta(t) - \varphi^n(\frv_t^n) |: t\in \cJ_\rmL^{n,-},t\in [-2,2]\} = 0;  \label{eq:E-unif-0-0}  \\
 &  \lim_{n\to\infty} \sup\big\{ \diam(\varphi^n(\mathfrak{f}_t^n) ):  \mathfrak{f}_t^n \text{ is equal  or adjacent to a face  } \mathfrak{f}_s^n \text{ with } s\in [-2,2]\cap \cJ_\rmR^{n,-}\big\} = 0;  \label{eq:E-unif-0-1}  \\ & 
   \lim_{n\to\infty}  \sup\{ |\eta(t)-{\mathfrak{tr}^n}(t)|: t\in [-2,2]\} = 0.  \label{eq:E-unif-0-2}
\end{align}
where the convergence~\eqref{eq:E-unif-0-2} is uniform over all choices of the path ${\mathfrak{tr}^n}$ on faces of $M^n$.

Using the scaling in Lemma~\ref{lem:def-cZ^n}, Proposition~\ref{prop:Tutte5.8} and  \cite[Lemma 5.36]{Tutte},  $t_{n,\rm mid}\to t_{\rm mid}$ in probability. By the independence between the dyadic system $\fD$ and $\eta$, almost surely, $\eta(t_{\rm mid})$ is in the interior of $\eta(I')$. By the second part of Theorem~\ref{thm:conv-0-cone-unif-Hcell} and~\eqref{eq:E-unif-0-1}, the curve distance~\eqref{eq:curve-metric-0} between $\partial T_{n,I_n'}$ under the embedding $\varphi^n$ and $\partial\eta(I_n')$ tends to 0 in probability as $n\to\infty$. Using $(a_n,b_n)\to (a,b)$,   the two components of $(\eta(I_n')\backslash\eta(I'))\cup(\eta(I')\backslash\eta(I'_n))$ stays in a small neighborhood of starting and ending points of $\eta(I')$ in probability, respectively. In particular, the curve distance~\eqref{eq:curve-metric-0} between $\partial\eta(I_n')$ and $\partial\eta(I')$ converges to 0 in probability.  Therefore~\eqref{eq:E-unif-0-0} and~\eqref{eq:E-unif-0-1}  further imply that with probability tending to 1 as $n\to\infty$, the vertex for $t_{n,\rm mid}$ in $M^n$ is on $T_{n,I_n'}$ and has at least graph distance 3 from the boundary of $T_{n,I_n'}$ and thus an interior vertex of the map $M_{n,I_n'}$. This verifies Condition~\ref{cond:def:uniftoMOT-(i)} of Definition~\ref{def:uniftoMOT}.

For each   edge $e$ on the right boundary of $T_{n,I_n'}$, we associate it with the vertex $v_{e,\rmL}$ on $M_{n,I_n'}$ which is on the same triangular face of $T_{n,I_n'}$ with $e$. Then $v_{e,\rmL}$ is also on the boundary of $M_{n,I_n'}$.  For two clockwise consecutive edges $e_1,e_2$ on the right boundary of $T_{n,I_n'}$ sharing a common vertex $v$, the clockwise boundary arc of $\partial M_{n,I_n'}$ between $v_{e_1,\rmL}$ and $v_{e_2,\rmL}$ is contained within the single face on $M^n$ whose vertex on $T^n$ is $v$. Thus~\eqref{eq:E-unif-0-1} and the second part of Theorem~\ref{thm:conv-0-cone-unif-Hcell} imply that when viewed as a loop that may traverse some edges for twice, $\partial \varphi^n( M_{n,I_n'})$ converges to $\partial\eta(I_n')$ and therefore $\partial\eta(I')$ in probability as $n\to\infty$ under the curve metric~\eqref{eq:curve-metric-0}. Furthermore, (the closure of) each connected component of $\partial \varphi^n(M_{n,I_n'})\backslash \partial \varphi^n(\wh M_{n,I_n'})$, when viewed as a closed loop, has winding number 0 around $\varphi^n(\frv_{I'_n}^{\mathrm{mid},n})$, and $\varphi^n(\frv_{I_n'}^{\mathrm{mid},n})\to \eta(t_{\rm mid})$ in probability as $n\to\infty$. As a consequence, by the curve convergence of  $\partial \varphi^n(M_{n,I_n'})$ to $\partial\eta(I')$, the maximal diameter of the connected component $\partial \varphi^n(M_{n,I_n'})\backslash \partial \varphi^n(\wh M_{n,I_n'})$ tends to 0 in probability.  This implies $\partial \varphi^n(\wh M_{n,I_n'})$ converges to $\partial\eta(I')$ under the metric~\eqref{eq:curve-metric-0} and $\sup_{t\in I_n'}|\frtr_{n,I'_n}(t)-\wh\frtr_{n,I_n'}(t)|\to 0$ in probability. 
Moreover, the vanishing maximal diameter of components of $\partial \varphi^n(M_{n,I_n'})\backslash \partial \varphi^n(\wh M_{n,I_n'})$ and the embedding convergence~\eqref{eq:E-unif-0-2} imply that for $I_n' = [t_{I_n'}^-,t_{I_n'}^+]$ and $I' = [t_{I'}^-,t_{I'}^+]$, we must have $|t_{I_n'}^--\wh t_{I_n'}^{n,\min}|\to 0$ and $|t_{I_n'}^+-\wh t_{I_n'}^{n,\max}|\to 0$ and thus $|t_{I'}^--\wh t_{I_n'}^{n,\min}|\to 0$ and $|t_{I'}^+-\wh t_{I_n'}^{n,\max}|\to 0$ in probability as $n\to\infty$.
%Furthermore, note that the boundary of $M_{n,I_n'}$ consists of the left boundary of $T_{n,I_n'}$, which is a simple path, and a possibly non-simple path on $M^n$. Let $I' = [t_{I'}^-,t_{I'}^+]$ and $I_n' = [t_{I_n'}^-,t_{I_n'}^+]$. Then the  left boundary of $\varphi^n(T_{n,I_n'})$ that is not on  $\varphi(\wh M_{n,I_n'})$ consist of a single initial segment and a final segment, and by  the vanishing diameter of components of $\partial \varphi^n(M_{n,I_n'})\backslash \partial \varphi^n(\wh M_{n,I_n'})$,  the two segments stay in a small neighborhood of $\eta(t_{I_n'}^-)$ and $\eta(t_{I_n'}^+)$ in probability as $n\to\infty$, respectively. Fix $\lambda\in(0,1)$. Since the dyadic system $\fD$ is independent from $(h,\eta)$, for each $\lambda$, $\eta(t^-_{I'})$ and $\eta(t^+_{I'})$   almost surely have positive distance from $\eta(\check{I}_\lambda')$, where $\check{I}_\lambda'$ is the interval concentric to $I'$ of length $(1-\lambda)|I'|$. Then the distance between  $\eta(t_{I_n'}^-)$ or $\eta(t_{I_n'}^+)$ and $\eta(\check{I}_\lambda')$ are bounded away from 0 in probability as $n\to\infty$, and if $t\in \check{I}_\lambda'$ and $\check{\frtr}_{n,I_n'}(t)$ is on an edge on the left boundary of $T_{n,I_n'}$, then this edge is also a boundary edge of $\wh M_{n,I_n'}$.
This gives Condition~\ref{cond:def:uniftoMOT-(v)}.  Let $\tau^n(z) = z- \varphi^n(\frv_{I_n'}^{\mathrm{mid},n})+\eta(t_{\rm mid})$. %$\varphi^n(\frv_I^{\mathrm{mid},n})$ to $\eta(t_{\rm mid})$, 
By Proposition~\ref{prop:boundary-perturb}, up to a rotation, the conformal map $\wh\varphi_{n,I_n'}:  \big(\tau^n\circ \varphi^n(\wh M_{n,I_n'}) \big)\to\bbD$ sending $\eta(t_{\rm mid})$ to 0 converges uniformly to the conformal map $\wh\varphi: \eta(I' )\to \bbD$ sending $\eta(t_{\rm mid})$ to 0 in probability. Then using $\wh t_{I_n'}^{n,\min}\to t_{I_n'}^-$, the distance between $\wh\varphi_{n,I_n'}(\check\frtr_{n,I_n'}(\wh t_{I_n'}^{n,\min}))$ and 1 tends to 0 in probability as $n\to\infty$, and the convergence of the conformal map  $\wh\varphi_{n,I_n'}$ remains true if we further rotate such that the vertex $\check\frtr_{n,I_n'}(\wh t_{I_n'}^{n,\min})$ is embedded at 1. Conditions~\ref{cond:def:uniftoMOT-(ii)}-\ref{cond:def:uniftoMOT-(iii)} in Definition~\ref{def:uniftoMOT} then follow by~\eqref{eq:E-unif-0-1} and~\eqref{eq:E-unif-0-2}   where we  take $\varphi_{n,I_n'} =  \wh\varphi_{n,I_n'}\circ \tau^n\circ\varphi^n$. Condition~\ref{cond:def:uniftoMOT-(iv)} follows from the second part of Theorem~\ref{thm:conv-0-cone-unif-Hcell}, where by Condition~\ref{cond:def:uniftoMOT-(v)}, the number of edges  that is on the left boundary of $T_{n,I_n'}$ but not on the boundary of $\wh M_{n,I_n'}$ divided by $2^{n/2}$ tends to 0 in probability as $n\to\infty$. This gives~\eqref{eq:q-q-prop:cZ^n-unif-E}, which further implies the desired convergence of Riemann uniformization to mating of trees by setting $(a_n,b_n)= (a,b)$.

%Now let $\Phi_\rmT(\cZ^n|_{I(0)}) = ( T^{n,0}, \frtr^{n,0})$, and embed $T^{n,0}$ on $\frM(M^n)$.  By~\eqref{eq:E-unif-0-2} and the second claim of Theorem~\ref{thm:conv-0-cone-unif-Hcell}, $\varphi^n(\partial T^{n,0})$ converges to $\partial\eta(I(0))$ in probability as $n\to\infty$ under the curve metric~\eqref{eq:curve-metric-0}. Since $\partial\eta(I(0))$ is a Jordan curve, these further imply that the times $(t_I^{n,-},t_I^{n,+})$ for $I=I(0)$ converge in probability to the left and right endpoints of $I(0)$, respectively. This verifies Condition~\ref{cond:def:uniftoMOT-(i)} of Definition~\ref{def:uniftoMOT}. Further combined with~\eqref{eq:E-unif-0-1}, we deduce that if we view $\frM(M_I^n)$ as a subsurface of $\frM(M^n)$, then $\varphi^n(\partial \frM(M_I^n))$ converges to $\partial\eta(I(0))$ in probability under the metric~\eqref{eq:curve-metric-0}. The rest of the conditions in Definition~\ref{def:uniftoMOT} can be verified by Theorem~\ref{thm:conv-0-cone-unif-Hcell} together with Proposition~\ref{prop:boundary-perturb}, where Proposition~\ref{prop:boundary-perturb} is  applied to $D^n = \varphi^n( \frM(M_I^n))$, $D = \eta(I(0))$ and $\wt D^n=\wt D=\bbD$.
\end{proof}

Using identical proof as in Proposition~\ref{prop:cZ^n-pack-Q} and Lemma~\ref{lem:cZ^n-pack-Q-n}, we have the following.
\begin{proposition}\label{prop:cZ^n-unif-Q} 
    The conclusions in    Lemma~\ref{lem:cZ^n-pack-Q-n} continue to hold if one replaces ``circle packing" with ``Riemann uniformization" as in Definition~\ref{def:uniftoMOT} in claim  $(\ref{cond:lem:cZ^n-pack-Q-n-(iv)})$ there, and replace the error~\eqref{eq:conv-dist-check} with the Riemann uniformization version as in~\eqref{eq:q-q-prop:cZ^n-unif-E}.
\end{proposition}

\subsection{The disk topology case}\label{subsec:unif-disk}

The goal of this subsection is to prove Theorem~\ref{thm:unif-disk}, where the proof is parallel to that of Theorem~\ref{thm:pack-disk}.

\begin{lemma}\label{lem:disk-unif-pf-0}
      There is a coupling of $(Z,\wt\cZ^1,\wt\cZ^2,...)$ such that
    \begin{enumerate}[(i)]
    \item claims $(\ref{cond:lem:disk-pf-0-(i)}),(\ref{cond:lem:disk-pf-0-(ii)}),(\ref{cond:lem:disk-pf-0-(v)})$ and $(\ref{cond:lem:disk-pf-0-(vi)})$ in Lemma~\ref{lem:disk-pf-0} hold;\label{cond:lem:disk-unif-pf-0-(i)}
        \item  {the total variation distance between the law of  the map $\wt M^{\disk,n}$ and the trace $\wt\frtr^{\disk,n}$ encoded by $\wt\cZ^n$ via $\Phi_\rmM^\disk$ and $\Phi_\rmT^\disk$  and the  law of  $(M^{\disk,n},\frtr^{\disk,n})$ in Theorem~\ref{thm:unif-disk} for $\ell=\ceil{\sqrt n}$ tends to 0 as $n\to\infty$ up to a reparameterization of the trace $\wt\frtr^{\disk,n}$};  \label{cond:lem:disk-unif-pf-0-(ii)}
        \item  {for each $\e\in(0,1/4)$,} Riemann uniformization converges to mating of trees for  {$(\wt\cZ^n|_{[A\e,A(1-\e)]},Z|_{[A\e,A(1-\e)]})_{n\in\bbN}$}  with $\nu(n)=n^{-1/2}$ in Condition~\ref{cond:def:uniftoMOT-(iv)};   \label{cond:lem:disk-unif-pf-0-(iii)}
        \item let $B^n_\rmL = \{t\in [0,A]: \wt\cL_t^n-\lim_{u\uparrow t}\wt\cL_u^n \neq 0 \}$ and  $B^n_\rmR = \{t\in [0,A]: \wt\cR_t^n-\lim_{u\uparrow t}\wt\cR_u^n \neq 0 \}$ denote the set of times at which $\wt\cL^n$ or $\wt\cR^n$  jumps, respectively. Conditioned on $A$ and $\#B_\rmL^n$, the elements of $B_\rmL^n$ have the law of $\#B_\rmL^n$ uniform and independent times in $[0,A]$, and the same holds for $B_\rmR^n$ in place of $B_\rmL^n$.   \label{cond:lem:disk-unif-pf-0-(iv)}
    \end{enumerate}
\end{lemma}

\begin{proof}
    The proof is identical to that of Lemma~\ref{lem:disk-pf-0}.
\end{proof}

\begin{lemma}\label{lem:disk-unif-pf-1}
    Consider the setting of Lemma~\ref{lem:disk-unif-pf-0}, and let $\frv^{\mathrm{mid},n}$ be the vertex on $\wt M^{\disk,n}$ for the first downward jump time after $A/2$ of the walk $\wt\cZ^n$ in the left coordinate. Let $\frM(\wt M^{\disk,n})$ be the Riemann surface associated with $\wt M^{\disk,n}$. Let $\wt\varphi^{\disk,n}:\frM(\wt M^{\disk,n})\to\bbD$ be  a conformal map that sends $\frv^{\mathrm{mid},n}$ to 0.  Then we have the following:
    \begin{enumerate}[(i)]
        \item The maximal diameter of faces under the map $\wt\varphi^{\disk,n}$ tends to 0 in probability as $n\to\infty$;
        \item Let $\wt\frtr^{\disk,n}:[0,A]\to\ol\bbD$ be the path on faces of $\wt T^{\disk,n}$ under the embedding $\wt\varphi^{\disk,n}(\frM(\wt M^{\disk,n}))$. Then up to a rotation,
        \eqb\label{eq:lem:disk-unif-pf-2}\sup_{t\in [0,A]} |\eta^\disk(t)-\wt\frtr^{\disk,n}(t)|\to 0 \ \ \ \text{in probability as } n\to\infty. \eqe
        \item  Let $\wt\nu^{\disk,n}$ be the measure on $\partial\bbD$ that assigns mass $n^{-1/2}$ to each   edge  $\wt\varphi^{\disk,n}(e)$ where $e\in \frM(\wt M^{\disk,n})$ is on the boundary. The Prohokov distance between $\wt\nu^{\disk,n}$ and $\nu_{h^\disk}$ tends to 0 in probability as $n\to\infty$.
    \end{enumerate}
\end{lemma}

\begin{proof}
  Let $0<\e_3<\e_1<\e<\delta$ and $\delta'$ be the constants in the proof of Lemma~\ref{lem:disk-pf-2} such that~\eqref{eq:lem:disk-pf-1-1} and~\eqref{eq:lem:disk-pf-1-2} hold with probability tending to 1 as $n\to\infty$ sufficiently faster than $\delta\to 0$. Consider the map $\wh M_{n,I'}$ and the conformal map $\varphi_{n,I'}$ defined in Definition~\ref{def:uniftoMOT} for $I' = [A\e, A(1-\e)]$. Then following Condition~\ref{cond:def:uniftoMOT-(v)} of Definition~\ref{def:uniftoMOT} and~\eqref{eq:lem:disk-pf-1-1} and~\eqref{eq:lem:disk-pf-1-2}, with probability tending to 1 as $n\to\infty$, for each $a\in [3\delta',1-3\delta']$ and $t_a = \inf\{s>0:\wt\cL^n_s\leq an^{1/2}\}$, the corresponding edge for $t_a$ is both on the boundary of  $\wt M^{\disk,n}$ and $\wh M_{n,I'}$. Similar to the proof of Lemma~\ref{lem:disk-pf-2}, using the continuity of $\eta^\disk$ near its starting and ending points and Condition~\ref{cond:def:uniftoMOT-(iii)} in Definition~\ref{def:uniftoMOT}, by the convergence of  Riemann uniformization   to mating of trees for  {$(\wt\cZ^n|_{[A\e,A(1-\e)]},Z|_{[A\e,A(1-\e)]})_{n\in\bbN}$} as in part (\ref{cond:lem:disk-unif-pf-0-(iii)}) of Lemma~\ref{lem:disk-unif-pf-0},  as we send $\delta\to 0$ sufficiently slower than $n\to\infty$ and $\delta'\to 0$ sufficiently slower than $\delta\to 0$, under the conformal map $\varphi_{n,I'}:\frM(\wh M_{n,I'})\to\bbD$, there is a common boundary arc of $\wh M_{n,I'}$ and $\wt M^{\disk,n}$ whose image has length tending to $2\pi$ in probability. Therefore up to a rotation, Proposition~\ref{prop:unif-disk} is applicable to the map $\wt\varphi^{\disk,n}\circ \varphi_{n,I'}^{-1}$, which now converges uniformly to the identity function in probability. Moreover, the diameter of  $\bbD\backslash \wt\varphi^{\disk,n}\circ \varphi_{n,I'}^{-1}(\wh M_{n,I'})$ tends to 0 and since we are sending $\delta\to 0$ sufficiently slower than $n\to\infty$,  the Porkhorov distance between $\wt\nu^{\disk,n}$ and  $\nu_{n,I',\rmL}$  tends to 0, while  $\nu_{h^\disk}(\eta([0,A]\backslash I'))\to 0$.   This along with the convergence of Riemann uniformization converges to mating of trees for  {$(\wt\cZ^n|_{[A\e,A(1-\e)]},Z|_{[A\e,A(1-\e)]})_{n\in\bbN}$} as in part (\ref{cond:lem:disk-unif-pf-0-(iii)}) of Lemma~\ref{lem:disk-unif-pf-0} proves Lemma~\ref{lem:disk-unif-pf-1}.
    %Let $\e>0$. Consider the map $\wh M_{n,I}$ and the conformal map $\varphi_{n,I}$ defined in Definition~\ref{def:uniftoMOT} for $I = [A\e, A(1-\e)]$.   Then $\frM(\wh M_{n,I})\subseteq \frM(\wt M^{\disk,n})$. Similar as in the proof of Lemma~\ref{lem:disk-pf-2}, using the continuity of $\eta^\disk$ near its starting and ending points and Condition~\ref{cond:def:uniftoMOT-(iii)} in Definition~\ref{def:uniftoMOT}, we see that by sending $\e\to0$ sufficiently slower than $n\to\infty$, one can draw a crosscut separating $\frM(M_I^n)$ from $ \frM(\wt M^{\disk,n})\backslash \frM(M_I^n)$ such that the diameter of its image under $\varphi_{n,I}$ tends to 0 in probability. The lemma then further follows from Proposition~\ref{prop:unif-disk}.
\end{proof}

\begin{proof}[Proof of Theorem~\ref{thm:unif-disk}]
Consider the setting of Lemma~\ref{lem:disk-unif-pf-1}. By Lemma~\ref{lem:disk-pf-0}, we have $\#B^n/n\to A$ in probability. On the other hand, for each given $n,\delta$, in the unconditioned standard simple random walk of $n$ steps, the probability that the total number of $(-1,0)$ steps is less than $(1-\delta)n/4$ is at most $c_1 e^{-n^{c_2}}$ for some $c_1,c_2>0$, while following Lemma~\ref{lem:walk-excursion}, the probability of the event we are conditioning on in the definition of $\cA_\ell^\disk$ where $\ell=\lceil\sqrt n\rceil$ is of order $n^{-c_3}$ for some $c_3$. Therefore after the conditioning, we continue to have  $\frac{\# B^n_\rmL}{\# B^n_\rmR}\to 1$ in probability as $n\to\infty$.  %Following the computations in Lemma~\ref{lem:walk-excursion}, it is not hard to prove that $\frac{\# B^n_\rmL}{\# B^n_\rmR}\to 1$ in probability as $n\to\infty$. 
By an identical proof, Lemma~\ref{lem:disk-pf-3} also holds for the sets of times encoding vertices or faces of $\frM(\wt M^{\disk,n})$, i.e., we can replace the random times encoding the map $\wt M^{\disk,n}$ in Lemma~\ref{lem:disk-unif-pf-1} with equally spaced times. Therefore Theorem~\ref{thm:unif-disk} follows from Lemma~\ref{lem:disk-unif-pf-1} by a further reparameterization as in the proof of Theorem~\ref{thm:pack-disk}. 
\end{proof}

\subsection{The sphere topology case}\label{subsec:unif-sph}

The goal of this subsection is to prove Theorem~\ref{thm:unif-sph}, where the proof is parallel to that of Theorem~\ref{thm:pack-sph}.

\begin{lemma}\label{lem:sph-unif-pf-1}
    The conclusions in Lemma~\ref{lem:sph-pack-pf-1} hold if one replaces ``circle packing" with ``Riemann uniformization".
\end{lemma}
\begin{proof}
    The proof is identical to that of Lemma~\ref{lem:sph-pack-pf-1}.
\end{proof}

\begin{lemma}\label{lem:sph-unif-pf-2}
%Let $\frv^{0,n}$ be an arbitrary endpoint of the root edge of $\wt T^{\sph,n}$, $\frv^{\mathrm{mid},n} = \frv^n_{t^{\mathrm{mid},n} }$ and     $\frv^{1/4,n} = \frv^n_{t^{1/4,n} }$, here $t^{\mathrm{mid},n}  = \inf\{s\geq An/2:s\in \wt B^{n,-}\}$ and $t^{1/4,n}  = \inf\{s\geq An/4:s\in \wt B^{n,-}\}$.
Consider the setting of Lemma~\ref{lem:sph-unif-pf-1}, and let $\wt M^{\sph,n}$ be the planar map encoded by $\wt\cZ^n$ via the map $\Phi_\rmM^\sph$. Let $B^{n,-}_\rmL$ be the set of times when $\wt\cL^n$ makes a downward jump. Let $\frv^{0,n} = \frv^n_{t^{0,n} }$, $\frv^{\mathrm{mid},n} = \frv^n_{t^{\mathrm{mid},n} }$ and     $\frv^{1/4,n} = \frv^n_{t^{1/4,n} }$, where $t^{0,n} = \inf B^{n,-}_\rmL$, $t^{\mathrm{mid},n}  = \inf\{s\geq 1/2:s\in   B^{n,-}_\rmL\}$ and $t^{1/4,n}  = \inf\{s\geq 1/4:s\in   B^{n,-}_\rmL\}$,  also viewed  as points on $\frM(\wt M^{\sph,n})$.
Let $\wt\varphi^{\sph,n}:\frM(\wt M^{\sph,n})\to\wh\bbC$ be the conformal map sending $\frv^{0,n}$, $\frv^{\mathrm{mid},n}$ and $\frv^{1/4,n}$ to $\infty,0,1$, respectively, and we view $\wt\frtr^{\sph,n}:[0,1]\to \wh\bbC$ as the path on the embedded map $\wt M^{\sph,n}$ via the embedding $\wt\varphi^{\sph,n}(\frM(\wt M^{\sph,n}))$. Then with probability tending to 1 as $n\to\infty$, $\frM(\wt M^{\sph,n})$ and $\wt\frtr^{\sph,n}$ are well-defined, and we have the following:
 \begin{enumerate}[(i)]
        \item The maximal spherical diameter of faces in $\wt\varphi^{\sph,n}(\frM(\wt M^{\sph,n}))$ tends to 0 in probability as $n\to\infty$;
        \item  We have \eqb\label{eq:lem:sph-unif-pf-2}\sup_{t\in [0,1]} d_{\sph}(\wt\eta^\sph(t),\wt\frtr^{\sph,n}(t))\to 0 \ \ \ \text{in probability as } n\to\infty. \eqe
    \end{enumerate}
%if we send $\e\to 0$ sufficiently slower than $n\to\infty$, the spherical diameter of the union of the disks for vertices that are adjacent to faces in $\cF \wt T^{\sph,n}\backslash\{\wt\frtr^{\sph,n}(t):t\in [\e,A-\e]\}$   tends to 0 in probability.
\end{lemma}

\begin{proof}
By the last claim of Lemma~\ref{lem:sph-double-edge-1}, with probability tending to 1 as $n\to\infty$, $B^{n,-}_\rmL$ intersects all of $(0.05,0.1)$, $(0.3,0.33)$, $(0.55,0.6)$ and $(0.8,0.9)$, so the vertices  $\frv^{0,n},\frv^{\mathrm{mid},n},\frv^{1/4,n}$ are distinct and  $\frM(\wt M^{\sph,n})$ and $\wt\frtr^{\sph,n}$ are well-defined.   Let $\e\in (0,0.01)$, $I' = [\e,1-\e]$,  $\varphi_{n,I'}$ and $\wh M_{n,I'}$ be as in Definition~\ref{def:uniftoMOT}. We scale $\varphi_{n,I}$ to get $\wt\varphi_{n,I'}$ such that  $\wt \varphi_{n,I'}(\frv^{1/4,n})=1$. Similar to the proof of Lemma~\ref{lem:sph-pack-pf-4}, using the convergence of Riemann uniformization to mating of trees for $(\wt\cZ^{\sph,n}|_{[\e,1-\e]},Z|_{[\e,1-\e]})$ along with the end-point continuity of $\wt\eta^\sph$, as we send $\e\to0$ sufficiently slower than $n\to\infty$, the maximal spherical diameter of faces in $\wt\varphi_{n,I'}(\wh M_{n,I'})$ tends to 0 in probability and for every $R>0$, with probability tending to 1, $B(0;R)\subseteq \wt\varphi_{n,I}(\wh M_{n,I'})$. The claims then follows from the conditions in Definition~\ref{def:uniftoMOT} together with Proposition~\ref{prop:unif-sph}.
\end{proof}

\begin{proof}[Proof of Theorem~\ref{thm:unif-sph}]
    Consider the setting of Lemma~\ref{lem:sph-unif-pf-2}. Let $\wt B^n_\rmL$ and $\wt B^n_\rmR$ be the set of times when $\wt\cL^n$ and $\wt\cR^n$ jumps, respectively. Then following the description of the law of $\wt\cZ^n$, the analog of Lemma~\ref{lem:sph-pack-pf-5} holds for $\wt B^n_\rmL$ and $\wt B^n_\rmR$, and one can prove by direct computation from Lemma~\ref{lem:walk-excursion} that $\frac{\#\wt B^n_\rmL}{\#\wt B^n_\rmR}\to 1$ in probability as $n\to\infty$. Theorem~\ref{thm:unif-sph} then follows from Lemma~\ref{lem:sph-unif-pf-2} and the same argument in the proof of Theorem~\ref{thm:pack-sph} via reparameterization.
\end{proof}

\subsection{The whole-plane topology case}\label{subsec:unif-cone}

In this subsection we prove Theorem~\ref{thm:unif-cone}.

\begin{lemma}
   Let $(\cZ_k)_{k\in\bbZ}$ be a standard nearest-neighbor walk on $\bbZ^2$ and $\Phi_\rmM(\cZ) = (M,e_0,\frT)$. Then the Riemann surface $\frM(M)$ is a.s.\ parabolic.
\end{lemma}

\begin{proof}
    It is shown in~\cite{Che17} that $M$ is ergodic, recurrent and can be written as the Benjamini-Schramm limit of uniformly rooted planar maps where the root is chosen proportionally to its degree. Let $\rho$ be the root vertex defined in this way and let $\bbP$ be the law of $(M,\rho)$, which is reversible in the sense of~\cite[Section 2.3]{AHNR18unimodular}/ We write $\bbP_*$ for the law of $(M,\rho)$ weighted by $\deg(\rho)^{-1}/\bbE[\deg(\rho)^{-1}]$. Then $\bbE_*[\deg(\rho)] = 1/\bbE[\deg(\rho)^{-1}]<\infty$ and $(M,\rho)$ is unimodular under $\bbP_*$ as in~\cite[Proposition 2.3]{AHNR18unimodular}.
  %  Furthermore, the root vertex $\rho$ satisfies $\bbE[\deg(\rho)]<\infty$. 
    The same continues to hold if we remove all the faces of $M$ with degree 1 or 2 as in the definition of $\frM(M)$. The claim then follows from~\cite[Theorem 1.1, (5)$\to$(9)]{AHNR18unimodular}.
\end{proof}

\begin{lemma}\label{lem:cone-unif-pf-2}
    Let $(\wt \cZ^n_t)_{t\in\bbR}$ be a rate $n$ continuous-time nearest-neighbor walk on $\bbZ^2$ such that $\wt \cZ^n_0=0$. Then there is a coupling for $(Z,(\wt\cZ^n)_{n\in\bbN})$ such that Assertions \ref{result:thm:unif-cone-1}-\ref{result:thm:unif-cone-3}\ in Theorem~\ref{thm:unif-cone} holds for $\wt\cZ^n$ instead of $\cZ^n$. Furthermore, for each $R>0$, with probability tending to 1 as $n\to\infty$ then $K\to\infty$,  there does not exist any time $t\in \cJ^{n,-}_\rmL$ such that $|t|\geq K$ while the vertex for the time $t$ is embedded inside $B(0;R)$, and  there does not exist any time $t\in \cJ^{n,-}_\rmR$ such that $|t|\geq K$ while the face for the time $t$ intersects $B(0;R)$ under the embedding $\varphi^n$. 
\end{lemma}

\begin{proof}
    By Proposition~\ref{prop:cZ^n-unif-Q}, Riemann uniformization converges to mating of trees for $(\wt\cZ^n|_{[-1,1]},Z|_{[-1,1]})$, and we let $\varphi_{n,I'}$ and $\wh M_{n,I'}$ be as in Definition~\ref{def:uniftoMOT}   for  $(\wt\cZ^n|_{[-1,1]},Z|_{[-1,1]})$ and $I'=[-1,1]$. Let $\cJ^{n,-}_\rmL$ be the set of times when $\wt\cL^n$ %or $\wt\cR^n$ 
    makes a downward jump. Let $\varphi:\eta([-1,1])\to\bbD$ be the conformal map fixing 0 and $\varphi'(0)>0$.
    
    Let $j\in\bbN$, $\frv^{n,j,0}$ (resp.\ $\frv^{n,j,1}$) be the vertex for the time $\inf(\cJ^{jn,-}_\rmL\cap\bbR_+)$ (resp.\ $\inf(\cJ^{jn,-}_\rmL\cap [1/j,\infty))$)  viewed as points on $\frM(M^{jn})\cap\frM(\wh M_{jn,I'})$. Then the conclusions as in Definition~\ref{def:uniftoMOT} still hold if we assume that $\varphi_{jn,I}(\frv^{n,j,0})=0$. Let $R^{n,j}>0$ and $\theta^{n,j}\in\bbR$ such that $R^{n,j}e^{i\theta^{n,j}}\varphi_{jn,I}(\frv^{n,j,1})=1$. Let $\varphi_{jn}: \frM(M^{jn})\to\bbC$ be the conformal map sending $\frv^{n,j,0}$ to 0 and $\frv^{n,j,1}$ to 1. The lemma  then follows from a similar argument as in the proof of Lemma~\ref{lem:pf-pack-cone-2} as well as~\eqref{eq:CONE-4-new} where we send $j\to\infty$ sufficiently slower than $n\to\infty$ and apply the conclusions in Definition~\ref{def:uniftoMOT} and   Proposition~\ref{prop:unif-sph} to the conformal map $\varphi_n\circ(R^{n,j}e^{i\theta^{n,j}}\varphi_{jn,I})^{-1}$.
\end{proof}

\begin{proof}[Proof of Theorem~\ref{thm:unif-cone}]
    The proof is identical to that of Theorem~\ref{thm:pack-cone} where we reparameterize the walk $\wt\cZ^n$ in the setting of Lemma~\ref{lem:cone-unif-pf-2}.
\end{proof}

\bibliographystyle{alpha}
\bibliography{circle}

\end{document}